\documentclass[12pt]{amsart}
\usepackage{amsfonts}
\usepackage{fullpage,multirow, amsmath, amsthm,amsfonts,amssymb,stmaryrd, mathrsfs,hhline}
\usepackage{mathdots}
\usepackage{pgf,tikz}
\usepackage{tikz-cd}
\usepackage{mathrsfs}
\usetikzlibrary{arrows}
\usetikzlibrary[patterns]
\usetikzlibrary{matrix,decorations.pathreplacing,calc}
\usepackage{hyperref}

\usepackage[all]{xy}

\tikzset{ 1
table/.style={
  matrix of math nodes,
  row sep=-\pgflinewidth,
  column sep=-\pgflinewidth,
  nodes={rectangle,text width=3em,align=center},
  text depth=1.25ex,
  text height=2.5ex,
  nodes in empty cells,
  left delimiter=[,
  right delimiter={]},
  ampersand replacement=\&
}
}
\pgfkeys{tikz/mymatrixenv/.style={decoration=brace,every left delimiter/.style={xshift=3pt},every right delimiter/.style={xshift=-3pt}}}
\pgfkeys{tikz/mymatrix/.style={matrix of math nodes,left delimiter=[,right delimiter={]},inner sep=2pt,column sep=1em,row sep=0.5em,nodes={inner sep=0pt}}}
\pgfkeys{tikz/mymatrixbrace/.style={decorate,thick}}

\newtheorem{theorem}[subsection]{Theorem}

\newtheorem{lemma}[subsection]{Lemma}
\newtheorem{lemmanotation}[subsection]{Lemma-Notation}
\newtheorem{corollary}[subsection]{Corollary}
\newtheorem{conjecture}[subsection]{Conjecture}

\newtheorem{proposition}[subsection]{Proposition}
\newtheorem{definition-lemma}[subsection]{Definition-Lemma}
\newtheorem{definition-proposition}[subsection]{Definition-Proposition}
\theoremstyle{definition}

\newtheorem{definition}[subsection]{Definition}
\newtheorem{hypothesis}[subsection]{Hypothesis}

\newtheorem{example}[subsection]{Example}
\newtheorem{question}[subsection]{Question}
\newtheorem{remark}[subsection]{Remark}
\newtheorem{convention}[subsection]{Convention}
\newtheorem{notation}[subsection]{Notation}

\numberwithin{equation}{subsection}

\makeatletter
\newcommand{\fakephantomsection}{%
  \Hy@MakeCurrentHref{\@currenvir.\the\Hy@linkcounter}
  \Hy@raisedlink{\hyper@anchorstart{\@currentHref}\hyper@anchorend}%
}
\makeatother

\def\calC{\mathcal{C}}

\def\calO{\mathcal{O}}

\def\calW{\mathcal{W}}

\def\gothm{\mathfrak{m}}

\def\bfB{\mathbf{B}}
\def\AAA{\mathbb{A}}

\def\CC{\mathbb{C}}
\def\FF{\mathbb{F}}
\def\GG{\mathbb{G}}

\def\NN{\mathbb{N}}

\def\QQ{\mathbb{Q}}
\def\RR{\mathbb{R}}

\def\ZZ{\mathbb{Z}}

\def\bfe{\mathbf{e}}

\def\bfk{\mathbf{k}}

\def\bfM{\mathbf{M}}

\def\rmB{\mathrm{B}}

\def\rmG{\mathrm{G}}
\def\rmH{\mathrm{H}}
\def\rmI{\mathrm{I}}
\def\rmK{\mathrm{K}}

\def\rmM{\mathrm{M}}

\def\rmU{\mathrm{U}}

\def\rmT{\mathrm{T}}

\def\k{\mathbf{k}}
\def\rmS{\mathrm{S}}

\newcommand{\Proj}{\mathrm{Proj}}

\DeclareMathOperator{\Ext}{Ext}
\DeclareMathOperator{\Gal}{Gal}
\DeclareMathOperator{\GL}{GL}

\DeclareMathOperator{\Hom}{Hom}

\DeclareMathOperator{\rank}{rank}

\DeclareMathOperator{\NP}{NP}

\DeclareMathOperator{\Sym}{Sym}

\DeclareMathOperator{\Dig}{Dig}
\newcommand{\op}{\mathrm{op}}

\newcommand{\AL}{\mathrm{AL}}

\newcommand{\et}{\mathrm{\acute{e}t}}

\newcommand{\new}{\mathrm{new}}
\newcommand{\old}{\textrm{-}\mathrm{old}}
\newcommand{\Iw}{\mathrm{Iw}}
\renewcommand{\det}{\mathrm{det}}
\newcommand{\unr}{\mathrm{ur}}

\newcommand{\pr}{\mathrm{pr}}

\newcommand{\rig}{\mathrm{rig}}

\newcommand{\wt}{\mathrm{wt}}

\newcommand{\ur}{\mathrm{ur}}

\newcommand{\univ}{\mathrm{univ}}
\newcommand{\JH}{\mathrm{JH}}
\newcommand{\Mult}{\mathrm{Mult}}

\newcommand{\nS}{\mathrm{nS}}

\newcommand{\midd}{\mathrm{mid}}
\newcommand{\even}{\mathrm{even}}
\newcommand{\odd}{\mathrm{odd}}

\DeclareMathOperator{\Char}{Char}
\DeclareMathOperator{\Spa}{Spa}
\DeclareMathOperator{\Spc}{Spc}

\DeclareMathOperator{\Ind}{Ind}
\newcommand{\Matrix}[4]{{\big(\begin{smallmatrix}#1&#2\\#3& #4\end{smallmatrix} \big)}}
\newcommand{\MATRIX}[4]{{\begin{pmatrix}#1&#2\\#3& #4\end{pmatrix}}}

\begin{document}

\definecolor{zzttqq}{rgb}{0.6,0.2,0.}
\definecolor{uuuuuu}{rgb}{0.26666666666666666,0.26666666666666666,0.26666666666666666}
\definecolor{xdxdff}{rgb}{0.49019607843137253,0.49019607843137253,1.}
\definecolor{ududff}{rgb}{0.30196078431372547,0.30196078431372547,1.}
\definecolor{cqcqcq}{rgb}{0.7529411764705882,0.7529411764705882,0.7529411764705882}
 
\title{A local analogue of the ghost conjecture of Bergdall--Pollack}
\author{Ruochuan Liu}
\address{Ruochuan Liu, School of Mathematical Sciences, Peking University, 5 Yi He Yuan Road, Haidian District, Beijing, 100871, China. }
\email{liuruochuan@math.pku.edu.cn}

\author{Nha Xuan Truong}
\address{Nha Xuan Truong, Department of Mathematics, University of Hawaii at Manoa, 2565 McCarthy Mall, Honolulu, HI 96822, U.S.A..}
\email{nxtruong@hawaii.edu}

\author{Liang Xiao}
\address{Liang Xiao, School of Mathematical Sciences,
Peking University.
5 Yi He Yuan Road, Haidian District,
Beijing, 100871, China.}
\email{lxiao@bicmr.pku.edu.cn}

\author{Bin Zhao}
\address{Bin Zhao, School of Mathematical Sciences, Capital Normal University, Beijing, 100048, China.}
\email{bin.zhao@cnu.edu.cn}
\date{\today}

\begin{abstract}
We formulate a local analogue of the ghost conjecture of Bergdall and Pollack, which essentially relies purely on the representation theory of $\GL_2(\QQ_p)$.  We further study the combinatorial properties of the ghost series as well as its Newton polygon, in particular, giving a characterization of the vertices of the Newton polygon and proving an integrality result of the slopes.
In a forthcoming sequel, we will prove this local ghost conjecture under some mild hypothesis and give arithmetic applications.
\end{abstract} 

\maketitle
\setcounter{tocdepth}{1}
\tableofcontents

\section{Introduction}

Let $p$ be a prime number and let $N$ be a positive integer relatively prime to $p$. The central object of this paper is \emph{$U_p$-slopes}, that is, the $p$-adic valuations of the eigenvalues of the $U_p$-operator acting on the space of (overconvergent) modular forms of level $\Gamma_1(Np)$, or on more general space of (overconvergent) automorphic forms essentially of $\GL_2(\QQ_p)$-type. In this paper, the $p$-adic valuation is normalized so that $v_p(p) =1$.
\begin{question}
What the slopes of modular forms are?
\end{question}
This seemingly naive question has surprisingly definite answers in some cases.
Inspired by the numerical pioneering works of Gouv\^ea and Mazur \cite{gouvea-mazur, gouvea}, as well as many foundational works of Buzzard--Calegari and many others (to name a few: \cite{buzzard-calegari}, \cite{buzzard-slope}, \cite{clay}, and  \cite{loeffler-slope}), a significant breakthrough on this topic was made in the recent series of works by Bergdall and Pollack \cite{bergdall-pollack2, bergdall-pollack3}, in which they proposed a new
conjecture on slopes: (under some regular conditions), the slopes of modular forms should be given by the slopes of the Newton polygon of a combinatorially defined power series, called \emph{ghost series}. On the one hand, this conjecture unifies all previous works, and on the other hand its combinatorial structure provides theoretical insight into these numerical data of slopes. The purpose of this manuscript is to reproduce this idea in a purely local setup and investigate several important properties of the corresponding ghost series.

In what follows, we shall first introduce our local version of ghost conjecture and then discuss its relation to conjectures that are historically important.

To motivate our construction, we first fix an absolutely irreducible residual Galois representation $\bar r: \Gal_\QQ\to \GL_2(\FF)$ with $\FF$ a finite extension of $\FF_p$.
For an open compact subgroup $K \subseteq \GL_2(\AAA_f)$, write $Y(K)$ for the corresponding open modular curve over $\QQ$. For $\calO$ a complete discrete valuation ring with residue field $\FF$ and uniformizer $\varpi$, let
$$
\rmH_1^\et\big(Y(K)_{\overline \QQ}, \calO\big)_{\gothm_{\bar r}}^{\mathrm{cplx}=1}
$$
denote the subspace of the first  \'etale homology of the modular curve $Y(K)$, localized at the maximal Hecke ideal $\gothm_{\bar r}$ corresponding to $\bar r$, on which a fixed complex conjugation acts by $1$.
Fix a neat tame level $K^p \subseteq \GL_2(\AAA_f^p)$. 
The \emph{$\bar r$-localized completed homology} with tame level $K^p$ is defined to be
\begin{equation}
\label{E:completed homology}
\widetilde \rmH(K^p)_{\gothm_{\bar r}}: = \varprojlim_{m}\; \rmH_1^\et\big(Y\big(K^p(1+p^m\rmM_2(\ZZ_p))\big)_{\overline \QQ},\, \calO\big)_{\gothm_{\bar r}}^{\mathrm{cplx}=1}.
\end{equation}
This completed homology is a  finitely generated projective (right) $\calO\llbracket \GL_2(\ZZ_p)\rrbracket$-module whose $\GL_2(\ZZ_p)$-action extends to a $\GL_2(\QQ_p)$-action.
As discovered by Emerton \cite{emerton2}, one can recover the information of classical modular forms from $\widetilde \rmH(K^p)_{\gothm_{\bar r}}$ as follows: for each weight $k \geq 2$, let $\Sym^{k-2}\calO^{\oplus 2}$ denote the $(k-2)$th symmetric power of the standard representation of $\GL_2(\ZZ_p)$ (the right action of $\GL_2(\ZZ_p)$ is the transpose of the standard left action), then
\begin{equation}
\label{E:completed homology vs classical forms}
\Hom_{\calO\llbracket\GL_2(\ZZ_p)\rrbracket}\big(\widetilde \rmH(K^p)_{\gothm_{\bar r}}, \ \Sym^{k-2}\calO^{\oplus 2}\big) \cong \rmH^1_\et\big(Y(K^p\GL_2(\ZZ_p))_{\overline \QQ},\, \Sym^{k-2}( R^1\pi_*\underline \calO)\big)_{\gothm_{\bar r}}^{\mathrm{cplx}=1},
\end{equation}
where $R^1\pi_*\underline \calO$ is the usual rank two \'etale local system on the open modular curve. It is well-known that the right hand side of \eqref{E:completed homology vs classical forms} is isomorphic to the $\gothm_{\bar r}$-localized space of modular forms as a Hecke module (e.g. \cite[Proposition~4.4.2]{emerton2}). Moreover, a similar isomorphism holds for $\GL_2(\ZZ_p)$ replaced by the Iwahori subgroup $\Iw_p =\Matrix{\ZZ_p^\times}{\ZZ_p}{p\ZZ_p}{\ZZ_p^\times}$, and the $U_p$-action on the space of classical forms can be realized easily on the left hand side of \eqref{E:completed homology vs classical forms}.
So in a sense, the completed homology $\widetilde \rmH(K^p)_{\gothm_{\bar r}}$ captures all information we need to study $p$-adic properties of classical modular forms. This motivates the following.

\begin{definition}
Consider a residual representation $\bar\rho: \rmI_{\QQ_p} \to \GL_2(\FF)$ of the \emph{inertia subgroup at $p$} that is \emph{reducible}, nonsplit, and generic. More precisely,  
\begin{equation}
\label{E:bar rho}
\bar \rho \simeq \MATRIX{\omega_1^{a+b+1}}{*\neq 0}{0}{\omega_1^b} \qquad \textrm{for } 1 \leq a \leq p-4 \textrm{ and } 0 \leq b \leq p-2,
\end{equation}
where $\omega_1: \rmI_{\QQ_p} \to \Gal(\QQ_p(\mu_{p})/\QQ_p)\cong \FF_p^\times$ is the first fundamental character.

A \emph{primitive $\calO\llbracket K_p\rrbracket$-projective augmented module of type $\bar\rho$} is a finitely generated projective (right) $\calO\llbracket \GL_2(\ZZ_p)\rrbracket$-module $\widetilde \rmH$ with a (right) $\GL_2(\QQ_p)\big/{\Matrix p{}{}p}^\ZZ$-action extending the $\GL_2(\ZZ_p)$-action from the module structure, such that

\begin{itemize}
\item $\overline \rmH:  =\widetilde \rmH / (\varpi, \rmI_{1+p\rmM_2(\ZZ_p)})$ is isomorphic to the projective envelope of $\Sym^a \FF^{\oplus 2} \otimes \det^b$ as an $\FF[\GL_2(\FF_p)]$-module, where $\rmI_{1+p\rmM_2(\ZZ_p)}$ is the augmentation ideal of the Iwasawa algebra $\calO\llbracket 1+p\rmM_2(\ZZ_p)\rrbracket$.
\end{itemize}
\end{definition}

By Serre weight conjecture, in the setup of modular forms, if $\bar r|_{\rmI_{\QQ_p}} \cong \bar \rho$, $\widetilde \rmH(K^p)_{\gothm_{\bar r}}$ is a primitive $\calO\llbracket K_p\rrbracket$-projective augmented module of type $\bar\rho$, if the following mod-$p$-multiplicity-one condition is met:
\begin{equation}
\label{E:multiplicity one}\rank_\calO  \Big( 
\rmH_1^\et\big(Y(K^p \Iw_{p,1})_{\overline \QQ}, \calO\big)^{\mathrm{cplx}=1}_{\gothm_{\bar r}} \Big)= 2;
\end{equation}
here $\Iw_{p,1} = \Matrix{1+p\ZZ_p}{\ZZ_p}{p\ZZ_p}{1+p\ZZ_p}$ and the rank $2$ module is the direct sum of two rank-one eigenspaces for the diamond operators in $(\FF_p^\times)^2$ with eigenvalues $\omega^a \times 1$ and $1 \times \omega^a$, respectively, where $\omega: \FF_p^\times \to \calO^\times$ is the Teichm\"uller lift.

{\it One of the key features of our new setup is that: it is purely local, and does not depend on the global automorphic input.}
\begin{remark}
\label{R:origins of arithmetic modules}
There are several origins of primitive $\calO\llbracket K_p\rrbracket$-projective augmented modules.
\begin{enumerate}
\item For any $U(2)$-Shimura varieties or quaternionic Shimura varieties (where $p$ splits completely in the corresponding fields), taking its completed homology (localized at an appropriate $\bar r$) similarly to \eqref{E:completed homology}, we obtain a primitive $\calO\llbracket K_p\rrbracket$-projective augmented module provided a mod-$p$-multiplicity-one assumption similar to \eqref{E:multiplicity one} hold.

\item 
In a global setup as in (1) with the mod-$p$-multiplicity-one assumption, assuming some further Taylor--Wiles conditions, one may construct a patched completed homology as in \cite[\S\,2]{six author}, which is a projective module over $\calO\llbracket z_1, \dots, z_g\rrbracket \llbracket \GL_2(\ZZ_p)\rrbracket$ and whose $\GL_2(\ZZ_p)$-action extends to an $\calO\llbracket z_1, \dots, z_g\rrbracket$-linear action of $\GL_2(\QQ_p)/p^\ZZ$ (assuming a certain central character condition). Evaluating the patching variables $z_1, \dots, z_g$ arbitrarily will give rise to a primitive $\calO\llbracket K_p\rrbracket$-projective augmented module of type $\bar\rho$, which does not have any automorphic origin. 

\item Unfortunately, we do not know, for a fixed $\bar \rho$ in \eqref{E:bar rho}, whether one can find a global $\bar r$ such that the restriction of $\bar r$ to some $p$-adic place isomorphic to $\bar \rho$ and that $\bar r$ satisfies the mod-$p$-multiplicity-one assumption similar to \eqref{E:multiplicity one}. Fortunately, the Pa\v sku\=nas functor provides yet another example of primitive $\calO\llbracket K_p\rrbracket$-projective augmented module.

Let $\tilde \rho = \Matrix{\omega_1^{a+b+1}\unr(\alpha_1)}{*\neq 0}{0}{\omega_1^b\unr(\alpha_2)}: \Gal_{\QQ_p} \to \GL_2(\FF)$ be a residual Galois representation whose restriction to $\rmI_{\QQ_p}$ is $\bar\rho$. Here $\omega_1$ can be extended to a character of $\Gal_{\QQ_p}$ in a natural way, and $\unr(\alpha_i)$ for $i=1,2$ is the unramified representation of $\Gal_{\QQ_p}$ sending the geometric Frobenius to $\alpha_i \in \FF^\times$. 
Let $\pi(\tilde \rho)$ be the smooth admissible representation associated to $\pi(\tilde \rho)$ a la $p$-adic local Langlands correspondence for $\GL_2(\QQ_p)$.
Pa\v sk\= unas \cite{paskunas-functor} constructed an injective envelope ${\bf J}_{\tilde \rho}$ of $\pi(\tilde \rho)$  in the category of locally admissible torsion $\calO$-representations of $\GL_2(\QQ_p)/p^\ZZ$. 
Its Pontryagin dual ${\bf P}_{\tilde \rho}$ can be viewed as a projective module over $\calO\llbracket x\rrbracket\llbracket \GL_2(\ZZ_p)\rrbracket$, where $x$ is a certain element in the universal deformation ring of $\tilde \rho$ (see \cite[Theorem~5.2]{paskunas-duke}). Then any evaluation of $x$ gives a primitive $\calO\llbracket K_p\rrbracket$-projective augmented module of type $\bar \rho$.

This example is extremely important as it allows us to relate the local ghost conjecture with questions on triangulline deformation space of Breuil--Hellmann-Schraen \cite{BHS}, using the known Breuil--M\'ezard conjecture for $\GL_2(\QQ_p)$.
\end{enumerate}
 
\end{remark}
One can immediately reduce the discussion to the case when $b=0$, which we assume from now on.

\begin{definition}Given a ``relevant" character $\varepsilon: (\FF_p^\times)^2 \to \calO^\times$ (which is used to parametrize weight disks; see Notation~\ref{N:relevant varepsilon}\,(1)), 
define the space of 
``abstract $p$-adic forms"
$$ \rmS_{p\textrm{-adic}}^{(\varepsilon)}: = \Hom_{\calO[\Iw_p]} (\widetilde \rmH,\ \calC^0(\ZZ_p, \calO\llbracket w\rrbracket)^{(\varepsilon)})\quad \textrm{for }\varepsilon \textrm{ a character of }\Delta^2,$$
where $\calC^0(\ZZ_p; \calO\llbracket w\rrbracket)^{(\varepsilon)}$ is the space of continuous functions on $\ZZ_p$, and the space is equipped with an action by $\Iw_p$ twisted by $\varepsilon$ in some way (see \S\,\ref{S:abstract p-adic forms}). 
This space carries an $\calO\llbracket w\rrbracket$-linear action of the $U_p$-operator, whose characteristic power series is denoted by  $C_{\widetilde \rmH}^{(\varepsilon)}(w,t) \in \calO\llbracket w, t\rrbracket$ (see \S\,\ref{S:char power series of Up} for the precise definitions).

We can define \emph{abstract classical forms} and \emph{abstract overconvergent forms} similarly; see \S\,\ref{S:family of overconvergent forms} and \ref{S:abstract classical forms}.
\end{definition}

The first goal of this article is to formulate the following.

\begin{conjecture}[Local ghost Conjecture~\ref{Conj:local ghost conjecture}]
\label{Conj:local ghost intro}
Let $\widetilde \rmH$ be a primitive $\calO\llbracket K_p\rrbracket$-projective augmented module of type $\bar\rho$ and let $\varepsilon: (\FF_p^\times)^2 \to \calO^\times$ be a relevant character. Then there exists a combinatorially defined \emph{ghost series} $$G^{(\varepsilon)}_{\bar \rho}(w,t): = \sum_{n \geq 0} g_n^{(\varepsilon)}(w)t^n \in \ZZ_p[w]\llbracket t\rrbracket\quad \textrm{with} \quad  g_n^{(\varepsilon)}(w) = \prod_{\substack{k \geq 2\\ k \equiv k_\varepsilon \bmod{p-1}}} (w-w_k)^{m_n^{(\varepsilon)}(k)} \in \ZZ_p[w],$$ where $w_k : = \exp((k-2)p)-1$, $k_\varepsilon\in \{2,\dots,p \}$ is an integer defined in Notation~\ref{N:relevant varepsilon} and the exponents $m_n^{(\varepsilon)}(k)$ are defined combinatorially in Definition~\ref{D:ghost series}, which only depends on $\bar \rho$, $\varepsilon$, (and $k$ and $n$), such that for any $w_\star \in \gothm_{\CC_p}$, the Newton polygon of $C_{\widetilde \rmH}^{(\varepsilon)}(w_\star,-)$ is the same as $G_{\bar \rho}^{(\varepsilon)}(w_\star,-)$.
\end{conjecture}

\emph{We advertise that we shall prove this conjecture under some mild hypotheses in a sequel to this paper.}

\begin{remark}
\begin{enumerate}
\item 
The definition of the exponents $m_n^{(\varepsilon)}(k)$ follows exactly the recipe laid out by \cite{bergdall-pollack2}, by harvesting input from the newform/old form theory. We refer to their paper as well as Remark~\ref{R:ghost pattern} for a more intuitive explanation.
\item
When $\widetilde \rmH$ comes from the case of modular forms as discussed above (and assuming certain mod-$p$-multiplicity-one condition), the local ghost conjecture is precisely the $\bar r$-ghost conjecture raised by Bergdall--Pollack \cite{bergdall-pollack2, bergdall-pollack3}.
\item
The condition that $\bar \rho$ is generic and reducible is essential to Conjecture~\ref{Conj:local ghost intro}.
If $\bar r|_{\Gal_{\QQ_p}}$ is irreducible, the situation is significantly different. We hope to address that in a future work. The genericity is also important as indicated  in \cite{bergdall-pollack4}.
These conditions on $\bar \rho$ first appeared in the literature in the form of so-called \emph{Buzzard-regular condition} \cite{buzzard-slope}; we refer to \cite[\S\,2]{bergdall-pollack4} for a detailed discussion of Buzzard-regular condition in terms of $\bar r|_{\Gal_{\QQ_p}}$.  See Remark~\ref{R:generalizations of ghost conjecture} for more discussion.
\item 
The primitive hypothesis is not a serious constraint: for each $\bar \rho$, one can always globalize it to a case in (1) above by a construction in \cite[Appendix~A]{gee-kisin}; using the known Breuil--M\'ezard conjecture, one can ``remember the slope information" on the Galois deformation side, and then bootstrap from it in a general situation. The same argument may be able to treat the case with reducible generic  split case.
We shall give more details on this argument in a subsequent paper.
\end{enumerate}

\end{remark}

\subsection{Main results of this paper}
\label{S:main result summarized}
The focus of this paper is to investigate the following basic properties of the ghost series.
\begin{enumerate}
\item 
We give an explicit and asymptotic computation of $m_n^{(\varepsilon)}(k)$. (Propositions~\ref{P:dimension of SIw} and \ref{P:dimension of Sunr}), making the ghost series explicit.
\item 
The local ghost conjecture would require that the slopes of the Newton polygon of $G_{\bar \rho}^{(\varepsilon)}(w_k, -)$ for $k \in \ZZ$ to include slopes of abstract classical forms. We check its compatibility with theta maps,  the Atkin--Lehner involutions, and the $p$-stabilizations (Proposition~\ref{P:ghost compatible with theta AL and p-stabilization}\,(1)--(3)).  
 \item
The \emph{ghost duality}, which is the most important property of ghost series, near a weight point where the space of abstract $p$-new forms  is nonzero. We refer to Proposition~\ref{P:ghost compatible with theta AL and p-stabilization}\,(4) for the statement.
\item For $k \in \ZZ_{\geq 2}$, the ``old form slopes" of $\NP(G(w_k,-))$ are less than or equal to $\big\lfloor\frac{k-1}{p+1}\big\rfloor$, as predicted by a conjecture of Gouv\^ea \cite{gouvea}.

\item  For $n \in \ZZ_{\geq 0}$ and $w_\star \in \gothm_{\CC_p}$, the point $(n, v_p(g_n(w_\star)))$ is a vertex on the Newton Polygon $\NP(G(w_\star, -))$ if and only if $(n, w_\star)$ is not near-Steinberg, a condition that roughly says that $w_\star$ is $p$-adically close to a weight point $w_k$ at which the $n$th slope is the slope of an abstract $p$-new form. (Theorem~\ref{T:near Steinberg = non-vertex}). 

\item For $k \in \ZZ$, the slopes of $\NP(G(w_k,-))$ are all integers if $a$ is even and belong to $\frac 12\ZZ$ if $a$ is odd (Corollary~\ref{C:integrality of slopes of G}). 
\end{enumerate}

The proofs of (1)--(4) are relatively straightforward, as soon as we setup the theory for abstract classical and overconvergent forms properly. The proofs of (5) and (6) are rather combinatorially involved; they are the main content of this paper.

\subsection{Relation with historical conjectures and results on slopes of modular forms}

The general study of slopes of (overconvergent) modular forms began with extensive computer enumeration by Gouv\^ea and Mazur \cite{gouvea-mazur, gouvea} in 1990's.
Let $N$ be a positive integer that is relatively prime to $p$.
The space $S_k(\Gamma_0(Np))$ of modular forms of weight $k \in \NN_{\geq 2}$ contains two copies of $S_k(\Gamma_0(N))$ as the space of $p$-old forms $S_k(\Gamma_0(Np))^{p\old}$, and the Hecke complement is the space of $p$-new forms $S_k(\Gamma_0(Np))^{\new}$. The eigenvalues of $U_p$-action on $S_k(\Gamma_0(Np))^{\new}$ are $\pm p^{(k-2)/2}$, so $U_p$ has slope $(k-2)/2$ on this space. On the other hand, Atkin--Lehner theory implies that the $U_p$-slopes on $S_k(\Gamma_0(Np))^{p\old}$ can be paired so that the sum of each pair is $k-1$.
Based on significant computation data, 
Gouv\^ea raised the following conjecture in \cite{gouvea}.

\begin{conjecture}
[Gouv\^ea]
\label{Conj:Gouvea} Let $\mu_k$ denote the density measure on $[0,1]$ given by the set of $U_p$-slopes on $S_k(\Gamma_0(Np))$, which are normalized by being divided by $k-1$ and are counted with multiplicities.
As $k \to \infty$, $\mu_k$ converges to the sum of the delta measure of mass $\frac{p-1}{p+1}$ at $\frac 12$, and the uniform measure on $[0, \frac 1{p+1}] \cup [\frac p{p+1}, 1]$ with total mass $\frac 2{p+1}$.
\end{conjecture}
The delta measure at $\frac 12$ clearly comes from the $p$-new forms, the $p$-old form part (supposedly responsible for the uniform measure) does not seem to have any theoretical explanation. In sharp contrast, numerically, 
Gouv\^ea observes that the lesser of the $p$-old form slopes (in the pairs mentioned above) are almost always less than or equal to $\big\lfloor\frac{k-1}{p+1}\big\rfloor$; this is known as Gouv\^ea's $\big\lfloor\frac{k-1}{p+1}\big\rfloor$-conjecture.

From modern point of view, it is best to study the two conjectures above after first localizing at a residual Galois representation $\bar r$.  Then as proved in \cite{bergdall-pollack3}, the local ghost Conjecture~\ref{Conj:local ghost intro} implies the $\bar r$-part of Conjecture~\ref{Conj:Gouvea} (at least under the mod-$p$-multiplicity-one assumption). On the other hand, \S\,\ref{S:main result summarized}(6) implies the $\bar r$-part of  Gouv\^ea's $\big\lfloor\frac{k-1}{p+1}\big\rfloor$-conjecture (at least under the mod-$p$-multiplicity-one assumption).

\begin{remark}
\begin{enumerate}
\item 
There has been a couple of interesting results on Gouv\^ea's $\big\lfloor \frac{k-1}{p+1}\big\rfloor$-conjecture, obtained by Berger--Li--Zhu \cite{berger-li-zhu} and Bergdall--Levin \cite{bergdall-levin}, respectively. They used $(\varphi, \Gamma)$-modules and Kisin modules techniques to show that if the slopes of the crystalline Frobenius is greater than or equal to $\big \lfloor \frac{k-1}{p-1}\big\rfloor$ and $\big\lfloor\frac{k-1}p\rfloor$, respectively, then the reduction of such representation is irreducible.

\item
We also mention a conjecture by Gouv\^ea--Mazur \cite[Conjecture 1]{gouvea-mazur} on radii of constancy of slopes on spectral curves. We do not know whether our local ghost conjecture implies the $\bar r$-part of Gouv\^ea--Mazur's conjecture. 
\end{enumerate}
\end{remark}

\subsection{Implication on Kisin's crystalline Galois deformation space}
As indicated by Remark~\ref{R:origins of arithmetic modules}(2) that applying local ghost Conjecture~\ref{Conj:local ghost intro} to the patched completed homology, we may deduce results on Kisin's crystalline deformation  space of $\tilde \rho$. One interesting conjecture in this direction was 
 originated in
emails among Breuil, Buzzard, and Emerton in 2005 (\cite{buzzard-gee}).

\begin{conjecture}[Breuil--Buzzard--Emerton]
\label{Conj:Breuil-Buzzard-Emerton}
Assume that $\tilde \rho : \Gal(\overline \QQ_p/\QQ_p) \to \GL_2(\FF)$ is reducible, nonsplit, and generic, then the $p$-adic valuations of the eigenvalues of the crystalline Frobenius action on the Kisin's framed crystalline deformation space of $\tilde \rho$ are integers (when $k$ is even) or half-integers (when $k$ is odd).
\end{conjecture}

Our property \S\,\ref{S:main result summarized} is related to this conjecture, and suggests that local ghost conjecture should imply the conjecture above.

\begin{remark}
It seems to be very difficult to directly compute the reduction of a crystalline representations when the (difference of) Hodge--Tate weights become larger than $2p$ (see for example, \cite{breuil-reduction}); retroactively speaking, this is expected as the ghost series is itself rather combinatorially involved.

When the slopes are less than equal to $3$, there have been several works on the computation of reduction of crystalline representations using mod $p$ local Langlands correspondence, see \cite{breuil-reduction, buzzard-gee-reduction, ghate1, ghate2, ghate3}.

Our approach may be viewed as coming from the automorphic side of the $p$-adic local Langlands correspondence, which harnesses the advantage of working with a family of abstract forms parametrized by the weight space. Working purely with the setup of $\calO\llbracket K_p\rrbracket$-projective augmented modules allows us to relate the slope question to Kisin's crystalline deformation space via the known Breuil--M\'ezard conjecture for $\GL_2(\QQ_p)$.
\end{remark}

\subsection{Relation to a refined spectral halo conjecture}
Note that if one evaluates the ghost series at $w_\star \in \gothm_{\CC_p}$ with $v_p(w_\star) \in (0,1)$,
we quickly deduce that
$$
v_p(g_n^{(\varepsilon)}(w_\star)) = v_p(w_\star)\cdot \deg g_n^{(\varepsilon)} .
$$
In Proposition~\ref{P:increment of degrees in ghost series} below, we compute explicitly $\deg g_{n+1}^{(\varepsilon)} - \deg g_n^{(\varepsilon)}$ and show that, under our assumption $1 \leq a \leq p-4$, the degree differences is strictly increasing in $n$. Thus, the local ghost conjecture can imply that the spectral curve associated to $\widetilde \rmH$, when restricted to the boundary of the weight disc: $\calW^{(0,1)}:= \big\{w_\star \in \gothm_{\CC_p}\; |\; v_p(w_\star) \in (0,1)\big\}$ is an infinite disjoint union of copies of $\calW^{(0,1)}$ and the slopes on each such copy is proportional to $v_p(w_\star)$. We also compute the predicted slope ratios in Proposition~\ref{P:increment of degrees in ghost series}. This may be viewed as a refined version of the spectral halo conjecture proved in \cite{liu-wan-xiao} (but only for the $\bar r$-part of the spectral curve) for which $\bar r$ satisfies the reducible, nonsplit, and generic conditions.

\subsection*{Roadmap of the paper}
Section~\ref{Sec:local conjecture} is devoted to formulating the local ghost conjecture (Conjecture~\ref{Conj:local ghost conjecture}). In the next section, we make standard constructions of abstract classical forms and abstract $p$-adic forms from an $\calO\llbracket K_p\rrbracket$-projective augmented module. In Section~\ref{Sec:properties of ghost series}, we prove basic properties of the ghost series, including \S\,\ref{S:main result summarized}(1)--(3),
most notably the \emph{ghost duality}. Section~\ref{Sec:finer properties of ghost series} is the most technical part of this paper, in which we prove many important results regarding vertices of Newton polygon of ghost series, namely \S\,\ref{S:main result summarized}(4)--(6). In the appendix, we recollect some basic facts on modular representations of $\GL_2(\FF_p)$ that we use throughout the paper.

\subsection*{Acknowledgments}
This paper will not be possible without the great idea from the work of John Bergdall and Robert Pollack \cite{bergdall-pollack2,bergdall-pollack3}.
Many proofs are inspired by their numerical evidences.  We especially thank them for sharing ideas and insights at an early stage and for many interesting conversations.
We thank Matthew Emerton and Yongquan Hu for multiple helpful discussions. We also thank Jiawei An, Christophe Breuil, Toby Gee, Bao Le Hung, Vincent Pilloni, and Rufei Ren for their interest in the project and inspiring comments.
We thank all the people contributing to the SAGE software, as lots of our argument rely on first testing by a computer simulation.

R. Liu is partially supported by the National Natural Science Foundation of China under agreement No. NSFC-11725101 and the Tencent Foundation. N.T. is partially supported by L.X.'s NSF grant DMS--1752703.
L.X. is  partially supported by Simons Collaboration Grant \#278433, NSF grant DMS--1502147 and DMS--1752703, the Chinese NSF grant NSFC--12071004, Recruitment Program of Global Experts of China, and a grant
from the Chinese Ministry of Education. B.Z. is partially supported by
AMS-Simons Travel Grant.

\subsection*{Notation}
\label{S:notation}
\begin{enumerate}
\item For a field $k$, we write $\overline k$ for its algebraic closure.


\item Throughout the paper, we fix a prime number $p\geq 5$ unless otherwise specified. Although some statements in this paper may also hold for smaller primes, we insist this assumption to avoid unnecessarily complicated argument.  
Let $\rmI_{\QQ_p}\subset \Gal(\overline{\QQ}_p/\QQ_p)$ denote the inertia subgroup, and  $\omega_1: \rmI_{\QQ_p} \twoheadrightarrow\Gal(\QQ_p(\mu_p)/\QQ_p) \cong \FF_p^\times$ the \emph{the $1$st fundamental character}.

\item Let $\Delta \cong (\ZZ/p\ZZ)^\times$ be the torsion subgroup of
$\ZZ_p^\times$, and let $\omega:\Delta\rightarrow \ZZ_p^\times$ be the
Teichm\"{u}ller character.
For an element $\alpha \in \ZZ_p^\times$, we often use $\bar \alpha \in \Delta $ to denote its reduction modulo $p$.
In particular, the Greek letters $\alpha, \beta, \gamma, \delta$ are reserved for elements in $\ZZ_p$, and $\bar \alpha, \bar \beta, \bar \gamma, \bar \delta$ denote elements in $\FF_p$.

\item Let $E$ be a finite extension of $\QQ_p(\sqrt p)$, which serves as the coefficient field so that $E$ contains $p^{\frac{k-2}{2}}$ for every integer $k\geq 2$. Let $\calO$, $\FF$, and $\varpi$ denote its ring of integers, residue field, and a uniformizer, respectively.

\item The $p$-adic valuation $v_p(-)$ is normalized so that $v_p(p) =1$.

\item We use $\lceil x\rceil$ to denote the ceiling function and $\lfloor x\rfloor$ to denote the floor function.

\item For $M$ a topological $\calO$-module, we write $\calC^0(\ZZ_p; M)$ for the space of continuous functions on $\ZZ_p$ with values in $M$. It is endowed with the compact-open topology.

All maps and characters are continuous with respect to the relevant topology. All hom spaces in this paper refer to the spaces of continuous homomorphisms.

\item We shall encounter both $p$-adic logs $\log(x) = (x-1) - \frac{(x-1)^2}2 + \cdots $ for $x$ a $p$-adic or formal element,  and the natural logs $\ln(-)$ in the real analysis. (The usual $p$-based logarithm will always be denoted by $\ln(-)/\ln(p)$.)

\item For each $m \in \ZZ$, we write $\{m\}$ for the unique integer
satisfying the conditions 
$$0\leq \{m\}\leq p-2 \quad \textrm{and} \quad m\equiv \{m\}
\bmod{p-1}.$$
For $m \in \ZZ_{\geq 0}$, let $\Dig(m)$ denotes the sum of all digits in the $p$-based expression of $m$.

\item For a square (possibly infinite) matrix $M$ with coefficients in a commutative ring $R$, its \emph{characteristic power series} (if it exists) is denoted by $\Char(M;t) : = \det(I- Mt) \in R\llbracket t\rrbracket$, where $I$ is the identity matrix.  When an $R$-linear action $U$ on an $R$-module $N$ is given by such a matrix $M$, we write $\Char (U, N; t)$ for $\Char (M;t)$.

\item For a power series $F(t) = \sum_{n \geq 0} c_nt^n \in \CC_p\llbracket t\rrbracket$ with $c_0=1$, we use $\NP(F)$ to denote its \emph{Newton polygon}, i.e. the convex hull of points $(n, v_p(c_n))$ for all $n$; the slopes the segments of $\NP(F)$ are often referred to as \emph{slopes} of $F(t)$.

\end{enumerate}

\section{Local ghost conjecture: formulation}
\label{Sec:local conjecture}

In this section, we give a detailed formulation of Conjecture~\ref{Conj:local ghost intro}, which is a local analogue of the ghost conjecture of Bergdall and Pollack
\cite{bergdall-pollack2, bergdall-pollack3}.
The novelty here is to work with a representation of $\GL_2(\QQ_p)$, without referencing to a global automorphic setup.  Such idea was inspired by a discussion with Matthew Emerton; we thank him heartily.

To be more precise, we start with a projective module $\widetilde \rmH$ over the Iwasawa algebra of $\GL_2(\ZZ_p)$ whose group action extends to an action of $\GL_2(\QQ_p)$; we often encounter such $\widetilde \rmH$  as the completed homology in the automorphic setup.  The information of $\bar \rho|_{G_{\QQ_p}}$ is ``manually" inserted by requiring $\widetilde \rmH$ to be the projective envelope as a $\calO\llbracket \GL_2(\ZZ_p)\rrbracket$-module of the corresponding Serre weight.  In this purely local setup, we can still properly define ``abstract classical forms", ``abstract overconvergent forms",  and ``abstract $p$-adic forms" associated to $\widetilde \rmH$.  This allows us to define the corresponding characteristic power series of the $U_p$-operator and the ghost series.   We state a more concrete statement of the local ghost conjecture (Conjecture~\ref{Conj:local ghost conjecture}) at the end of the section.

\subsection{Recollection of Iwasawa algebras}
\label{S:Iwasawa algebra}
Let $G$ be topological group that contains a finite-index \emph{analytic} pro-$p$-Sylow subgroup $H$ (in the sense of Lazard \cite[\S\,III.3]{lazard}).  We write $$\calO\llbracket G\rrbracket : = \varprojlim_{N \unlhd G} \calO[G/N]$$ for the corresponding (possibly non-commutative) Iwasawa algebra.  It is left and right noetherian by \cite[Proposition~V.2.2.4]{lazard}.
We write $\rmI_G$ for its augmentation ideal. For an element $g \in G$, we write $[g]$ for the group element in $\calO\llbracket G \rrbracket$.

In this paper, we shall consider right $\calO\llbracket G\rrbracket$-modules as opposed to  left ones because the completed homology groups are naturally right modules over the Iwasawa algebra of $\GL_2(\ZZ_p)$. 

\begin{lemma}
\label{L:Iwasawa algebra properties} Keep the notation from above.
The following are equivalent for a finitely generated  \emph{right} $\calO\llbracket G\rrbracket$-module $M$:
\begin{enumerate}
\item $M$ is a finite projective right $\calO\llbracket G\rrbracket$-module,
\item $M$ is a finite projective right $\calO\llbracket H \rrbracket$-module,
\item $M$ is a finite free right $\calO\llbracket H\rrbracket$-module, and
\item $\Ext^1_{\calO\llbracket H\rrbracket}(M, \FF) = 0$.
\end{enumerate}
\end{lemma}
\begin{proof}

Indeed, (1) obviously implies (2). Conversely, for any (right) $\calO\llbracket G\rrbracket$-module $N$, the composition
\[
\Ext^i_{\calO\llbracket G\rrbracket}(M, N) \xrightarrow{\mathrm{Restriction}}
\Ext^i_{\calO\llbracket H\rrbracket}(M, N)
\xrightarrow{\mathrm{Trace}}
\Ext^i_{\calO\llbracket G\rrbracket}(M, N)
\]
is the multiplication by the index $[G:H]$ (\cite[Chapter I, \S~2.4]{Serre-Galois cohomology}), which is relatively prime to $p$. So the vanishing of $\Ext^{>0}_{\calO\llbracket H\rrbracket}(M, N)$ implies that of $\Ext^{>0}_{\calO\llbracket G\rrbracket}(M, N)$.

(3) obviously implies (2). The converse implication holds because $\calO\llbracket H\rrbracket$ is a noetherian local ring (whose maximal ideal is generated by $\varpi$ and the augmentation ideal).

(3) obviously implies (4). Conversely, since $\Ext^i_{\calO\llbracket H\rrbracket}(M,-)$ commutes with taking inverse limit on finite objects, it reduces to showing that $\Ext^1_{\calO\llbracket H\rrbracket}(M, N) =0$ for any finite length $\calO\llbracket H\rrbracket$-modules $N$. But all such $N$ are successive extensions of $\FF$ as $\calO\llbracket H\rrbracket$ is a local ring. So (4) implies (3). 
\end{proof}

\begin{definition}
We will mostly work with the following subgroups of $\GL_2(\QQ_p)$:
$$
\rmK_p: =\GL_2(\ZZ_p) \supset \Iw_p := \begin{pmatrix}
\ZZ_p^\times & \ZZ_p \\ p \ZZ_p & \ZZ_p^\times
\end{pmatrix} \supset
\Iw_{p,1} := \begin{pmatrix}
1+p\ZZ_p & \ZZ_p \\ p \ZZ_p & 1+p\ZZ_p
\end{pmatrix}.$$

By an \emph{ $\calO\llbracket K_p\rrbracket$-projective augmented module}, we mean a finite projective \emph{right} $\calO\llbracket \rmK_p\rrbracket$-module $\widetilde \rmH$ together with an action of $\GL_2(\QQ_p)$ extending the $\rmK_p$-action. Such a module is endowed with the canonical topology as a finitely generated $\calO\llbracket \rmK_p\rrbracket$-module defined in \cite[Definition~2.1.4]{emerton-ordinary}. (The notation is an adaptation of the notion of pro-augmented representations in \cite[Definition~2.1.5]{emerton-ordinary}.)
We fix one for this and the next four sections.

By \cite[Corollary~3.4]{breuil}, extending the $\rmK_p$-action on $\widetilde \rmH$ to a $\GL_2(\QQ_p)$-action is equivalent to giving a (right) action of
$\Pi=\Matrix 01p0$ on $\widetilde \rmH$ which satisfies the relation
$$
m\Pi \Matrix \alpha\beta{p\gamma}\delta= m \Matrix \delta\gamma{p\beta}\alpha\Pi, \quad \textrm{ for all }m \in \widetilde \rmH \textrm{ and } \Matrix \alpha\beta{p\gamma}\delta \in \Iw_p.
$$
\end{definition}

\subsection{Abstract $p$-adic forms}
\label{S:abstract p-adic forms}
We follow the setup of \cite[\S\,2.3]{liu-wan-xiao} with small variation. In particular, we consider \emph{right} inductions from the lower triangular matrices (as opposed to the upper triangular matrices) but with the natural action (as opposed to the twisted action used in \cite[\S~2.3]{liu-wan-xiao}). Luckily, we end up with the same formula.

Set $\Delta : = \FF_p^\times$
and let $\omega: \Delta \to \ZZ_p^\times$ be the Teichm\"uller character. 
Write $\log(-)$ and $\exp(-)$ for the formal $p$-adic logarithmic function and formal $p$-adic exponential function, respectively. For an element $\alpha \in \ZZ_p$, write $\bar \alpha \in \FF_p$ for its reduction modulo $p$.

Fix a character $\varepsilon: \Delta^2 \to \calO^\times$ and put
\[
\calO\llbracket w \rrbracket^{(\varepsilon)}: = \calO\llbracket \Delta \times \ZZ_p^\times\rrbracket \otimes_{\calO[\Delta^2], \varepsilon} \calO. 
\]
It is isomorphic to the ring of formal power series ring $\calO\llbracket w \rrbracket$ with $w$ corresponding to $[(1, \exp(p))]-1$. (In \cite{liu-wan-xiao}, we used the letter $T$ for $w$. In this paper, we follow the convention of \cite{bergdall-pollack2, bergdall-pollack3}.)

Consider the universal character
\[
\begin{tikzcd}[row sep = 0pt]
\chi^{(\varepsilon)}_\univ: B^\mathrm{op}(\ZZ_p) ={\Big( \begin{smallmatrix}
\ZZ_p^\times &0  \\p\ZZ_p & \ZZ_p^\times
\end{smallmatrix} \Big)} \ar[r]
&\big( \calO\llbracket w \rrbracket^{(\varepsilon)} \big) ^\times
\\
{\Matrix {\alpha}{0}{\gamma}{\delta}} \ar[r, mapsto] & {[(\bar \alpha, \delta)]} \otimes 1 = \varepsilon(\bar \alpha, \bar \delta) \cdot (1+w)^{\log(\delta/\omega( \bar \delta))/p},
\end{tikzcd}
\]
where the superscript in $\calO\llbracket \omega\rrbracket^{(\varepsilon)}$ indicates the twisted action. The induced representation (for the \emph{right} action convention)
\begin{align*}
\Ind_{B^\op(\ZZ_p)}^{\Iw_p}(\chi_\univ^{(\varepsilon)}): = \big\{ &\textrm{continuous functions }f: \Iw_p \to  \calO\llbracket w\rrbracket^{(\varepsilon)};
\\ & f(gb) = \chi_\univ^{(\varepsilon)}(b) \cdot f(g) \textrm{ for }b \in B^\op(\ZZ_p) \textrm{ and } g \in \Iw_p\big\}
\end{align*}
carries a \emph{right} action of $\Iw_p$ given by $f|_{\Matrix \alpha\beta{\gamma}\delta}(g): = f\big(\Matrix \alpha \beta \gamma \delta g\big)$ 
for $\Matrix \alpha\beta{\gamma}\delta \in \Iw_p$.
The Iwasawa decomposition allows us to make the following identification:
\[
\begin{tikzcd}[row sep =0pt]
\Ind_{B^\op(\ZZ_p)}^{\Iw_p}(\chi_\univ^{(\varepsilon)}) \ar[r, "\cong"]& \calC^0(\ZZ_p; \calO\llbracket w \rrbracket^{(\varepsilon)})
\\
f \ar[r, mapsto] & h(z): = f(\Matrix 1{z}01).
\end{tikzcd}
\] 
The natural $\Iw_p$-action on the left hand side then translates to the following action on the right hand side
\begin{eqnarray}
\label{E:induced representation action}
h\big|_{\Matrix \alpha\beta{\gamma}\delta}(z) & =& \big[(\bar \alpha, \gamma z+\delta)\big] \cdot h\Big( \frac{\alpha z+\beta}{\gamma z+\delta}\Big)
\\
\nonumber
&=& \varepsilon(\bar \alpha, \bar \delta) \cdot (1+w)^{\log\left((\gamma z+\delta )/ \omega(\bar \delta)\right) / p} \cdot h\Big( \frac{\alpha z+\beta}{\gamma z+\delta}\Big)
\quad \textrm{for } \Matrix \alpha\beta{\gamma}\delta \in \Iw_p.
\end{eqnarray}
This action extends to the action of a bigger monoid (where the notation comes from \cite[\S\,4]{buzzard2})
$$
\mathbf{M}_1=\big\{\Matrix \alpha\beta\gamma\delta \in \rmM_2(\ZZ_p);  \ p|\gamma,\,p\nmid \delta,\, \alpha \delta-\beta \gamma \neq 0\big\}
$$
through the following formula analogous to \eqref{E:induced representation action}: for $\Matrix \alpha \beta \gamma \delta \in \bfM_1$ with determinant $\alpha \delta - \beta\gamma = p^r d$ ($d \in \ZZ_p^\times$),
\begin{equation}
\label{E:induced representation action extended}
h\big|_{\Matrix \alpha\beta{\gamma}\delta}(z) = \varepsilon(\bar d /  \bar \delta, \bar \delta) \cdot (1+w)^{\log\left((\gamma z+\delta )/ \omega(\bar \delta)\right)/p} \cdot h\Big( \frac{\alpha z+\beta}{\gamma z+\delta}\Big).
\end{equation}

For the fixed $\calO\llbracket K_p\rrbracket$-projective augmented module $\widetilde \rmH$ and a character $\varepsilon$ as above, we define the space of \emph{abstract $p$-adic forms} to be
\begin{equation*}
\label{E:Sp-adic epsilon}
\rmS^{(\varepsilon)}_{p\textrm{-adic}} = \rmS^{(\varepsilon)}_{\widetilde \rmH, p\textrm{-adic}}: = \Hom_{\calO[\Iw_{p}]}\big(\widetilde \rmH, \, \Ind_{B^\op(\ZZ_p)}^{\Iw_p}(\chi_\univ^{(\varepsilon)})\big) \cong \Hom_{\calO[\Iw_p]}\big(\widetilde \rmH, \, \calC^0(\ZZ_p; \calO\llbracket w\rrbracket^{(\varepsilon)})\big).
\end{equation*}
\begin{remark}
	\begin{enumerate}
		\item Recall that the hom spaces in this paper refer to the space of continuous homomorphisms. We recall the topologies of the source and target in the definition of $\rmS^{(\varepsilon)}_{p\textrm{-adic}}$ for the audience's convenience. The topology on $\widetilde \rmH$ is the canonical topology as a finitely generated $\calO\llbracket K_p \rrbracket$-module and the topology on $\calC^0(\ZZ_p; \calO\llbracket w\rrbracket^{(\varepsilon)})$ is the open-compact topology. In particular, if we have a subset $S$ such that the $\calO[K_p]$-submodule of $\widetilde \rmH$ generated by $S$ is dense in $\widetilde \rmH$, then any element in $\rmS^{(\varepsilon)}_{p\textrm{-adic}}$ is uniquely determined by its values on $S$.
		\item When $\widetilde \rmH = \varprojlim_n \calO\big[ D^\times \backslash D(\AAA_f)^\times / K^p (1+ p^n \rmM_2(\ZZ_p))\big]$ for a definite quaternion algebra $D$ over $\QQ$ split at $p$ and an open compact neat subgroup $K^p\subset (D\otimes \AAA_f^{(p)})^\times$ (we refer to \cite[\S2.4]{liu-wan-xiao} for the precise definition), $\widetilde \rmH$ is an $\calO\llbracket K_p\rrbracket$-projective augmented module and (the direct sum over characters $\varepsilon$ of)  $\rmS_{\widetilde \rmH, p\textrm{-adic}}^{(\varepsilon)}$ is the same as the space of $p$-adic automorphic forms $S^D_\mathrm{int}$ of \cite[\S\,2.7]{liu-wan-xiao}, after properly adjusted for normalizations.
	\end{enumerate}
\end{remark}

The space $\rmS_{p\textrm{-adic}}^{(\varepsilon)}$ is a countable  infinite completed direct sum of copies of $ \calO\llbracket w\rrbracket$,  carrying
an $ \calO\llbracket w\rrbracket$-linear $U_p$-action: fixing a decomposition of the double coset $\Iw_p \Matrix {p^{-1}}001 \Iw_p = \coprod_{j=0}^{p-1}  v_j\Iw_p$ (e.g. $v_j = \Matrix {p^{-1}}0j1$ with $v_j^{-1} = \Matrix p0{-jp}1 \in \bfM_1$), the $U_p$-operator sends $\varphi \in \rmS_{p\textrm{-adic}}^{(\varepsilon)}$ to
\begin{equation}
\label{E:Up action}
U_p(\varphi)(x) = \sum_{j=0}^{p-1}
\varphi(xv_j)|_{v_j^{-1}}
\quad \textrm{for all }x \in \widetilde \rmH.
\end{equation}
The $U_p$-operator does not depend on the choice of coset representatives.


\begin{remark}
The next logical step is show that the characteristic power series of the $U_p$-action on $\rmS_{p\textrm{-adic}}^{(\varepsilon)}$ has coefficients in $\calO\llbracket w\rrbracket$.
Unfortunately, such well-believed result is not documented in the literature under our abstract setup. The situation is further complicated by two subtleties: the $U_p$-action is \emph{not} compact on $\rmS_{p\textrm{-adic}}^{(\varepsilon)}$ (see \cite[Example 5.2 and Remark 5.3]{liu-wan-xiao}), and certain non-neat phenomenon appears because we only assumed $\widetilde \rmH$ to be finite projective (as opposed to finite free) over $\calO\llbracket\rmK_p\rrbracket$.
The argument we provide below in \S\,\ref{S:explicit Muniv}--\ref{S:char power series of Up} is tailored to the need of defining the characteristic power series. We shall revisit finer estimates of the $U_p$-actions analogous to \cite{liu-wan-xiao} in the next paper.
\end{remark}

\begin{convention}
In what follows, we will frequently study the action of the $U_p$-operator on various $\calO$-modules or vector spaces. Following the convention in this area, we say \emph{finite $U_p$-slopes} to mean the $p$-adic valuations of the nonzero eigenvalues of the $U_p$-action on the corresponding space, usually counted with multiplicity.
\end{convention}

\subsection{Explicit description of $\rmS_{p\textrm{-}\mathrm{adic}}^{(\varepsilon)}$}
\label{S:explicit Muniv}
Unlike in \cite{liu-wan-xiao}, our $\widetilde \rmH$ may not be free over $\calO\llbracket\Iw_p\rrbracket$. We now discuss this subtle point.

We may view $\bar \rmT= \Matrix \Delta 0 0 \Delta \cong \Delta^2$ as a subgroup of $\Iw_p$ via  Teichm\"uller lifts.
The $\calO\llbracket K_p\rrbracket$-projective augmented module $\widetilde \rmH$ is a finite projective right $\calO\llbracket\rmK_p\rrbracket$-module, and hence a finite projective (right) $\calO\llbracket\Iw_p\rrbracket$-module. So there exists a
(non-canonical) isomorphism of right $\calO\llbracket\Iw_p\rrbracket$-modules:
\begin{equation}
\label{E:expression of tilde rmH}
\widetilde \rmH \simeq \bigoplus_{i=1}^r  e_i\calO  \otimes_{\chi_i,\calO[\bar \rmT]} \calO\llbracket\Iw_p\rrbracket,
\end{equation}
where $\chi_i: \bar \rmT \to \calO^\times$ is a character, and $r$ is equal to the rank of $\widetilde \rmH$ as a free $\calO\llbracket\Iw_{p,1}\rrbracket$-module.
Using the evaluations at the  elements $e_i$, we have identifications
\begin{equation}
\label{E:explicit Mepsilon}
\rmS_{p\textrm{-adic}}^{(\varepsilon)} \simeq \bigoplus_{i=1}^r e_i^*\cdot  \calC^0(\ZZ_p; \calO\llbracket w\rrbracket^{(\varepsilon)})^{\bar \rmT = \chi_i}.
\end{equation}
Here for a function $h(z)\in \calC^0(\ZZ_p; \calO\llbracket w\rrbracket^{(\varepsilon)})^{\bar \rmT = \chi_i}$, the notation $e_i^\ast \cdot h(z)$ denotes the (unique) element in $\rmS_{p\textrm{-adic}}^{(\varepsilon)} $ that sends $e_i$ to $h(z)$, and $e_j$ to $0$ for all $j\neq i$.

In explicit terms, $\calC^0(\ZZ_p; \calO\llbracket w\rrbracket^{(\varepsilon)})^{\bar \rmT = \chi_i}$ consists of continuous functions $h(z)$ on $\ZZ_p$ with values in $\calO\llbracket w \rrbracket^{(\varepsilon)}$ such that
\[
\textrm{for any }\bar \alpha, \bar \delta \in \Delta, \quad \textrm{we have} \quad \varepsilon(\bar \alpha, \bar \delta)  \cdot h( \bar \alpha z / \bar \delta ) = \chi_i(\bar \alpha, \bar \delta)\cdot h(z).
\]
So $\calC^0(\ZZ_p; \calO\llbracket w\rrbracket^{(\varepsilon)})^{\bar \rmT = \chi_i}$ 
is non-zero precisely when $\varepsilon\chi^{-1}_i(\bar x, \bar x) = 1$ for all $\bar x \in \Delta$, and in this case, it is the subspace of continuous functions on $\ZZ_p$ such that
\begin{equation}
\label{E:non-neat condition}
h(\bar \alpha z) = \chi_i\varepsilon^{-1}( \bar \alpha,1) h(z),\quad \textrm{for all }\bar \alpha \in \Delta.
\end{equation}
This can be viewed as an eigenspace for the $\Delta$-action on $\calC^0(\ZZ_p; \calO\llbracket w\rrbracket^{(\varepsilon)})$.

The reason that we see this eigenspace (as opposed to the whole space) is analogous to the situation of non-neat level structure.

\subsection{Explicit description of the $U_p$-action}
\label{S:Up in terms of matrices}
Using the explicit description \eqref{E:explicit Mepsilon} of $\rmS_{p\textrm{-adic}}^{(\varepsilon)}$, we rewrite the action of $U_p$ on $\rmS_{p\textrm{-adic}}^{(\varepsilon)}$ as follows.
For each $e_i$ and $v_j$, we write
\[
e_i v_j = \sum_{k=1}^r e_k c_{ijk} \quad \textrm{for} \quad c_{ijk} \in \calO \otimes_{\chi_k,\calO[\bar \rmT]}\calO\llbracket\Iw_p\rrbracket.
\]
Then via the isomorphism \eqref{E:explicit Mepsilon}, the $U_p$-action on $\rmS_{p\textrm{-adic}}^{(\varepsilon)}$ is the same as the  $\calO\llbracket w \rrbracket$-linear endomorphism
\[
\bigoplus_{i=1}^r e_i^* \cdot  \calC^0(\ZZ_p;\calO\llbracket w\rrbracket^{(\varepsilon)})^{\bar \rmT = \chi_i} \longrightarrow \bigoplus_{i=1}^r e_i^* \cdot  \calC^0(\ZZ_p;\calO\llbracket w\rrbracket^{(\varepsilon)})^{\bar \rmT = \chi_i}
\]
represented by an $r \times r$ matrix, whose $(i,k)$-entry is given by
\begin{equation}
\label{E:entries of Up}
h(z) \longmapsto \sum_{j=0}^{p-1} h|_{c_{ijk}v_j^{-1}}(z).
\end{equation}
Each expression in \eqref{E:entries of Up} is a uniform limit of finite sums of actions given by elements from $\Matrix{p\ZZ_p}{\ZZ_p}{p\ZZ_p}{\ZZ_p^\times}{}^{\det \in p\ZZ_p^\times}$, and hence converges to a well-defined $\calO\llbracket w \rrbracket$-linear map.


\subsection{Characteristic power series of $U_p$}
\label{S:char power series of Up}
To define the characteristic power series of the $U_p$-operator, we need to circumvent the issue that the $U_p$-action on $\rmS_{p\textrm{-adic}}^{(\varepsilon)}$ is not compact (see \cite[\S\,5]{liu-wan-xiao} for further discussion).  This is well know to experts; we include it here for completeness.

In this subsection, we write $\Lambda= \calO\llbracket w \rrbracket$ to simplify notation. We put
\[
\Lambda^{\geq 1/p}: = \calO\llbracket w \rrbracket\langle p/w\rangle, \quad \Lambda^ {=1/p} : = \calO\langle w/p, p/w \rangle, \quad \textrm{and} \quad \Lambda^{\leq 1/p}: = \calO\langle w/p\rangle.
\]
Then $\Lambda^{\geq 1/p}$ and $\Lambda^{\leq 1/p}$ are noetherian Banach--Tate rings (in the sense of \cite[Annexe B]{AIP}) with pseudo-uniformizers $w$ and $p$, respectively.
In particular, we note that the intersection of $\Lambda^{\geq 1/p}$ and $\Lambda^{\leq 1/p}$ in $\Lambda^{=1/p}$ is exactly $\Lambda$.
We shall define
an $\Lambda^?$-linear subspace $\rmS_{p\textrm{-adic}}^{(\varepsilon), ?}$ of $ \rmS_{p\textrm{-adic}}^{(\varepsilon)} \otimes_{\Lambda} \Lambda^? $ where $?$ stands for $\geq 1/p$ or $\leq 1/p$, such that we have a natural equality
\begin{equation}
\label{E:Mvarepsilon gluing}
\rmS_{p\textrm{-adic}}^{(\varepsilon),\geq 1/p} \otimes_{\Lambda^{\geq 1/p}} \Lambda^{=1/p}  = \rmS_{p\textrm{-adic}}^{(\varepsilon),\leq 1/p} \otimes_{\Lambda^{\leq 1/p}} \Lambda^{=1/p}
\end{equation}
inside $\rmS_{p\textrm{-adic}}^{(\varepsilon)} \otimes_\Lambda \Lambda^{=1/p}$.

Set $ \Lambda^{(\varepsilon),?}: = \calO\llbracket w\rrbracket^{(\varepsilon)}\otimes_{\Lambda}\Lambda^?$.
Recall that $\calC^0(\ZZ_p; \calO)$ has a natural orthonormal Mahler basis $1, z, \binom z2, \dots$.
We consider two subspaces:
\begin{align}
\label{E:small subspaces}
\calC^0 \big(\ZZ_p;  \Lambda^{(\varepsilon),\geq 1/p}\big)^\mathrm{mod}& := \widehat \bigoplus_{n \geq 0} w^{\lfloor n/p\rfloor} \tbinom zn \cdot \Lambda^{(\varepsilon),\geq 1/p} \subseteq \calC^0\big(\ZZ_p; \Lambda^{(\varepsilon),\geq 1/p}\big), \quad \textrm{and}
\\
\label{E:small subspaces2}
\calC^0\big(\ZZ_p; \Lambda^{(\varepsilon),\leq 1/p}\big)^\mathrm{mod} &:= \widehat \bigoplus_{n \geq 0} p^{\lfloor n/p\rfloor} \tbinom zn\cdot \Lambda^{(\varepsilon),\leq 1/p} \subseteq \calC^0\big(\ZZ_p; \Lambda^{(\varepsilon),\leq 1/p}\big).
\end{align}
(Our choice of basis elements differs from \cite[\S\,5.4]{liu-wan-xiao}, where a basis $w^n\binom zn$ is used. This to circumvent a technical issue as indicated in the errata of the paper in \S\,\ref{ASS:errata integral model}.)
Following the argument in \cite[\S\,5.4]{liu-wan-xiao}, we now show that both subspaces \eqref{E:small subspaces} and \eqref{E:small subspaces2} are stable under the $\bfM_1$-action, and the action of $\Matrix{p\ZZ_p}{\ZZ_p}{p\ZZ_p}{\ZZ_p^\times}{}^{\det \neq 0}$ is compact on these two subspaces. 

Indeed, by \cite[Proposition~3.14\,(2)]{liu-wan-xiao}, if $P = (P_{m,n})_{m,n \geq 0}$ denotes the infinite matrix for the (right) action of $\Matrix \alpha\beta\gamma\delta \in \bfM_1$ on $\calC^0(\ZZ_p;\Lambda^{(\varepsilon),\geq 1/p})$ (resp. $\calC^0(\ZZ_p;\Lambda^{(\varepsilon),\leq 1/p})$)  with respect to the basis $w^{\lfloor n/p\rfloor}\binom zn$ (resp. $p^{\lfloor n/p \rfloor} \binom zn$), i.e. we have $w^{\lfloor n/p \rfloor}\binom zn |_{\Matrix \alpha\beta\gamma\delta}=\sum\limits_{m\geq 0} P_{m,n}\cdot w^{\lfloor m/p \rfloor}\binom zm$ (resp. $p^{\lfloor n/p \rfloor}\binom zn |_{\Matrix \alpha\beta\gamma\delta}=\sum\limits_{m\geq 0} P_{m,n}\cdot p^{\lfloor m/p \rfloor}\binom zm$)
 then
\[
P_{m,n} \in w^{\lfloor n/p\rfloor-\lfloor m/p\rfloor} (p, w)^{\max\{m-n,0\}} \subseteq w^{\lfloor n/p\rfloor-\lfloor m/p\rfloor+ \max\{0,m-n\}} \Lambda^{\geq 1/p}
\]
\[
\textrm{(resp. }P_{m,n} \in p^{\lfloor n/p\rfloor-\lfloor m/p\rfloor} (p, w)^{\max\{m-n,0\}} \subseteq p^{\lfloor n/p\rfloor-\lfloor m/p\rfloor+ \max\{0,m-n\}} \Lambda^{\leq 1/p}\quad ).
\]

For a fixed $n$, we let $m \to \infty$ and then the exponents on $w$ (resp. $p$) in the above estimations tend to $\infty$; so the subspace $\calC^0(\ZZ_p; \Lambda^{(\varepsilon), ?})^{\mathrm{mod}}$ is stable under the $\Iw_p$-action.
Now, if $\Matrix \alpha\beta\gamma\delta \in \Matrix{p\ZZ_p}{\ZZ_p}{p\ZZ_p}{\ZZ_p^\times}{}^{\det \neq 0}$, \cite[Proposition~3.14\,(1)]{liu-wan-xiao} gives a stronger estimate:
$$
P_{m,n} \in w^{\lfloor n/p\rfloor-\lfloor m/p\rfloor} (p, w)^{\max\{m-\lfloor n/p\rfloor,0\}} \subseteq w^{m-\lfloor m/p\rfloor} \Lambda^{\geq 1/p}
$$
\begin{equation}
\label{E:estimate of Pmn}
\textrm{(resp. }
P_{m,n} \in p^{\lfloor n/p\rfloor-\lfloor m/p\rfloor} (p, w)^{\max\{m-\lfloor n/p\rfloor,0\}} \subseteq p^{m-\lfloor m/p\rfloor} \Lambda^{\leq 1/p}
 \quad ).
\end{equation}
This shows the compactness of the $\Matrix{p\ZZ_p}{\ZZ_p}{p\ZZ_p}{\ZZ_p^\times}{}^{\det \neq 0}$-action on \eqref{E:small subspaces} and \eqref{E:small subspaces2}.

We further remark that, since the order of $\Delta$ is prime-to-$p$, an eigensubspace for the $\Delta$-action on $\calC^0(\ZZ_p; \Lambda^?)^\mathrm{mod}$ is a direct summand integrally.
So the  the action of a uniform limit of finite $\ZZ$-linear combinations of $\Matrix{p\ZZ_p}{\ZZ_p}{p\ZZ_p}{\ZZ_p^\times}{}^{\det \neq 0}$-actions on the $\Delta$-eigenspaces of \eqref{E:small subspaces} and \eqref{E:small subspaces2} is compact. (The action of  $\Matrix{p\ZZ_p}{\ZZ_p}{p\ZZ_p}{\ZZ_p^\times}{}^{\det \neq 0}$ may not preserve the $\Delta$-eigenspaces; here we meant the action after projecting back to the $\Delta$-eigenspaces.)

Now, we define
\[
\rmS_{p\textrm{-adic}}^{(\varepsilon),?}: = \Hom_{\calO[\Iw_p]} \big( \widetilde \rmH, \calC^0\big(\ZZ_p; \Lambda^{(\varepsilon), ?} \big)^{\mathrm{mod}}\big) \cong \bigoplus_{i=1}^r e_i^*\cdot  \calC^0\big(\ZZ_p; \Lambda^{(\varepsilon), ?} \big)^{\mathrm{mod}, \bar \rmT = \chi_i}.
\]
Then $\rmS_{p\textrm{-adic}}^{(\varepsilon),\geq 1/p}$ and $\rmS_{p\textrm{-adic}}^{(\varepsilon),\leq 1/p}$ clearly satisfy the condition \eqref{E:Mvarepsilon gluing}, because in $\Lambda^{=1/p}$, $p$ and $w$ are differed by a unit.
Moreover, the actions of $U_p$ on $\rmS_{p\textrm{-adic}}^{(\varepsilon),\geq 1/p}$ and $\rmS_{p\textrm{-adic}}^{(\varepsilon),\leq 1/p}$ are represented by $r \times r$-matrices whose entries are uniform limit of finite sums of actions given by elements from $\Matrix{p\ZZ_p}{\ZZ_p}{p\ZZ_p}{\ZZ_p^\times}{}^{\det \neq 0}$. So they are compact.

In conclusion, the $U_p$-actions on $\rmS_{p\textrm{-adic}}^{(\varepsilon),\geq 1/p}$ and $\rmS_{p\textrm{-adic}}^{(\varepsilon),\leq 1/p}$ are compact and are compatible when base changed to $\Lambda^{=1/p}$.
Since both spaces are completed (infinite) direct sum of copies of $\Lambda^{\geq 1/p}$ and $\Lambda^{\leq 1/p}$, respectively, the characteristic power series of $U_p$ (by choosing any orthonormal basis of $\rmS_{p\textrm{-adic}}^{(\varepsilon), ?}$ over $\Lambda^?$) is well-defined and is independent of the choice of the basis. They are
\begin{equation}
\label{E:char power series of Up}
\Char\big(U_p,\, \rmS_{p\textrm{-adic}}^{(\varepsilon),\geq 1/p}; t\big) \in \Lambda^{\geq 1/ p}\llbracket t \rrbracket  \quad \textrm{and} \quad \Char\big(U_p,\, \rmS_{p\textrm{-adic}}^{(\varepsilon),\leq 1/p}; t\big) \in \Lambda^{\leq 1/ p} \llbracket t \rrbracket.
\end{equation}
They are Fredholm series in the sense of \cite[D\'efinition~B.1]{AIP}, and agree in $\Lambda^{=1/p}\llbracket t \rrbracket$ because of the isomorphism \eqref{E:Mvarepsilon gluing}. So both power series in \eqref{E:char power series of Up} are equal and lie in $\Lambda\llbracket t \rrbracket$.
We denote it by
\begin{equation}
\label{E:char power series}
C^{(\varepsilon)}(w,t) =C^{(\varepsilon)}_{\widetilde \rmH}(w,t) = \sum_{n \geq 0} c_n^{(\varepsilon)}(w)t^n \in \Lambda\llbracket t \rrbracket = \calO\llbracket w,t \rrbracket.
\end{equation}
It is called the \emph{characteristic power series of the $U_p$-action on the abstract $p$-adic forms} $\rmS_{p\textrm{-adic}}^{(\varepsilon)}$.
This is the central object we study in this paper.

\begin{notation}
\label{N:spectral curve}
Let $\calW^{(\varepsilon)}$ denote the rigid analytic open unit disk associated to $\calO\llbracket w \rrbracket^{(\varepsilon)}: = \calO\llbracket \Delta \times \ZZ_p^\times\rrbracket  \otimes_{\calO[\Delta^2], \varepsilon} \calO$. It is the \emph{weight disk of character $\varepsilon$}.  
The set of points $\kappa \in \calW^{(\varepsilon)}(\CC_p)$ is the same as the set of continuous characters $\kappa: \Delta \times \ZZ_p^\times \to \calO_{\CC_p}^\times$ such that $\kappa|_{\Delta^2} = \varepsilon$.

The \emph{spectral curve} $\Spc_{\widetilde \rmH}^{(\varepsilon)} = \Spc^{(\varepsilon)}$ is the scheme theoretical zero locus of the characteristic power series $C^{(\varepsilon)}(w,t)$ in $\calW^{(\varepsilon)} \times \GG_m^\rig$. The spectral curve, as a closed subspace of $\calW^{(\varepsilon)} \times \GG_m^\rig$, carries two natural maps, the \emph{weight map} and the \emph{slope map}
\[
\wt: \Spc^{(\varepsilon)} \to \calW^{(\varepsilon)} \quad \textrm{and} \quad a_p: \Spc^{(\varepsilon)} \to \GG_m^\rig,
\]
sending a point $(w,t) \in \Spc^{(\varepsilon)}$ to $w \in \calW^{(\varepsilon)}$ and $t^{-1} \in \GG_m^\rig$, respectively.
\end{notation}


\begin{remark}
Following Andreatta--Iovita--Pilloni's adic Fredholm theory \cite{AIP}, one can reinterpret our construction as follows.
The $U_p$-action on $\rmS_{p\textrm{-adic}}^{(\varepsilon)}$ is not compact,
but thinking of it as a Banach sheaf over $\Spa(\Lambda, \Lambda)$, we constructed a Banach subsheaf by gluing $\rmS_{p\textrm{-adic}}^{(\varepsilon),\geq 1/p}$ over $\Spa(\Lambda^{\geq 1/p}, \Lambda^{\geq 1/p})$ and $\rmS_{p\textrm{-adic}}^{(\varepsilon),\leq 1/p}$ over $\Spa(\Lambda^{\leq 1/p}, \Lambda^{\leq 1/p})$, such that the $U_p$-action on this Banach subsheaf is compact.
Then the characteristic power series of the $U_p$-action on this Banach subsheaf is well-defined and has coefficients in $\Lambda$.
\end{remark}




\subsection{Family of abstract overconvergent forms}
\label{S:family of overconvergent forms}
The construction of abstract $p$-adic forms using $\calC^0(\ZZ_p; \Lambda^{(\varepsilon), ?})^\mathrm{mod}$ above allows us to prove that the characteristic power series $C^{(\varepsilon)}(w,t)$ has coefficients in $\calO\llbracket w\rrbracket$. To relate the (abstract) $p$-adic forms with the (abstract classical) forms, it is more natural to consider a construction using the space $\Lambda^{(\varepsilon),\leq 1/p}\langle z\rangle = \calO\langle w/p, z\rangle \otimes_{\Lambda} \Lambda^{(\varepsilon)}$, that is, the space of analytic functions on the closed unit disk with values in $\calO\langle w/p\rangle$.

We define the space of \emph{family of abstract overconvergent forms} to be
$$
\rmS^{\dagger, (\varepsilon)} = 
\rmS^{\dagger, (\varepsilon)}_{\widetilde \rmH}: = \Hom_{\calO[\Iw_p]}\big( \widetilde \rmH,\Lambda^{(\varepsilon),\leq 1/p}\langle z\rangle \big),
$$
where $\Matrix\alpha \beta \gamma \delta \in \bfM_1$ acts on the right hand side by the same formula as \eqref{E:induced representation action extended}, i.e. if $\alpha \delta - \beta \gamma = p^r d$ with $d \in \ZZ_p^\times$, then
\begin{align}
\label{E:action on overconvergent forms}
&h\big|_{\Matrix \alpha\beta{\gamma}\delta}(z) = \varepsilon(\bar d / \bar \delta, \bar \delta) \cdot (1+w)^{\log\left((\gamma z+\delta )/ \omega(\bar \delta)\right) / p} \cdot h\Big( \frac{\alpha z+\beta}{\gamma z+\delta}\Big)
\\
\nonumber
&=\; \varepsilon(\bar d/ \bar \delta, \bar \delta) \cdot h\Big( \frac{\alpha z+\beta}{\gamma z+\delta}\Big) \cdot \sum_{n\geq 0}  \binom {\log\left((\gamma z+\delta )/ \omega(\bar \delta)\right) / p}n p^n\Big(\frac wp\Big)^n  \ \in \ \calO\langle w/p, z\rangle.
\end{align}
We define the $U_p$-operator on $\rmS^{\dagger, (\varepsilon)}$
via the same formula as in \eqref{E:Up action}.

Note that we have a natural inclusion
\begin{equation}
\label{E:O<z> subset of C(Zp)small}
\Lambda^{(\varepsilon),\leq 1/p}\langle z\rangle \subset \calC^0\big(\ZZ_p;  \Lambda^{(\varepsilon), \leq 1/p}\big)^{\textrm{mod}},
\end{equation}
because each power basis element $z^i$ can be written as a $\ZZ_p$-linear combination of $p^{\lfloor n/p\rfloor}\binom zn$'s.  
In addition, this inclusion is equivariant for the $\bfM_1$-action, inducing a natural $U_p$-equivariant inclusion
$$ \rmS^{\dagger,(\varepsilon)} \subseteq \rmS_{p\textrm{-adic}}^{(\varepsilon), \leq 1/p}.
$$

\begin{lemma}
\label{L:char power series agree}
The $U_p$-action on $\rmS^{\dagger, (\varepsilon)}$ is compact, and the characteristic power series  of this $U_p$-action agrees with $C^{(\varepsilon)}(w, t)$ in $\Lambda^{\leq 1/p}\llbracket t\rrbracket$.
\end{lemma}
\begin{proof}
Consider a different subspace:
$$
\calC^0\big(\ZZ_p; \Lambda^{(\varepsilon),\leq 1/p}\big)^{\mathrm{mod},\flat} := \widehat \bigoplus_{n \geq 0} p^{\lfloor n/p\rfloor + n/2} \tbinom zn\cdot \Lambda^{(\varepsilon),\leq 1/p} \subset \calC^0\big(\ZZ_p; \Lambda^{(\varepsilon),\leq 1/p}\big)
$$
with basis $p^{\lfloor n / p \rfloor +n/2}\binom zn$ instead. 
It is easy to see the inclusions
$$ \calC^0\big(\ZZ_p;  \Lambda^{(\varepsilon), \leq 1/p}\big)^{\textrm{mod},\flat} \subset
\Lambda^{(\varepsilon),\leq 1/p}\langle z\rangle \subset \calC^0\big(\ZZ_p;  \Lambda^{(\varepsilon), \leq 1/p}\big)^{\textrm{mod}},
$$ which is
compatible with the $\bfM_1$-actions. Moreover, if $\Matrix \alpha \beta \gamma \delta \in {\Matrix{p\ZZ_p}{\ZZ_p}{p\ZZ_p}{\ZZ_p^\times}}^{\det \neq 0}$, then the estimate \eqref{E:estimate of Pmn} implies that $|_{\Matrix \alpha \beta \gamma \delta}$ takes $\calC^0\big(\ZZ_p;  \Lambda^{(\varepsilon), \leq 1/p}\big)^{\textrm{mod}}$ into the subspace $\calC^0\big(\ZZ_p;  \Lambda^{(\varepsilon), \leq 1/p}\big)^{\textrm{mod},\flat}$ because $n -\lfloor n/p\rfloor \geq n/2$.

The sequence of inclusions above induces $U_p$-equivariant and compact inclusions
$$
\rmS_{p\textrm{-}\mathrm{adic}}^{(\varepsilon), \leq 1/p,\flat}: =  \Hom_{\calO[\Iw_p]}(\widetilde \rmH, \calC^0\big(\ZZ_p; \Lambda^{(\varepsilon),\leq 1/p}\big)^{\mathrm{mod},\flat}) \subseteq \rmS^{\dagger, (\varepsilon)} \subseteq \rmS_{p\textrm{-adic}}^{(\varepsilon), \leq 1/p}.
$$
Yet the $U_p$-operator on $ \rmS_{p\textrm{-adic}}^{(\varepsilon), \leq 1/p}$ factors through the inclusion $ \rmS_{p\textrm{-adic}}^{(\varepsilon), \leq 1/p,\flat} \subseteq \rmS_{p\textrm{-adic}}^{(\varepsilon), \leq 1/p}$; namely the following diagram commutes
$$
\begin{tikzcd}
\rmS_{p\textrm{-adic}}^{(\varepsilon), \leq 1/p, \flat} \ar[r, hookrightarrow] \ar[d, "U_p"]& \rmS^{\dagger, (\varepsilon)} \ar[r, hookrightarrow] & \rmS_{p\textrm{-adic}}^{(\varepsilon), \leq 1/p} \ar[dll, dashed, "\psi"'] \ar[d, "U_p"]
\\
\rmS_{p\textrm{-adic}}^{(\varepsilon), \leq 1/p, \flat} \ar[r, hookrightarrow] & \rmS^{\dagger, (\varepsilon)} \ar[r, hookrightarrow] & \rmS_{p\textrm{-adic}}^{(\varepsilon), \leq 1/p}
\end{tikzcd}
$$
so that the compositions $\rmS^{\dagger, (\varepsilon)} \to \rmS^{\dagger, (\varepsilon)}$ is the $U_p$-operator. Now, the lemma follows from the link property of Buzzard \cite[Lemma~2.7]{buzzard}.
\end{proof}

\begin{remark}
There is no analogous construction of $\rmS^{\dagger, (\varepsilon)}$ over the entire boundary of the weight disk $\Spa(\Lambda^{(\varepsilon), \geq 1/p}, \Lambda^{(\varepsilon), \geq 1/p})$, because for the expression \eqref{E:induced representation action extended} to converge in power series, we need to require $v_p(w)$ to be bigger than at least a fixed positive number.
\end{remark}

\begin{notation}
\label{N:tilde eta}
For the rest of this paper, \emph{the weights of our classical forms will be $k$, and the corresponding point on the weight disk is $w_k: = \exp(p(k-2))-1$.}  (This may appear cumbersome in this and the next sections, yet this choice seems to be the best for the rest of the paper.)

For a character $\eta$ of $\Delta$, we write $\tilde \eta$ for the character $\eta \times \eta$ of $\Delta^2$. For an integer $k \geq 2$, we write $\calO[z]^{\leq k-2}$ for the space of polynomials with coefficients in $\calO$ of degree $\leq k-2$.
\end{notation}

\subsection{Abstract classical and overconvergent forms}
\label{S:abstract classical forms}
For an integer $k\geq 2$ and a character $\psi: \Delta^2 \to \calO^\times$, the space of \emph{abstract overconvergent forms of weight $k$ and character $\psi$} is defined to be
\[
\rmS_{k}^\dagger(\psi)=\rmS_{\widetilde \rmH, k}^{\dagger}(\psi): =  \Hom_{\calO[\Iw_p ]}\big(\widetilde \rmH,\,\calO\langle z\rangle\otimes \psi\big),
\]
and its subspace of \emph{abstract classical forms of weight $k$ and character $\psi$} is
\[\rmS_{k}^\Iw(\psi)=\rmS_{\widetilde \rmH, k}^{\Iw}(\psi) 
: =  \Hom_{\calO[\Iw_p ]}\big(\widetilde \rmH,\, \calO[z]^{\leq k-2} \otimes \psi\big).
\]
Here a matrix $\Matrix \alpha \beta \gamma \delta \in \Iw_p$, or more generally $\Matrix \alpha \beta \gamma \delta  \in \bfM_1$, with determinant $p^r d$ (for $r \in \ZZ$ and $d\in \ZZ_p^\times$) acts \emph{from the right} on $\calO\langle z \rangle$ by
\begin{equation}
\label{E:modular action of weight k}
f|_{\Matrix \alpha \beta \gamma \delta} (z): =   ( \gamma z + \delta)^{k-2} \cdot f \big(\frac{\alpha z + \beta }{\gamma z + \delta }\big),
\end{equation}
stabilizing the subspace  $\calO[z]^{\leq k-2}$; and we view $\psi$ as a character of $\bfM_1$ by setting 
\begin{equation}
\label{E:psi as a character of M1}\psi\big( \Matrix \alpha \beta \gamma \delta\big) = \psi(\bar d/ \bar \delta, \bar \delta).
\end{equation}
The space $\calO[z]^{\leq k-2} \otimes \psi$ is a finite free $\calO$-module, and is endowed with the $p$-adic topology.

We define the $U_p$-operator on $\rmS_{k}^{\Iw}(\psi) \subset \rmS_k^\dagger(\psi)$  via the same formula as in \eqref{E:Up action}. In particular, $\rmS_{k}^{\Iw}(\psi)$ is a subspace of $\rmS_k^\dagger(\psi)$ invariant under the $U_p$-operator.

Set $\varepsilon:= \varepsilon(k,\psi): = \psi\cdot (1\times \omega^{k-2})$ as a character of $\Delta^2$. Then we have a natural isomorphism
\begin{equation}
\label{E:family overconvergent forms specialization}\rmS_k^{\dagger}(\psi)\cong
\rmS^{\dagger, (\varepsilon)} \otimes_{\Lambda^{\leq 1/p}, w \mapsto w_k} \calO,
\end{equation}
because  the action \eqref{E:action on overconvergent forms} exactly specializes to 
the tensor product of \eqref{E:modular action of weight k} and \eqref{E:psi as a character of M1} per the definition of $\varepsilon$.  
An immediate consequence of this is that the characteristic power series of the $U_p$-action on $\rmS_{k}^\Iw(\psi)$ (which is in fact a polynomial) divides the characteristic power series $C^{(\varepsilon)}(w_k, t)$.  We defer further discussions of these inclusions and classicality to  \S\,\ref{SS:power basis} and Proposition~\ref{P:theta map}.

\medskip
For a character $\eta: \Delta \to \calO^\times$, we define the space of \emph{abstract classical forms of weight $k$ with full level structure} to be
$$
\rmS_k^\unr(\eta): = \Hom_{\calO[\rmK_p]}\big( \widetilde \rmH, \calO[z]^{\leq k-2} \otimes \eta \circ \det\big),
$$
where a matrix $\Matrix \alpha \beta \gamma \delta \in \rmM_2(\ZZ_p)^{\det \neq 0}$ with determinant $p^r d$ (for $r \in \ZZ$ and $d\in \ZZ_p^\times$) acts from the right on $\calO[z]^{\leq k-2}$ by
\begin{equation}
\label{E:modular action of weight k GL2}
f|_{\Matrix \alpha \beta \gamma \delta}(z) : =  \eta(\bar d) \cdot ( \gamma z + \delta)^{k-2} \cdot f \big(\frac{\alpha z + \beta }{\gamma z + \delta }\big),
\end{equation}
and we view $\eta \circ \det$ as a character of $\rmM_2(\ZZ_p)^{\det \neq 0}$ by setting 
\begin{equation}
\label{E:eta det as a character of M2(Zp)}
\eta\circ \det \big( \Matrix \alpha \beta \gamma \delta \big) = \eta (\bar d).
\end{equation}
As before, we endow the space $\calO[z]^{\leq k-2} \otimes \eta $ with the $p$-adic topology.

This space carries an action of $T_p$-operator:
taking a  coset decomposition $\rmK_p \Matrix p001 \rmK_p = \coprod_{j=0}^p \rmK_p u_j$ (e.g. $u_j = \Matrix 1j0p$ for $j=0, \dots, p-1$ and $u_p = \Matrix p001$), then
$$
T_p(\varphi)(m) = \sum_{j=0}^p \varphi(mu_j^{-1})|_{u_j} \quad \textrm{for all } m \in \widetilde \rmH.$$
The $T_p$-operator does not depend on the choice of the coset decomposition.

\begin{notation}
Since $\widetilde \rmH$ is a projective $\calO\llbracket \rmK_p\rrbracket$-module, these spaces $\rmS_{k}^\unr(\eta)$ and $\rmS_{k}^\Iw(\psi)$ are finite free $\calO$-modules (but possibly zero). We record their ranks:
\[
d_{k}^\unr(\eta): = \rank_\calO \rmS_{k}^\unr(\eta) \quad \textrm{and} \quad d_{k}^\Iw(\psi): = \rank_\calO \rmS_{k}^\Iw(\psi).
\]

The character $\eta$ induces a character $\tilde \eta$ of $\Delta^2$ per Notation~\ref{N:tilde eta}; it is straightforward to see that $\rmS_k^\unr(\eta) \subseteq \rmS_k^\Iw(\tilde \eta)$. (We defer to the forthcoming paper for a detailed study of the corresponding $p$-stabilization process, which is a key input to proving local ghost conjecture.) For now, we simply put 
\[
d_{k}^{\new}(\eta) : = d_{k}^\Iw(\tilde \eta) - 2 d_{k}^\unr(\eta).
\]
We will call this the \emph{rank of abstract $p$-new forms}.
\end{notation}

\begin{remark}
\label{R:abstract vs classical}
If $D$ is a definite quaternion algebra $D$ over $\QQ$ split at $p$, and if $\widetilde \rmH = \varprojlim_n \calO\big[ D^\times \backslash D(\AAA_f)^\times / K^p (1+ p^n \rmM_2(\ZZ_p))\big]$, then $\rmS_{k}^\unr(\mathrm{triv})$ is the same as the space of automorphic forms $S_{k}^D(K^p\rmK_p)$ on $D$ with level $K^p\rmK_p$ and weight $k$ defined in \cite[\S~4]{buzzard2}. Similarly,  $\rmS_{k}^\Iw(\psi)$ is the same as $S_{k}^D(K^p\Iw_p; \psi)$, the space of automorphic forms of level $K^p \Iw_p$ and nebentypus character $\psi$ (cf. \cite[$\mathit{loc. cit.}$]{buzzard2}).  The similar statement holds for the space of abstract overconvergent forms.

One can also define abstract classical forms of higher level and more general nebentypus characters. We will not make essential use of them in this paper; so we leave that to interested readers.
\end{remark}

\begin{remark}
\label{R:homogenized classical forms}
One can consider a homogenized version of the space of abstract classical forms; explicitly, we have natural isomorphisms:
\[
\rmS_{k}^\Iw(\psi)\cong \Hom_{\calO[\Iw_p ]}\big(\widetilde \rmH,\, \Sym^{k-2}\calO^{\oplus 2}  \otimes \psi\big)\  \textrm{and}\ 
\rmS_{k}^\unr(\eta) \cong \Hom_{\calO[ \rmK_p]}(\widetilde \rmH,\, \Sym^{k-2}\calO^{\oplus 2} \otimes \eta \circ \det).
\]
Here we  view $\psi$ as a character of $\Iw_p$ (resp. $\eta \circ \det$ as a character of $\GL_2(\ZZ_p)$) via \eqref{E:psi as a character of M1} (resp. \eqref{E:eta det as a character of M2(Zp)}). The space $\Sym^{k-2} \calO^{\oplus 2}$ is provided with the natural action of $\rmK_p$ from the right, namely,
\begin{equation}
\label{E:right action on Sym}
\tilde h \big|_{\Matrix \alpha \beta \gamma \delta} (X, Y): = \tilde h(\alpha X + \beta Y, \gamma X + \delta Y).
\end{equation}
(This is the same as the transpose of the left action considered in \cite{breuil} and \cite{paskunas}.)

The first isomorphism above follows from the identification
\begin{equation}
\label{E:sym corresponds to polynomials}
\begin{tikzcd}[row sep=0pt]
\Sym^{k-2}\calO^{\oplus 2} \otimes \psi \ar[r, leftrightarrow] & \calO[z]^{\leq k-2} \otimes \psi\\
\sum\limits_{i =0}^{k-2} a_i X^i Y^{k-2-i} \ar[r, leftrightarrow] & \sum\limits_{i=0}^{k-2} a_i z^i.
\end{tikzcd}
\end{equation}
equivariant for the $\Iw_p$-action \eqref{E:right action on Sym} above on the left hand side and the action~\eqref{E:modular action of weight k} on the right hand side. The second isomorphism follows from similar identifications.
\end{remark}

\begin{notation}
\label{N:Serre weights}
For a pair of non-negative integers $(a,b)$, we use $\sigma_{a,b}$ to denote the \emph{right} $\FF$-representation $\Sym^a\FF^{\oplus 2} \otimes \det^b$ of $\GL_2(\FF_p)$; (in particular, $b$ is taken modulo $p-1$.) As argued in Remark~\ref{R:homogenized classical forms}, we may identify the representation $\sigma_{a,b}$ with the
$\FF$-vector space $\FF[z]^{\deg \leq a}$ of polynomials of degree less than or equal to $a$, where the $\GL_2(\FF_p)$-action is given by
$$
h|_{\Matrix {\bar \alpha}{\bar\beta}{\bar\gamma}{\bar\delta}}(z) =(\bar\alpha\bar\delta-\bar\beta\bar\gamma)^b \cdot  (\bar\gamma z+\bar\delta)^{a}\cdot  h\Big( \frac{\bar\alpha z+\bar\beta}{\bar\gamma z+\bar\delta}\Big), \quad \textrm{for }\Matrix {\bar \alpha}{\bar\beta}{\bar\gamma}{\bar\delta} \in \GL_2(\FF_p).
$$

When $a \in \{0, \dots, p-1\}$ and $b \in \{0, \dots, p-2\}$, $\sigma_{a,b}$ is irreducible. These representations exhaust all irreducible (right) $\FF$-representations of $\GL_2(\FF_p)$ (\cite[Proposition~1]{Barthel-Livne}). We call them the \emph{Serre weights}.
We use $\Proj_{a,b}$ to denote the projective envelope of $\sigma_{a,b}$ as a (right) $\FF[\GL_2(\FF_p)]$-module. We refer to Lemma~\ref{L:projective envelope} in the appendix for an explicit description of $\Proj_{a,b}$.
\end{notation}

Now, we switch to stating the local version of Bergdall and Pollack's ghost conjecture. This is largely inspired by their work
\cite{bergdall-pollack2,bergdall-pollack3}, with one difference: we confine ourselves to the reducible, non-split, and generic type below.

\begin{definition}
\label{D:primitive type}
For the rest of this paper, fix a \emph{reducible, nonsplit, and generic} residual representation $\bar\rho: \rmI_{\QQ_p} \to \GL_2(\FF)$ of the inertia subgroup:
\begin{equation}
\label{E:bar rhop}
\bar \rho \simeq \MATRIX{\omega_1^{a+b+1}}{*\neq 0}{0}{\omega_1^b} \qquad \textrm{for } 1 \leq a \leq p-4 \textrm{ and } 0 \leq b \leq p-2,
\end{equation}
where $\omega_1$ is the first fundamental character (recall $p\geq 5$).
We say an $\calO\llbracket K_p\rrbracket$-projective augmented module $\widetilde \rmH$ is of \emph{type $\bar\rho$} with \emph{multiplicity} $m(\widetilde \rmH) \in \NN$ if
\begin{enumerate}
\item (Serre weight)
$\overline \rmH: = \widetilde \rmH / (\varpi, \rmI_{1+p\rmM_2(\ZZ_p)})$ is isomorphic to a direct sum of $m(\widetilde \rmH)$ copies of $\Proj_{a,b}$ as a right $\FF[\GL_2(\FF_p)]$-module, and

\item (Central character \uppercase\expandafter{\romannumeral1})
the action of $\Matrix p00p$ on $\widetilde \rmH$ is the multiplication by an invertible element $\xi \in \calO^\times$.
\item (Central character \uppercase\expandafter{\romannumeral2}) there exists an isomorphism $\widetilde \rmH \cong \widetilde \rmH_0 \widehat \otimes_\calO \calO\llbracket (1+p\ZZ_p)^\times \rrbracket$ of  $\calO[\GL_2(\QQ_p)]$-modules, where $\widetilde \rmH_0$ carries an action of $\GL_2(\QQ_p)$ which is trivial on elements of the form $\Matrix \alpha 00 \alpha$ for $\alpha \in (1+p\ZZ_p)^\times$, and the latter factor $\calO\llbracket (1+p\ZZ_p)^\times \rrbracket$ carries the natural action of $\GL_2(\QQ_p)$ through the map $\GL_2(\QQ_p) \xrightarrow{\det}\QQ_p^\times \xrightarrow{p^r\delta \mapsto \delta/\omega(\bar d)} (1+p\ZZ_p)^\times$.
\end{enumerate}

We say $\widetilde \rmH$ is \emph{primitive} if $m(\widetilde \rmH) = 1$.
\end{definition}

\begin{remark}\fakephantomsection\label{R:after the definition of completed homology piece}
\begin{enumerate}
\item
If $D$ is a definite quaternion algebra over $\QQ$ that is split at $p$ and if $\bar r: \Gal_\QQ \to \GL_2(\FF)$ is an absolutely irreducible global residual Galois representation such that $\bar r|_{\rmI_{\QQ_p}} \cong \bar \rho$, then the $\bar r$-localized completed homology
$$
\widetilde \rmH: = \varprojlim_n \calO\big[ D^\times \backslash D(\AAA_f)^\times / K^p (1+ p^n \rmM_2(\ZZ_p))\big]_{\gothm_{\bar r}}
$$
is an $\calO\llbracket K_p\rrbracket$-projective augmented module of type $\bar\rho$; the primitive condition is equivalent to the mod-$p$ multiplicity-one hypothesis: $\rmS_2^D(K^p \Iw_p; \omega^{a+b} \times \omega^b)$ has rank $1$ over $\calO$. Indeed, condition (1) of Definition~\ref{D:primitive type} is determined by the Serre weight conjecture. The other two conditions are natural properties of central characters.  One have similar construction for modular curves, as indicated in the introduction.

\item If $\tilde{\rmH}$ is primitive, we will see in Proposition~\ref{P:reducible => ordinary} that the space $\rmS_2^{\Iw}(\omega^b \times \omega^{a+b})$ has rank $1$ over $\calO$ and the $U_p$-action on it is given by an element in $\calO^\times$. In particular, this space is ordinary. Since we will not use this result in this paper, we leave it to the Appendix (Proposition~\ref{P:reducible => ordinary}). For a general $\tilde{\rmH}$, we still expect the space $\rmS_2^{\Iw}(\omega^b \times \omega^{a+b})$ is ordinary although we cannot give a proof in this paper.

One can see why this should be true as follows: in the automorphic situation above, 
the mod-$p$-multiplicity-one condition automatically puts us in the situation with one Serre weight, which (in many cases) means
$\bar r|_{\rmI_{\QQ_p}} \cong \bar \rho$ for some $\bar\rho$ above; in this case, the ordinary condition is automatic. Thus, it is conceivable that such ordinarity result can be deduced directly on purely on $\calO\llbracket K_p\rrbracket$-projective augmented modules.
\end{enumerate}
\end{remark}

We shall primarily focus on the case of a primitive $\calO\llbracket K_p\rrbracket$-projective augmented modules, but we shall indicate from time to time how our theory adapts to the non-primitive case.

\begin{notation}
\label{N:relevant varepsilon}
For the rest of this paper, we fix a \emph{primitive} $\calO\llbracket K_p\rrbracket$-projective augmented module $\widetilde \rmH$ of type $\bar \rho$, unless otherwise specified. In particular, $\overline \rmH = \widetilde \rmH/ (\varpi, \rmI_{1+p \rmM_2(\ZZ_p)}) \cong \Proj_{a,b}$.
\begin{enumerate}
\item
A character $\varepsilon$ of $\Delta^2$ is called \emph{relevant} to $\bar \rho$ if $\varepsilon(\bar \alpha, \bar \alpha) = \bar \alpha^{a+2b}$ for all $\bar \alpha \in \Delta$.
By comparing the central action of $\Matrix {\omega(\bar \alpha)}00{\omega(\bar \alpha)}$ for $\bar \alpha \in \Delta$ on $\widetilde \rmH$ and on $\Ind_{B^\op(\ZZ_p)}^{\Iw_p}(\chi_\univ^{(\varepsilon)})$, one sees that  $\rmS_{p\textrm{-adic}}^{(\varepsilon)}$ as defined in \eqref{E:Sp-adic epsilon} is nonzero if and only if $\varepsilon$ is relevant to $\bar \rho$.

\item
We write each relevant character $\varepsilon$ of $\Delta^2$ in the form
$$
\varepsilon =\varepsilon_1 \times \varepsilon_2 =  \omega^{-s_\varepsilon+b} \times \omega^{a+s_\varepsilon + b}
$$
for some $s_\varepsilon \in \{0, \dots, p-2\}$.
A very important invariant we shall use later is
\begin{equation}
\label{E:kepsilon}
k_\varepsilon : = 2+\{a+2s_\varepsilon\} \ \in \ \{2,3, \dots, p\}.
\end{equation}
here $\{ a+2s_\varepsilon \}$ is the unique integer in $\{0,\dots, p-2 \}$ satisfying the condition 
$$a+2s_\varepsilon\equiv \{ a+2s_\varepsilon \} \mod p-1 .$$
If we write $\tilde \varepsilon_1$ for the character $\varepsilon_1\times \varepsilon_1$ of $\Delta^2$ (per Notation~\ref{N:tilde eta}), then
\begin{equation}
\label{E:epsilon as omega times (1times k)}
\varepsilon = \tilde \varepsilon_1 \cdot (1\times \omega^{k_\varepsilon-2}).
\end{equation}

\item
By the discussion in \S\,\ref{S:abstract classical forms} and \eqref{E:epsilon as omega times (1times k)}, for integer $k \geq 2$ such that $k \equiv k_\varepsilon \bmod{(p-1)}$, the space of abstract classical forms $\rmS_{k}^\Iw(\tilde \varepsilon_1)$ and hence $\rmS_{k}^\unr( \varepsilon_1)$ are subspaces of  $\rmS_{p\textrm{-}\mathrm{adic}}^{(\varepsilon)} \otimes_{\calO\llbracket w\rrbracket , w \mapsto w_k} \calO$, where we recall that $w_k = \exp((k-2)p)-1$.
\end{enumerate}

\emph{For the rest of this section, we fix a character $\varepsilon$ of $\Delta^2$ relevant to $\bar \rho$.}
\end{notation}

\begin{definition}
\label{D:ghost series}
Following \cite{bergdall-pollack2,bergdall-pollack3}, we define the \emph{ghost series} for $\widetilde \rmH$ over $\calW^{(\varepsilon)}$  to be the formal power series
$$
G^{(\varepsilon)}(w,t)=G^{(\varepsilon)}_{\widetilde\rmH}(w,t) = 1+\sum_{n=1}^\infty
g_n^{(\varepsilon)}(w)t^n\in \calO[w]\llbracket t\rrbracket,
$$
where each coefficient $g_n^{(\varepsilon)}(w)$ is a product
\begin{equation}
\label{E:gi varepsilon}
g_n^{(\varepsilon)}(w) = \prod_{\substack{k \geq 2\\ k \equiv k_\varepsilon \bmod{p-1}}} (w - w_k)^{m_n^{(\varepsilon)}(k)} \in \ZZ_p[w]
\end{equation}
with exponents $m_n^{(\varepsilon)}(k)$ given by the following recipe
\[
m_n^{(\varepsilon)}(k) = \begin{cases}
\min\big\{ n - d_{k}^\unr( \varepsilon_1) , d_{k}^\Iw(\tilde \varepsilon_1)- d_{k}^\unr(\varepsilon_1) - n\big\} & \textrm{ if }d_{k}^\unr(\varepsilon_1) < n < d_{k}^\Iw(\tilde\varepsilon_1) - d_{k}^\unr(\varepsilon_1)
\\
0 & \textrm{ otherwise.}
\end{cases}
\]
In particular, for a fixed $k$, the sequence $(m_n^{(\varepsilon)}(k))_{n \geq 1}$ is given by the following palindromic pattern
\begin{equation}
\label{E:cascading pattern}
\underbrace{0, \dots, 0}_{d_{k}^{\unr}(\varepsilon_1)}, 1, 2, 3, \dots, \tfrac12 d_{k}^{\textrm{new}}(\varepsilon_1) -1, \tfrac12{d_{k}^{\textrm{new}}(\varepsilon_1)}, \tfrac12{d_{k}^{\textrm{new}}(\varepsilon_1)} -1,\dots,  3, 2, 1, 0, 0, \dots,
\end{equation}
where the maximum $\frac 12 d_{k}^{\new}(\varepsilon_1)$ appears at the $\frac 12 d_{k}^{\Iw}(\tilde\varepsilon_1)$th place.
(We shall see later in Corollary~\ref{C:dIw is even} that $d_k^\mathrm{new}(\varepsilon_1) = d_k^\Iw(\tilde \varepsilon_1) - 2d_k^\unr(\varepsilon_1) $ is always an even integer.)
As we shall late prove in Corollary~\ref{C:ghost series depends on bar rho} that, for a fixed $n$, $m_n^{(\varepsilon)}(k)$ is only nonzero for finitely many $k$. So the product \eqref{E:gi varepsilon} defining $g_n^{(\varepsilon)}(w)$ is a finite product.


We remark that, the ghost pattern \eqref{E:cascading pattern} implies that, for any integer $k \geq 2$ such that $k \equiv k_\varepsilon \bmod {(p-1)}$, we have, for $n\geq 1$,
\begin{equation}
\label{E:increment of multiplicity}
m_{n+1}^{(\varepsilon)}(k) -m_{n}^{(\varepsilon)}(k) = \begin{cases}
1 & \textrm{ if }  d_{k}^\unr(\varepsilon_1)  \leq n< \frac 12 d_{k}^\Iw(\tilde \varepsilon_1)
\\
-1 & \textrm{ if }\frac 12 d_{k}^\Iw(\tilde \varepsilon_1) \leq n < d_{k}^\Iw(\tilde \varepsilon_1) - d_{k}^\unr(\varepsilon_1)
\\
0 & \textrm{ otherwise,}
\end{cases}
\end{equation}
and for $n \geq 2$,
\begin{equation}
\label{E:second order increment of multiplicity}
m_{n+1}^{(\varepsilon)}(k) -2m_n^{(\varepsilon)}(k) +m_{n-1}^{(\varepsilon)}(k) = \begin{cases}
-2 & \textrm{ if }  n = \frac 12 d_{k}^\Iw(\tilde \varepsilon_1)
\\
1 & \textrm{ if }n =  d_{k}^\unr(\varepsilon_1) \textrm{ or } d_{k}^\unr(\varepsilon_1)+d_{k}^{\new}(\varepsilon_1)
\\
0 & \textrm{ otherwise.}
\end{cases}
\end{equation}
\end{definition}



\begin{example}
\label{Ex:p=7a=3}
We give an example of the ghost series and refer to \S\,\ref{Sec:properties of ghost series} for the computation of the dimensions.
Suppose that $p=7$, $a=2$, $b=0$, and that $\widetilde \rmH$ is primitive of type $\bar\rho = \Matrix{\omega_1^3}{*\neq 0}{0}{1}$; so $\overline \rmH \cong \Proj_{2,0}$. We list below the dimensions $d_k^{\Iw}(\psi)$ for small $k$'s.

\begin{center}
\renewcommand{\arraystretch}{1.2}
\begin{tabular}{|c|c|c|c|c|c|c|c|c|c|c|c|c|c|c|c|}
\hline $\varepsilon$ &
$k$ & 2& 3 & 4 &  5 & 6 & 7 & 8 & 9 & 10 & 11 & 12&13&14 
\\ \hline
$1 \times \omega^2$ &
$d_{k}^\Iw(1 \times \omega^{4-k}) = \lfloor \frac{k+2}6\rfloor + \lfloor \frac{k+4}6\rfloor$ &
1& 1& 2${}^*$& 2& 2& 2& 3 & 3& 4${}^*$ & 4& 4& 4& 5 
\\ \hline
$\omega^5 \times \omega^3$ &
$d_{k}^\Iw(\omega^5 \times \omega^{5-k})= \lfloor \frac{k+1}6\rfloor + \lfloor \frac{k+3}6\rfloor$ & 0&
1& 1& 2& 2${}^*$& 2& 2& 3 & 3& 4 & 4${}^*$& 4& 4
\\ \hline
$\omega^4 \times \omega^4$ &
$d_{k}^\Iw(\omega^4 \times \omega^{-k})= \lfloor \frac{k}6\rfloor + \lfloor \frac{k+2}6\rfloor$ & 0${}^*$& 0&
1& 1& 2& 2& 2${}^*$& 2& 3 & 3& 4 & 4& 4${}^*$
\\ \hline
$\omega^3 \times \omega^5$ &
$d_{k}^\Iw(\omega^3 \times \omega^{1-k}) = \lfloor \frac{k-1}6\rfloor + \lfloor \frac{k+1}6\rfloor$ & 0 & 0 & 0${}^*$&
1& 1& 2& 2& 2& 2${}^*$& 3 & 3& 4 & 4
\\ \hline
$\omega^2 \times 1$ &
$d_{k}^\Iw(\omega^2 \times \omega^{2-k}) =  \lfloor \frac{k+4}6\rfloor + \lfloor \frac{k}6\rfloor$ & 1& 1& 1& 1& 2${}^*$& 2& 3& 3& 3& 3& 4${}^*$& 4& 5
\\ \hline
$\omega \times \omega$ &
$d_{k}^\Iw(\omega \times \omega^{3-k}) = \lfloor \frac{k+3}6\rfloor + \lfloor \frac{k-1}6\rfloor$ & 0${}^*$&  1& 1& 1& 1& 2& 2${}^*$& 3& 3& 3& 3& 4& 4${}^*$
\\ \hline
\end{tabular}
\end{center}

The superscript $*$ indicates where the character $\psi$ is equal to $\tilde \varepsilon_1$, in which case $d_{k}^\unr(\varepsilon_1)$ makes sense.
In the table below, we list the information on dimensions of abstract classical forms with level $\rmK_p$ and $\Iw_p$.
\begin{center}\renewcommand{\arraystretch}{1.2}
\begin{tabular}{|c|c|c|c|c|c|c|c|c|c|c|c|c|c|c|c|}
\hline
$\varepsilon$ & \multicolumn{7}{|c|}{Triples $\big(k, \ d_{k}^\unr(\varepsilon_1),\ d_{k}^{\textrm{new}}(\varepsilon_1)\big)$ on the corresponding weight disk}
\\
\hline
$1 \times \omega^2$ & 
 $(4, 1, 0)$  & $(10, 1, 2) $ & $(16, 1, 4)$ & $(22, 1, 6)$ & $(28, 2, 6)$ & $(34, 2,8)$ & $(40, 2, 10)$ \\
 \hline
 $\omega^5 \times \omega^3$ &
 $(6, 0, 2)$ & $(12, 1, 2)$ & $(18, 1, 4)$ & $(24, 1, 6)$ & $(30, 1, 8)$ & $(36, 2, 8)$ & $(42, 2, 10)$
\\ \hline
 $\omega^4 \times \omega^4$ &
  $(2,0,0)$  & $(8,0, 2)$ & $(14, 0, 4)$ & $(20, 1, 4)$ & $(26, 1, 6)$ & $(32, 1, 8)$ & $(38, 1, 10)$
\\ \hline
 $\omega^3 \times \omega^5$ &
$(4, 0, 0)$  & $(10, 0, 2) $ & $(16, 0, 4)$ & $(22, 0, 6)$ & $(28, 1, 6)$& $(34, 1,8)$ & $(40, 1, 10)$
\\ \hline
 $\omega^2 \times 1$ &
$(6, 0, 2)$ & $(12, 1, 2)$ & $(18, 1, 4)$ & $(24, 1, 6)$ & $(30, 1, 8)$& $(36, 2, 8)$ & $(42, 2, 10)$
\\ \hline
 $\omega \times \omega$ &
$(2,0,0)$  & $(8,0, 2)$ & $(14, 0, 4)$ & $(20, 1, 4)$ & $(26, 1, 6)$ & $(32, 1, 8)$ & $(38, 1, 10)$
\\ \hline
\end{tabular}
\end{center}

Unwinding the definition of ghost series in Definition~\ref{D:ghost series}, we give  the first four terms of the ghost series on the $\varepsilon = (1 \times \omega^2)$-weight disk (corresponding to the first rows in the above two tables).
\begin{align*}
g_1^{(\varepsilon)}(w) &=1,
\\
g_2^{(\varepsilon)}(w) &=(w-w_{10})(w-w_{16})(w-w_{22}),
\\
g_3^{(\varepsilon)}(w) &=(w-w_{16})^2(w-w_{22})^2(w-w_{28})(w-w_{34})(w-w_{40})(w-w_{46}),
\\
g_4^{(\varepsilon)}(w) & = (w-w_{16})(w-w_{22})^3(w-w_{28})^2 \cdots (w-w_{46})^2(w-w_{52}) \cdots (w-w_{70}).
\end{align*}

The following graph gives the asymptotic behavior of the dimensions $d_{k}^\unr = \frac {2k}{p^2-1}+O(1)$ and $d_{k}^\Iw = \frac{2k}{p-1}+O(1)$ when $p=7$. (The vertical axis was stretched approximately $5$ times.) The degree of $g_n(w)$ is proportional to the area of the shaded region, which corresponds to the contribution of the ghost zeros.
\begin{center}
\begin{tikzpicture}[line cap=round,line join=round,>=triangle 45,x=1.0cm,y=5.0cm]
\draw [color=cqcqcq,, xstep=1.0cm,ystep=1.0cm] (-0.5,-0.1) grid (7.,1.);
\draw[->,color=black] (0.,0.) -- (7.,0.);
\foreach \x in {,1.,2.,3.,4.,5.,6.}
\draw[shift={(\x,0)},color=black] (0pt,2pt) -- (0pt,-2pt);
\draw[->,color=black] (0.,0.) -- (0.,1.);
\foreach \y in {,0.2,0.4,0.6,0.8}
\draw[shift={(0,\y)},color=black] (2pt,0pt) -- (-2pt,0pt);
\clip(-0.5,-0.1) rectangle (7.,1.);
\fill[line width=1.2pt,color=uuuuuu,fill=uuuuuu,pattern=north east lines,pattern color=uuuuuu] (1.5,0.25) -- (0.8571428571428571,0.25) -- (1.5,0.4375) -- cycle;
\fill[line width=1.2pt,color=uuuuuu,fill=uuuuuu,pattern=north east lines,pattern color=uuuuuu] (1.5,0.25) -- (6.,0.25) -- (1.5,0.0625) -- cycle;
\draw [line width=1.2pt,dash pattern=on 2pt off 3pt] (6.,0.25)-- (0.,0.25);
\draw [line width=1.2pt,dash pattern=on 2pt off 3pt] (1.5,0.4375)-- (1.5,0.);
\draw [line width=2.pt,color=red,domain=0.0:7.0] plot(\x,{(-0.--0.25*\x)/6.});
\draw [line width=2.pt,color=red,domain=0.0:7.0] plot(\x,{(-0.--1.75*\x)/6.});
\draw [line  width=2.pt,color=red,domain=0.0:7.0] plot(\x,{(-0.--2.*\x)/6.});
\begin{scriptsize}
\draw [fill=blue] (0.,0.25) circle (1.5pt);
\draw[color=blue] (-0.2,0.27) node {$n$};
\draw [fill=blue] (1.5,0.) circle (1.5pt);
\draw[color=blue] (1.5,-0.05) node {$k$};
\draw [fill=black] (1.5,0.0625) circle (2.0pt);
\draw [fill=black] (1.5,0.4375) circle (2.0pt);
\draw[color=red] (5.8,0.15) node {$d_{k}^\unr \approx\frac{2k}{p^2-1}$};
\draw[color=red] (4.15,0.73) node {$d_{k}^\Iw-d_{k}^\unr\approx \frac{2pk}{p^2-1}$};
\draw[color=red] (1.7,0.9) node {$\frac{2k}{p-1} \approx d_{k}^\Iw$};
\end{scriptsize}
\end{tikzpicture}

Graph 1: Functions $d_{k}^\unr$, $d_{k}^\Iw$, and $d_{k}^\Iw-d_{k}^\unr$, and the contribution of ghost zeros.
\end{center}

The following table gives the increments in degrees $ \deg g_n^{(\varepsilon)}$ over different weight disks.

\begin{center}\renewcommand{\arraystretch}{1.2}
\begin{tabular}{|c|c|c|c|c|c|c|c|c|c|c|c|c|c|c|c|c|}
\hline
$\varepsilon$ & \multicolumn{16}{|c|}{$\deg g_n^{(\varepsilon)}- \deg g_{n-1}^{(\varepsilon)}$ for $n \geq 1$}
\\
\hline
$1\times \omega^2$ & 0 & 3 & 5 & 8 & 10 & 13 & 15 & 18 & 21 & 23 & 26 & 28&31&33&36 & ... \\ \hline
$\omega^5\times \omega^3$ &  1 & 3 & 6 & 9 & 11 & 14 & 16 & 19 & 21&24&27&29&32&34&37& ...
\\
 \hline
$\omega^4\times \omega^4$  & 2 & 4 & 7 & 9 & 12 &15 & 17 & 20 & 22&25&27&30&33&35&38& ...
\\ \hline
$\omega^3\times \omega^5$  &3 & 5 & 8 & 10 & 13 & 15 & 18 & 21 &23&26& 28&31&33&36&39& ...
\\ \hline
$\omega^2\times 1$ & 1 & 3 & 6 & 9 & 11 & 14 & 16 & 19 &21&24&27&29&32&34&37&...
\\ \hline
$\omega\times \omega$  & 2 & 4 & 7 & 9 & 12 &15 & 17 & 20 &22&25&27&30& 33&35&38&...
\\ \hline
\end{tabular}
\end{center}
These are expected to be the halo $w$-adic slopes ``near the boundary" of the weight disks; see Corollary~\ref{C:ghost versus halo} and the surrounding remark for more discussions.
\end{example}
\begin{remark}
\label{R:ghost pattern}
In the definition of ghost series, the palindromic pattern in the middle of \eqref{E:cascading pattern} is what Bergdall and Pollack referred to as the ``ghost pattern". They discovered this through a numerical computation. We refer to \cite[\S1]{bergdall-pollack2} for a discussion on the heuristic behind the construction of ghost series.
\end{remark}

\begin{notation}
\label{N:Newton polygon}
For a power series $F(t) = \sum_{n \geq 0} f_n t^n\in
\CC_p\llbracket t\rrbracket$ (with $f_0=1$), we use $\NP(F)$ to denote its \emph{Newton polygon}, that is the convex hull of points $(n, v_p(f_n))$ for all $n$.
The slopes of its segments are exactly the minuses of the $p$-adic valuations of zeros of $F(t)$, counted with multiplicities.
\end{notation}

The following is the local analogue of the ghost conjecture of Bergdall and Pollack. \begin{conjecture}[Local ghost conjecture]
\label{Conj:local ghost conjecture}
Let $\bar \rho = \Matrix{\omega_1^{a+b+1}}{*\neq 0}{0}{\omega_1^b}: \rmI_{\QQ_p} \to \GL_2(\FF)$ be a character with $a \in \{1, \dots, p-4\}$ and $b \in \{0, \dots, p-2\}$.
Let $\widetilde \rmH$ be a primitive $\calO\llbracket K_p\rrbracket$-projective augmented module of type $\bar\rho$, and let $\varepsilon$ be a character of $\Delta^2$ relevant to $\bar \rho$. We define the characteristic power series $C^{(\varepsilon)}(w,t)$ of $U_p$-action and the ghost series $G^{(\varepsilon)}(w,t)$ for $\widetilde \rmH$ as in this section.  Then
\begin{enumerate}
\item
the ghost series $G^{(\varepsilon)}(w,t)$ depends only $\bar \rho$, and $\varepsilon$, and
\item 

for every $w_\star \in \gothm_{\CC_p}$, we have
$\NP(G^{(\varepsilon)}(w_\star, -))=\NP(C^{(\varepsilon)}(w_\star,-))$.
\end{enumerate} 
\end{conjecture}

Part (1) of this conjecture will be proved in Corollary~\ref{C:ghost series depends on bar rho}; in fact the dimensions $d_k^\Iw(\varepsilon_1)$ and $d_k^\unr(\varepsilon_1)$ depend only on $k$, $\bar \rho$, and $\varepsilon$. Because of this, one may rename $G_{\widetilde \rmH}^{(\varepsilon)}(w,t)$ into $G_{\bar \rho}^{(\varepsilon)}(w,t)$, and call it \emph{the ghost series attached to $\bar\rho$ and $\varepsilon$}.

We will address part (2) in a sequel to this paper.

\begin{remark}\fakephantomsection
\label{R:remark after ghost series}
\begin{enumerate}
\item
The ghost series $G^{(\varepsilon)}(w,t)$ may be viewed as a good enough approximation of the characteristic power series $C^{(\varepsilon)}(w,t)$.
At least the $p$-adic valuations of their zeros are the same.

\item 
The definition of the ghost series depends on the choice of the topological generator $\exp(p)$ of $(1+p\ZZ_p)^\times$. In other words, for any $\eta \in 1+p\ZZ_p^\times$, we may replace in \eqref{E:induced representation action extended} the exponent $\log\left((\gamma z+\delta)/\omega(\bar \delta)\right) /p$ by $\log \left((\gamma z+\delta)/\omega(\bar \delta)\right)/\log \eta $, and replace $w_k = \exp((k-2)p)-1$ by $\eta^{k-2}-1$, to get a essentially different ghost series. So ghost series is in this sense not canonical; yet the Newton polygon $\NP(G^{(\varepsilon)}(w_\star,t))$ (for all $w_\star \in \gothm_{\CC_p}$) does not depend on the choice of the topological generator.
So we may work with our ``preferred" topological generator $\exp(p)$ of $(1+p\ZZ_p)^\times$.

\item 
Recall that $\Matrix p00p$ acts by multiplication with a unit $\xi \in\calO^\times$ on $\widetilde \rmH$.
Moreover, by Definition~\ref{D:primitive type}\,(1), the diagonal matrices $\Matrix {\bar\alpha}00{\bar \alpha}$ for $ \bar \alpha \in \Delta$ acts on $\widetilde \rmH$ by multiplication with $\omega(\bar \alpha)^{a+2b}$.
(By possibly enlarging $\calO$), consider a character $\zeta: \QQ_p^\times \to \calO^\times$ such that $\zeta(p)^2 = \xi^{-1}$ and $\zeta (\alpha) = \omega(\bar \alpha)^{-b}$ for all $ \alpha \in \ZZ_p^\times$. If twist the $\GL_2(\QQ_p)$-action on $\widetilde \rmH$ by the character $\zeta\circ \det$, then the resulting completed homology piece $\widetilde \rmH': = \widetilde \rmH \otimes \zeta \circ \det$ is of type $\Matrix{\omega_1^{a+1}}{*\neq 0}{0}{1}$, and the action of $\Matrix p00p$ is trivial.

It follows easily from the construction of the characteristic power series of the $U_p$-action and the definition of the ghost series that
\[
C_{\widetilde \rmH}^{(\varepsilon)}(w,t) = C_{\widetilde \rmH'}^{(\varepsilon \cdot\tilde \omega^{-b})}(w,\zeta(p)^{-1} t) \quad \textrm{and} \quad G_{\widetilde \rmH}^{(\varepsilon)}(w,t) = G_{\widetilde \rmH'}^{(\varepsilon\cdot\tilde \omega^{-b} )}(w,t).
\]
So Conjecture~\ref{Conj:local ghost conjecture} can be reduced to the case when $b = 0$ and $\xi(p) = 1$.

\end{enumerate}

\end{remark}

\begin{remark}

Our definition of ghost series follows Bergdall and Pollack \cite{bergdall-pollack2, bergdall-pollack3}. In fact, a variant of the ghost series already appeared in a much earlier paper by Buzzard and Calegari \cite{buzzard-calegari} (see also \cite{loeffler}) but only in a special case: $p=2$ and level $N=1$. In our language, one can set
$u = \log (1+w) +2p$ and then it is easy to check that $\calO\llbracket w/ p^{1/(p-1)}\rrbracket  \cong \calO\llbracket u/p^{1/(p-1)}\rrbracket$, where we assume that $\calO$ contains a $(p-1)$th root of $p$. Then one can generalize Buzzard and Calegari's construction to define a variant of ghost series:
$$
A^{(\varepsilon)}(u,t) = 1+ \sum_{n=1}^\infty a_n^{(\varepsilon)}(u)t^n \in \calO\llbracket u/p^{1/(p-1)}\rrbracket\llbracket t\rrbracket 
$$
with coefficients given by
$$
a_n^{(\varepsilon)}(u): = \prod_{k \geq 2, k \equiv k_\varepsilon \bmod {p-1}} (u- pk)^{m_n^{(\varepsilon)}(k)} \in \ZZ_p[u].
$$
(In Buzzard--Calegari case, the numbers $m_n^{(\varepsilon)}(k)$ are explicitly computed, as it is for one example $p=2$ and level $N=1$.)
One can prove that for every $w_\star \in \calO_{\CC_p}$ with $v_p(w_\star) > \frac 1{p-1}$, setting 
\[
u_\star = \log(1+w_\star)+2p
\]
(which has the same $p$-adic valuation as $w_\star$), the two Newton polygons
$\NP(G^{(\varepsilon)}(w_\star, -))$ and $\NP(A^{(\varepsilon)}(u_\star, -))$ coincide. Thus proving Conjecture~\ref{Conj:local ghost conjecture} is equivalent to proving that, for every pair $w_\star$ and $u_\star$ above (with valuation $>\frac 1{p-1}$), 
\[
\NP(C^{(\varepsilon)}(w_\star, -)) = \NP(A^{(\varepsilon)}(u_\star, -)).
\]
\end{remark}

\begin{remark}
\label{R:generalizations of ghost conjecture}
It is natural to ask what one should expect when $\bar\rho$ is nongeneric, reducible split, or of the form $\bar \rho = \omega_2^a \oplus \omega_2^{pa}$, where $\omega_2$ is second fundamental character. There are several places of the theory that needs modifications:
\begin{enumerate}
\item a primitive $\calO\llbracket K_p\rrbracket$-projective augmented module
$\widetilde \rmH$ should be instead isomorphic to the direct sum of the projective envelopes of Serre weights, viewed as an $\calO\llbracket \GL_2(\ZZ_p)\rrbracket$-module (which would then change the dimension combinatorics);

\item some extra condition should be imposed to give initial ``seeds" of the theory corresponding to the Hodge--Tate weight $(0,1)$ crystabelian deformation rings of $\bar\rho$ (in the Fontaine-Laffaille range, the crystalline deformations of a reducible $\bar{\rho}$ are ordinary, while for the irreducible $\rho$'s, we need extra information to record the initial slopes); and
\item the combinatorics may be different and more complicated.
\end{enumerate}
It seems that, with appropriate modifications, the ghost series framework might still partially work. 
We hope to address some of these cases in future works.

We do not know how to generalize the ghost conjecture framework beyond $\GL_2(\QQ_p)$.
\end{remark}

\begin{remark}
\label{R:ghost for direct sums}
If $\widetilde \rmH_1$ and $\widetilde \rmH_2$ are two primitive $\calO\llbracket K_p\rrbracket$-projective augmented modules of type $\bar\rho_1$ and $\bar\rho_2$, respectively, with $\det \bar \rho_1 = \det \bar\rho_2$, then one may use the recipe in Definition~\ref{D:ghost series} to define a ghost series $G_{\widetilde \rmH_1\oplus \widetilde \rmH_2}^{(\varepsilon)}(w,t)$ for $\widetilde \rmH_1 \oplus \widetilde \rmH_2$. While it is clear that $C^{(\varepsilon)}_{\widetilde \rmH_1 \oplus \widetilde \rmH_2}(w,t) = C^{(\varepsilon)}_{\widetilde \rmH_1} (w,t)C_{\widetilde \rmH_2}^{(\varepsilon)}(w,t)$, it is not obvious whether  $\NP\big(G^{(\varepsilon)}_{\widetilde \rmH_1 \oplus \widetilde \rmH_2}(w_\star,-)\big) = \NP\big(G^{(\varepsilon)}_{\widetilde \rmH_1} (w_\star,t)G_{\widetilde \rmH_2}^{(\varepsilon)}(w_\star,-)\big)$ for all $w_\star \in \gothm_{\CC_p}$.

In fact, as a recent preprint of Rufei Ren \cite{ren} suggests, this equality of Newton polygon of ghost series should hold if and only if $\bar\rho_1^{ss} = \bar \rho_2^{ss}$, where $\bar\rho_?^{ss}$ stands for the semisimplification of $\bar\rho_?$. When $\bar \rho_1^{ss} \neq \bar \rho_2^{ss}$, the two Newton polygons can disagree at some $n$th slope of a certain weight $k$, where both $n$ and $k$ are very large integers.
In turn, this suggests the original ghost conjecture for the full space of modular forms may not be correct as stated in \cite[Conjecture~2.5]{bergdall-pollack2}, where the numerical computation is not able to reach very large $n$. 

Now if $\widetilde \rmH'$ is an $\calO\llbracket K_p\rrbracket$-projective augmented module of type $\bar \rho$ with multiplicity $m(\widetilde \rmH')$, then Rufei Ren's result \cite{ren} would imply that $\NP(G_{\widetilde \rmH'}^{(\varepsilon)}(w_\star,-))$ is the same as $\NP(G_{\bar \rho}^{(\varepsilon)}(w_\star,-))$ with $x$-axis stretched $m(\widetilde \rmH')$ times, for every $w_\star \in \gothm_{\CC_p}$. It is natural to conjecture that $\NP(C_{\widetilde \rmH'}^{(\varepsilon)}(w_\star,-))$ agrees with this polygon.
\end{remark}

\section{Properties of abstract classical, overconvergent, and $p$-adic forms}
\label{Sec:more abstract automorphic forms}

In this section, we discuss finer properties of abstract classical, overconvergent,  and $p$-adic forms:
\begin{enumerate}
\item 
a careful choice of the basis $e_1$ and $e_2$ of $\widetilde \rmH$ as a free $\calO\llbracket \Iw_{p,1}\rrbracket$-module (Notation~\ref{N:e1 = e2Pi}),
\item a description of the power basis of abstract classical, overconvergent forms (Proposition~\ref{P:power basis}),
\item classicality of abstract overconvergent forms, and the theta operator (Proposition~\ref{P:theta map}), and
\item 
the Atkin--Lehner involution for abstract classical forms (Proposition~\ref{P:Atkin-Lehner duality}).
\end{enumerate}
Property (1) is well known. Property (2) will be useful later for establishing the dimension formula for abstract classical forms with Iwahori level structure.
Properties (3) and (4) are all well known in the context of modular forms, and more generally automorphic forms; we just reproduce them in our abstract setup.

In view of the discussion of Remark~\ref{R:remark after ghost series}\,(3), we assume the following.

\begin{hypothesis}
\label{H:b=0}
For the rest of this paper, we assume that $b=0$, $\xi =1$, and that $\widetilde \rmH$ is a \emph{primitive} $\calO\llbracket K_p\rrbracket$-projective augmented module of type $\bar \rho =\Matrix{\omega_1^{a+1}}{*\neq 0}{0}{1}$, on which $\Matrix p00p$ acts trivially. In addition, we assume that $1\leq a \leq p-4$.
\end{hypothesis}
We reserve the Greek letter
$\varepsilon$ to denote a character of $\Delta^2$ relevant to $\bar \rho$.

\subsection{$\widetilde \rmH$ as an $\calO\llbracket \Iw_p\rrbracket$-module}
\label{S:basis of widetilde H}
Let $\widetilde \rmH$ be as above. 
As explained in \S\,\ref{S:explicit Muniv}, we may write $\widetilde \rmH$ as an $\calO\llbracket \Iw_p\rrbracket$-module
\begin{equation}
\label{E:tilde H as Iwp module}
\widetilde \rmH \simeq e_1 \calO\otimes_{\chi_1, \calO[\bar \rmT]} \calO\llbracket \Iw_p\rrbracket \oplus e_2 \calO\otimes_{\chi_2, \calO[\bar \rmT]} \calO\llbracket \Iw_p\rrbracket 
\end{equation}
for two characters $\chi_1$ and $\chi_2$ of $\Delta^2$, which we determine now.

Taking \eqref{E:tilde H as Iwp module} modulo $(\varpi,\rmI_{1+p\rmM_2(\ZZ_p)})$ and using condition (1) of Definition~\ref{D:primitive type}, we obtain an isomorphism of (right) $\FF[\bar \rmB]$-modules (for $\bar \rmB := \Matrix{\FF_p^\times}{\FF_p}{0}{\FF^\times_p}$)
$$ \Proj_{a,0} \cong e_1 \FF \otimes_{\chi_1, \FF[\bar \rmT]} \FF[\bar \rmB] \oplus e_2 \FF \otimes_{\chi_2, \FF[\bar \rmT]}\FF[\bar \rmB].$$
Taking the $\bar \rmU: =\Matrix{1}{\FF_p}{0}{1} $-invariants of both sides, the left hand side has $2$-dimensional $\bar \rmU$-invariants with characters $1 \times \omega^{a}$ and $\omega^{a} \times 1$ under the $\bar \rmT$-action by Lemma~\ref{L:U invariants of Proj}; the $\bar \rmU$-invariants of the right hand side are spanned by $e_i \cdot \big( \Matrix 1001 + \Matrix 1101 + \cdots + \Matrix 1{p-1}01\big)$ for $i =1,2$, which have characters $\chi_1$ and $\chi_2$, respectively.
So it follows that (after possibly exchanging the $\chi_i$'s)
$$
\chi_1 = 1 \times \omega^{a}\quad \textrm{and}\quad \chi_2 = \omega^{a} \times1.
$$

Although the characters $\chi_1$ and $\chi_2$ are uniquely determined, the two generators $e_1$ and $e_2$ are far from unique.
We further rigidify the generators $e_1$ and $e_2$ as follows.

\begin{lemma}
\label{L:choice of e1 and e2}
In \eqref{E:tilde H as Iwp module}, we may replace $e_2$ by $e'_2 := e_1 \Matrix 01p0$ (so that $e_1 = e'_2 \Matrix 01p0$).
\end{lemma}
\begin{proof}
Note that the action of $\bar \rmT$ on $e_2$ and $e'_2$ are the same.
Indeed, for the given basis $e_1$ and $e_2$, we can write
\begin{equation}
\label{E:e1Pi and e2Pi}
e_1 \Matrix 01p0 = e_1 A + e_2 B \quad \textrm{and} \quad e_2 \Matrix 01p0  = e_1 C + e_2 D,
\end{equation}
with $A, B, C, D \in \calO\llbracket \Iw_{p, 1}\rrbracket$.
Let $\bar A$, $\bar B$, $\bar C$, and $\bar D$ denote the reduction in $\FF$ of $A$, $B$, $C$, and $D$ modulo the maximal ideal  $(\varpi,\rmI_{\Iw_{p,1}})$ of $\calO\llbracket \Iw_{p,1}\rrbracket$.
Consider the $\bar \rmT$-action on \eqref{E:e1Pi and e2Pi} modulo $(\varpi, \rmI_{\Iw_{p,1}})$ (noting that $e_1$ and $e_2$ have different characters $1\times \omega^a$ and $\omega^a \times 1$, respectively).  We get $\bar A = \bar D=  0$ and $\bar B, \bar C \in \FF^\times$. Thus it follows that $e_1$ and $e_1 \Matrix 01p0$ spans $\widetilde \rmH$ over $\calO\llbracket \Iw_{p,1}\rrbracket $. 
\end{proof}

\begin{notation}
\label{N:e1 = e2Pi}
\emph{For the rest of the paper, we fix 
an isomorphism}
$$
\widetilde \rmH \simeq e_1 \calO\otimes_{1 \times \omega^a, \calO[\bar \rmT]} \calO\llbracket \Iw_p\rrbracket \oplus e_2 \calO\otimes_{\omega^a \times 1, \calO[\bar \rmT]} \calO\llbracket \Iw_p\rrbracket 
,
$$
where the generators satisfy $e_1 = e_2 \Matrix 01p0$ and $e_2 = e_1 \Matrix 01p0$.
\end{notation}
\begin{remark}
Choosing generators with the above property will simplify our proof in this paper; and will be essential for sequel.
\end{remark}

\subsection{Power basis of abstract overconvergent and classical forms}
\label{SS:power basis}
Fix a \emph{relevant} character $\varepsilon = \omega^{-s_\varepsilon}
\times \omega^{a+s_\varepsilon}$ of $\Delta^2$ as in Notation~\ref{N:relevant varepsilon} (in particular, $\varepsilon_1 = \omega^{-s_\varepsilon}$).
We give a natural power basis of the space of the family of abstract overconvergent forms $\rmS^{\dagger, (\varepsilon)}$ defined in \S\,\ref{S:family of overconvergent forms}; then by specializing the parameters, we obtain a power basis of the space of abstract overconvergent forms
$\rmS_k^\dagger\big(\varepsilon \cdot (1\times \omega^{2-k})\big)$ of weight $k$ and its subspace of abstract classical forms $\rmS_k^{\Iw}\big(\varepsilon \cdot (1\times \omega^{2-k})\big)$ as defined in \S\,\ref{S:abstract classical forms}.

Following a discussion similar to \S\,\ref{S:explicit Muniv}, we have
\begin{align*}
&\rmS^{\dagger, (\varepsilon)} 
= \Hom_{\calO[ \Iw_p]}\big(\widetilde \rmH,\, \Lambda^{\leq 1/p}\langle z \rangle \otimes (\omega^{-s_\varepsilon} \times \omega^{a+s_\varepsilon}) \big)\\
\cong \ & e_1^* \cdot \big(\Lambda^{\leq 1/p}\langle z \rangle\otimes (\omega^{-s_\varepsilon}\times \omega^{a+s_\varepsilon}) \big)^{\bar \rmT = 1 \times \omega^{a}}  \oplus e_2^* \cdot \big(\Lambda^{\leq 1/p}\langle z \rangle \otimes (\omega^{-s_\varepsilon}\times \omega^{a+s_\varepsilon}) \big)^{\bar \rmT = \omega^{a} \times 1}.
\end{align*}
The monomial $z^j \in \calO\langle z \rangle$ is an eigenvector for the $\bar \rmT$-action with character $\omega^j \times \omega^{-j}$. So we deduce the following results on basis of $\rmS^{\dagger, (\varepsilon)}$.

\begin{proposition}
\label{P:power basis}
For the relevant character $\varepsilon = \omega^{-s_\varepsilon} \times \omega^{a+s_\varepsilon}$ with $s_\varepsilon \in \{0, \dots, p-2\}$ and for $k$ an integer, the following list is a basis of $\rmS^{\dagger, (\varepsilon)}$ as a Banach $\Lambda^{\leq 1/p}$-module and also a basis of $\rmS_{k}^\dagger\big(\varepsilon\cdot (1\times \omega^{2-k})\big)$ as a Banach $\calO$-module:
\begin{equation}
\label{E:basis of Sdagger}e_1^* z^{s_\varepsilon}, e_1^* z^{p-1+s_\varepsilon},e_1^* z^{2(p-1)+s_\varepsilon}, \dots;\; e_2^* z^{\{a+s_\epsilon\}}, e_2^* z^{p-1+\{a+s_\epsilon\}},e_2^* z^{2(p-1)+\{a+s_\epsilon\}}, \dots.
\end{equation}
When $k\geq 2$, the subsequence consisting of terms whose power in $z$ is less than or equal to $k-2$ forms a basis of $\rmS_k^\Iw\big(\varepsilon\cdot(1\times \omega^{2-k})\big)$ as an $\calO$-module.
\end{proposition}
\begin{proof}
The list of basis \eqref{E:basis of Sdagger} follows from the discussion above.  The case for $\rmS_k^{\dagger}(\varepsilon\cdot (1\times \omega^{2-k}))$ follows from the family case and the isomorphism \eqref{E:family overconvergent forms specialization}.
\end{proof}
This result will immediately give a dimension formula for $d_k^\Iw(\psi)$, which we defer to Proposition~\ref{P:dimension of SIw}.

\begin{notation}\fakephantomsection
\label{N:power basis}
\begin{enumerate}
\item 
We call \eqref{E:basis of Sdagger} the \emph{power basis} of $\rmS^{\dagger, (\varepsilon)}$ and  $\rmS_{k}^{\dagger}\big(\varepsilon\cdot (1\times \omega^{2-k})\big)$, denoted by $\bfB^{ (\varepsilon)}$ (as it formally does not depend on $k$).
\item 
The \emph{degree} of each basis element $\bfe = e_i^*z^j \in \bfB^{ (\varepsilon)}$ is the exponents on $z$, i.e., $\deg (e_i^*z^j) = j$.
We order the elements in $\bfB^{(\varepsilon)}$ as $\bfe_1^{(\varepsilon)}, \bfe_2^{(\varepsilon)}, \dots$ with increasing degrees. In particular each $\bfe_\ell^{(\varepsilon)}$ is some $e_i^*z^j$ and $i$ depends only on the parity of $\ell$.
Moreover, under our generic assumption $1\leq a \leq p-4$ (indeed $1\leq a \leq p-2$ is already enough), the degrees of elements of $\bfB^{ (\varepsilon)}$ are pairwise distinct.
\item 
Write $ \bfB_{k}^{(\varepsilon)} $ for the subset of elements of $\bfB^{(\varepsilon)}$ with degree less than or equal to $k-2$. It is a basis of $\rmS_{k}^\Iw\big(\varepsilon\cdot (1\times \omega^{2-k})\big)$ by Proposition~\ref{P:power basis}. When fixing the $k$, elements in $\bfB_k^{(\varepsilon)}$ are often called \emph{classical} and elements in $\bfB^{(\varepsilon)} \backslash \bfB_k^{(\varepsilon)}$ are often called \emph{non-classical}.

\item
We write $\rmU^{\dagger,(\varepsilon)} \in \rmM_\infty(\Lambda^{\leq 1/p})$ for the matrix of $\Lambda^{\leq 1/p}$-linear $U_p$-action on $\rmS^{\dagger,(\varepsilon)}$ with respect to the power basis $\bfB^{(\varepsilon)}$; its evaluation at $w = w_k$ is the matrix $\rmU_k^{\dagger, (\varepsilon)}$ of $U_p$-action on $\rmS_k^\dagger\big(\varepsilon\cdot(1\times \omega^{2-k})\big)$ (with respect to $\bfB^{ (\varepsilon)}$).

Since the $U_p$-operator on $\rmS_k^\dagger\big(\varepsilon\cdot(1\times \omega^{2-k})\big)$ leaves the subspace $\rmS_{k}^\Iw\big(\varepsilon\cdot (1\times \omega^{2-k})\big)$ stable, the matrix $\rmU^{\dagger, (\varepsilon)}_k$ is a block-upper-triangular matrix, whose $d_k^\Iw\big(\varepsilon\cdot (1\times \omega^{2-k})\big) \times d_k^\Iw\big(\varepsilon\cdot (1\times \omega^{2-k})\big)$ upper-left block $\rmU^{\Iw, (\varepsilon)}_k$ is the matrix for the $U_p$-action on $\rmS_{k}^\Iw\big(\varepsilon\cdot (1\times \omega^{2-k})\big)$ with respect to $\bfB_k^{ (\varepsilon)}$.

\end{enumerate}
\end{notation}

\begin{remark}
\label{R:comparison of char power series for power basis and small basis}
As we proved in Lemma~\ref{L:char power series agree} that the $U_p$-action on $\rmS^{\dagger, (\varepsilon)}$ is compact, and the inclusion $\rmS_{p\textrm{-adic}}^{(\varepsilon), \leq 1/p} \subseteq \rmS^{\dagger, (\varepsilon)}$ induces an identification of  characteristic power series: 
\begin{equation}
\label{E:char Up of power is equal to char Up of small basis}\Char(\rmU^{\dagger,(\varepsilon)},t) = C^{(\varepsilon)}(w, t) \quad \textrm{and} \quad \Char(\rmU_k^{\dagger,(\varepsilon)}, t) = C^{(\varepsilon)}(w_k, t).
\end{equation}
\end{remark}

The following is an analogue of the classicality result of overconvergent automorphic forms.

\begin{proposition}[Theta map]
\label{P:theta map}
Let $\varepsilon = \omega^{-s_\varepsilon} \times \omega^{a+s_\varepsilon}$ be a character of $\Delta^2$ relevant to $\bar \rho$, and let $k \geq 2$ be an integer. Put $\psi = \varepsilon \cdot (1\times \omega^{2-k})$, $\varepsilon' = \varepsilon \cdot (\omega^{k-1} \times \omega^{1-k})$, and $\psi' = \varepsilon'\cdot (1 \times \omega^k) = \psi\cdot \tilde \omega^{k-1}$.

\begin{enumerate}
\item
There is a short exact sequence
\begin{equation}
\label{E:classicality short exact sequence}
0 \to \rmS^\Iw_{k}(\psi) \longrightarrow  \rmS^{\dagger}_{k}(\psi)\xrightarrow{(\frac{d}{dz})^{k-1} \circ}\rmS_{2-k}^{\dagger}(\psi'),
\end{equation}
which is equivariant for the usual $U_p$-action on the first two terms and the $p^{k-1}U_p$-action on the third term. Here the map $\big( \frac d{dz}\big)^{k-1} \circ$ sends an element $\varphi \in \rmS_{k}^{\dagger}(\psi)$ to the function that takes $m \in \widetilde \rmH$ to the $(k-1)$th derivative of the function $\varphi(m) \in \calO\langle z\rangle$.
\item
We have an equality of characteristic power series
\[
C^{(\varepsilon)}(w_k, t) = C^{(\varepsilon')}(w_{2-k}, p^{k-1}t) \cdot \mathrm{Char}\big(U_p;\, \rmS^\Iw_k(\psi) \big).
\]
\item
All finite $U_p$-slopes of $\rmS^{\dagger}_{k}(\psi)$ strictly less than $k-1$ are the same as the finite $U_p$-slopes of $\rmS_{k}^\Iw( \psi)$ (which are strictly less than $k-1$). The multiplicity of $k-1$ as $U_p$-slopes of $\rmS^{\dagger}_{k}(\psi)$ is the sum of the multiplicity of $k-1$ as $U_p$-slopes of $\rmS_{k}^\Iw(\psi)$ and the multiplicity of $0$ as $U_p$-slopes of $\rmS^{\dagger}_{2-k}(\psi')$.
\end{enumerate}

\end{proposition}
\begin{proof} The results of the proposition are well-known. But we do not know any proof in our abstract setup. We include one for completeness.
\emph{Only in this proof, for $a$ an integer, we write $|^{(a)}_{\Matrix \alpha \beta \gamma \delta}$ for the action of $\Matrix \alpha \beta \gamma \delta \in \bfM_1$ on $h(z) \in \calO\langle z\rangle$ given by }
$$h\big|^{(a)}_{\Matrix \alpha \beta \gamma \delta}(z) = (\gamma z+\delta)^a h\big(\frac{\alpha z+\beta}{\gamma z+\delta} \big).$$

For (1), 
it follows from the discussion in \cite[\S7]{buzzard} that for $h(z) \in \calO\langle z\rangle$,
we have the following equality for $\Matrix\alpha \beta \gamma \delta \in \bfM_1$,
\begin{align*}
&\frac {d^{k-1}}{dz^{k-1}} \Big( h\big|^{(k-2)}_{\Matrix\alpha\beta\gamma\delta}(z)\Big)=\frac {d^{k-1}}{dz^{k-1}} \Big((\gamma z +\delta)^{k-2} h\big( \frac{\alpha z+\beta}{\gamma z+\delta}\big) \Big)
\\
=\;& (\alpha \delta - \beta \gamma)^{k-1}(\gamma z+\delta )^{-k} \big(\frac{d^{k-1}h}{dz^{k-1}}\big)\big( \frac{\alpha z+\beta}{\gamma z+\delta}\big)
= (\alpha \delta - \beta \gamma)^{k-1} \big(\frac{d^{k-1}h}{dz^{k-1}}\big)\big|^{(-k)}_{\Matrix \alpha \beta \gamma \delta}.
\end{align*}
In other words, $\big(\frac d{dz}\big)^{k-1}$ induces a natural morphism (equivariant for the $\bfM_1$-action)
\begin{equation}
\label{E:derivatives commute with actions}
\big(\frac d{dz}\big)^{k-1}: \Big( \calO\langle z \rangle,\, \big |^{(k-2)} \Big) \longrightarrow\Big( \calO\langle z \rangle,\, \big |^{(-k)} \otimes \det^{k-1} \Big).
\end{equation}
Applying $\Hom_{\calO[\Iw_p]}(\widetilde \rmH, - \otimes \psi)$ to \eqref{E:derivatives commute with actions} gives a natural morphism
\begin{equation}
\label{E:derivatives commute with action 2}
\big(\frac d{dz}\big)^{k-1}\circ : \rmS_{k}^{\dagger, (\varepsilon)}  \longrightarrow\Hom_{\calO[\Iw_p]}\Big(\widetilde \rmH,\big( \calO\langle z \rangle,\, \big |^{(-k)} \otimes \det^{k-1} \otimes \psi \big) \Big)  .
\end{equation}
Write the character $\det$ (on $\GL_2(\QQ_p)$) as $\det_1 \times \det_2$ so that if $\det \Matrix \alpha \beta \gamma \delta = p^r d$ for $r \in \ZZ$ and $d\in \ZZ_p^\times$, then $\det_1\Matrix \alpha \beta \gamma \delta =p^r \omega(\bar d)$ and $\det_2 \Matrix \alpha\beta\gamma\delta= d / \omega(\bar d)$.
Then we may rewrite \eqref{E:derivatives commute with action 2} as
\begin{equation}
\label{E:theta map local}
\big(\frac d{dz}\big)^{k-1}\circ : \rmS_{k}^{\dagger, (\varepsilon)}  \longrightarrow\Hom_{\calO[\Iw_p]}\Big(\widetilde \rmH \otimes \det_2^{1-k},\big( \calO\langle z \rangle,\, \big |^{(-k)} \otimes \det_1^{k-1} \otimes \psi \big) \Big).
\end{equation}

By assumption $(3)$ in Definition~\ref{D:primitive type}, the $\calO\llbracket K_p\rrbracket$-projective augmented module $\widetilde \rmH$ is continuously and $\calO[\GL_2(\QQ_p)]$-equivariantly isomorphic to $\widetilde \rmH \otimes \det_2^{1-k}$. Using this isomorphism, we deduce that the right hand side of \eqref{E:theta map local} is isomorphic to $\rmS_{2-k}^{\dagger}(\psi')$ with $\psi' = \psi \cdot \tilde \omega^{k-1}$, and the corresponding weight disk character
$$\varepsilon' = (1\times \omega^{-k}) \cdot \psi' =  (1 \times \omega^{-k})\cdot \varepsilon \cdot (1\times \omega^{2-k}) \cdot (\omega^{k-1} \times \omega^{k-1}) = \varepsilon \cdot (\omega^{k-1} \times \omega^{1-k}).
$$
The $U_p$-action on the right hand side of \eqref{E:theta map local} is $p^{k-1}$ times the usual $U_p$-action on $\rmS_{2-k}^{\dagger}(\psi')$ (where the extra factor $p^{k-1}$ comes from the image of $p$ under $\det_1^{k-1}$). 
This proves (1) of the Proposition.

\medskip
(2) 
Consider the standard power basis \eqref{E:basis of Sdagger} of $\rmS_{k}^{\dagger}(\psi)$. The kernel of the map $(\frac d{dz})^{k-1}\circ$, namely $\rmS_{k}^\Iw(\psi)$, is exactly spanned by those power basis elements $\bfe \in \bfB^{(\varepsilon)}$ such that $\deg (\bfe) \leq k-2$, i.e $\bfB^{(\varepsilon)}_k$. Moreover, the image of a basis element $\bfe = e_i^* z^j \in \bfB^{(\varepsilon)}-\bfB^{(\varepsilon)}_k$ under $(\frac d{dz})^{k-1}\circ$ is exactly $j (j-1)\cdots (j-k+2) e_i^*z^{j-k+1}$, namely a multiple of the power basis element in $\bfB^{(\varepsilon')}$ for $\rmS_{2-k}^{\dagger}(\psi')$ (if nonzero).
Therefore, it follows that
\begin{equation}
\label{E:Ukdagger is block upper triangular}
\rmU_k^{\dagger,(\varepsilon)} = \begin{pmatrix}
\rmU_k^{\Iw,(\varepsilon)}& * \\
0 & p^{k-1}D^{-1} \rmU_{2-k}^{\dagger,(\varepsilon')} D
\end{pmatrix},
\end{equation}
where $D$ is the diagonal matrix whose diagonal entries are indexed by $\bfe = e_i^*z^j \in \bfB^{(\varepsilon)}$ with $i=1,2$ and $j \geq k-1$, and are given by $j(j-1) \cdots (j-k+2)$. From this and the equality \eqref{E:char Up of power is equal to char Up of small basis}, we see immediately that
\begin{align*}
C^{(\varepsilon)}(w_k,t) &\stackrel{\eqref{E:char Up of power is equal to char Up of small basis}}= \Char (\rmU_k^{\dagger,(\varepsilon)},t) = \Char ( \rmU_k^{\Iw,(\varepsilon)
},t) \cdot \Char(\rmU_{2-k}^{\dagger,(\varepsilon')},p^{k-1}t)
\\ & \stackrel{\eqref{E:char Up of power is equal to char Up of small basis}}=\Char\big(U_p;\,\rmS^\Iw_k(\psi) \big)\cdot  C^{(\varepsilon')}(w_{2-k}, p^{k-1}t) .
\end{align*}

(3) By (2), the multiset of the finite $U_p$-slopes of $\rmS_{k}^{\dagger}(\psi)$ is the disjoint union of the multiset of $U_p$-slopes of $\rmS_{k}^{\Iw}(\psi)$ and the multiset of $k-1+\alpha$ for each $U_p$-slope $\alpha$ of $\rmS_{2-k}^{\dagger}(\psi')$, counted with multiplicity. (3) is a corollary of this and the fact that the $U_p$-slope on $\rmS_{2-k}^{\dagger}(\psi')$ are non-negative (as $C^{(\varepsilon')}(w_{2-k}, t)$ has integral coefficients).
\end{proof}

\begin{remark} The short exact sequence \eqref{E:classicality short exact sequence} is not exact on the right, because ``integrating" a power series in $\calO\langle z\rangle $ no longer belong to $\calO\langle z\rangle$ (inverting $p$ will not fix this problem either due to existence of arbitrarily large $p$-powers in the denominators). To get the surjectivity, one needs to work with functions on open disks (as opposed to closed disks), or to take certain limit. Since we will not use this later, we leave this as an exercise for the readers.

\end{remark}

The following is an analogue of the classical Atkin--Lehner theory. 

\begin{proposition}[Atkin--Lehner involution] \label{P:Atkin-Lehner duality}
Let $\psi = \psi_1 \times \psi_2$ be a character of $\Delta^2$ (where we allow $\psi_1 = \psi_2$). Put $\psi'' = \psi_2 \times \psi_1$ and $\varepsilon''= \varepsilon \cdot \psi'' \cdot \psi^{-1} = \omega^{-s_{\varepsilon''}} \times \omega^{a+s_{\varepsilon''}}$ (with $s_{\varepsilon''} = \{k-2-a-s_\varepsilon\}$).

\begin{enumerate}
\item 
Then we have a well-defined natural morphism
\begin{equation}
\label{E:Atkin-Lehner map}
\begin{tikzcd}[row sep =0pt]
\AL_{(k, \psi)}: \rmS_{k}^{\Iw}( \psi)  \ar[r] &  \rmS_{k}^{\Iw}( \psi'')
\\
\varphi \ar[r, mapsto] & (\ \AL_{(k, \psi)}(\varphi): x \mapsto \varphi \big(x \Matrix 0{p^{-1}}10 \big)\big|_{\Matrix 01p0} \ ).
\end{tikzcd}
\end{equation}
Here the last $|_{\Matrix 01p0}$ is the usual action on $\calO[z]^{\leq k-2}$ and the \emph{trivial} action on the factor $\psi''$. 

\item 
We have an explicit formula, for $i=1,2$ and any $j$,
$$
\AL_{(k,\psi)} (e_i^* z^j) = p^{k-2-j} \cdot e_{3-i}^* z^{k-2-j}.
$$
Or equivalently we have
\begin{equation}
\label{E:AL swaps power basis elements}\AL_{(k,\psi)}(\bfe_\ell^{(\varepsilon)}) = p^{k-2-\deg \bfe_\ell^{(\varepsilon)}}\bfe^{(\varepsilon'')}_{d_k^\Iw(\psi'')+1-\ell},
\end{equation}
where we added superscripts to the power basis element to indicate the corresponding character.
In particular, we have
\begin{equation}
\label{E:AL circ AL}
\AL_{(k, \psi'')} \circ \AL_{(k, \psi)} = p^{k-2}.
\end{equation}
So $\AL_{(k, \psi)}$ induces an isomorphism $\AL_{(k,\psi)}: \rmS_{k}^\Iw(\psi)\otimes E \xrightarrow{\cong}\rmS_{k}^\Iw(\psi'') \otimes E$.  

\item
When $\psi_1 \neq \psi_2$ (or equivalently $k \not\equiv k_\varepsilon \bmod{(p-1)}$), we have an equality
\begin{equation}
\label{E:Atkin-Lehner involution}
U_p \circ \AL_{(k, \psi)} \circ U_p = p^{k-1} \cdot \AL_{(k, \psi)}
\end{equation}
as maps from $\rmS_{k}^{\Iw}( \psi)$ to $\rmS_{k}^{\Iw}( \psi'')$.
Consequently, when $\psi_1 \neq \psi_2$,  we can pair the slopes for the $U_p$-action on $\rmS_{k}^{\Iw}( \psi)$ and the slopes for the $U_p$-action on $\rmS_{k}^{\Iw}( \psi'')$ so that each pair adds up to $k-1$. In particular all slopes belong to $[0, k-1]$.
\end{enumerate}
\end{proposition}
\begin{proof}
This is well known, but not documented in the literature for our abstract setup. We include a proof for completeness. First of all, $\psi = \varepsilon \cdot (1\times \omega^{2-k}) = \omega^{-s_\varepsilon} \times \omega^{a+s_\varepsilon+2-k}$ and
$$
\varepsilon'' = \varepsilon \cdot \psi''\cdot \psi^{-1} = \omega^{-s_\varepsilon + (a+s_\varepsilon +2-k) +s_\varepsilon} \times \omega^{a+s_\varepsilon+s_\varepsilon - (a+s_\varepsilon +2-k)} = \omega^{-s_{\varepsilon''}} \times \omega^{a+s_{\varepsilon''}},
$$
where $s_{\varepsilon''} = \{k-2-a-s_\varepsilon\}$.

(1) We need to show that if $\varphi \in \rmS_{k}^{\Iw}( \psi)$, $\AL_{(k,\psi)}(\varphi)$ as a map from $\widetilde \rmH$ to $\calO[z]^{\leq k-2} \otimes \psi''$ is equivariant for the $\Iw_{p}$-action. Indeed,  for $\Matrix \alpha \beta \gamma \delta \in \Iw_p$ and $x \in \widetilde \rmH$,
\begin{align*}
\AL_{(k,\psi)}(\varphi)&\big(x\Matrix \alpha \beta \gamma \delta\big) = \varphi\big(x \Matrix \alpha \beta \gamma \delta\Matrix 0{p^{-1}}10 \big)\big|_{\Matrix 01p0} = \varphi \big( x\Matrix 0{p^{-1}}10\Matrix \delta{p^{-1}\gamma} {p\beta} \alpha\big)\big|_{\Matrix 01p0}
\\
& = \psi(\bar \delta, \bar \alpha) \cdot \varphi(x  \Matrix 0{p^{-1}}10\big)\big|_{\Matrix \delta{p^{-1}\gamma} {p\beta} \alpha \Matrix 01p0} = \psi''( \bar \alpha, \bar \delta) \cdot \varphi(x  \Matrix 0{p^{-1}}10\big)\big|_{ \Matrix 01p0 \Matrix\alpha \beta \gamma \delta}.
\end{align*}
Here the third equality uses the fact that $\varphi$ is equivariant for the right action of $\Matrix \delta{p^{-1}\gamma} {p\beta} \alpha \in \Iw_p$ on $\widetilde \rmH$ and the product right action of $\Matrix \delta{p^{-1}\gamma} {p\beta} \alpha \in \Iw_p$ on $\calO[z]^{\leq k-2} \otimes \psi$.

(2) Recall from Notation~\ref{N:e1 = e2Pi} that the generators of $\widetilde \rmH$ are chosen so that $e_1 \Matrix 01p0 = e_2$ and $e_2 \Matrix 01p0 =e_1$ (as $\Matrix p00p$ acts trivially on $\widetilde \rmH$ by Hypothesis~\ref{H:b=0}). So if $\varphi = e_i^*z^j\in \rmS_{k}^\Iw(\psi)$, then
$$\AL_{(k,\psi)}(\varphi)(e_{i'}) = \varphi \big(e_{i'} \Matrix 0{p^{-1}}10\big) \big|_{\Matrix 01p0} = \varphi(e_{3-i'})\big|_{\Matrix 01p0}=
\begin{cases}
z^j\big|_{\Matrix 01p0} = (pz)^{k-2-j} & \textrm{if } i=3-i'
\\
0 & \textrm{if } i = i'.
\end{cases} 
$$

(3)
Let us check that $\psi_1 = \psi_2$ if and only if $k \equiv k_\varepsilon \bmod {(p-1)}$. Indeed, $\psi = \varepsilon \cdot (1\times \omega^{2-k}) = \omega^{-s_\varepsilon} \times \omega^{a+s_\varepsilon +2-k}$. So $\psi_1 = \psi_2 $ if and only if $-s_\varepsilon \equiv a+s_\varepsilon+2 -k \bmod{(p-1)}$, which is further equivalent to $k \equiv a+2s_\varepsilon +2\equiv k_\varepsilon \bmod{(p-1)}$.

Now we assume $\psi_1 \neq \psi_2$ and check the equality \eqref{E:Atkin-Lehner involution} and then the last sentence is clearly a corollary of this.
Thanks to \eqref{E:AL circ AL}, we prove $ U_p \circ \AL_{(k, \psi)} \circ U_p\circ \AL_{(k, \psi'')} = p^{2k-3}$ instead.
We compute using the formula \eqref{E:Up action}, for $\varphi \in \rmS_{k}^\Iw(\psi)$, 
\begin{eqnarray}
\label{E:U AL U AL computation}
&&U_p \circ \AL_{k, \psi} \circ U_p \circ  \AL_{k, \psi''} (\varphi)
\\
\nonumber
&=& \sum_{i,j  =0}^{p-1} \varphi \big(\bullet{\Matrix 01p0}^{-1} {\Matrix p0{ip} 1}^{-1} {\Matrix 01p0}^{-1} {\Matrix p0{jp}1}^{-1} \big) \big|_{{ \Matrix p0{jp} 1} \Matrix 01p0 {\Matrix p0{ip}1}\Matrix 01p0}
\\ \nonumber
&= & \sum_{i,j  =0}^{p-1} \varphi \big(\bullet {\Matrix p00p}^{-2} {\Matrix 1ij{1+ij}}^{-1} \big) \big|_{{ \Matrix 1ij{1+ij}{\Matrix p00p}}^2 }
\\ \nonumber
&= & p^{2k-4}\sum_{i,j  =0}^{p-1} \varphi \big(\bullet {\Matrix 1ij{1+ij}}^{-1} \big) \big|_{\Matrix 1ij{1+ij}}
\end{eqnarray}

We separate the discussion depending on whether $1+ij$ is divisible by $p$.
When $ij \equiv -1 \bmod p$, 
\begin{align*}
\sum_{ j = 1}^{p-1} \varphi  \big(\bullet {\Matrix 1{i_j}j0}^{-1} \big) \big|_{\Matrix 1{i_j} j0} &= \sum_{ j = 1}^{p-1} \varphi  \big(\bullet{\Matrix0{-1}10}^{-1} {\Matrix {-i_j}10j}^{-1} \big) \big|_{\Matrix {-i_j}10j \Matrix 0{-1}10} 
\\
&= \sum_{ j = 1}^{p-1} \psi_1(\bar j)\psi_2(\bar j)^{-1} \varphi  \big(\bullet{\Matrix0{-1}10}^{-1}\big) \big|_{\Matrix 0{-1}10} 
\end{align*}
where $i_j$ denote the unique integer in $\{1, \dots, p-1\}$ such that $j i_j \equiv -1 \bmod p$. Since $\psi_1 \neq \psi_2$, the sum of $\psi_1(\bar j) \psi_2(\bar j)^{-1}$ over all $\bar j  \in \Delta$ is zero. So the sum above is zero.

When $ij \not\equiv -1 \bmod p$, we may write
$$\MATRIX 1ij{1+ij}  = \MATRIX {1/(1+ij)}i0{1+ij} \MATRIX 10{j/(1+ij)} 1.$$
So we deduce
\begin{eqnarray}
\label{E:sum when ij neq -1}
&&\sum_{ij \not\equiv -1 \bmod p} \varphi \big(\bullet {\Matrix 1ij{1+ij}}^{-1} \big) \big|_{\Matrix 1ij{1+ij}}
\\
\nonumber
&=& \sum_{ij \not\equiv -1 \bmod p} \varphi \big(\bullet {{\Matrix 10{j/(1+ij)} 1}^{-1}\Matrix{1/(1+ij)}i0{1+ij}}^{-1} \big) \big|_{\Matrix {1/(1+ij)}i0{1+ij} \Matrix 10{j/(1+ij)} 1}
\\
\nonumber
& = &
\sum_{ij \not\equiv -1 \bmod p} \psi_1(\overline{1+ij})\psi_2(\overline{1+ij})^{-1} \varphi \big(\bullet {\Matrix 10{j/(1+ij)} 1}^{-1} \big) \big|_{ \Matrix 10{j/(1+ij)} 1}
\end{eqnarray}
Since $\psi_1 \neq \psi_2$, the sum of $\psi_1(\bar \alpha) \psi_2(\bar\alpha)^{-1}$ over all $\bar \alpha  \in \Delta$ is zero. So when $j \neq 0$, if we consider the above sum with $j/(1+ij)$ fixed yet letting $1+ij$ to run over all nonzero residual classes mod $p$,
we would get zero. 
Indeed, there is a bijection
$$
\begin{tikzcd}[row sep = 0pt]
\{\bar i, \bar j \in \FF_p \;|\; \bar j \neq 0 \textrm{ and }1+\bar i \bar j \neq 0 \} \ar[r, leftrightarrow] & \{( \bar \alpha, \bar \beta) \in (\FF_p^\times)^2\}
\\
(\bar i,\, \bar j) \ar[r, mapsto] &  (1+ \bar i \bar j,\, \bar j / (1+\bar i \bar j))
\\
(\bar\alpha \bar \beta,\, (\bar \alpha-1) / \bar \alpha \bar \beta ) & \ar[l, mapsto] (\bar \alpha,\, \bar \beta).
\end{tikzcd}
$$
So the sum in \eqref{E:sum when ij neq -1} is equal to the sum over the pairs $(i,0)$, i.e.
$$\sum_{ij \not\equiv -1 \bmod p} \varphi \big(\bullet {\Matrix 1ij{1+ij}}^{-1} \big) \big|_{\Matrix 1ij{1+ij}} = \sum_{i =0}^{p-1} \varphi \big(\bullet {\Matrix 1i0{1}}^{-1} \big) \big|_{\Matrix 1i0{1}} = p \varphi.
$$
Combining this with the discussion when $ij \equiv -1 \bmod p$ and \eqref{E:U AL U AL computation}, we deduce that
$$U_p \circ \AL_{k, \psi} \circ U_p \circ  \AL_{k, \psi''} (\varphi) = p^{2k-4} \cdot p \varphi = p^{2k-3} \varphi.$$
This concludes the proof of the Proposition.
\end{proof}
\begin{remark}\fakephantomsection\label{R:remark after AL for abstract forms}
\begin{enumerate}
\item 
We point out that $\rmS_{k}^\Iw(\psi)$ and $\rmS_{k}^\Iw(\psi'')$ often correspond to points on \emph{different} weight disks. More precisely, if $\varepsilon = \omega^{-s_\varepsilon} \times \omega^{a+s_\varepsilon}$ is a character of $\Delta^2$ relevant to $\bar\rho$.
Then applying above to $\psi = \omega^{-s_\varepsilon} \times \omega^{a+s_\varepsilon-k+2}$ we get an Atkin--Lehner involution, for $\varepsilon'' = \omega^{-s''_\varepsilon} \times \omega^{a+s''_\varepsilon}$ with $s''_\varepsilon: = \{k-  2-a-s_\varepsilon\}$,
$$\AL_{(k, \psi)}:  \rmS_{k}^\Iw\big(\varepsilon\cdot (1\times \omega^{2-k})\big) \to \rmS_{k}^\Iw\big(\varepsilon'' \cdot (1\times \omega^{2-k})\big).
$$
In particular, for different $k$'s, the Atkin--Lehner involutions are between different pairs of disks. (This phenomenon was known before, see e.g. \cite[Step III of \S3.23]{liu-wan-xiao}.)
\item
One can also deduce an Atkin--Lehner involution for abstract classical forms of higher level at $p$ (with primitive nebentypus characters). We do not need this generalization in this paper, so we leave it as an exercise for interested readers.
\item
For the formulas in Propositions~\ref{P:theta map} and \ref{P:Atkin-Lehner duality},  we essentially used the fact that the action of $\Matrix p00p$ on $\widetilde \rmH$ is given by multiplication with $\xi = 1$.  If $\xi \neq 1$ in general, the formulas above should be modified to
$$ \AL_{(k, \psi'')} \circ \AL_{(k, \psi)} = p^{k-2} \xi^{-1}, \quad \textrm{and}\quad U_p \circ \AL_{(k, \psi)} \circ U_p = p^{k-1} \xi^{-1} \cdot \AL_{(k, \psi)}.$$

\item 
The case when $\psi_1 \neq \psi_2$ is the usual Atkin--Lehner involution. However, what is particularly useful to us is the case when $\psi_1 = \psi_2$. While the equality \eqref{E:Atkin-Lehner involution} no longer holds, the operator $\AL_{(k, \psi)}$ is closely related to the $p$-stabilization and $T_p$- and $U_p$-operators. We refer to the sequel to this paper for a detailed discussion.
\end{enumerate}
\end{remark}



\section{Basic properties of ghost series}
\label{Sec:properties of ghost series}

In this section, we discuss some basic properties of the ghost series:
\begin{itemize}
\item We give the dimension formulas for the space of abstract classical forms (Propositions~\ref{P:dimension of SIw} and \ref{P:dimension of Sunr}).
\item 
From the dimension formulas, we deduce a formula in Proposition~\ref{P:increment of degrees in ghost series} describing the increment in degrees of the coefficients of the ghost series, comparing this to the halo estimate (\cite[Theorem 3.16]{liu-wan-xiao}) for the modified Mahler basis. In particular, we see that the halo estimate for the modified Mahler basis is ``surprisingly sharp".
\item 
Next, we discuss the compatibility of the ghost series with theta maps,  the Atkin--Lehner involutions, and the $p$-stabilizations (Proposition~\ref{P:ghost compatible with theta AL and p-stabilization}\,(1)--(3)).  
 \item
At the same time, we prove the most important property of ghost series, namely the \emph{ghost duality}, near a weight point where the space of abstract $p$-new forms  is nonzero. We refer to Proposition~\ref{P:ghost compatible with theta AL and p-stabilization}\,(4) for the statement.

\item Finally, we show that the ``old form slopes" of the ghost series at a classical weight $k$ is $\leq \lfloor \frac{k-1}{p+1}\rfloor$ (Proposition~\ref{P:gouvea k-1/p+1 conjecture}). This is the essential condition for proving folklore conjecture of Gouv\^ea on maximal old form slopes from local ghost conjecture.
\end{itemize}
Some of the proofs are rather combinatorially involved and somewhat tedious; we recommend passing the technical proofs for the first reading.

\medskip
Recall that we have fixed $\bar \rho = \Matrix{\omega_1^{a+1}}{*\neq 0}{}{1}$ with $1\leq a \leq p-4$, and  $\widetilde \rmH$ a \emph{primitive} $\calO\llbracket K_p\rrbracket$-projective augmented module of type $\bar \rho$ on which $\Matrix p00p$ acts trivially.
Fix a character $\varepsilon = \omega^{-s_\varepsilon} \times \omega^{a+s_\varepsilon}$ of $\Delta^2$ relevant to $\bar \rho$ with $s_\varepsilon \in \{0, \dots, p-2\}$ (as in Notation~\ref{N:relevant varepsilon}\,(2)); in particular, $\varepsilon_1 = e^{-s_\varepsilon}$.

We first deduce the dimension formula for $\rmS_k^\Iw(\psi)$ as a corollary of Proposition~\ref{P:power basis}.

\begin{proposition}
\label{P:dimension of SIw}We have the following dimension formulas
\begin{equation}
\label{E:dimension of SIw}
d_{k}^\Iw\big(\varepsilon\cdot (1\times \omega^{2-k})\big) = \Big\lfloor \frac{k-2-s_\varepsilon}{p-1}\Big\rfloor + \Big\lfloor \frac{k-2-\{a+s_\varepsilon\}}{p-1}\Big\rfloor +2.
\end{equation}
We may write this differently as
\begin{itemize}
\item When $a+s_\varepsilon<p-1$,
$$
d_{k}^\Iw\big(\varepsilon\cdot (1\times \omega^{2-k})\big) = 2\Big\lfloor \frac{k-2-s_\varepsilon}{p-1}\Big\rfloor +1+ \begin{cases}
1 & \textrm{ if } \{k-2-s_\varepsilon\} \geq a\\ 0 & \textrm{ otherwise.}
\end{cases}
$$
\item When $a+s_\varepsilon\geq p-1$,
$$
d_{k}^\Iw\big(\varepsilon\cdot (1\times \omega^{2-k})\big) = 2\Big\lfloor \frac{k-2-\{a+s_\varepsilon\}}{p-1}\Big\rfloor +1+ \begin{cases}
1 & \textrm{ if } \{k-2-\{a+s_\varepsilon\}\} \geq p-1-a\\ 0 & \textrm{ otherwise.}
\end{cases}
$$
\end{itemize}
As a special case of either case, we have
\[
d_{k+p-1}^\Iw\big(\varepsilon\cdot (1\times \omega^{2-k})\big) - d_{k}^\Iw\big(\varepsilon\cdot (1\times \omega^{2-k})\big) = 2.
\]
Asymptotically, we see that $$d_k^\Iw \big(\varepsilon\cdot(1 \times \omega^{2-k})\big) = \frac{2k}{p-1} + O(1).$$
\end{proposition}
\begin{proof}
The dimension formula \eqref{E:dimension of SIw} follows from Proposition~\ref{P:power basis} by a simple counting. The alternative writing also follows because
\begin{itemize}
\item when $a+s_\varepsilon<p-1$, $\{a+s_\varepsilon\} = a+s_\varepsilon > s_\varepsilon$; and
\item when $a+s_\varepsilon \geq p-1$, $\{a+s_\varepsilon\} = a+s_\varepsilon-(p-1)< s_\varepsilon$.
\end{itemize}
The last two statements are clear from the dimension formulas.
\end{proof}

\begin{notation}
\label{N:kbullet}
For $\varepsilon$ relevant, we set
$$
\delta_\varepsilon = \Big\lfloor \frac{s_\varepsilon + \{a+s_\varepsilon\}}{p-1}\Big \rfloor = \begin{cases}
0 & \textrm{ if }s_\varepsilon + \{a+s_\varepsilon\} < p-1,\\
1 & \textrm{ if }s_\varepsilon + \{a+s_\varepsilon\} \geq p-1.
\end{cases}
$$
We shall often consider the case when $k \equiv k_\varepsilon =2+ \{a+2s_\varepsilon \} \bmod{(p-1)}$ (in which case $\varepsilon\cdot (1\times \omega^{2-k}) = \tilde \varepsilon_1$). In this case, we set $$k = k_\varepsilon + k_\bullet(p-1) \quad \textrm{for }k_\bullet \in \ZZ.$$
\end{notation}

\begin{remark}
\label{R:delta}
An equality frequently used later is
\[
(p-1)\delta_{\varepsilon}+\{a+2s_{\varepsilon}\}=s_{\varepsilon}+\{a+s_{\varepsilon}\}.
\]
\end{remark}

\begin{corollary}
\label{C:dIw is even}
When $k = k_\varepsilon + k_\bullet (p-1)$, we have
$$
d_{k}^\Iw(\tilde \varepsilon_1) = 2k_\bullet +2 -2 \delta_\varepsilon.
$$

In particular,
the dimensions $d_{k}^\Iw(\tilde \varepsilon_1)$ and $d_{k}^{\mathrm{new}}(\varepsilon_1)$ are even integers, and asymptotically, $$d_k^\Iw(\tilde \varepsilon_1) = \frac{2k}{p-1}+O(1).$$ 
\end{corollary}
\begin{proof}
When $k \equiv k_\varepsilon \bmod{(p-1)}$, we have
$$\{k-2\} = \{k_\varepsilon-2\} = \{a+2s_\varepsilon\} \equiv s_\varepsilon + \{a+s_\varepsilon\} \bmod{(p-1)}.$$
So by \eqref{E:dimension of SIw}, 
\begin{itemize}
\item when $s_\varepsilon+\{a+s_\varepsilon\}<p-1$, we have $\{k-2\} = s_\varepsilon+\{a+s_\varepsilon\} \geq \max\big( s_\varepsilon, \{a+s_\varepsilon\}\big)$, in which case $d_{k}^\Iw(\tilde \varepsilon_1) = 2 \cdot \big\lfloor \frac{k-2}{p-1}\big\rfloor +2 =2k_\bullet +2$,
\item when $s_\varepsilon+\{a+s_\varepsilon\} \geq p-1$, we have $\{k-2\} = s_\varepsilon+\{a+s_\varepsilon\} - p-1< \min\big( s_\varepsilon, \{a+s_\varepsilon\} \big)$, in which case $d_{k}^\Iw(\tilde \varepsilon_1) = 2 \cdot \big\lfloor \frac{k-2}{p-1}\big\rfloor =2 k_\bullet $.
\end{itemize}
In either case, $d_{k}^\Iw(\tilde \varepsilon_1)$ is an even integer; so $d_{k}^{\textrm{new}}(\varepsilon_1) = d_{k}^\Iw(\tilde \varepsilon_1)- 2 d_{k}^\unr(\tilde \varepsilon_1)$ is also an even integer.
The asymptotic of $d_{k}^\Iw(\tilde \varepsilon_1)$ is clear.
\end{proof}

\begin{notation}
For a (right) $\FF$-representation $\sigma$ of $\GL_2(\FF_p)$, we write $\JH(\sigma)$ for the multiset of its Jordan--H\"older factors.  For a Serre weight $\sigma_{a',b'}$ (see Notation~\ref{N:Serre weights}), we write $\Mult_{\sigma_{a',b'}}(\sigma)$ for the multiplicity of $\sigma_{a',b'}$ in $\JH(\sigma)$.
\end{notation}

\subsection{Explicit computation of $\rmS_{k}^\unr$}
\label{S:explicit Sunr}
Let $\widetilde \rmH$ and $\varepsilon$ be as above.
Since $\overline \rmH$ is a (right) projective $\FF[\GL_2(\FF_p)]$-module, we have the following sequence of equalities, for $k \equiv k_\varepsilon \bmod {(p-1)}$
\begin{align}
\label{E:dimension formula dkunr}
d_k^{\unr}(\varepsilon_1) &= \rank_\calO \rmS_{k}^{\unr}(\varepsilon_1) = \rank \Hom_{\calO[ \rmK_p]}\big( \widetilde \rmH,\, \Sym^{k-2}\calO^{\oplus 2} \otimes \varepsilon_1 \circ \det\big)
\\
\nonumber
& = \dim_\FF \Hom_{\FF[\GL_2(\FF_p)]} \big( \overline \rmH, \sigma_{k-2, -s_\varepsilon} \big) = \dim_\FF \Hom_{\FF[\GL_2(\FF_p)]} \big( \Proj_{a,s_\varepsilon}, \sigma_{k-2,0} \big)
\\
\nonumber
& =  \Mult_{\sigma_{a,s_\varepsilon}}(\sigma_{k-2,0}). 
\end{align}
The last equality used the exactness of $\Hom_{\FF[\GL_2(\FF_p)]}(\Proj_{a,s_\varepsilon},-)$.
Write
$
k= k_\varepsilon + k_\bullet (p-1)$ as in Notation~\ref{N:kbullet}.
Recall the definition of $\delta_\varepsilon$ from Corollary~\ref{C:dIw is even}.
Using the notation above, we introduce two integers $t_1=t_1^{(\varepsilon)}$ and $t_2  =t_2^{(\varepsilon)}$ as follows:
\begin{itemize}
\item when $a+s_\varepsilon<p-1$, $t_1 = s_\varepsilon+\delta_\varepsilon$ and $t_2 = a+s_\varepsilon + \delta_\varepsilon+2$;
\item when $a+s_\varepsilon \geq p-1$, $t_1 =  \{a+s_\varepsilon\} + \delta_\varepsilon+1$ and $t_2 = s_\varepsilon +\delta_\varepsilon+ 1$.
\end{itemize}

\begin{proposition}
\label{P:dimension of Sunr}
We have
\begin{eqnarray}
\label{E:dkunr}
d_{k}^\unr(\varepsilon_1) &=&
\Big\lfloor \frac{k_\bullet - t_1}{p+1}\Big\rfloor + \Big\lfloor \frac{k_\bullet - t_2}{p+1}\Big\rfloor + 2
\\ \nonumber &=& 2\Big\lfloor \frac{k_\bullet - t_1}{p+1}\Big\rfloor + 1 +\begin{cases}
1 & \textrm{ when }k_\bullet - (p+1) \lfloor \frac{k_\bullet-t_1}{p+1}\rfloor  \geq t_2
\\
0 & \textrm{ otherwise.}
\end{cases}
\end{eqnarray}
In particular,  $d_{k}^\unr(\varepsilon_1)$ is non-decreasing in the class $k \equiv k_\varepsilon \bmod{(p-1)}$, and we have
$$d_{k+p^2-1}^\unr(\varepsilon_1) - d_{k}^\unr(\varepsilon_1) = 2.$$
Asymptotically, $d_{k}^\unr(\varepsilon_1) = \frac{2k}{p^2-1}+O(1)$ and  $d_k^\new(\varepsilon_1) = \frac{2k}{p+1}+O(1)$.
\end{proposition}
\begin{proof}It is enough to compute the Jordan--H\"older factors of $\sigma_{k-2,0}$.
By Lemmas~\ref{L:exact sequence of (p+1)-steps jump} and \ref{L:exact sequence of induced representation}\,(1) (and the congruence $k_\varepsilon-2 \equiv a+2s_\varepsilon \bmod{(p-1)}$), we have the following equality in the Grothendieck group  $\mathrm{Groth}(\FF[\GL_2(\FF_p)])$ of (right) $\FF[\GL_2(\FF_p)]$-modules:
$$
[\sigma_{k-2,0}]-[\sigma_{k-2-(p+1), 1}]=[\sigma_{\{a+2s_\varepsilon\},0}]+[\sigma_{p-1-\{a+2s_\varepsilon\}, \{a+2s_\varepsilon\}}].
$$
If we replace $k-2$ by $k-2-(p+1)$ in the above equality and then twist by the character $\det$, we deduce 
$$
[\sigma_{k-2-(p+1), 1}]-[\sigma_{k-2-2(p+1), 2}]=[\sigma_{\{a+2s_\varepsilon-2\},1}]+[\sigma_{p-1-\{a+2s_\varepsilon-2\}, \{a+2s_\varepsilon\}-1}].
$$
		
Now assume that $k-2\geq (p+1)(p-1)$, then we can repeat the above		computation $p-1$ times to get $p-1$ equalities:
\begin{align}
\label{E:inductive formulas for symmetric squares}
[\sigma_{k-2,0}]-[\sigma_{k-2-(p+1), 1}]&=[\sigma_{\{a+2s_\varepsilon\},0}]+[\sigma_{p-1-\{a+2s_\varepsilon\}, \{a+2s_\varepsilon\}}]
\\
\nonumber
[\sigma_{k-2-(p+1), 1}]-[\sigma_{k-2-2(p+1), 2}] &=[\sigma_{\{a+2s_\varepsilon-2\},1}]+[\sigma_{p-1-\{a+2s_\varepsilon-2\}, \{a+2s_\varepsilon\}-1}].
\\
\nonumber
\cdots & \cdots\cdots\\
\nonumber
[\sigma_{k-2-(p-2)(p+1), p-2}]-[\sigma_{k-2-(p-1)(p+1), 0}]&=[\sigma_{\{a+2s_\varepsilon+2\},p-2}]+[\sigma_{p-1-\{a+2s_\varepsilon+2\}, \{a+2s_\varepsilon\}-(p-2)}].
\end{align}
The sum of the left hand sides of \eqref{E:inductive formulas for symmetric squares} is equal to $[\sigma_{k-2,0}] - [\sigma_{k-2-(p-1)(p+1), 0}]$.
On the right sides of \eqref{E:inductive formulas for symmetric squares}, the Serre weight $\sigma_{a, s_\varepsilon}$ appears exactly twice: one time, it appears in the first summand of the $(s_\varepsilon+1)$th equality, and at the other time, it appears in the second summand of the $(\{a+s_\varepsilon\}+1)$th equality. So combining this with \eqref{E:dimension formula dkunr}, we deduce that $$d_k^{\unr}(\varepsilon_1)=d_{k-(p-1)(p+1)}^{\unr}(\varepsilon_1)+2.$$
It then remains to prove \eqref{E:dkunr} when $k-2< p^2-1$ (or equivalently $k_\bullet < p+1$).

In this case, we write $k-2=s(p+1)+r$ with $s\in \{0, \dots, p-2\}$ and $r \in \{0, \dots, p\}$. We can then repeat the argument above but only using the first $s$ rows of \eqref{E:inductive formulas for symmetric squares} to deduce that $[\sigma_{k-2,0}] - [\sigma_{r,s}]$ is equal to the sum of the right hand sides of first $s$ rows of \eqref{E:inductive formulas for symmetric squares}.
\begin{enumerate}
\item 
The ``remainder" term $[\sigma_{r,s}]$ contains $\sigma_{a,s_\varepsilon}$ as a Jordan--H\"older factor  if and only if $s=s_\varepsilon$ and $a  \equiv r \bmod{(p-1)}$ (note that the latter includes the case that $a=1$ and $r=p$).  We immediately remark that when $s=s_\varepsilon$, since $\{a+2s_\varepsilon\} + k_\bullet(p-1) = s_\varepsilon (p+1) +r$, the condition $a \equiv r \bmod{(p-1)}$ is automatic.

\item 
If $s\geq s_\varepsilon +1$, the Serre weight $\sigma_{a, s_\varepsilon}$ will appear in the first summand of the right hand side of (the $(s_\varepsilon + 1)$th row) of \eqref{E:inductive formulas for symmetric squares}. This will contribute a $\sigma_{a, s_\varepsilon}$ in the Jordan--H\"older factors of $[\sigma_{k-2,0}]$.

\item 
If $s \geq \{a+s_{\varepsilon}\} + 1$, 
 the Serre weight $\sigma_{a, s_\varepsilon}$ will appear in the second summand of the right hand side of (the $(\{a+s_{\varepsilon}\} + 1)$th row) of \eqref{E:inductive formulas for symmetric squares}. This will contribute a $\sigma_{a, s_{\varepsilon}}$ in the Jordan--H\"older factors of $[\sigma_{k-2,0}]$.
\end{enumerate}
Combining (1) and (2), precisely when $s \geq s_\varepsilon$, (1) and (2) will contribute one $\sigma_{a, s_\varepsilon}$ to the Jordan--H\"older factor of $[\sigma_{k-2,0}]$. Note that $\{a+2s_\varepsilon \}+k_\bullet(p-1)\equiv s_\varepsilon(p+1)+a\equiv a+2s_\varepsilon \bmod (p-1)$. Since $\lfloor \frac{\{a+2s_\varepsilon\} + k_\bullet(p-1)}{p+1} \rfloor=\lfloor \frac{k-2}{p+1}\rfloor=s$ and $\lfloor\frac{s_\varepsilon(p+1)+a}{p+1}\rfloor=s_\varepsilon>\lfloor\frac{s_\varepsilon(p+1)+a-(p-1)}{p+1}\rfloor$, it is straightforward to see that $s \geq s_\varepsilon$ if and only if the following equivalent conditions hold:
\begin{equation}
\label{E:t1}
\{a+2s_\varepsilon\} + k_\bullet(p-1) \geq s_\varepsilon(p+1)  + a \quad  \Longleftrightarrow \quad k_\bullet \geq s_\varepsilon + \frac{a+2s_\varepsilon - \{a+2s_\varepsilon\}}{p-1}.
\end{equation}
On the other hand, the condition $s \geq \{a+s_\varepsilon\}+1$ of (3) is equivalent to
$$
\{a+2s_\varepsilon\} + k_\bullet(p-1) \geq (p+1) \big( \{a+s_\varepsilon\}+1\big),
$$
which is further equivalent to
\begin{equation}
\label{E:t2}
k_\bullet \geq \Big\lceil \frac{(p+1)\{a+s_\varepsilon\}-\{a+2s_\varepsilon\}+p+1}{p-1}\Big\rceil = \{a+s_\varepsilon\} + 2 + \frac{2\{a+s_\varepsilon\} - \{a+2s_\varepsilon\} - a}{p-1}.
\end{equation}

Now, when $a+s_\varepsilon<p-1$, \eqref{E:t1} is equivalent to $k_\bullet \geq s_\varepsilon+\delta_\varepsilon = t_1$ and \eqref{E:t2} is equivalent to $k_\bullet \geq a+s_\varepsilon + 2+\delta_\varepsilon = t_2$. So \eqref{E:dkunr} follows.
When $a+s_\varepsilon \geq p-1$,
\eqref{E:t1} is equivalent to $k_\bullet \geq s_\varepsilon+\delta_\varepsilon + 1 = t_2$ and \eqref{E:t2} is equivalent to $k_\bullet \geq \{a+s_\varepsilon\} + 1+\delta_\varepsilon = t_1$. So \eqref{E:dkunr} follows.
\end{proof}

\begin{corollary}\fakephantomsection
\label{C:ghost series depends on bar rho}
For $\varepsilon$ a relevant character for $\bar\rho$ and $\widetilde \rmH$ a primitive $\calO\llbracket K_p\rrbracket$-projective augmented module, the dimensions $d_k^\Iw(\tilde \varepsilon_1)$ and $d_k^\unr(\varepsilon_1)$ for every $k \in \NN_{\geq 2}$ depend only on $\bar \rho$, $k$, and $\varepsilon$, but not on (the $\GL_2(\QQ_p)$-action on) $\widetilde \rmH$.

Moreover, for a fixed $n$, the ghost multiplicity $m_n^{(\varepsilon)}(k)$ in Definition~\ref{D:ghost series} is zero when $d_{k}^\unr(\varepsilon_1) \geq n$ (which is certainly the case when $k/(p^2-1) > n/2+2$). So the coefficients of the ghost series $g_n^{(\varepsilon)}(w)$ are polynomials in $w$.
\end{corollary}
\begin{proof}
This is clear from Propositions~\ref{P:dimension of SIw} and \ref{P:dimension of Sunr}.
\end{proof}

\begin{remark}
\label{R:nonprimitive ghost}
\begin{enumerate}

\item 
If $\widetilde \rmH'$ is an $\calO\llbracket K_p\rrbracket$-projective augmented module of type $\bar\rho$ with multiplicity $m(\widetilde \rmH')$, we note that as $\calO\llbracket \GL_2(\ZZ_p)\rrbracket$-modules, $\widetilde \rmH' \simeq \widetilde \rmH^{\oplus m(\widetilde \rmH')}$. Yet formulas \eqref{E:dimension of SIw} and \eqref{E:dimension formula dkunr} only makes use of the $\calO\llbracket \GL_2(\ZZ_p)\rrbracket$-module structure of $\widetilde \rmH$. It then follows that, for $k \in \NN_{\geq 2}$,  a character $\psi$ of $\Delta^2$, and a character $\varepsilon_1$ of $\Delta$,
$$
\rank_\calO \rmS^\Iw_{\widetilde \rmH', k}(\psi) = m(\widetilde \rmH') d_k^\Iw(\psi) \quad \textrm{and}\quad \rank_\calO \rmS_{\widetilde \rmH',k}^\ur(\varepsilon_1) = m(\widetilde \rmH') d_k^\unr(\varepsilon_1).
$$
So these ranks also only depend on $\bar \rho$, $\psi$, $\varepsilon_1$, $k$, and (linearly on) $m(\widetilde \rmH')$.
\item
The fact that $d_{k}^\unr(\varepsilon_1)$ is increasing in $k$ relies on the assumption that the parameter $a$ is generic. The analogous statement for non-generic $a$ is false (as in the original examples considered by Bergdall--Pollack \cite{bergdall-pollack2}).
\end{enumerate}
\end{remark}

Now, we move to discuss the slopes when $w_\star \in \gothm_{\CC_p}$ has valuation $v_p(w_\star) \in (0,1)$, the so-called \emph{spectral halo range}. Note that for every $k \in \ZZ$, $v_p(w_k) = v_p(\exp(p(k-2))-1) \geq 1$; so $v_p(w_\star- w_k) = v_p(w_\star)$ in this case, and thus, for the $n$th term of the ghost series
$$
v_p(g_n^{(\varepsilon)}(w_\star)) =  v_p(w_\star) \cdot \deg g_n^{(\varepsilon)}(w).
$$
In particular, this valuation is linear in $v_p(w_\star)$. 
This is a direct analogue of Coleman--Mazur--Buzzard--Kilford's spectral halo conjecture, as partially proved in \cite{liu-wan-xiao}. The key to that proof is to provide a lower bound on the Newton polygon of the  characteristic power series, called the \emph{Hodge polygon} in that paper.
Formally translating the construction from that paper, we expect that when $v_p(w_\star) \in (0,1)$, the Newton polygon of the ghost series should lie above the polygon, which we call \emph{ghost Hodge polygon}, whose $n$th slope is given by $v_p(w_\star) \cdot \lambda_n^{(\varepsilon)}$ with  $$\lambda_n^{(\varepsilon)}(w)= \deg \bfe_n^{(\varepsilon)} - \big\lfloor \tfrac{\deg \bfe_n^{(\varepsilon)}}p\big\rfloor.$$
It turns out that $\NP(C(w_\star, -))$ is ``very close" to this ghost Hodge polygon, as proved in Proposition~\ref{P:increment of degrees in ghost series} below.

\begin{notation}
\label{N:tnvarepsilon}
For $n \in \ZZ$ and $\varepsilon$ a relevant character, we set
$\beta_{[n]}^{(\varepsilon)} = \begin{cases}
t_1^{(\varepsilon)} & \textrm{if }n \textrm{ is even}
\\
t_2^{(\varepsilon)}-
\tfrac{p+1}2 & \textrm{if }n \textrm{ is odd}
\end{cases}$.
Since this depends only on the parity of $n$, we often write $\beta_\even^{(\varepsilon)}$ and $\beta_\odd^{(\varepsilon)}$ for the values.
\end{notation}

\begin{proposition}
\label{P:increment of degrees in ghost series}
Fix $\varepsilon$ a relevant character.  If $a+s_\varepsilon<p-1$, then we have
\begin{equation}
\label{E:increment of degree a+s<p-1}
\deg g_{n+1}^{(\varepsilon)}(w) - \deg g_{n}^{(\varepsilon)}(w) - \lambda_{n+1}^{(\varepsilon)} =\begin{cases}
1 & \textrm{ if } n - 2s_\varepsilon \equiv 1, 3, \dots, 2a+1 \bmod {2p},
\\
-1 & \textrm{ if } n- 2s_{\varepsilon} \equiv 2, 4, \dots, 2a+2 \bmod{2p},
\\
0 & \textrm{ otherwise.}
\end{cases}
\end{equation}
If $a+s_\varepsilon\geq p-1$, then we have
\begin{equation}
\label{E:increment of degree a+s>=p-1}
\deg g_{n+1}^{(\varepsilon)}(w) - \deg g_{n}^{(\varepsilon)}(w) - \lambda_{n+1}^{(\varepsilon)} =\begin{cases}
1 & \textrm{ if } n - 2s_\varepsilon \equiv 2, 4, \dots, 2a+2 \bmod {2p},
\\
-1 & \textrm{ if } n- 2s_{\varepsilon} \equiv 3, 5, \dots, 2a+3 \bmod{2p},
\\
0 & \textrm{ otherwise.}
\end{cases}
\end{equation}
In either case, we have
\begin{equation}
\label{E:degree approx halo bound}
\deg g_n^{(\varepsilon)}(w) - (\lambda_1^{(\varepsilon)}+\cdots + \lambda_n^{(\varepsilon)}) = \begin{cases}
0 & \textrm{ if }\deg \bfe_{n+1}-\deg\bfe_n = a,\\
0 \textrm{ or }1 & \textrm{ if }\deg \bfe_{n+1}-\deg\bfe_n = p-1-a.
\end{cases}
\end{equation}
Moreover, the differences $\deg g_{n+1}^{(\varepsilon)}(w) - \deg g_n^{(\varepsilon)}(w)$ are strictly increasing in $n$.
\end{proposition}
\begin{proof}
In this proof, we shall suppress the $\varepsilon$ from the notation when no confusion arises.
Since $\deg g_n(w)$ is the sum of ghost multiplicities $m_n(k)$ for all $k \equiv k_\varepsilon \bmod{(p-1)}$, the formula \eqref{E:increment of multiplicity} for the increments of the ghost multiplicities implies that
\begin{equation}
\label{E:gn+1 - gn}
\deg g_{n+1} (w) - \deg g_n(w) = \hspace{-10pt}
\sum_{\substack{k \equiv k_\varepsilon \bmod{(p-1)}
\\
d_{k}^\unr \leq n < \frac 12 d_{k}^\Iw}}
\hspace{-10pt}1\qquad -  \hspace{-10pt} \sum_{\substack{k \equiv k_\varepsilon \bmod{(p-1)}
\\
\frac 12 d_{k}^\Iw\leq n < d_{k}^{\Iw} - d_{k}^\unr }}
\hspace{-15pt}1.
\end{equation}
Note that $d_{k}^\unr$, $d_{k}^\Iw$, and $ d_{k}^{\Iw} - d_{k}^\unr$ are increasing functions in $k$ within the congruence class (for the last one, as $k$ is increased by $p-1$, $d_{k}^\Iw$ is increased by $2$ and $d_{k}^\unr$ is increased by at most $1$).

The following picture illustrates the contribution to the degree increment $\deg g_{n+1}(w) -\deg g_n(w)$. The contribution of the blue interval from $k_{\min}$ to $k_\midd$ is negative and the contribution of the red interval from $k_\midd$ to $k_{\max}$ is positive. More precisely, we define these three integers as
\begin{align}\label{E:first definition of extremal ks}
k_{\min}&=k_{\min}(n):=\min\{k|~k\equiv k_\varepsilon \bmod (p-1) \text{~and~} d_k^\Iw-d_k^\ur>n \}\\
k_{\midd}&=k_{\midd}(n):=\min\{k|~k\equiv k_\varepsilon \bmod (p-1) \text{~and~} \frac 12 d_k^\Iw\leq n \}\nonumber\\
k_{\max}&=k_{\max}(n):=\min\{k|~k\equiv k_\varepsilon \bmod (p-1) \text{~and~} d_k^\ur\leq n \}\nonumber
\end{align}

\begin{center}
\begin{tikzpicture}[line cap=round,line join=round,>=triangle 45,x=1cm,y=5cm]
\draw [color=cqcqcq,, xstep=1cm,ystep=1cm] (-0.5,-0.1) grid (7,0.6);
\draw[->,color=black] (0,0) -- (7,0);
\foreach \x in {1,2,3,4,5,6}
\draw[shift={(\x,0)},color=black] (0pt,2pt) -- (0pt,-2pt);
\draw[->,color=black] (0,0) -- (0,0.6);
\foreach \y in {0.2,0.4,0.6}
\draw[shift={(0,\y)},color=black] (2pt,0pt) -- (-2pt,0pt);
\clip(-0.5,-0.1) rectangle (7,0.6);
\draw [line width=1pt,dash pattern=on 2pt off 3pt] (6,0.25)-- (0,0.25);
\draw [line width=1pt,dash pattern=on 2pt off 3pt] (1.5,0.4375)-- (1.5,0);
\draw [line width=1pt,domain=0:7] plot(\x,{(-0--0.25*\x)/6});
\draw [line width=1pt,domain=0:7] plot(\x,{(-0--1.75*\x)/6});
\draw [line width=1pt,domain=0:7] plot(\x,{(-0--2*\x)/6});
\draw [line width=2pt,color=red] (6,0.25)-- (1.5,0.25);
\draw [line width=2pt,color=blue] (0.87,0.25)-- (1.5,0.25);
\draw [line width=1pt,dash pattern=on 2pt off 3pt] (6,0.25)-- (6,0);
\draw [line width=1pt,dash pattern=on 2pt off 3pt] (0.86,0.25)-- (0.86,0);
\begin{scriptsize}
\draw [fill=black] (0,0.25) circle (1pt);
\draw[color=black] (-0.2,0.26) node {$n$};
\draw [fill=black] (1.5,0) circle (1pt);
\draw[color=black] (1.5,-0.05) node {$k_{\midd}$};
\draw[color=black] (6,-0.05) node {$k_{\max}$};
\draw[color=black] (0.8,-0.05) node {$k_{\min}$};
\draw [fill=black] (1.5,0.0625) circle (1.5pt);
\draw [fill=black] (1.5,0.4375) circle (1.5pt);
\draw[color=black] (4.7,0.13) node {$d_{k}^\unr$};
\draw[color=black] (2.86,0.55) node {$d_{k}^\Iw-d_{k}^\unr$};
\draw[color=black] (1.13,0.55) node {$d_{k}^\Iw$};
\draw[color=blue] (1.2,0.21) node {$-1$};
\draw[color=red] (3.7,0.29) node {$+1$};
\end{scriptsize}
\end{tikzpicture}

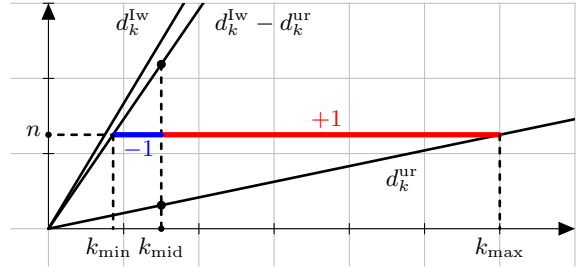
Graph 2: contribution to increments of degrees.
\end{center}

It remains to express $k_{\midd \bullet}$, $k_{{\max}\bullet}$, and $k_{{\min}\bullet}$ in terms of $n$. Since this result will be used later, we state (a more refined version of) it separately here.
\begin{lemmanotation}
\label{L:extremal ks}
We only consider those integers $k \equiv k_\varepsilon \bmod{(p-1)}$ and write $k=k_\varepsilon+(p-1)k_{\bullet}$. Fix $n\in \ZZ_{ \geq 0}$.
Recall the definition of $\delta_\varepsilon$ in Notation~\ref{N:kbullet}, $t_1^{(\varepsilon)}$ and $t_2^{(\varepsilon)}$ from Proposition~\ref{P:dimension of Sunr}, and $\beta_{[n]}^{(\varepsilon)}$ from Notation~\ref{N:tnvarepsilon}

\begin{enumerate}
\item There is a unique integer $k_\bullet$ such that  $n=\frac 12 d_{k}^\Iw(\tilde \varepsilon_1)$; it is
$$k_\bullet = k_{\midd\bullet}^{(\varepsilon)}(n): = n+\delta_\varepsilon-1.$$

\item If $k_\bullet \in \ZZ_{\geq 0}$ satisfies $n = d_{k}^\unr(\varepsilon_1)$, then
\begin{align*}
\textrm{when } n \textrm{ is even,}&\quad (p+1)\tfrac {n-2}2 + t_2^{(\varepsilon)} \leq  k_\bullet  < (p+1)\tfrac {n}2 + t_1^{(\varepsilon)},
\\
\textrm{when } n \textrm{ is odd,}&\quad (p+1)\tfrac {n-1}2 + t_1^{(\varepsilon)} \leq k_\bullet < (p+1)\tfrac {n-1}2 + t_2^{(\varepsilon)}.
\end{align*}
We denote the largest such $k_\bullet$ by  $$k_{{\max}\bullet}^{(\varepsilon)}(n)=\frac{p+1}2 n + \beta_{[n]}^{(\varepsilon)}-1.$$

\item 
There is at most one $k_\bullet \in \ZZ_{\geq 0}$ such that  $n = d_{k}^\Iw(\tilde \varepsilon_1)-d_{k}^\unr(\varepsilon_1)$.
Such $k_\bullet$ exists if and only if the interval
\begin{align}
\label{E:condition on kmin}
\big(\,(p+1)( \tfrac {n-2}2 + \delta_\varepsilon)  -t_1^{(\varepsilon)} , \, (p+1)( \tfrac n2 + \delta_\varepsilon)  -t_2^{(\varepsilon)}\, \big]& \textrm{ if }n \textrm{ is even, and}
\\
\nonumber \big(\,(p+1)( \tfrac {n-1}2 + \delta_\varepsilon)  -t_2^{(\varepsilon)} , \, (p+1)( \tfrac {n-1}2 + \delta_\varepsilon)  -t_1^{(\varepsilon)}\, \big] &\textrm{ if }n \textrm{ is odd}
\end{align}
contains a multiple of $p$, which would be $pk_\bullet$. We write $$\tilde k_{{\min}\bullet}^{(\varepsilon)}(n)= \frac{p+1}2(n-1+2\delta_\varepsilon) - \beta_{[n-1]}^{(\varepsilon)}+1$$ for the \emph{right end} of the above interval plus $1$; and set $k_{{\min}\bullet}^{(\varepsilon)}(n) = \lceil \tilde k_{{\min}\bullet}^{(\varepsilon)} (n)/ p\rceil$.

\emph{\underline{Warning:} despite the name, $k_{{\min}\bullet}^{(\varepsilon)}(n)$ is never  a $k_\bullet$ such that $n = d_{k}^\Iw(\tilde \varepsilon_1) - d_{k}^\unr(\varepsilon_1)$.}
\end{enumerate}
Finally, for 
$? = \midd$, $\max$, and $\min$, we put $k_?^{(\varepsilon)}(n) = k_\varepsilon + (p-1)k_{?\bullet}^{(\varepsilon)}(n)$. We set
$$\tilde k_{{\min}}^{(\varepsilon)}(n)  = pk_\varepsilon + (p-1)\tilde k_{{\min}\bullet}^{(\varepsilon)}(n).
$$
\end{lemmanotation}
\begin{remark}
\label{R:meaning of kmidminmax}
From the definition of $k_{\mathrm{mid}\bullet}^{(\varepsilon)}(n)$, $k_{\mathrm{min}\bullet}^{(\varepsilon)}(n)$, and $k^{(\varepsilon)}_{\mathrm{max} \bullet}(n)$, we see that
\begin{itemize}
\item $\frac 12d_k^\Iw(\tilde \varepsilon_1) = n \Leftrightarrow k_\bullet = k_{\mathrm{mid}\bullet}^{(\varepsilon)}(n)$.
\item $d_k^\ur(\varepsilon_1) \leq  n \Leftrightarrow k_\bullet \leq k_{\mathrm{max}\bullet}^{(\varepsilon)}(n)$ and $d_k^\ur(\varepsilon_1) \geq   n \Leftrightarrow k_\bullet > k_{\mathrm{max}\bullet}^{(\varepsilon)}(n-1)$.
\item 
$d_k^\Iw(\tilde \varepsilon_1) - d_k^\ur(\varepsilon_1) \leq  n \Leftrightarrow k_\bullet < k_{\mathrm{min}\bullet}^{(\varepsilon)}(n)$ and $d_k^\Iw(\tilde \varepsilon_1) -d_k^\ur(\varepsilon_1) \geq   n \Leftrightarrow k_\bullet \geq k_{\mathrm{min}\bullet}^{(\varepsilon)}(n-1)$.
\end{itemize}
In particular, we see that the numbers defined in Lemma-Notation~\ref{L:extremal ks} coincide with those defined in (\ref{E:first definition of extremal ks}).
\end{remark}
\begin{proof}
(1) This follows from the dimension formula in Corollary~\ref{C:dIw is even}: $d_{k}^{\Iw}(\tilde \varepsilon_1) = 2k_\bullet + 2-2\delta_\varepsilon$.

(2) This follows from inverting the dimension formula in Proposition~\ref{P:dimension of Sunr}. (There is no problem with the statement when $n=0$, as we asked for $k_\bullet \geq 0$.)

(3) By Corollary~\ref{C:dIw is even}, $d_{k}^{\Iw}(\tilde \varepsilon_1) = 2k_\bullet + 2-2\delta_\varepsilon$. So the equality $n = d_{k}^\Iw(\tilde \varepsilon_1)-d_{k}^\unr(\varepsilon_1)$ is equivalent to
$$
d_{k}^\unr(\varepsilon_1)\ =  (2k_\bullet+2-2\delta_\varepsilon)-n.
$$
When $n$ is even, this means that
$$(p+1)\tfrac { (2k_\bullet+2-2\delta_\varepsilon)-n-2}2 + t_2 \leq k_\bullet  < (p+1)\tfrac {(2k_\bullet+2-2\delta_\varepsilon)-n}2 + t_1,
$$
or equivalently
$$ (p+1)( \tfrac {n-2}2 + \delta_\varepsilon) -t_1 < pk_\bullet  \leq (p+1)( \tfrac n2 + \delta_\varepsilon)  -t_2.
$$
The case when $n$ is odd follows from the same argument.
\end{proof}

Now, we continue the proof of Proposition~\ref{P:increment of degrees in ghost series}.  By \eqref{E:gn+1 - gn}, we write
$$
\deg g_{n+1}(w) - \deg g_n(w) = \big(k_{{\max}\bullet}(n) - k_{\midd \bullet}(n)\big) - \big(k_{\midd \bullet}(n) -  k_{{\min}\bullet}(n)+1 \big).$$
We shall compare $k_{{\max}\bullet}(n)-k_{\midd \bullet}(n)$ with $\deg \bfe_{n+1}$ and   $k_{\midd \bullet}(n)-k_{{\min}\bullet}(n)$ with $\lfloor \frac{\deg\bfe_{n+1}}{p} \rfloor$.

By Lemma-Notation~\ref{L:extremal ks}\,(1)(2),  $k_{{\max}\bullet}(n)-k_{\midd \bullet}(n)=\frac{p-1}{2}\cdot n+\beta_{[n]}-\delta_{\varepsilon}$, we list the values of $k_{{\max}\bullet}(n)-k_{\midd \bullet}(n)$ and $\deg \bfe_{n+1}$ in the following table:
\begin{center}
\begin{tabular}{|c|c|c|c|}
\hline
 & &$k_{{\max} \bullet}(n)-k_{\midd \bullet}(n)$ & $\deg \bfe_{n+1}$ 
\\
\hline
\multirow{2}{*}{$a+s_\varepsilon < p-1$} & $n=2n_0$ & $(p-1)n_0+s_\varepsilon$& $s_\varepsilon+(p-1)n_0$
\\
\cline{2-4}
& $n=2n_0+1$ &  $(p-1)n_0+a+s_\varepsilon+1$ & $a+s_\varepsilon+(p-1)n_0$
\\
\hline 
\multirow{2}{*}{
$a+s_\varepsilon \geq p-1$} & $n=2n_0$ &$(p-1)n_0+\{a+s_\varepsilon\}+1$ & $\{a+s_\varepsilon\}+(p-1)n_0$ 
\\
\cline{2-4} & $n=2n_0+1$ & $(p-1)n_0+s_\varepsilon$ & $s_\varepsilon+(p-1)n_0$ 
\\
\hline
\end{tabular}
\end{center}
In other words, by the explicit description of power basis \eqref{E:basis of Sdagger},
\begin{equation}
\label{E:kmax-kmid}
k_{\max {\bullet}}(n) - k_{\midd\bullet}(n)  - \deg \bfe_{n+1} = \begin{cases}
0& \textrm{ if }\bfe_{n+1} = e_1^*z^{s_\varepsilon+ (p-1)\lfloor n/2\rfloor},
\\
1 & \textrm{ if }\bfe_{n+1} = e_2^*z^{\{a+s_\varepsilon\}+ (p-1)\lfloor n/2\rfloor}.
\end{cases}
\end{equation}

By Lemma-Notation~\ref{L:extremal ks}\,(1)(3), we have 
$$
k_{{\midd}\bullet}(n)-k_{{\min} \bullet}(n)+1=n-\Big\lceil \frac{\frac{p+1}{2}(n-1)+\delta_\varepsilon-\beta_{[n-1]}+1}{p}\Big\rceil.$$
We list the values of $k_{{\midd}\bullet}(n)-k_{{\min}\bullet}(n)+1$ and $\lfloor\frac {\deg \bfe_{n+1}}p\rfloor$ in the following table:

\begin{center}
\begin{tabular}{|c|c|c|c|}
\hline
 & &$k_{\midd \bullet}(n)-k_{{\min} \bullet}(n)+1$ & $\lfloor\frac{\deg \bfe_{n+1}}p\rfloor$ 
\\
\hline
\multirow{2}{*}{$a+s_\varepsilon < p-1$} & $n=2n_0$ &  $n_0-\lceil \frac{n_0-a-s_\varepsilon-1}{p}\rceil$ & $n_0-\lceil \frac{n_0-s_\varepsilon}p\rceil$
\\
\cline{2-4}
& $n=2n_0+1$ & $n_0+1-\lceil\frac{n_0-s_\varepsilon+1}{p}\rceil$ &$n_0-\lceil \frac{n_0-a-s_\varepsilon}p\rceil$
\\
\hline 
\multirow{2}{*}{
$a+s_\varepsilon \geq p-1$} & $n=2n_0$ & $n_0-\lceil \frac{n_0-s_\varepsilon}{p} \rceil$  &  $n_0-\lceil \frac{n_0-\{a+s_\varepsilon\}}p\rceil$
\\
\cline{2-4} & $n=2n_0+1$ & $ n_0+1-\lceil \frac{n_0-\{a+s_\varepsilon\}}{p} \rceil$& $n_0-\lceil \frac{n_0-s_\varepsilon}p\rceil$
\\
\hline
\end{tabular}
\end{center}
Here we used the obvious formula $\lfloor \frac{(p-1)a+b}p\rfloor = a - \lceil \frac{a-b}p\rceil$. 
The equalities \eqref{E:increment of degree a+s<p-1} and \eqref{E:increment of degree a+s>=p-1} follow from this by a case-by-case argument. We only spell out the case when $a+s_\varepsilon<p-1$ (that is \eqref{E:increment of degree a+s<p-1}) and the other case is similar.

When $a+s_\varepsilon<p-1$, 
$$
k_{\midd\bullet}(n) - k_{{\min}\bullet}(n)+1 - \Big\lfloor \frac{\deg \bfe_{n+1}}p\Big \rfloor=
\begin{cases}
0 & \textrm{ if } n = 2n_0 \textrm{ and } n_0  -s_\varepsilon \equiv a+2, \dots, p \bmod p\\
& \textrm{ or } n = 2n_0+1 \textrm{ and } n_0 -s_\varepsilon \equiv 0, \dots, a \bmod p\\
1 & \textrm{ otherwise}
\end{cases}
$$ 
Combining this with \eqref{E:kmax-kmid}, we deduce that
$$
\deg g_{n+1}(w) - \deg g_n(w) - \lambda_{n+1} = \begin{cases}
1 & \textrm{ if } n - 2s_\varepsilon \equiv 1, 3, \dots, 2a+1 \bmod {2p},
\\
-1 & \textrm{ if } n- 2s_{\varepsilon} \equiv 2, 4, \dots, 2a+2 \bmod{2p}
\\
0 & \textrm{ otherwise.}
\end{cases}
$$

The statement \eqref{E:degree approx halo bound} follows from \eqref{E:increment of degree a+s<p-1} and \eqref{E:increment of degree a+s>=p-1} immediately, by observing that the difference being $1$ only appears in the case when $\deg \bfe_{n+1}-\deg \bfe_n = p-1-a$.

To see the last statement, we note that, as shown in the table above 
$$\big( k_{{\max}\bullet}(n+1) - k_{\midd\bullet}(n+1)\big) - \big(  k_{{\max}\bullet}(n) - k_{\midd\bullet}(n)\big)$$ is always equal to either $a+1$ or $p-2-a$, which is $\geq 2$ under the genericity condition $1\leq a \leq p-4$. On the other hand, the next table shows that
$$
\Big|\big(k_{{\midd}\bullet}(n+1) - k_{{\min}\bullet}(n+1)\big) - \big(  k_{{\midd}\bullet}(n) - k_{{\min}\bullet}(n)\big)\Big| \leq 1.
$$
Combining these two, we deduce that the increments of degrees $\deg g_{n+1}(w)-\deg g_n(w)$ is always strictly increasing.
\end{proof}

\begin{corollary}[Ghost conjecture versus spectral halo conjecture]
\label{C:ghost versus halo}
Let $\varepsilon$ be a relevant character, then for $w_\star \in \gothm_{\CC_p}$ with $v_p(w_\star) \in (0,1)$, the $n$th slope of $\NP(G(w_\star, -))$ is
$$
v_p(w_\star) \cdot (\deg g_{n+1}^{(\varepsilon)}-\deg g_n^{(\varepsilon)})
$$
Moreover, these slopes are strictly increasing with respect to $n$.
\end{corollary}

\begin{example}
\label{E:halo bound vs g}
In the setup of Example~\ref{Ex:p=7a=3}, the following table indicates the difference between $\deg g_{n+1}^{(\varepsilon)}- \deg g_{n}^{(\varepsilon)}$ and the halo Hodge slope $\lambda_{n+1} : = \deg \bfe_{n+1}^{(\varepsilon)} - \big\lfloor \deg \bfe_{n+1}^{(\varepsilon)} / p\big\rfloor$, on two of the disks.

\begin{center}
\begin{tabular}{|c|c|c|c|c|c|c|c|c|c|c|c|c|c|c|c|c|c|}
\hline
& \multicolumn{17}{|c|}{$\varepsilon = 1\times \omega^2$ disk}
\\
\hline
$\deg g_{n+1}^{(\varepsilon)}- \deg g_{n}^{(\varepsilon)}$  & 0 & 3 & 5 & 8 & 10 & 13 & 15 & 18 & 21 & 23 & 26 & 28&31&33&36 & 39& ... \\ \hline
$\deg \bfe_{n+1}^{(\varepsilon)}$ &  0 & 2 & 6 & 8 & 12 & 14 & 18 & 20 & 24&26&30&32&36&38&42&44& ...
\\
 \hline
$\lambda_{n+1}^{(\varepsilon)}$  & 0 & 2 & 6 & 7 & 11 & 12 & 16 & 18 & 21 & 23 & 26 & 28&31&33&36 & 38& ...
\\ \hline
& \multicolumn{17}{|c|}{$\varepsilon= \omega^2 \times 1$ disk}
\\
\hline
$\deg g_{n+1}^{(\varepsilon)}- \deg g_{n}^{(\varepsilon)}$  &1 & 3 & 6 & 9 & 11 & 14 & 16 & 19 &21&24&27&29&32&34&37&39& ...
\\ \hline
$\deg \bfe^{(\varepsilon)}_{n+1}$ & 0 & 4 & 6 & 10 & 12 & 16 & 18 & 22 &24&28&30&34&36&40& 42 &46& ...
\\ \hline
$\lambda_{n+1}^{(\varepsilon)}$  & 0 & 4 & 6 & 9 & 11 & 14 & 16 & 19 &21&24&26&30&31&35&36&40& ...
\\ \hline
\end{tabular}
\end{center}
\end{example}

\begin{remark}
\label{R:modified halo bound sharp}
As indicated by this Corollary, the halo bound is surprising sharp ``after we distributed it over to each weight disk". It seems to be possible to give a much more robust proof of the halo conjecture similar to \cite{liu-wan-xiao}. We will return to this topic in a future work.
\end{remark}

The next part of this section is devoted to checking that the slopes of the ghost series are compatible with the theta maps, the Atkin--Lehner involutions, and $p$-stabilizations (which will be elaborated in the forthcoming paper; but we do not need them logically here). In other words, 
\begin{itemize}
\item 
the existence of the theta maps predicts that the slopes of abstract overconvergent but non-classical forms of weight $k_0$ is $k_0-1$ plus the slopes of abstract overconvergent forms of weight $2-k_0$,
\item 
the Atkin--Lehner involution predicts that the slopes of abstract classical forms of weight $k_0$ with ``opposite characters" can be paired so that the sum is equal to $k_0-1$, and
\item
the $p$-stabilization process predicts that the slopes of abstract $p$-old forms of weight $k_0$ can be paired so that the sum is equal to $k_0-1$.
\end{itemize}
Below, we show that the ghost series satisfies the analogous properties.
In addition, we will prove a very important yet mysterious property of the ghost series, which we call the \emph{ghost duality}. Roughly speaking, for a fixed weight $k_0 \equiv k_\varepsilon \bmod{(p-1)}$, not only the slopes of the abstract $p$-new forms  are equal to $\frac{k_0-2}2$, but certain slopes computed using appropriate derivatives in $w$ of the ghost series have slopes $\frac{k_0-2}2$. This was first observed by Bergdall and Pollack in their computational data, and was communicated to us. As in many other places, we thank them for sharing their ideas on this project.

\begin{notation}
\label{N:gnhatk}
For $k \equiv k_\varepsilon \bmod {(p-1)}$, we write 
$$
g^{(\varepsilon)}_{n,\hat k}(w): = g_n^{(\varepsilon)}(w) \big/ (w-w_k)^{m_n^{(\varepsilon)}(k)}.
$$
In particular, $g^{(\varepsilon)}_{n,\hat k}(w_k)$ is the leading coefficient of the Taylor expansion of $g_n^{(\varepsilon)}(w)$ at $w=w_k$.

More generally, if $\bfk = \{k_1, \dots, k_r\}$ is a set of ghost zeros with each $k_i \equiv k_\varepsilon \bmod (p-1)$, we put
$$
g^{(\varepsilon)}_{n,\hat {\bfk}}(w): = g_n^{(\varepsilon)}(w) \big/ \prod_{i=1}^r(w-w_{k_i})^{m_n^{(\varepsilon)}(k_i)}.
$$
\end{notation}

We point out a formula that we will use frequently in this section and later:
\begin{equation}
\label{E:vp of k1-k2}
v_p(w_{k_1} - w_{k_2}) = v_p\big(\exp(p(k_2-2)) \cdot (\exp(p(k_1-k_2))-1)\big) =  1+ v_p(k_1-k_2).
\end{equation}
\begin{proposition}
\label{P:ghost compatible with theta AL and p-stabilization}
Fix $k_0 \geq 2$ and a character $\varepsilon= \omega^{-s_\varepsilon} \times \omega^{a+s_\varepsilon}$ of $\Delta^2$ relevant to $\bar \rho$ as before. Write $d: = d_{k_0}^\Iw(\varepsilon\cdot (1\times \omega^{2-k_0}))$ in this theorem.
\begin{enumerate}
\item (Compatibility with theta maps)
Put $\varepsilon' := \varepsilon \cdot (\omega^{k_0-1} \times \omega^{1-k_0})$ with $s_{\varepsilon'} = \{s_\varepsilon +1-k_0\}$. For every $\ell\geq 1$, the $(d+\ell)$th slope of $\NP(G^{(\varepsilon)}(w_{k_0}, -))$ is $k_0-1$ plus the  $\ell$th slope of $\NP(G^{(\varepsilon')}(w_{k_0}, -))$.
In particular, the $(d+\ell)$th slope of $\NP(G^{(\varepsilon)}(w_{k_0}, -))$ is at least ${k_0}-1$.

\item (Compatibility with Atkin--Lehner involutions)
Assume that ${k_0}\not \equiv k_\varepsilon \bmod{(p-1)}$.
Put $\varepsilon'' = \omega^{-s_{\varepsilon''}} \times \omega^{a+s_{\varepsilon''}}$ with $s_{\varepsilon''}: = \{{k_0}-  2-a-s_\varepsilon\}$.  Then for every $\ell \in \{1, \dots, d\}$, the sum of the $\ell$th slope of $\NP(G^{(\varepsilon)}(w_{k_0}, -))$ and the $(d-\ell+1)$th slope of $\NP(G^{(\varepsilon'')}(w_{k_0}, -))$ is exactly ${k_0}-1$.
In particular, the $\ell$th slope of $\NP(G^{(\varepsilon)}(w_{k_0}, -))$ is at most $k_0-1$.

\item (Compatibility with $p$-stabilizations)
Assume that ${k_0}\equiv k_\varepsilon \bmod {(p-1)}$. Then for every $\ell \in \{1, \dots, d_{k_0}^\unr(\varepsilon_1)\}$, the sum of the $\ell$th slope of $\NP(G^{(\varepsilon)}(w_{k_0}, -))$ and the $(d-\ell+1)$th slope of $\NP(G^{(\varepsilon)}(w_{k_0}, -))$ is exactly ${k_0}-1$.
In particular, the $\ell$th slope of $\NP(G^{(\varepsilon)}(w_{k_0}, -))$ is at most ${k_0}-1$.
\item (Ghost duality)
Assume that ${k_0} \equiv k_\varepsilon \bmod{(p-1)}$. Write $d_{k_0}^\unr$, $d_{k_0}^\Iw$, and $d_{k_0}^{\new}$ for $d_{k_0}^\unr(\varepsilon_1)$, $d_{k_0}^\Iw(\tilde \varepsilon_1)$, and $d_{k_0}^{\new}(\varepsilon_1)$, respectively.
Then for each $\ell = 0, \dots,\frac 12 d_{k_0}^{\new}-1$,
\begin{equation}
\label{E:ghost duality}
v_p\big(g_{d_{k_0}^\Iw - d_{k_0}^\unr -\ell, \hat k_0}^{(\varepsilon)} (w_{k_0}) \big) - v_p\big(g_{ d_{k_0}^\unr +\ell, \hat k_0}^{(\varepsilon)} (w_{k_0}) \big) = ({k_0}-2) \cdot (\tfrac 12d_{k_0}^{\new} - \ell).
\end{equation}

In particular, the $(d_{k_0}^\unr+1)$th to the $(d_{k_0}^\Iw-d_{k_0}^\unr)$th slopes of $\NP(G^{(\varepsilon)}(w_{k_0},-))$ are all equal to $\frac{{k_0}-2}2$. 
\end{enumerate}
\end{proposition}
For an alternative explanation of name ``ghost duality", we refer to Notation~\ref{N:Delta kell} later.
\begin{proof}[Proof of Proposition~\ref{P:ghost compatible with theta AL and p-stabilization}]
We first list the concrete statements on $g_n^{(\varepsilon)}(w)$ that we shall prove in order to deduce the Theorem:
\begin{itemize}
\item[(a)] In the setup of (1),
we will prove in \S\,\ref{S:proof of ghost compatible with theta} that
\begin{equation}
\label{E:ghost compatible with theta}
v_p(g^{(\varepsilon)}_{d+\ell+1}(w_{k_0}))-v_p(g^{(\varepsilon)}_{d+\ell}(w_{k_0})) = v_p(g^{(\varepsilon')}_{\ell+1}(w_{2-{k_0}}))-v_p(g^{(\varepsilon')}_{\ell}(w_{2-{k_0}})) +{k_0}-1.
\end{equation}
\item[(b)]
Set $\varepsilon'' = \omega^{-s_{\varepsilon''}} \times \omega^{a+s_{\varepsilon''}}$ with $s_{\varepsilon''}: = \{{k_0}-  2-a-s_\varepsilon\}$ (allowing $k_0\equiv k_\varepsilon \bmod{(p-1)}$ in which case $\varepsilon''= \varepsilon$).
Then for every integer $\ell =1, \dots, d$, we will prove in \S\,\ref{S:proof of ghost compatible with AL} that
\begin{align}
\label{E:ghost compatible with AL}
&\big(v_p(g^{(\varepsilon)}_{d+1-\ell, \hat k_0}(w_{k_0})) - v_p(g^{(\varepsilon)}_{d-\ell, \hat k_0}(w_{k_0})) \big) + \big(v_p(g^{(\varepsilon'')}_{\ell, \hat k_0}(w_{k_0})) - v_p(g^{(\varepsilon'')}_{\ell-1, \hat k_0}(w_{k_0})) \big) 
\\
\nonumber
=\ & \begin{cases}
{k_0}-2 & \textrm{if }k_0\equiv k_\varepsilon \bmod{(p-1)} \textrm{ and }d_{k_0}^\unr +1\leq  \ell \leq d_{k_0}^\Iw-d_{k_0}^\unr
\\
{k_0}-1 & \textrm{otherwise.}
\end{cases}
\end{align}
\end{itemize}
Replacing $\ell$ by $\ell'$ in \eqref{E:ghost compatible with AL}, the equality \eqref{E:ghost duality} follows from summing the resulting equality for $\ell'$'s being $d_{k_0}^{\unr} + \ell+1, d_{k_0}^{\unr} +\ell+2, \dots, \frac12 d$. To deduce $(1)$ and $(2)$ from \eqref{E:ghost compatible with AL} and \eqref{E:ghost compatible with theta} rigorously, we still need to show that the point $(d, v_p(g_d^{(\varepsilon)}(w_{k_0})))$ actually lies on the Newton polygon $\NP(G^{(\varepsilon)}(w_{k_0}, -))$, which we prove now.

Replacing  $\ell$ by $\ell'$ in \eqref{E:ghost compatible with theta} and summing the resulting equality for $\ell' =0, 1, \dots, \ell-1$ implies that for all $\ell \geq 1$,
$$
v_p(g^{(\varepsilon)}_{d+\ell}(w_{k_0}))-v_p(g^{(\varepsilon)}_{d}(w_{k_0})) = v_p(g^{(\varepsilon')}_{\ell}(w_{2-{k_0}}))  +\ell({k_0}-1) \geq (k_0-1)\ell.
$$
Similarly, taking appropriate sums of \eqref{E:ghost compatible with AL} implies that, for $\ell = 1, \dots, d$ if ${k_0} \not\equiv k_\varepsilon \bmod{(p-1)}$ and for $\ell = 1, \dots, d_{k_0}^\unr$ if ${k_0} \equiv k_\varepsilon \bmod{(p-1)}$,
$$
\big(v_p(g^{(\varepsilon)}_d(w_{k_0})) - v_p(g^{(\varepsilon)}_{d-\ell}(w_{k_0})) \big) + v_p(g_\ell^{(\varepsilon'')}(w_{k_0})) = (k_0-1) \ell.
$$
In particular, the first summand is less than or equal to $(k_0-1)\ell$.
These two estimates together show that Newton polygon $\NP(G^{(\varepsilon)}(w_{k_0}, -))$ passes through the point $(d, v_p(g_d^{(\varepsilon)}(w_{k_0})))$, and the slopes are less than or equal to ${k_0}-1$ before this point, and are greater than or equal to ${k_0}-1$ after this point.

Finally, to address $(3)$ and the last statement of (4), we need to show that the first $d_{k}^\unr$th slopes of $\NP(G^{(\varepsilon)}(w_{k_0}, -))$ are strictly less than $\frac{k_0-2}2$. In fact, we will prove these slopes are less than or equal to $\frac{k_0-3}{p+1}$ in Proposition~\ref{P:gouvea k-1/p+1 conjecture} below. Moreover, in the proof of Proposition~\ref{P:gouvea k-1/p+1 conjecture}, we obtain the inequality
$$
v_p(g^{(\varepsilon)}_{d_{k_0}^\unr}(w_{k_0})) - v_p(g^{(\varepsilon)}_{d_{k_0}^\unr-i}(w_{k_0})) \leq i \cdot \frac{k_0-3}{p+1}.
$$
for $i=1,\dots,d_{k_0}^\unr$. It follows that $(d_{k_0}^\ur,v_p(g_{d_{k_0}^\ur}^{(\varepsilon)}(w_{k_0})))$ is a vertex of $\NP(G^{(\varepsilon)}(w_{k_0}),-)$. Combining with (\ref{E:ghost compatible with AL}), we see that $(d-d_{k_0}^\ur,v_p(g_{d-d_{k_0}^\ur}^{(\varepsilon)}(w_{k_0})))$ is also a vertex of $\NP(G^{(\varepsilon)}(w_{k_0}),-)$.

 This then concludes the proof of the Theorem assuming \eqref{E:ghost compatible with theta}, \eqref{E:ghost compatible with AL} and Proposition~\ref{P:gouvea k-1/p+1 conjecture}, which will be proved below after a  couple of  lemmas first.
\end{proof}

\begin{lemma}
\label{L:increment of ghost valuation}
Fix $k_0 \in \ZZ$ and a relevant character $\varepsilon$. 
We keep the convention that $k = 2+\{k_\varepsilon-2\} + (p-1)k_\bullet$. Recall $k_\midd^{(\varepsilon)}(n)$, $k_{\max}^{(\varepsilon)}(n)$, and  $k_{\min}^{(\varepsilon)}(n)$ from Lemma-Notation~\ref{L:extremal ks}.
For $n \geq 0$, we have
\begin{equation}
\label{E:increment of ghost series valuation 1}
v_p(g_{n+1, \hat k_0}^{(\varepsilon)}(w_{k_0})) - v_p(g_{n, \hat k_0}^{(\varepsilon)}(w_{k_0})) =  \hspace{-25pt}
\sum_{\substack{
k \neq k_0
\\
k^{(\varepsilon)}_{\midd\bullet}(n) < k_\bullet \leq k^{(\varepsilon)}_{{\max}\bullet}(n)}}
\hspace{-25pt}\big( v_p(k-k_0)+1 \big)\ -  \hspace{-25pt} \sum_{\substack{
k \neq k_0\\
k^{(\varepsilon)}_{{\min}\bullet}(n) \leq k_\bullet \leq k^{(\varepsilon)}_{{\midd}\bullet}(n)}}
\hspace{-25pt}\big( v_p(k-k_0)+1 \big).
\end{equation}
For $n \geq 1$, we have 
\begin{align}
\label{E:increment of ghost series valuation 2}
& v_p(g_{n+1, \hat k_0}^{(\varepsilon)}(w_{k_0})) -2v_p(g_{n, \hat k_0}^{(\varepsilon)}(w_{k_0}))  + v_p(g_{n-1, \hat k_0}^{(\varepsilon)}(w_{k_0})) 
\\
\nonumber
=\hspace{-30pt}
\sum_{\substack{
k \neq k_0\\
k^{(\varepsilon)}_{{\max}\bullet}(n-1) < k_\bullet \leq k^{(\varepsilon)}_{{\max}\bullet}(n)}} 
\hspace{-30pt}\big(& v_p(k-k_0)+1 \big)
+\hspace{-30pt}
\sum_{\substack{
k \neq k_0\\
k^{(\varepsilon)}_{{\min}\bullet}(n-1) \leq k_\bullet < k^{(\varepsilon)}_{{\min}\bullet}(n)}}
\hspace{-30pt}\big( v_p(k-k_0)+1 \big)
- 
2 \big( v_p(k_\midd^{(\varepsilon)}(n) - k_0) + 1\big),
\end{align}
where the last term is undefined when $k_{\midd}^{(\varepsilon)}(n)= k_0$, in which case we interpret it as zero.
\end{lemma}
\begin{proof}
We first prove \eqref{E:increment of ghost series valuation 1}.
By definition, we have
\begin{align*}
v_p(g_{n+1, \hat k_0}^{(\varepsilon)}(w_{k_0})) - v_p(g_{n, \hat k_0}^{(\varepsilon)}(w_{k_0})) = 
\sum_{\substack{k \neq k_0\\k \equiv k_\varepsilon \bmod{(p-1)}}} v_p(w_k-w_{k_0}) \cdot (m_{n+1}^{(\varepsilon)}(k) - m_{n}^{(\varepsilon)}(k))
\\=  \hspace{-20pt}
\sum_{\substack{
k \neq k_0
\\k \equiv k_\varepsilon \bmod{(p-1)}
\\
d_{k}^\unr(\varepsilon_1) \leq n < \frac 12 d_{k}^\Iw(\tilde \varepsilon_1)}}
\hspace{-20pt}\big( v_p(k-k_0)+1 \big)\ -  \hspace{-30pt} \sum_{\substack{
k \neq k_0
\\k \equiv k_\varepsilon \bmod{(p-1)}
\\
\frac 12 d_{k}^\Iw(\tilde \varepsilon_1)  \leq n < d_{k}^{\Iw}(\tilde \varepsilon_1) - d_{k}^\unr(\varepsilon_1)  }}
\hspace{-30pt}\big( v_p(k-k_0)+1 \big).
\end{align*}
The second equality uses the expression for \eqref{E:increment of multiplicity} to compute $m_{n+1}^{(\varepsilon)}(k) - m_{n}^{(\varepsilon)}(k)$ and the equality $v_p(w_k-w_{k_0}) =v_p(k-k_0)+1$ from \eqref{E:vp of k1-k2}.
Combining this with Lemma-Notation~\ref{L:extremal ks}, we deduce \eqref{E:increment of ghost series valuation 1}.

Using a similar argument and citing \eqref{E:second order increment of multiplicity} in place of \eqref{E:increment of multiplicity}, we deduce \eqref{E:increment of ghost series valuation 2}.
\end{proof}

\begin{remark}
\label{R:graph of kmin and kmax}For \eqref{E:increment of ghost series valuation 1}, the contribution of various $k$ is similar to the case of degrees, as shown in Graph 2 in the proof of Proposition~\ref{P:increment of degrees in ghost series}, namely the blue interval from $k_{\min}$ to $k_\midd$ contributes negatively and the red interval from $k_\midd$ to $k_{\max}$ contributes positively. For \eqref{E:increment of ghost series valuation 2}, the contribution is shown in the picture below, namely, the contributions of the red points at $k_{\min}$ and $k_{\max}$ are positive and the contribution of the blue point at $k_\midd$ is negative with a multiplier of $2$.

\begin{center}
\begin{tikzpicture}[line cap=round,line join=round,>=triangle 45,x=1cm,y=5cm]
\draw [color=cqcqcq,, xstep=1cm,ystep=1cm] (-0.5,-0.1) grid (7,.6);
\draw[->,color=black] (0,0) -- (7,0);
\foreach \x in {,1,2,3,4,5,6}
\draw[shift={(\x,0)},color=black] (0pt,2pt) -- (0pt,-2pt);
\draw[->,color=black] (0,0) -- (0,.6);
\foreach \y in {0.2,0.4,0.6}
\draw[shift={(0,\y)},color=black] (2pt,0pt) -- (-2pt,0pt);
\clip(-0.5,-0.1) rectangle (7,.6);
\draw [line width=1pt,dash pattern=on 2pt off 3pt] (6,0.25)-- (0,0.25);
\draw [line width=1pt,dash pattern=on 2pt off 3pt] (1.5,0.4375)-- (1.5,0);
\draw [line width=1pt,domain=0:7] plot(\x,{(-0--0.25*\x)/6});
\draw [line width=1pt,domain=0:7] plot(\x,{(-0--1.75*\x)/6});
\draw [line width=1pt,domain=0:7] plot(\x,{(-0--2*\x)/6});
\draw [line width=1pt,dash pattern=on 2pt off 3pt] (6,0.25)-- (6,0);
\draw [line width=1pt,dash pattern=on 2pt off 3pt] (0.86,0.25)-- (0.86,0);
\begin{scriptsize}
\draw [fill=red] (6,0.25) circle (3pt);
\draw[color=red] (6.,0.3) node {$+1$};
\draw [fill=black] (0,0.25) circle (1pt);
\draw[color=black] (-0.2,0.26) node {$n$};
\draw [fill=red] (0.8571428571428571,0.25) circle (3pt);
\draw[color=red] (0.6,0.3) node {$+1$};
\draw [fill=black] (1.5,0) circle (1pt);
\draw[color=black] (1.5,-0.05) node {$k_{\midd}$};
\draw[color=black] (6,-0.05) node {$k_{\max}$};
\draw[color=black] (0.8,-0.05) node {$k_{\min}$};
\draw [fill=black] (1.5,0.0625) circle (1.5pt);
\draw [fill=black] (1.5,0.4375) circle (1.5pt);
\draw [fill=blue] (1.5,0.25) circle (3pt);
\draw[color=blue] (1.75,0.3) node {$-2$};
\draw[color=black] (4.7,0.14) node {$d_{k}^\unr$};
\draw[color=black] (2.6,0.5) node {$d_{k}^\Iw-d_{k}^\unr$};
\draw[color=black] (1.3,0.53) node {$d_{k}^\Iw$};
\end{scriptsize}
\end{tikzpicture}

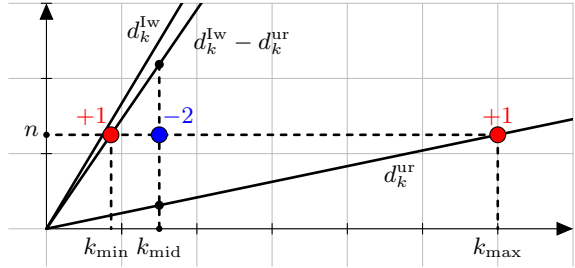
Graph 3: contribution to the second difference of $p$-adic valuations of ghost series
\end{center}
\end{remark}

\begin{lemma}
\label{L:d-2k} We have the following equalities
\begin{eqnarray}
\label{E:d-2k varepsilon''}
(p-1)(d + \delta_\varepsilon + \delta_{\varepsilon''}-2) &=& 2(k_0 -2)  - \{k_\varepsilon -2\} - \{k_{\varepsilon''} - 2\}.
\\
\label{E:d-2k varepsilon'}
(p-1)(d + \delta_\varepsilon - \delta_{\varepsilon'}) &=& 2(k_0-1) - \{k_\varepsilon - 2\} + \{k_{\varepsilon'}-2\}.
\end{eqnarray}
Here we recall
$$
\delta_? = \Big\lfloor \frac{s_? + \{a+s_?\}}{p-1}\Big \rfloor = \begin{cases}
0 & \textrm{ if }s_? + \{a+s_?\} < p-1,\\
1 & \textrm{ if }s_? + \{a+s_?\} \geq p-1.
\end{cases}
$$
for $?=\varepsilon$, $\varepsilon'$ or $\varepsilon''$.
\end{lemma}
\begin{proof}
By the dimension formula of $d = d_{k_0}^{\Iw}(\varepsilon \cdot(1\times \omega^{2-k_0}))$, we have
\begin{align*}
d\, &=2+ \frac{k_0-2-s_\varepsilon-\{k_0-2-s_\varepsilon\}}{p-1} + \frac{k_0-2-\{a+s_\varepsilon\}-\{k_0-2-a-s_\varepsilon\}}{p-1}
\\
& = 2 + \frac{2(k_0-2) - s_\varepsilon - \{a+s_\varepsilon\} -s_{\varepsilon''} -  \{a+s_{\varepsilon''}\}}{p-1}
\\
& = 2-\delta_\varepsilon -\delta_{\varepsilon''}+ \frac{2(k_0-2) - \{k_\varepsilon -2\} - \{k_{\varepsilon''} - 2\}}{p-1},
\end{align*}
where the last equality makes use of the equality (see Remark \ref{R:delta})
\begin{equation}
\label{E:k-delta}
\{k_\varepsilon - 2\} = s_\varepsilon + \{a+s_\varepsilon\} - (p-1)\delta_\varepsilon
\end{equation} 
and the similar equality with $\varepsilon''$ in place of $\varepsilon$. This proves \eqref{E:d-2k varepsilon''}.
We may alternatively use the equality $\{c\} + \{-1-c\} = p-2$ for any integer $c$ to rewrite the first row above as
\begin{align*}
d\, &= \frac{k_0-1-s_\varepsilon+\{s_\varepsilon+1-k_0\}}{p-1} + \frac{k_0-1-\{a+s_\varepsilon\}+\{a+s_\varepsilon+1-k_0\}}{p-1}
\\
& = - \delta_{\varepsilon} + \delta_{\varepsilon'} + \frac{2(k_0-1) -\{k_\varepsilon -2\} +\{k_{\varepsilon'} - 2\}}{p-1}.
\end{align*}
This proves \eqref{E:d-2k varepsilon'}.
\end{proof}

Now we are ready to prove the two key equalities \eqref{E:ghost compatible with theta} and \eqref{E:ghost compatible with AL}.

\subsection{Proof of \eqref{E:ghost compatible with AL}}
\label{S:proof of ghost compatible with AL}
For $\ell = 1, \dots, d$, applying formula \eqref{E:increment of ghost series valuation 1} to the case of $n=d-\ell$ with $\varepsilon$, and the case of $n=\ell-1$ with $\varepsilon''$, we deduce 
\begin{align}
\label{E:ghost compatible with AL first step}
&\big(v_p(g_{d-\ell+1, \hat k_0}^{(\varepsilon)}(w_{k_0})) - v_p(g_{d-\ell, \hat k_0}^{(\varepsilon)}(w_{k_0}))\big) + \big(v_p(g_{\ell, \hat k_0}^{(\varepsilon'')}(w_{k_0})) - v_p(g_{\ell-1, \hat k_0}^{(\varepsilon'')}(w_{k_0}))\big)
\\
\nonumber
=&
\hspace{-25pt}
\sum_{\substack{k\neq k_0 \\
k^{(\varepsilon)}_{\midd\bullet}(d-\ell) < k_\bullet \leq k^{(\varepsilon)}_{{\max}\bullet}(d-\ell)}}
\hspace{-25pt}\big( v_p(k-k_0)+1 \big)\quad -  \hspace{-25pt} \sum_{\substack{k\neq k_0 \\
k^{(\varepsilon)}_{{\min}\bullet}(d-\ell) \leq  k_\bullet \leq k^{(\varepsilon)}_{{\midd}\bullet}(d-\ell)}}
\hspace{-25pt}\big( v_p(k-k_0)+1 \big)\\ \nonumber
&+
\hspace{-25pt}
\sum_{\substack{k\neq k_0 \\
k^{(\varepsilon'')}_{\midd\bullet}(\ell-1) < k_\bullet \leq k^{(\varepsilon'')}_{{\max}\bullet}(\ell-1)}}
\hspace{-25pt}\big( v_p(k-k_0)+1 \big)\quad -  \hspace{-25pt} \sum_{\substack{k\neq k_0 \\
k^{(\varepsilon'')}_{{\min}\bullet}(\ell-1) \leq k_\bullet \leq k^{(\varepsilon'')}_{{\midd}\bullet}(\ell-1)}}
\hspace{-25pt}\big( v_p(k-k_0)+1 \big).
\end{align}
For the first and the third terms, we separate out the $1$'s in the sum; and for the second and the fourth terms, we include the $1$'s into the valuation as a multiple of $p$. This way, (after rearranging the terms and possibly replacing $v_p(c)$ by $v_p(-c)$,) the sum above becomes the following sum.
\begin{align}
\label{E:degree term AL}
&\big( k^{(\varepsilon)}_{{\max}\bullet}(d-\ell) - k^{(\varepsilon)}_{{\midd}\bullet}(d-\ell) \big)+ \big( k^{(\varepsilon'')}_{{\max}\bullet}(\ell-1) - k^{(\varepsilon'')}_{{\midd}\bullet}(\ell-1)\big) - \upsilon
\\
\label{E:factorial term + AL}
+\ & \hspace{-15pt}\sum_{\substack{k \neq k_0\\
k^{(\varepsilon)}_{\midd\bullet}(d-\ell) < k_\bullet \leq k^{(\varepsilon)}_{{\max}\bullet}(d-\ell)}}
\hspace{-15pt}v_p(k-k_0)\quad + \hspace{-15pt} \sum_{\substack{k \neq k_0\\
k^{(\varepsilon'')}_{{\midd}\bullet}(\ell-1) < k_\bullet \leq k^{(\varepsilon'')}_{{\max}\bullet}(\ell-1)}}
\hspace{-15pt}v_p(k_0-k)\quad 
\\
\label{E:factorial term - AL}
-\ & \hspace{-15pt} \sum_{\substack{k \neq k_0\\
k^{(\varepsilon)}_{{\min}\bullet}(d-\ell) \leq k_\bullet \leq k^{(\varepsilon)}_{{\midd}\bullet}(d-\ell)}}
\hspace{-15pt}v_p(pk-pk_0)\quad - \hspace{-15pt}\sum_{\substack{k \neq k_0\\
k^{(\varepsilon'')}_{{\min}\bullet}(\ell-1) \leq  k_\bullet \leq k^{(\varepsilon'')}_{{\midd}\bullet}(\ell-1)}}
\hspace{-15pt}v_p(pk_0-pk),
\end{align}
where term $\upsilon$ is zero unless $k_0 \equiv k_\varepsilon \bmod{(p-1)}$ and $d_{k_0}^\unr+1\leq  \ell\leq d_{k_0}^\Iw - d_{k_0}^\unr$,  in which case $\upsilon = 1$ because the extra $1$ in the sum of \eqref{E:ghost compatible with AL first step} is one less due to the removal of the factor when $k = k_0$.

Then to prove \eqref{E:ghost compatible with AL}, it suffices to show that \eqref{E:degree term AL} is equal to $k_0-1-\upsilon$, and the sum \eqref{E:factorial term + AL} cancels with the sum \eqref{E:factorial term - AL}.

We prove the former statement. Indeed, by Lemma-Notation~\ref{L:extremal ks}, we need to show that
$$
\big(\tfrac{p+1}2 (d-\ell) + \beta_{[d-\ell]}^{(\varepsilon)}-1\big) - (d-\ell + \delta_\varepsilon -1)  + \big(\tfrac{p+1}2 (\ell-1) + \beta_{[\ell-1]}^{(\varepsilon'')}-1\big) - (\ell-1 + \delta_{\varepsilon''} -1)=k_0-1.
$$
This is equivalent to proving
$$
k_0-1-\tfrac{p-1}2 (d-1) +\delta_\varepsilon + \delta_{\varepsilon''} = \beta_{[d-\ell]}^{(\varepsilon)} + \beta_{[\ell-1]}^{(\varepsilon'')}.
$$
By Lemma~\ref{L:d-2k}, this is further equivalent to proving
$$
\beta_{[d-\ell]}^{(\varepsilon)} + \beta_{[\ell-1]}^{(\varepsilon'')} \ =\tfrac12\big(\{k_\varepsilon - 2\} + \{k_{\varepsilon''}-2\} \big)  + \tfrac{p+1}2 (\delta_\varepsilon + \delta_{\varepsilon''}) -\tfrac{p-3}2.
$$
This boils down to a case-by-case check, which we do in Lemma~\ref{L:case-by-case check AL} later.

\medskip
Now we show that the sum \eqref{E:factorial term + AL} cancels with the sum \eqref{E:factorial term - AL}.
For this, we need to establish three identities.
\begin{enumerate}
\item $k_\midd^{(\varepsilon)} (d-\ell)+p-1 -k_0 =k_0- k_{\midd}^{(\varepsilon'')}(\ell-1)$,
\item 
$k_{\max}^{(\varepsilon)} (d-\ell) -k_0 = pk_0-\tilde  k_{\min}^{(\varepsilon'')}(\ell-1)$, and
\item 
$ \tilde k_{\min}^{(\varepsilon)} (d-\ell) -pk_0=k_0-   k_{\max}^{(\varepsilon'')}(\ell-1) $.
\end{enumerate}
Assume these identities for a moment.
Then the first sum in \eqref{E:factorial term + AL} is exactly the sum of the valuations of the numbers in the sequence (with step $p-1$) from $ k^{(\varepsilon)}_{\max}(d-\ell) - k_0$ (downwards) to $k^{(\varepsilon)}_{\midd}(d-\ell)+(p-1) - k_0 \stackrel{(1)}= k_0 - k_{{\midd}}^{(\varepsilon'')}(\ell-1)$, and then is continued with the second sum of \eqref{E:factorial term + AL} which is the sum of valuations of the numbers (with step size $p-1$) from $k_0 -(k_{{\midd}}^{(\varepsilon'')}(\ell-1)+p-1)$ to $k_0 - k_{\max}^{(\varepsilon'')}(\ell-1)$. In this sum we remove the term if the number happens to be zero.  In comparison, in the expression \eqref{E:factorial term - AL}, the second sum is the sum of valuations of only $p$-multiples in the sequence (with step size $p-1$) from $pk_0-\tilde  k_{\min}^{(\varepsilon'')}(\ell-1)$ to $pk_0 - pk^{(\varepsilon'')}_{\midd \bullet}(\ell-1) 
\stackrel{(1)}=  pk^{(\varepsilon)}_{\midd \bullet}(d-\ell) - pk_0 + p(p-1)$; whereas the first sum is the sum of the $p$-multiples in the continuing sequence (with step size $p-1$) from $
pk^{(\varepsilon)}_{\midd \bullet}(d-\ell) - pk_0$ to $\tilde k^{(\varepsilon'')}_{\min}(d-\ell) - pk_0$. Again we skip the valuation of zero in the sequence (if it appears).
So both the sums \eqref{E:factorial term + AL} and \eqref{E:factorial term - AL} are exactly the sum of valuations of (the $p$-multiples in) the same sequence from $k^{(\varepsilon)}_{\max}(d-\ell) - k_0\stackrel{(2)}= pk_0-\tilde  k_{\min}^{(\varepsilon'')}(\ell-1)$ to $ k_0 - k^{(\varepsilon'')}_{\max}(\ell-1) \stackrel{(3)}=\tilde k^{(\varepsilon'')}_{\min}(d-\ell) - pk_0$. Clearly they cancel each other.

\medskip
It remains to prove the equalities (1), (2), and (3).
By Lemma-Notation~\ref{L:extremal ks}, Equality (1) is equivalent to
\begin{align*}
2+\{k_\varepsilon -&2\} + (p-1)(d-\ell +\delta_\varepsilon) -k_0 = k_0 - \big( 2+\{k_{\varepsilon ''}-2\} + (p-1)(\ell-1 + \delta_{\varepsilon''} -1)\big).
\\
&
\textrm{or equivalently, }(p-1)(d+\delta_\varepsilon + \delta_{\varepsilon''}-2) = 2k_0 -4 -\{k_\varepsilon-2\}+\{k_{\varepsilon''}-2\}.
\end{align*}
But this is exactly Lemma~\ref{L:d-2k}.

We prove (2). By Lemma-Notation~\ref{L:extremal ks}, this is equivalent to
\begin{align*}
2 +\{k_{\varepsilon} -2\} + &(p-1) \big(\tfrac{p+1}2 (d-\ell) +\beta_{[d-\ell]}^{(\varepsilon)}-1\big) - k_0
\\
& = pk_0 - 
2p -p\{k_{\varepsilon''} -2\} - (p-1) \big(\tfrac{p+1}2 (\ell-2+2\delta_{\varepsilon''}) -\beta_{[\ell]}^{(\varepsilon'')}+1\big) .
\end{align*}
Rearranging, this is equivalent to
$$
(p+1) \big( \tfrac{p-1}2 (d-2+2\delta_{\varepsilon''})-k_0 + 2\big) + \{k_\varepsilon - 2\} +p\{k_{\varepsilon''}-2\} = (p-1)\big( \beta_{[\ell]}^{(\varepsilon'')}- \beta_{[d-\ell]}^{(\varepsilon)}\big).
$$
Feeding in the formula \eqref{E:d-2k varepsilon''}, this is further equivalent to
$$
\tfrac{p-1}2\{k_{\varepsilon''} - 2\}-\tfrac{p-1}2\{k_\varepsilon-2\} + \tfrac{p^2-1}2(\delta_{\varepsilon''} - \delta_\varepsilon) =  (p-1)\big( \beta_{[\ell]}^{(\varepsilon'')}- \beta_{[d+\ell]}^{(\varepsilon)}\big).
$$
We may cancel the factor $p-1$ on both sides, and then after that, we need to do a case-by-case study which we defer to Lemma~\ref{L:case-by-case check AL} later.

Finally, we note that (3) is the same as (2) when $\varepsilon$ and $\varepsilon''$ are swapped and $\ell$ is replaced by $d-\ell + 1$.
This then completes the proof of \eqref{E:ghost compatible with AL} assuming the following Lemma.
 \hfill $\Box$

\begin{lemma}
\label{L:case-by-case check AL}
We have the following two identities for every $\ell$
\begin{eqnarray*}
\beta_{[d+\ell]}^{(\varepsilon)}-\beta_{[\ell]}^{(\varepsilon'')} &=&\tfrac12\big(\{k_\varepsilon-2\}-\{k_{\varepsilon''}-2\}\big) + \tfrac{p+1}2( \delta_\varepsilon-\delta_{\varepsilon''} ), \quad\textrm{and}
\\
\beta_{[d-\ell]}^{(\varepsilon)} + \beta_{[\ell-1]}^{(\varepsilon'')}&=& \tfrac12\big(\{k_\varepsilon - 2\} + \{k_{\varepsilon''}-2\} \big)  + \tfrac{p+1}2 (\delta_\varepsilon + \delta_{\varepsilon''}) - \tfrac{p-3}2.
\end{eqnarray*}
\end{lemma}
\begin{proof}
In view of the equality $\{k_\varepsilon -2 \} + (p-1)\delta_\varepsilon = s_\varepsilon +\{a+s_\varepsilon\}$ and the similar equality for $\varepsilon''$, we need to show that
\begin{eqnarray}
\label{E:beta + beta and beta - beta}
\beta_{[d+\ell]}^{(\varepsilon)}-\beta_{[\ell]}^{(\varepsilon'')} & =&\tfrac12\big(s_\varepsilon + \{a+s_\varepsilon\}\big) -\tfrac 12\big(s_{\varepsilon''} + \{a+s_{\varepsilon''}\}\big) +  \delta_\varepsilon-\delta_{\varepsilon''}, \quad\textrm{and}
\\
\label{E:beta + beta and beta - beta-continued}
\beta_{[d-\ell]}^{(\varepsilon)} + \beta_{[\ell-1]}^{(\varepsilon'')} &=&\tfrac12\big(s_\varepsilon+ \{a+s_\varepsilon \} \big) + \tfrac12\big(s_{\varepsilon''} + \{a+s_{\varepsilon''}\}\big) + \delta_\varepsilon + \delta_{\varepsilon''} -\tfrac{p-3}2.
\end{eqnarray}
Recall that $s_{\varepsilon''} = \{k_0-2-a-s_\varepsilon \}$.
When $a+s_\varepsilon < p-1$, we list our calculation  in the following table. 
\begin{center}
\begin{tabular}{|c|c|c|}
\hline
Condition & $s_\varepsilon\leq  \{k_0-2\} < a+s_\varepsilon$ &$\{k_0-2\}< s_\varepsilon $ or $\{k_0-2\} \geq a+s_\varepsilon$
\\
\hline 
$d$ &  odd & even
\\ 
\hline
$a+ s_{\varepsilon''}$ &  $\geq p-1$ & $< p-1$
\\
\hline 
$\beta^{(\varepsilon)}_\even$ & $s_\varepsilon + \delta_\varepsilon$& $s_\varepsilon + \delta_\varepsilon$
\\
\hline 
$\beta^{(\varepsilon)}_\odd$ & $a+s_\varepsilon + \delta_\varepsilon - \frac{p-3}2$ & $a+s_\varepsilon + \delta_\varepsilon - \frac{p-3}2$ 
\\
\hline 
$\beta^{(\varepsilon'')}_\even$ & $\{a+s_{\varepsilon''}\}  + \delta_{\varepsilon''}+1$& $s_{\varepsilon''} + \delta_{\varepsilon''}$
\\
\hline 
$\beta^{(\varepsilon'')}_\odd$ & $s_{\varepsilon''} + \delta_{\varepsilon''} -\frac{p-1}2$&$a+s_{\varepsilon''} + \delta_{\varepsilon''} - \frac{p-3}2$
\\
\hline  $\beta_{[d+\ell]}^{(\varepsilon)} - \beta_{[\ell]}^{(\varepsilon'')}$ & $s_\varepsilon -s_{\varepsilon''} +  \delta_\varepsilon - \delta_{\varepsilon''} +\frac{p-1}2$ & $s_\varepsilon - s_{\varepsilon''} + \delta_{\varepsilon } - \delta_{\varepsilon''}$
\\
\hline
$\beta_{[d-\ell]}^{(\varepsilon)} + \beta_{[\ell-1]}^{(\varepsilon'')}$  & $a+  s_\varepsilon + s_{\varepsilon''} + \delta_\varepsilon + \delta_{\varepsilon ''} +2-p$ & $a+  s_\varepsilon + s_{\varepsilon''} + \delta_\varepsilon + \delta_{\varepsilon ''} -\frac{p-3}2$
\\
\hline
\end{tabular}
\end{center}
Note that the expressions in the last two rows are valid regardless of the parity of $\ell$. We explain how to get the results in the above table and use them to prove (\ref{E:beta + beta and beta - beta}) and (\ref{E:beta + beta and beta - beta-continued}). We divide the discussion into two cases. When $s_\varepsilon\leq  \{k_0-2\} < a+s_\varepsilon$, it follows from Proposition~\ref{P:dimension of SIw} that $d: = d_{k_0}^\Iw(\varepsilon\cdot (1\times \omega^{2-k_0}))$ is odd. We also have $a+s_{\varepsilon''}\geq p-1$ in this case. We use Notation~\ref{N:tnvarepsilon} to compute the last six numbers in the table. The penultimate line computes the left hand side of (\ref{E:beta + beta and beta - beta}), while the right hand side of (\ref{E:beta + beta and beta - beta}) equals to $\frac 12(s_\varepsilon+a+s_\varepsilon)-\frac 12(s_{\varepsilon''}+a+s_{\varepsilon''}-(p-1))+\delta_\varepsilon-\delta_{\varepsilon''}$, and (\ref{E:beta + beta and beta - beta}) is proved in this case. We use the last line in the table to prove (\ref{E:beta + beta and beta - beta-continued}) by a similar computation. The other case when $\{k_0-2\}< s_\varepsilon $ or $\{k_0-2\} \geq a+s_\varepsilon$ can be proved in the same way. 

Similarly, when $a + s_\varepsilon \geq p-1$, we have the following table.

\begin{center}
\begin{tabular}{|c|c|c|}
\hline
Condition & $\{a+s_\varepsilon\}\leq  \{k_0-2\} < s_\varepsilon$ &$\{k_0-2\}< \{a+s_\varepsilon\} $ or $\{k_0-2\} \geq s_\varepsilon$
\\
\hline 
$d$ &  odd & even
\\ 
\hline
$a+ s_{\varepsilon''}$  & $< p-1$&  $\geq p-1$
\\
\hline 
$\beta^{(\varepsilon)}_\even$ & $\{a+s_\varepsilon\} + \delta_\varepsilon+1$& $\{a+s_\varepsilon\} + \delta_\varepsilon+1$
\\
\hline 
$\beta^{(\varepsilon)}_\odd$ & $s_\varepsilon + \delta_\varepsilon - \frac{p-1}2$ & $s_\varepsilon + \delta_\varepsilon - \frac{p-1}2$
\\
\hline 
$\beta^{(\varepsilon'')}_\even$ & $s_{\varepsilon''} + \delta_{\varepsilon''}$ & $\{a+s_{\varepsilon''}\}  + \delta_{\varepsilon''}+1$
\\
\hline 
$\beta^{(\varepsilon'')}_\odd$&$a+s_{\varepsilon''} + \delta_{\varepsilon''} - \frac{p-3}2$ & $s_{\varepsilon''} + \delta_{\varepsilon''} -\frac{p-1}2$
\\
\hline  $\beta_{[d+\ell]}^{(\varepsilon)} - \beta_{[\ell]}^{(\varepsilon'')}$ & $s_\varepsilon -s_{\varepsilon''} +  \delta_\varepsilon - \delta_{\varepsilon''} -\frac{p-1}2$ & $s_\varepsilon - s_{\varepsilon''} + \delta_{\varepsilon } - \delta_{\varepsilon''}$
\\
\hline
$\beta_{[d-\ell]}^{(\varepsilon)} + \beta_{[\ell-1]}^{(\varepsilon'')}$  & $a+  s_\varepsilon + s_{\varepsilon''} + \delta_\varepsilon + \delta_{\varepsilon ''} +2-p$ & $a+  s_\varepsilon + s_{\varepsilon''} + \delta_\varepsilon + \delta_{\varepsilon ''} -\frac{3p-5}2$
\\
\hline
\end{tabular}
\end{center}
As above, the expressions in the last two rows are valid regardless of the parity of $\ell$. Again the two equalities (\ref{E:beta + beta and beta - beta}) and (\ref{E:beta + beta and beta - beta-continued}) can be proved via a case-by-case computation.
\end{proof}

\subsection{Proof of \eqref{E:ghost compatible with theta}}
\label{S:proof of ghost compatible with theta}
The proof is similar to \S\,\ref{S:proof of ghost compatible with AL}, so we only sketch the proof and focus on the differences of the two proofs.
Applying formula \eqref{E:increment of ghost series valuation 1} to the case of $n=d+\ell$ with $\varepsilon$, and the case of $n=\ell$ with $\varepsilon'$ and rearranging terms in a way similar to \S\,\ref{S:proof of ghost compatible with AL}, we deduce that
\begin{align}
\nonumber
\big( v_p(g^{(\varepsilon)}_{d+\ell+1}& (w_{k_0}))-v_p(g^{(\varepsilon)}_{d+\ell}(w_{k_0})) \big) - \big( v_p(g^{(\varepsilon')}_{\ell+1}(w_{2-{k_0}}))-v_p(g^{(\varepsilon')}_{\ell}(w_{2-{k_0}})) \big)
\\
\label{E:degree term theta}
=\ &\big( k^{(\varepsilon)}_{{\max}\bullet}(d+\ell) - k^{(\varepsilon)}_{{\midd}\bullet}(d+\ell) \big)- \big( k^{(\varepsilon')}_{{\max}\bullet}(\ell) - k^{(\varepsilon')}_{{\midd}\bullet}(\ell)\big)
\\
\label{E:factorial term + theta}
&+  \hspace{-15pt}\sum_{
k^{(\varepsilon)}_{\midd\bullet}(d+\ell) < k_\bullet \leq k^{(\varepsilon)}_{{\max}\bullet}(d+\ell)}
\hspace{-15pt}v_p(k-k_0)\quad + \hspace{-15pt} \sum_{
k^{(\varepsilon')}_{{\min}\bullet}(\ell) \leq k_\bullet \leq k^{(\varepsilon')}_{{\midd}\bullet}(\ell)}
\hspace{-15pt}v_p(pk-p(2-k_0))\quad 
\\
\label{E:factorial term - theta}
&-  \hspace{-15pt} \sum_{
k^{(\varepsilon')}_{{\midd}\bullet}(\ell) < k_\bullet \leq k^{(\varepsilon')}_{{\max}\bullet}(\ell)}
\hspace{-15pt}v_p(k-(2-k_0))\quad - \hspace{-15pt}\sum_{
k^{(\varepsilon)}_{{\min}\bullet}(d+\ell) \leq k_\bullet \leq k^{(\varepsilon)}_{{\midd}\bullet}(d+\ell)}
\hspace{-15pt}v_p(pk-pk_0).
\end{align}
Then to prove \eqref{E:ghost compatible with theta}, it suffices to show that \eqref{E:degree term theta} is equal to $k_0-1$, and the sum \eqref{E:factorial term + theta} cancels with the sum \eqref{E:factorial term - theta}.

We prove the former statement. Indeed, by Lemma-Notation~\ref{L:extremal ks}, we need to show that
$$
\big(\tfrac{p+1}2 (d+\ell) + \beta_{[d+\ell]}^{(\varepsilon)}-1\big) - (d+\ell + \delta_\varepsilon -1)  = (k_0-1) + \big(\tfrac{p+1}2 \ell + \beta_{[\ell]}^{(\varepsilon')}-1\big) - (\ell + \delta_{\varepsilon'} -1).
$$
This is equivalent to proving
\begin{equation}
k_0-1-\tfrac{p-1}2 d +\delta_\varepsilon - \delta_{\varepsilon'} = \beta_{[d+\ell]}^{(\varepsilon)} - \beta_{[\ell]}^{(\varepsilon')}.
\end{equation}
By \eqref{E:d-2k varepsilon'}, we are left to show that
$$
\beta_{[d+\ell]}^{(\varepsilon)} - \beta_{[\ell]}^{(\varepsilon')} = \tfrac{p+1}2( \delta_\varepsilon -\delta_{\varepsilon'}) +\tfrac 12\big( \{k _\varepsilon -2\} - \{k_{\varepsilon '}-2\}\big).
$$
This boils down to a case-by-case check, which we refer to Lemma~\ref{L:case-by-case check theta} later.

Now we show that the sum \eqref{E:factorial term + theta} cancels with the sum \eqref{E:factorial term - theta}.
For this, we need the following lemma:
\begin{lemma}\label{L:three elementary identities}
	We have the following three identities:
\begin{enumerate}
\item $k_\midd^{(\varepsilon)} (d+\ell) -k_0 = k_{\midd}^{(\varepsilon')}(\ell) - (2-k_0)$,
\item 
$k_{\max}^{(\varepsilon)} (d+\ell) -k_0 = \tilde  k_{\min}^{(\varepsilon')}(\ell) -p (2-k_0)-(p-1)$, and
\item 
$\tilde k_{\min}^{(\varepsilon)} (d+\ell) -pk_0 - (p-1) =   k_{\max}^{(\varepsilon')}(\ell) - (2-k_0)$.
\end{enumerate}
\end{lemma}
Assume these identities for a moment.
The first sum of \eqref{E:factorial term + theta} is the sum of the  valuations of the sequence (with step size $p-1$) from $ k^{(\varepsilon)}_{\midd}(d+\ell) - k_0+(p-1)$ to $k^{(\varepsilon)}_{\max}(d+\ell) - k_0 \stackrel{(2)}= \tilde k_{{\min}\bullet}^{(\varepsilon')}(\ell) -p(2-k_0) - (p-1)$, and the second sum of \eqref{E:factorial term + theta} then continues with the valuations (of only the $p$-multiples) in the sequence (with step size $p-1$) from $\tilde k_{{\min}\bullet}^{(\varepsilon')}(\ell) -p(2-k_0)$ to $pk_{\midd\bullet}^{(\varepsilon')}(\ell) - p(2-k_0)$.
In a similar way, the first sum of \eqref{E:factorial term - theta} is the sum of the  valuations of the sequence (with step size $p-1$) from $ k^{(\varepsilon')}_{\midd}(\ell) - (2-k_0)+(p-1)$ to $k^{(\varepsilon')}_{\max}(\ell) - (2-k_0) \stackrel{(3)}= \tilde k_{{\min}\bullet}^{(\varepsilon)}(d+\ell) -pk_0-(p-1)$, and then the second sum of \eqref{E:factorial term - theta} continues with the sum of valuations of (only $p$-multiples in) the sequence (of step size $p-1$) from $\tilde k_{{\min}\bullet}^{(\varepsilon)}(d+\ell) -pk_0$ to $pk_{\midd\bullet}^{(\varepsilon)}(d+\ell) - pk_0$.
So these the sums \eqref{E:factorial term + theta} and \eqref{E:factorial term - theta} are exactly the sum of valuations of the same sequence from $k_{\midd}^{(\varepsilon)}(d+\ell) - k_0+p-1 \stackrel{(1)}= k_{\midd}^{(\varepsilon')}(\ell) - (2-k_0)+p-1$ to $pk_{\midd}^{(\varepsilon)}(d+\ell) - pk_0 \stackrel{(1)}= pk_{\midd}^{(\varepsilon')}(\ell) -p (2-k_0)$. Clearly they cancel each other.

\medskip
It remains to prove the equalities (1), (2), and (3) in Lemma~\ref{L:three elementary identities}.
By Lemma-Notation~\ref{L:extremal ks}, Equality (1) is equivalent to
\begin{align*}
2+\{k_\varepsilon -&2\} + (p-1)(d+\ell +\delta_\varepsilon -1) -k_0 = 2+\{k_{\varepsilon '}-2\} + (p-1)(\ell + \delta_{\varepsilon'} -1) - (2-k_0).
\\
&
\textrm{or equivalently, }(p-1)(d+\delta_\varepsilon - \delta_{\varepsilon'}) = 2k_0 -2 -\{k_\varepsilon-2\}+\{k_{\varepsilon'}-2\}.
\end{align*}
But this is exactly Lemma~\ref{L:d-2k}.

We prove (2). By Lemma-Notation~\ref{L:extremal ks}, this is equivalent to
\begin{align*}
2 +\{k_{\varepsilon} -2\} + &(p-1) \big(\tfrac{p+1}2 (d+\ell) +\beta_{[d+\ell]}^{(\varepsilon)}-1\big) - k_0\\ & = 
2p +p\{k_{\varepsilon'} -2\} + (p-1) \big(\tfrac{p+1}2 ( \ell-1+2\delta_{\varepsilon'})- \beta_{[\ell-1]}^{(\varepsilon')}\big)   -p(2-k_0).
\end{align*}
Rearranging and using Lemma~\ref{L:d-2k}, this is equivalent to
$$  (p-1)(\beta_{[d+\ell]}^{(\varepsilon)}+ \beta_{[\ell-1]}^{(\varepsilon')} )=
\tfrac{p-1}2\{k_\varepsilon-2\} +\tfrac{p-1}2\{k_{\varepsilon'} - 2\} + \tfrac{p^2-1}2(  \delta_\varepsilon+\delta_{\varepsilon'}) -(p-1)\tfrac{p-3}2.
$$
Dividing this by $p-1$, this follows from Lemma~\ref{L:case-by-case check theta} below.

We prove (3). By Lemma-Notation~\ref{L:extremal ks}, this is equivalent to
\begin{align*}
2p +p\{k_{\varepsilon} -2\} + &(p-1) \big(\tfrac{p+1}2 (d+\ell-1+2\delta_{\varepsilon}) -\beta_{[d+\ell-1]}^{(\varepsilon)}\big) - pk_0
\\ & = 
2 +\{k_{\varepsilon'} -2\} + (p-1) \big(\tfrac{p+1}2 \ell +\beta_{[\ell]}^{(\varepsilon')}-1\big)   -(2-k_0).
\end{align*}
Rearranging and using Lemma~\ref{L:d-2k}, this is further equivalent to
$$(p-1)(\beta_{[d+\ell-1]}^{(\varepsilon)}+ \beta_{[\ell]}^{(\varepsilon')})=
\tfrac{p-1}2\{k_\varepsilon-2\} +\tfrac{p-1}2\{k_{\varepsilon'} - 2\} + \tfrac{p^2-1}2(\delta_{\varepsilon'} + \delta_\varepsilon) -(p-1)\tfrac{p-3}2.
$$
Dividing by $p-1$, this will follow from Lemma~\ref{L:case-by-case check theta} below.

To sum up, we have proved \eqref{E:ghost compatible with AL} provided the following Lemma. \hfill $\Box$

\begin{lemma}
\label{L:case-by-case check theta}
We have the following two identities for every $\ell$
\begin{eqnarray*}
\beta_{[d+\ell]}^{(\varepsilon)} - \beta_{[\ell]}^{(\varepsilon')} &=&  \tfrac 12\big(\{k _\varepsilon -2\} - \{k_{\varepsilon '}-2\}\big)+\tfrac{p+1}2( \delta_\varepsilon-\delta_{\varepsilon'})  \quad\textrm{and}
\\
\beta_{[d+\ell-1]}^{(\varepsilon)}+\beta_{[\ell]}^{(\varepsilon')}
&=&
\tfrac12\big(\{k_\varepsilon-2\} +\{k_{\varepsilon'}  - 2\}\big) + \tfrac{p+1}2(\delta_{\varepsilon'} + \delta_\varepsilon)-\tfrac{p-3}2.
\end{eqnarray*}
\end{lemma}
\begin{proof}
This follows from a case-by-case study as in Lemma~\ref{L:case-by-case check AL}. In fact, we get exactly the \emph{same} tables as in Lemma~\ref{L:case-by-case check AL} with $\varepsilon'$ in places of $\varepsilon''$ everywhere. We leave the details to the readers.
\end{proof}

Lastly, we address the last piece needed for Proposition~\ref{P:ghost compatible with theta AL and p-stabilization}. This is also an essential ingredient to prove that the local ghost conjecture implies Gouv\^ea's $\lfloor\frac{k-1}{p+1}\rfloor$-conjecture (see the discussion following Conjecture~\ref{Conj:Gouvea}). The readers may choose to skip  the proof for the first time reading.

\begin{notation}
For a positive integer $m$, let $\Dig(m)$ denote the sum of all digits in the $p$-based expression of $m$. Then the sums of valuations of consecutive integers in $(m_1, m_2]$ with $m_2 > m_1>0$ is
\begin{equation}
\label{E:sum of consecutive valuations}
\sum_{m_1< i \leq m_2} v_p(i) = \frac{(m_2-\Dig(m_2) )-(m_1 -\Dig(m_1))} {p-1}.
\end{equation}

\end{notation}

\begin{proposition}
\label{P:gouvea k-1/p+1 conjecture}
Fix a relevant character $\varepsilon$ of $\Delta^2$, and let $k_0 \equiv k_\varepsilon \bmod{(p-1)}$ be a weight. Then the first $d_{k_0}^\unr(\varepsilon_1)$ slopes of $\NP(G^{(\varepsilon)}(w_{k_0}, -))$ are all less than or equal to $\big\lfloor\frac{k_0-1-\min\{a+1,p-2-a\}}{p+1}\big\rfloor$. More precisely, they are all less than or equal to
\begin{equation}
\label{E:gouvea maximal slope}
\frac{p-1}2(d_{k_0}^\unr(\varepsilon_1)-1)-\delta_\varepsilon + \beta_{[d_{k_0}^\unr(\varepsilon_1)-1]}.
\end{equation}
\end{proposition}
\begin{proof}
In this proof, we shall suppress $\varepsilon$ and $\varepsilon_1$ from the notation for simplicity.
We first check that $\eqref{E:gouvea maximal slope}$ is less than or equal to $\frac{k_0-1-\min\{a+1,p-2-a\}}{p+1}$.

Write $k_0 = 2+\{a+2s\} + (p-1)k_{0\bullet}$ as before. Using Proposition \ref{P:dimension of Sunr} and Remark \ref{R:delta}, we may separate the argument in the following two cases.

\begin{itemize}
\item If $k_{0\bullet} - (p+1) \lfloor \frac{k_{0\bullet}-t_1}{p+1}\rfloor  \geq t_2$, then 
\begin{align*}
\eqref{E:gouvea maximal slope}
\nonumber 
=\ & (p-1)\Big\lfloor \frac{k_{0\bullet}-t_1}{p+1}\Big\rfloor -\delta + t_2-1   \leq (p-1)\frac{(k_{0\bullet}-t_2)}{p+1} -\delta +t_2 -1 \\
 =\ & \frac{k_0-1}{p+1} + \frac{2t_2-\{ a+2 s\} -1}{p+1}-\delta-1 \\ 
\nonumber
=\ & \begin{cases}
	\frac{k_0-1}{p+1}- \frac{p-2-a}{p+1} & \textrm{if } a+s_\varepsilon < p-1
	\\
	\frac{k_0-1}{p+1}- \frac{a+1}{p+1} & \textrm{if } a+s_\varepsilon \geq p-1.
\end{cases}
\end{align*}
\item If $k_{0\bullet} - (p+1) \lfloor \frac{k_{0\bullet}-t_1}{p+1}\rfloor  < t_2$, then 
\begin{align*}
\eqref{E:gouvea maximal slope} 
\nonumber
=\ & (p-1)\Big\lfloor \frac{k_{0\bullet}-t_1}{p+1}\Big\rfloor -\delta + t_1   \leq (p-1)\frac{(k_{0\bullet}-t_1)}{p+1} - \delta + t_1 \\ 
=\ &\frac{k_0-1}{p+1} + \frac{2t_1-\{ a+2 s\} -1}{p+1}-\delta \\ 
\nonumber
=\ & \begin{cases}
	\frac{k_0-1}{p+1}- \frac{a+1}{p+1}  & \textrm{if } a+s_\varepsilon < p-1
	\\
	\frac{k_0-1}{p+1}- \frac{p-2-a}{p+1}& \textrm{if } a+s_\varepsilon \geq p-1.
\end{cases}
\end{align*}
\end{itemize}
The claim follows immediately. 
 
Now, we prove that the first $d_{k_0}^\unr$ slopes of $\NP(G(w_{k_0}, -))$ are all less than or equal to \eqref{E:gouvea maximal slope}. It suffices to show that for $i = 1, \dots, d_{k_0}^\unr$,
$$
v_p(g_{d_{k_0}^\unr}(w_{k_0})) - v_p(g_{d_{k_0}^\unr-i}(w_{k_0})) \leq i \cdot \eqref{E:gouvea maximal slope}.
$$

Recall from \eqref{E:increment of ghost series valuation 1} that for
$n = d_{k_0}^\unr -1,\dots,1, 0$ (so $k_{0\bullet} > k_{{\max}\bullet}(n)$ by Remark \ref{R:meaning of kmidminmax}),
\begin{eqnarray}
\label{E:vp difference k-1/p+1 conjecture}
\nonumber
 &v_p(g_{n+1}(w_{k_0})) - v_p(g_n(w_{k_0})) = \hspace{-15pt}\displaystyle\sum_{
k_{\midd\bullet}(n) < k_\bullet \leq k_{{\max}\bullet}(n)}
\hspace{-20pt}\big( v_p(k_{0\bullet} - k_\bullet)+1 \big)\ \,-  \hspace{-25pt} \displaystyle\sum_{
k_{{\min}\bullet}(n) \leq k_\bullet \leq k_{{\midd}\bullet}(n)}
\hspace{-20pt}\big( v_p(k_{0\bullet}-k_\bullet)+1 \big)
\\
&\qquad\qquad \quad
\stackrel{\eqref{E:sum of consecutive valuations}}= \dfrac{p(k_{{\max}\bullet}(n) - k_{\midd\bullet}(n)) - \Dig (k_{0\bullet}-k_{\midd\bullet}(n)-1)
+ \Dig(k_{0\bullet} - k_{{\max}\bullet}(n)-1)
}{p-1}
\\
\nonumber
& \qquad\qquad\qquad\quad -\ \dfrac{p(k_{{\midd}\bullet}(n) - k_{{\min}\bullet}(n)+1) - \Dig(k_{0\bullet} - k_{{\min}\bullet}(n)) + \Dig ( k_{0\bullet}- k_{\midd\bullet}(n)-1)
}{p-1}.
\end{eqnarray}
Using basic properties of $p$-based numbers and Lemma-Notation \ref{L:extremal ks}, one can deduce that
\begin{equation*}
pk_{{\min}\bullet}(n)+ \Dig(k_{0\bullet} - k_{{\min}\bullet}(n)) = \tilde k_{{\min}\bullet}(n)+ \Dig(pk_{0\bullet} - \tilde k_{{\min}\bullet}(n)).
\end{equation*}
From this, we see that \eqref{E:vp difference k-1/p+1 conjecture} is equal to
\begin{align*}
&\frac{\tilde k_{{\min}\bullet}(n) +pk_{{\max}\bullet}(n)- 2pk_{{\midd}\bullet}(n)-p}{p-1} + \frac{\Dig(k_{0\bullet}- k_{{\max}\bullet}(n)-1) + \Dig(pk_{0\bullet}  - \tilde k_{{\min}\bullet}(n))}{p-1}
\\
& \quad - \frac{2\Dig(k_{0\bullet}- k_{{\midd}\bullet}(n)-1)}{p-1}.
\end{align*}
Plugging in the formulae of Lemma-Notation~\ref{L:extremal ks} and setting 
\[
A_n = k_{0\bullet}- k_{{\max}\bullet}(n)-1, \quad B_n = pk_{0\bullet}  - \tilde k_{{\min}\bullet}(n), \quad C_n = k_{0\bullet}- k_{{\midd}\bullet}(n)-1, 
\]
we deduce that \eqref{E:vp difference k-1/p+1 conjecture} is equal to
$$
\frac{(p-1)n-1}2 - \delta_\varepsilon +\frac{p\beta_{[n]} - \beta_{[n-1]}}{p-1}
+ \frac{\Dig A_n+\Dig B_n-2\Dig C_n}{p-1}.$$

Therefore, we have
\begin{eqnarray}
\label{E:error}
&v_p(g_{d_{k_0}^\unr}(w_{k_0})) - v_p(g_{d_{k_0}^\unr-i}(w_{k_0})) -  i \cdot \big( \frac{p-1}2(d_{k_0}^\unr-1)-\delta_\varepsilon + \beta_{[d_{k_0}^\unr -1]}\big) \\
\nonumber &
= -\frac{i(i-1)}{4}(p-1)-\frac{i}{2}+ 
\begin{cases}
\frac{i}{2} (\beta_{[d_{k_0}^\unr]}-\beta_{[d_{k_0}^\unr -1]}) & \textrm{if $i$ is even}\\
\frac{i-1}{2} (\beta_{[d_{k_0}^\unr ]}-\beta_{[d_{k_0}^\unr -1]})+ \frac{\beta_{[d_{k_0}^\unr -1]}-\beta_{[d_{k_0}^\unr]}}{p-1} & \textrm{if $i$ is odd}
\end{cases}\\
\nonumber
& + \displaystyle \sum_{1 \leq j \leq i} \left( \frac{\Dig \big(A_{d_{k_0}^\unr - j}\big) +\Dig\big(B_{d_{k_0}^\unr - j}\big) - 2 \Dig\big(C_{d_{k_0}^\unr - j}\big)}{p-1}\right).
\end{eqnarray}
Note that $\beta_{[d_{k_0}^\unr]}-\beta_{[d_{k_0}^\unr-1]} \in \{ \pm( t_2-t_1-\frac{p+1}{2})\} = \{\pm \frac{p-3-2a}{2}\}$; so it is at most $\frac{p-5}{2}$. Since $p\geq5$, it is straightforward to see that the second line of \eqref{E:error} is less than or equal to 
\begin{eqnarray}
\label{E:second-line}
\begin{cases}
-\frac{2}{p-1} & \textrm{if $i=1$}\\
-i^2+\tfrac12 i & \textrm{if $i>1$}.
\end{cases}
\end{eqnarray}
Moreover, since \eqref{E:error} is an integer, it reduces to show 
\begin{equation}
\label{E:main-est}
\displaystyle \sum_{1 \leq j \leq i} \left( \frac{\Dig \big(A_{d_{k_0}^\unr - j}\big) +\Dig\big(B_{d_{k_0}^\unr - j}\big) - 2 \Dig\big(C_{d_{k_0}^\unr - j}\big)}{p-1}\right) + \eqref{E:second-line}<1. 
\end{equation}

From Lemma-Notation~\ref{L:extremal ks}, we have the relation
\[
B_n = (p+1) C_n - A_{n+1}-1, \qquad n\geq0.
\]
For $j=1,2,...,d^\unr_{k_0}$, we may compute $A_{d_{k_0}^\unr - j},B_{d_{k_0}^\unr - j},C_{d_{k_0}^\unr - j}$ in terms of $A_{d_{k_0}^\unr - 1},A_{d_{k_0}^\unr - 2}$, and $C= C_{d_{k_0}^\unr - 1}:$
\begin{equation*}
\label{E:A_n and C_n}
A_{d_{k_0}^\unr - j} =
\begin{cases}
A_{d_{k_0}^\unr -1}+ \frac{p+1}{2} (j-1) & \textrm{ if } j \textrm{ is odd}
\\
A_{d_{k_0}^\unr -2} +\frac{p+1}{2}(j-2) & \textrm{ if } j \textrm{ is even}
\end{cases} 
\textrm {    and  }C_{d_{k_0}^\unr - j} = C+ (j-1).
\end{equation*}
Hence
\begin{equation*}
 B_{d_{k_0}^\unr - j} =
\begin{cases}
(p+1)(C+\frac{j-1}{2})+p- A_{d_{k_0}^\unr -2} &  \textrm{ if } j \textrm{ is odd}
\\
(p+1)(C+\frac{j-2}{2})+p- A_{d_{k_0}^\unr -1} & \textrm{ if } j \textrm{ is even}.
\end{cases} 
\end{equation*}
Note that $k_{\max\bullet}(n)$ is an increasing function in $n$ by Remark \ref{R:meaning of kmidminmax}.  Thus $A_{d_{k_0}^\unr -1}\leq A_{d_{k_0}^\unr -2}=\tfrac 12(p+1)<p$. Then using the inequalities $\Dig(A+B)\leq\Dig(A) + \Dig(B)  $ and $\Dig((p+1)D)\leq 2\Dig(D) \leq 2D$, we conclude that
\begin{equation*}
\nonumber
\Dig\big(A_{d_{k_0}^\unr - j}\big) \leq
\begin{cases}
A_{d_{k_0}^\unr -1}+ (j-1) &  \textrm{ if } j \textrm{ is odd}
\\
A_{d_{k_0}^\unr -2}+ (j-2) & \textrm{ if } j \textrm{ is even},
\end{cases} 
\end{equation*}
and
\begin{equation*}
\label{Eq1}
\Dig\big(B_{d_{k_0}^\unr - j}\big) \leq
\begin{cases}
\Dig\big((p+1)(C+\frac{j-1}{2})\big)+p- A_{d_{k_0}^\unr -2}&  \textrm{ if } j \textrm{ is odd}
\\
\Dig\big((p+1)(C+\frac{j-2}{2})\big)+p- A_{d_{k_0}^\unr -1}& \textrm{ if } j  \textrm{ is even}.
\end{cases} 
\end{equation*}  
Summing up these inequalities for $j=1,\dots, i$, we get

\begin{align}
\label{E:diga+digb}
\displaystyle \sum_{1 \leq j \leq i}{\big( \Dig A_{d_{k_0}^\unr - j} +\Dig B_{d_{k_0}^\unr-j}}\big) \leq &\ pi + 2\lfloor\tfrac{i-1}{2}\rfloor\lfloor\tfrac{i+1}{2}\rfloor+ 2\displaystyle \sum_{0 \leq j \leq \lfloor\frac{i-1}{2}\rfloor}{\Dig \big((p+1)(C+j)\big)}\\
\nonumber
&- \delta(i)\big(2\lfloor\tfrac{i-1}{2}\rfloor+\Dig \big((p+1)(C+\lfloor\tfrac{i-1}{2}\rfloor)\big)\big),
\end{align}
where $\delta(i)\in\{0,1\}$ is determined by $\delta(i)\equiv i\mod 2$. 

We need the following lemma.

\begin{lemma}
	\label{Diglemma}
	For $p\geq3$ and $m,N \in \mathbb{N}$,  we have
	\[
	\Dig\big((p+1)N \big) \leq \Dig(N)+ \Dig(N+m)+(p-2)m 
	\]
\end{lemma}
\begin{proof}
Denote by $a$ and $b$ the number of times we carry over digits to the next one when computing the sums $N+pN$ and $N+m$ respectively. It follows that
\[
\Dig\big((p+1)N \big)=\Dig(N)+\Dig(pN)-(p-1)a=2\Dig(N)-(p-1)a
\]
and
\[
\Dig(N+m)=\Dig(N)+\Dig(m)-(p-1)b.
\]
Let $d_m$ be the number of digits of $m$, namely $\lfloor\frac{\ln m}{\ln p}\rfloor+1$, if in the $p$-based expression of $m$ the leading term is not $p-1$; otherwise, we define $d_m$ to be one plus the number of digits of $m$. (Here and later, $\ln(-)$ denotes the natural real logarithmic.) It is not difficult to check that if $b-d_m \geq 1$, then the there must be at least $b-d_m$ consecutive $(p-1)$'s in $N$, and moreover they will force carrying over digits at least $b-d_m$ times. So $a \geq b-d_m$. 
Now, combining with the inequalities above, we deduce that
$$
\Dig((p+1)N) - \Dig(N) -\Dig(N+m) \leq (p-1)(b-a)-\Dig(m) \leq  (p-1)d_m -\Dig(m).
$$
It is not difficult to verify that $(p-1)d_m \leq \Dig(m) + (p-2)m$. The lemma is proved.
\end{proof}

Applying Lemma \ref{Diglemma} to $N=C,C+1,\dots,C +\lfloor\frac{i-1}{2}\rfloor$ and $m=\lfloor\frac{i}{2}\rfloor$ or $\lfloor\frac{i+1}{2}\rfloor$, and adding them up, we deduce
that 
\begin{align*}
&2 \sum_{0 \leq j \leq \lfloor\frac{i-1}{2}\rfloor}{\Dig \big((p+1)(C+j)\big)}-\delta(i)\Dig \big((p+1)(C+\lfloor\tfrac{i-1}{2}\rfloor)\big) \\
\leq\ &
2\displaystyle \sum_{0 \leq j \leq i-1}{\Dig\big((C+j)\big)}+(p-2)\lfloor\tfrac{i+1}{2}\rfloor\lfloor\tfrac{i}{2}\rfloor.
\end{align*}
Note that $\lfloor\frac{i+1}{2}\rfloor-\delta(i)=\lfloor\frac{i}{2}\rfloor$. Combining with \eqref{E:diga+digb}, we deduce that
\begin{align*}
\displaystyle \sum_{1 \leq j \leq i} 
\frac{\Dig A_{d_{k_0}^\unr - j} +\Dig B_{d_{k_0}^\unr - j} - 2 \Dig C_{d_{k_0}^\unr - j}}{p-1}&\leq \frac{pi +2\lfloor\frac{i-1}{2}\rfloor\lfloor\frac{i}{2}\rfloor+ 2(p-2)\lfloor\frac{i+1}{2}\rfloor\lfloor\frac{i}{2}\rfloor}{p-1} \\
\nonumber
&=2\lfloor\tfrac{i+1}{2}\rfloor\lfloor\tfrac{i}{2}\rfloor+i+\tfrac{\delta(i)}{p-1}.
\end{align*}

Put
\[
S(i):=2\lfloor\tfrac{i+1}{2}\rfloor\lfloor\tfrac{i}{2}\rfloor+i+\tfrac{\delta(i)}{p-1}+ \eqref{E:second-line}.
\]
A short computation shows that
$S(1)=1-\frac{1}{p-1}, S(2)=1$, and for $i\geq3$,
\[
S(i)\leq \frac{i^2}{2}+i+\eqref{E:second-line}=\frac{-i^2+3i}{2}\leq0.
\]
By \eqref{E:main-est}, we conclude the desired result except for $i=2$. Moreover, if \eqref{E:main-est} fails for $i=2$, we must have 
\[
\Dig\big((p+1)C+(p-A_{d_{k_0}^\unr -2})\big)=\Dig\big((p+1)C\big)+p- A_{d_{k_0}^\unr -2}
\]
and
\[
\Dig((p+1)C)=\Dig(C)+\Dig(C+1)+(p-2).
\]
From the second equality, we deduce that $\Dig(C+1)<\Dig(C)$, yielding $p|(C+1)$. However, this contradicts the first equality since $p-A_{d_{k_0}^\unr -2}=\tfrac12(p-1)>1$. 
\end{proof}

\section{Vertices of the Newton polygon of ghost series}
\label{Sec:finer properties of ghost series}

In this section, we investigate further properties of ghost series.
We keep Hypothesis~\ref{H:b=0} on $\widetilde \rmH$, i.e. $\widetilde \rmH$ is a primitive $\calO\llbracket K_p\rrbracket$-projective augmented module of type $\bar\rho = \Matrix {\omega_1^{a+1}}{*\neq 0}{0}{1}$ with $1 \leq a \leq p-4$, on which $\Matrix p00p$ acts trivially.

\emph{We fix  $\varepsilon = \omega^{-s_\varepsilon} \otimes \omega^{a+s_\varepsilon}$ a relevant character of $\Delta^2$ in this section. We will suppress $\varepsilon$ from all superscripts when no confusion arises. Let $G(w,t) = G^{(\varepsilon)}(w,t) = \sum_{n\geq 0} g_n(w) t^n \in \ZZ_p[w]\llbracket t\rrbracket$ denote the corresponding ghost series.}

\medskip
We prove the following two results in this section:
\begin{enumerate}
\item For $n \in \ZZ_{\geq 0}$ and $w_\star \in \gothm_{\CC_p}$, the point $(n, v_p(g_n(w_\star)))$ is a vertex on the Newton Polygon $\NP(G(w_\star, -))$ if and only if $(n, w_\star)$ is not near-Steinberg (Theorem~\ref{T:near Steinberg = non-vertex}). Here, near-Steinberg roughly means that $w_\star$ is ``close" to a weight point $w_k$ such that ``$n$ lies between $d_k^\unr$ and $d_k^\Iw-d_k^\unr$".
The precise definition of this concept in Definition~\ref{D:near steinberg} is an essential tool to parcel the delicate combinatorics of the ghost series.

\item For $k \in \ZZ$, the slopes of $\NP(G(w_k,-))$ are all integers if $a$ is even and belong to $\frac 12\ZZ$ if $a$ is odd (Corollary~\ref{C:integrality of slopes of G}). This sheds some lights on the folklore Breuil--Buzzard--Emerton Conjecture (Conjecture~\ref{Conj:Breuil-Buzzard-Emerton}) on slopes of Kisin deformation spaces.
\end{enumerate}

Before diving into the more detailed discussion, let us first explain one of the serious subtleties we face in this section.  While most numerical information of ghost series can be easily described asymptotically through linear or quadratic terms, the precise values of these numerical data are somewhat subtle and often involve the $p$-based expansion of the weight $k$ (or rather $k_\bullet$). This means that for ``majority" of the case where the $p$-based expansion of $k_\bullet$ is not ``special", the behavior of the Newton polygon of the ghost series is ``reasonable". In contract, near these ``special" weights $k$ (e.g. $k_\bullet$ is an integer like $p^{p^p}-p$), we need to be extra careful.  See Example~\ref{Ex:pathological Steinberg range} for a discussion in an example of such pathological case.

\begin{notation}
\label{N:Delta kell}
For integer $k \geq 2$ and $k \equiv k_\varepsilon \bmod{(p-1)}$, we set
\begin{equation}
\label{E:definition of Delta'}
\Delta'_{k, \ell} : = v_p \big(g_{\frac 12 d_{k}^\Iw+\ell, \hat k}(w_k) \big) -\tfrac{k-2}2 \ell, \quad \textrm{for }\ell = -\tfrac 12 d_{k}^\new, -\tfrac 12 d_{k}^\new+1, \dots, \tfrac 12d_{k}^{\new}.
\end{equation}
Then the ghost duality statement in Proposition~\ref{P:ghost compatible with theta AL and p-stabilization}\,(4) is equivalent to the following
equality
\begin{equation}
\label{E:ghost duality alternative}
\Delta'_{k,\ell} = \Delta'_{{k}, -\ell} \quad\textrm{ for all }\ell = -\tfrac 12 d_{k}^\new, \dots, \tfrac 12d_{k}^{\new}.
\end{equation}

If we connect the points $(\ell, \Delta'_{k,\ell})$ with line segments on the plane to get the graph $\underline \Delta'_k$ of a function, the function is in general not convex (although for ``majority" $k$ it is; see Lemma~\ref{L:failure of convexity of Delta'}). Consider the convex hull $\underline \Delta_k = \underline \Delta^{(\varepsilon)}_k$ of such graph and let $(\ell, \Delta_{k, \ell}) = (\ell, \Delta^{(\varepsilon)}_{k, \ell})$ denote the corresponding points on this convex hull. In particular, the ghost duality \eqref{E:ghost duality alternative} implies that $\Delta_{k,\ell} = \Delta_{k, -\ell}$.

\end{notation}


We start with several technical lemmas giving estimates of $\Delta'_{k, \ell}$ and $\Delta_{k,\ell}$.  The proofs of these lemmas have complicated notations though not difficult themselves. The readers may skip the proofs for the first reading (but do read Notation~\ref{N:linear shift down}). The motivation to study this polygon $\underline \Delta_k$ is explained by Definition~\ref{D:near steinberg} and the following remarks and examples.
\begin{lemma}
\label{L:estimate Delta kell}
For $k =k_\varepsilon + (p-1)k_{\bullet}$ and $\ell = 1,\dots, \frac 12 d_{k}^\new$, we have
$$
\Delta'_{k, \ell} - \Delta'_{k, \ell-1}\;\geq\; \frac {\min\{a+2, p-1-a\}}2+ \frac {p-1}2 (\ell-1) \;\geq\;  \frac 32 + \frac {p-1}2(\ell-1).
$$
\end{lemma}
\begin{proof}
By ghost duality \eqref{E:ghost duality alternative}, it is equivalent to consider $\Delta'_{k,-\ell} -\Delta'_{k, -\ell+1} $.
We use the formula \eqref{E:increment of ghost series valuation 1} for $n = \frac 12 d_{k}^\Iw - \ell = k_{\bullet}+1-\delta_\varepsilon -\ell$, (and write $k_{*\bullet}$ for $k_{* \bullet}(n)$ with $* = \min, \max, \midd$)
\begin{small}
\begin{align}
\nonumber
&\Delta'_{k,-\ell} - \Delta'_{k,-\ell+1}  = 
\frac{k-2}2-  \hspace{-20pt}
\sum_{\substack{k'_\bullet \neq k_\bullet \\
k_{\midd\bullet} < k'_\bullet  \leq k_{{\max}\bullet}}}
\hspace{-20pt}\big( v_p(k'_\bullet - k_{\bullet})+1 \big)\ +  \hspace{-20pt} \sum_{
k_{{\min}\bullet} \leq k'_\bullet \leq k _{{\midd}\bullet}}
\hspace{-20pt}\big( v_p(k'_\bullet - k_{\bullet})+1 \big).
\\
\nonumber
\stackrel{\eqref{E:sum of consecutive valuations}}=\ & \frac{k-2}2 - \big( k_{{\max}\bullet}-k_{\midd\bullet }-1\big) + \big( k_{\midd \bullet} -k_{{\min}\bullet} +1 \big)-\frac{k_\bullet-k_{{\midd}\bullet}-1- \Dig(k_\bullet-k_{{\midd}\bullet}-1)}{p-1}
\\
\nonumber
&- \frac{k_{{\max}\bullet}-k_\bullet - \Dig(k_{{\max}\bullet}-k_\bullet)}{p-1}+\frac{k_{{\midd}\bullet} - k_{{\min}\bullet} +1-\Dig(k_\bullet - k_{{\min}\bullet})+ \Dig(k_\bullet-k_{{\midd}\bullet}-1)}{p-1}
\\
\label{E:Delta'-Delta'}
=\ & \frac{k-2}2+ \frac {p(2k_{\midd \bullet}- k_{{\max}\bullet}- k_{{\min}\bullet}+2)}{p-1}+ \frac{\Dig(k_{{\max}\bullet} - k_{\bullet})- \Dig (k_{\bullet} - k_{{\min}\bullet})+2\Dig (\ell-1)}{p-1}.
\end{align}\end{small}
Here $\Dig(m)$ is the sum of all digits in the $p$-based expression of $m$, and we made use of formula \eqref{E:sum of consecutive valuations} for sums of valuations of continuous integers, and also noted that $k_\bullet - k_{\midd\bullet} = \ell$.
One can play the digital expansion game to deduce that
\begin{equation}
\label{E:digital expansion game}
pk_{{\min}\bullet}+ \Dig(k_{\bullet} - k_{{\min}\bullet}) = \tilde k_{{\min}\bullet}+ \Dig(pk_{\bullet} - \tilde k_{{\min}\bullet}).
\end{equation}
So \eqref{E:Delta'-Delta'} is equal to \begin{equation}
\label{E:Delta'-Delta'2}
\frac{k-2}2+ \frac {2pk_{\midd \bullet}-p k_{{\max}\bullet}- \tilde k_{{\min}\bullet}+2p}{p-1}+\frac{\Dig(k_{{\max}\bullet} - k_{\bullet})- \Dig (pk_{\bullet} - \tilde  k_{{\min}\bullet})+2\Dig (\ell-1)}{p-1}.
\end{equation} 

By Lemma-Notation~\ref{L:extremal ks}, 
\begin{align*}
&k_{\midd \bullet} =k_{\bullet} -\ell, \quad k_{{\max}\bullet} =\tfrac {p+1}2 (k_{\bullet}+1-\delta_\varepsilon -\ell) + \beta_{[n]}-1,
\\
& \textrm{and} \quad \tilde k_{{\min}\bullet}  =  \tfrac{p+1}2 (k_{\bullet}-\ell+\delta_\varepsilon) - \beta_{[n-1]}+1.
\end{align*}
Then the sum of the first two terms in \eqref{E:Delta'-Delta'2} is equal to
\begin{small}
\begin{align*}
&\frac{k-2}2 +
\frac{2 p\big( k_\bullet -\ell\big)  -p \big(\tfrac{p+1}2(k_\bullet+1 - \delta_\varepsilon -\ell) + \beta_{[n]}-1\big) 
- \big(\tfrac{p+1}2(k_\bullet - \ell+ \delta_\varepsilon) - \beta_{[n-1]}+1 \big)+2p}{p-1}
\\
&= \frac{p-1}2 (\ell-1)+ \frac{k_\varepsilon - 2}2+1+\frac{p+1}2 \delta_\varepsilon -\beta_{[n]}+ \frac{\beta_{[n-1]}-\beta_{[n]}+\frac{p+1}2}{p-1}.
\end{align*}\end{small}We can list the values of $\frac{k_\varepsilon - 2}2+1+\frac{p+1}2 \delta_\varepsilon -\beta_{[n]}$ and $\beta_{[n-1]}-\beta_{[n]}+\frac{p+1}2$ in the following table.
\begin{center}
\begin{tabular}{|c|c|c|c|}
\hline
& &$\frac{k_\varepsilon - 2}2+1+\frac{p+1}2 \delta_\varepsilon -\beta_{[n]}$ & $\beta_{[n-1]}-\beta_{[n]}+\frac{p+1}2$
\\
\hline
\multirow{2}{*}{$a+s_\varepsilon < p-1$} & $n$ even & $\frac 12(a+2)$ & $a+2$
\\
\cline{2-4}
& $n$ odd & $\frac 12(p-1-a)$ & $p-1-a$
\\
\hline 
\multirow{2}{*}{
$a+s_\varepsilon \geq p-1$} & $n$ even & $\frac 12(p-1-a)$ & $p-1-a$
\\
\cline{2-4} & $n$ odd & $\frac 12(a+2)$ & $a+2$
\\
\hline
\end{tabular}
\end{center}
We set $\theta: = \beta_{[n-1]}-\beta_{[n]}+\frac{p+1}2$, which is equal to either $a+2$ or $p-1-a$; in particular, $3\leq \theta \leq p-2$. Then the sum of the first two terms in \eqref{E:Delta'-Delta'2} is equal to
$$
\frac{p-1}2(\ell-1) + \frac 12\cdot  \theta + \frac{\theta}{p-1}.
$$

On the other hand, 
for the digital sum in the last term of \eqref{E:Delta'-Delta'2}, if we set
\begin{equation}
\label{E:eta}
\eta: = \frac{p-1}2k_\bullet -\frac{p+1}2\delta_\varepsilon + \beta_{[n]}-1,
\end{equation}
then we can write in a uniform fashion that
\begin{equation}
\label{E:kmax - k}
k_{{\max}\bullet}-k_\bullet = \eta - \frac{p+1}2 (\ell-1) , \quad pk_\bullet -\tilde k_{{\min}\bullet} =\eta+\theta +\frac{p+1}2(\ell-1).
\end{equation}

To sum up, we deduce from \eqref{E:Delta'-Delta'2} and the discussion above that $\Delta'_{k, \ell} - \Delta'_{k, \ell-1}$ is equal to
\begin{equation}
\label{E:Delta'-Delta'3}
\frac {(p-1) (\ell-1) + \theta}2+ \frac{\theta + \Dig\big(\eta -\frac{p+1}2(\ell-1)\big) + 2\Dig(\ell-1) - \Dig\big(\eta +\theta+\frac{p+1}2(\ell-1)\big)}{p-1}.
\end{equation}
But it is clear that 
\begin{align}
\nonumber
&\theta + \Dig\big(\eta -\tfrac{p+1}2(\ell-1)\big) + 2\Dig(\ell-1)
\\
\nonumber
=\ &\Dig \big(\eta -\tfrac{p+1}2(\ell-1)\big) + \Dig(\ell-1) + \Dig (p\ell-p+\theta)
\\
\label{E:inequality in difference of delta} \geq\ & \Dig\big(\eta +\theta+\tfrac{p+1}2(\ell-1)\big).
\end{align}
It follows that
\[
\Delta'_{k, \ell} - \Delta'_{k, \ell-1} \geq \tfrac 12\big((p-1) (\ell-1) + \theta \big) \geq \tfrac{p-1}2(\ell-1) + \tfrac 32. \qedhere
\]
\end{proof}

The following corollary in fact slightly strengthens the previous lemma. Despite the technical nature, it will be used in multiple important proofs later (for Proposition~\ref{P:shifting points wstar to wk} and for results in the forthcoming paper).
\begin{corollary}
\label{C:finer inequality for Delta}
Let $k = k_\varepsilon + (p-1)k_\bullet$ and $\ell =1, \dots, \frac 12 d_k^\new$ be as above. Let $k' = k_\varepsilon + (p-1) k'_\bullet$ 
be any weight such that 
\begin{equation}
\label{E:dk' boundaries belongs to NS range} 
d_{k'}^\ur, \textrm{ or }d_{k'}^\Iw - d_{k'}^\ur\textrm{  belongs to }\big(\tfrac 12d_k^\Iw-\ell, \tfrac 12 d_k^\Iw+\ell\big)
,
\end{equation}
or $\frac 12d_{k'}^\Iw \in \big[\tfrac 12d_k^\Iw-\ell, \tfrac 12 d_k^\Iw+\ell\big]$,
then we have
\begin{equation}
\label{E:Delta kl difference strong}
\Delta'_{k, \ell} - \Delta'_{k,\ell-1} -v_p(w_k - w_{k'}) \geq \frac 12 + \frac {p-1}2(\ell-1) -\Big \lfloor \frac{\ln((p+1) \ell)}{\ln p}\Big\rfloor  \geq \frac 12 (2\ell-1),
\end{equation}
except when $\ell = 1$, the last inequality fails but $\Delta'_{k, 1} - \Delta'_{k,0} - v_p(w_{k} - w_{k'}) \geq \frac 12$ still holds.

When $p \geq 7$ and $\ell \geq 2$, we have
$\Delta'_{k, \ell} - \Delta'_{k,\ell-1}-v_p(w_{k}-w_{k'}) \geq \frac 12(2\ell+1)$.
\end{corollary}
\begin{proof}
The second inequality in \eqref{E:Delta kl difference strong} is obvious (when $\ell >1$); we will prove the first inequality. (Similarly, the last statement also follows from the first inequality.)

We first treat the easy case: $\frac 12d_{k'}^\Iw \in \big[\frac 12d_k^\Iw-\ell, \ \frac 12d_k^\Iw+\ell\big]$.
Since $\frac 12 d_{k}^\Iw - \frac 12 d_{k'}^\Iw = k_\bullet - k'_\bullet$, we deduce that $|k'_\bullet - k_\bullet|\leq \ell$. In particular, $v_p(w_{k_\bullet}-w_{k'_\bullet}) \leq 1+ \big\lfloor \ln \ell / \ln p\big\rfloor$. So the first inequality of \eqref{E:Delta kl difference strong} follows from Lemma~\ref{L:estimate Delta kell} immediately.

We now assume \eqref{E:dk' boundaries belongs to NS range} holds.  First note that, by Remark~\ref{R:meaning of kmidminmax}, the condition \eqref{E:dk' boundaries belongs to NS range} is equivalent to
\begin{equation}
\label{E:k' so that dk'unr in Steinberg}
k'_\bullet \in \big[k_{{\min}\bullet}(\tfrac 12 d_{k}^\Iw - \ell), k_{{\min}\bullet}(\tfrac 12 d_{k}^\Iw + \ell-1) \big) \cup \big(k_{{\max}\bullet}(\tfrac 12 d_{k}^\Iw - \ell), k_{{\max}\bullet}(\tfrac 12 d_{k}^\Iw + \ell-1) \big],
\end{equation}
Using the equality (2) in \S\,\ref{S:proof of ghost compatible with AL}:
\begin{equation}
\label{E:near Steinberg duality}
k_{{\max}\bullet}(d_{k}^\Iw - \ell') - k_\bullet= pk_\bullet - \tilde k_{{\min}\bullet}(\ell' -1 ), \quad \textrm{for }\ell'  = \tfrac 12 d_{k}^\Iw- \ell+1 \textrm{ or } \tfrac 12 d_{k}^\Iw +\ell,
\end{equation}
we see that for every $k'_\bullet \in \big[k_{{\min}\bullet}(\tfrac 12 d_{k}^\Iw - \ell), k_{{\min}\bullet}(\tfrac 12 d_{k}^\Iw + \ell-1) \big)$, there exists $k''_\bullet \in \big(k_{{\max}\bullet}(\tfrac 12 d_{k}^\Iw - \ell), k_{{\max}\bullet}(\tfrac 12 d_{k}^\Iw + \ell-1) \big]$ such that
$
k''_\bullet -k_\bullet = p(k_\bullet  -k'_\bullet).
$
So to prove the lemma, one may assume that 
\begin{equation}
\label{E:condition on k'}
k'_\bullet \in \big(k_{{\max}\bullet}(\tfrac 12 d_{k}^\Iw - \ell), k_{{\max}\bullet}(\tfrac 12 d_{k}^\Iw + \ell-1) \big].
\end{equation}

Now we import the notation from the proof of Lemma~\ref{L:estimate Delta kell}: $n: = \frac 12d_k^\Iw -\ell$, $\theta := \beta_{[n-1]}-\beta_{[n]}+\frac{p+1}2$, and $\eta$ as defined in \eqref{E:eta}. Then condition \eqref{E:condition on k'} is equivalent to
$$
k'_\bullet \in \big(k_{\mathrm{max}\bullet}(n), \ k_{\mathrm{max}\bullet}(n+ 2\ell-1)\big].
$$
So by \eqref{E:kmax - k}, we deduce that
\begin{eqnarray*}
k'_\bullet - k_\bullet & > & \eta-\tfrac{p+1}2(\ell-1), \quad \textrm{ and }
\\
k'_\bullet - k_\bullet & \leq & \eta-\tfrac{p+1}2(\ell-1) + \big( k_{\mathrm{max}\bullet}(n+2\ell-1) - k_{\mathrm{max}\bullet}(n)\big)
\\
 &=& \eta -\tfrac{p+1}2(\ell-1) + (p+1)(\ell-1) + \theta = \eta+\theta + \tfrac{p+1}2(\ell-1),
\end{eqnarray*}
where the first equality uses the formula for $k_{\mathrm{max}\bullet}(n)$ in Lemma-Notation~\ref{L:extremal ks}(2).

Now we argue as in Lemma~\ref{L:estimate Delta kell} till \eqref{E:Delta'-Delta'3} to get
\begin{small}
\begin{equation}
\label{E:Delta kl - Deltakl-1 prime}
\Delta'_{k, \ell}- \Delta'_{k, \ell-1} = 
\frac {(p-1) (\ell-1) + \theta}2+ \frac{\theta + \Dig\big(\eta -\frac{p+1}2(\ell-1)\big) + 2\Dig(\ell-1) - \Dig\big(\eta +\theta+\frac{p+1}2(\ell-1)\big)}{p-1}.
\end{equation}\end{small}Then we simply note that
\begin{align*}
&\tfrac{1}{p-1}\Big(\theta + \Dig\big(\eta -\tfrac{p+1}2(\ell-1)\big) + 2\Dig(\ell-1)\Big)
\\\geq \ &\tfrac{1}{p-1}\Big(\Dig \big(\eta -\tfrac{p+1}2(\ell-1)\big) + \Dig((p+1)(\ell-1)+\theta)\Big)
\\ \geq\ & \tfrac{1}{p-1}\Dig\big(\eta +\theta+\tfrac{p+1}2(\ell-1)\big) + v_p(k'_\bullet - k_\bullet) - \lfloor \ln((p+1)(\ell-1)+\theta)/\ln p\rfloor.
\end{align*}

Here the last inequality uses the following elementary fact (proved in Lemma~\ref{L:elementary Digits lemma} below): for two positive integers $A$ and $B$, we have
$$
\frac{\Dig(A) + \Dig(B) -\Dig(A+B)}{p-1} +\Big\lfloor \frac{\ln B}{\ln p}\Big\rfloor \geq \max_{A<x\leq A+B}\big\{v_p(x)\big\}.
$$
From this, we deduce that
\begin{align*}
\Delta'_{k, \ell} - \Delta'_{k,\ell-1} \geq\tfrac {(p-1) (\ell-1) + \theta}2 +v_p(k'_\bullet - k_\bullet)  - \big\lfloor \ln((p+1)(\ell-1)+\theta) / \ln p\big\rfloor.
\end{align*}
The corollary follows when $\ell >1$ because $\theta \in [3,p-2]$.

In the special case of $\ell=1$, the above inequality becomes
\begin{align*}
\Delta'_{k,1} - \Delta'_{k,0}\geq \frac \theta 2+v_p(k'_\bullet - k_\bullet)
\end{align*}
as $\theta<p$. So we deduce that
\[
\Delta'_{k,1}-\Delta'_{k,0}-v_p(w_{k} - w_{k'})\geq \frac \theta 2 - 1 \geq \frac 12. \qedhere
\]

\end{proof}

\begin{lemma}
\label{L:elementary Digits lemma}For two positive integers $A$ and $B$, we have 
	\[
	\frac{\Dig(A)+\Dig(B)-\Dig(A+B)}{p-1}+\Big\lfloor \frac{\ln B}{\ln p}\Big\rfloor\geq \max_{A<x\leq A+B}\{ v_p(x) \}.
	\]
\end{lemma}
\begin{proof}
Let $n=\lfloor \frac{\ln B}{\ln p}\rfloor$. First note that
\[
	\frac{\Dig(A)+\Dig(B)-\Dig(A+B)}{p-1}=v_p((A+B)!)-v_p(A!)-v_p(B!)=v_p\Big(\frac{\prod_{j=1}^B (A+j)}{B!} \Big).
	\]
If $\max_{A<x\leq A+B}\{ v_p(x) \}\leq n$, then we are done. Otherwise, suppose 
\[
n<v_p(A+i)=\max_{A<x\leq A+B}\{ v_p(x)\}
\]
for some $i\in [1, B]$. Note that this implies that $v_p(A+i\pm j)=v_p(j)$ for $j\in[1, B]$. It follows that 
\[
v_p\Big(\frac{\prod_{j=1}^B (A+j)}{B!} \Big)=v_p((i-1)!)+v_p((B-i)!)+v_p(A+i)-v_p(B!).
\]
For every $\ell\geq1$, it is clear that $\lfloor \frac{i-1}{p^\ell}\rfloor+\lfloor \frac{B-i}{p^\ell}\rfloor-\lfloor \frac{B}{p^\ell}\rfloor\geq-1$. We therefore conclude that
\[
v_p\Big(\frac{\prod_{j=1}^B (A+j)}{B!} \Big)\geq v_p(A+i)-n,
\]
yielding the desired result.
\end{proof}

In our forthcoming paper where we prove the local ghost conjecture, we need to establish similar bounds for $\Delta_{k, \ell} - \Delta'_{k, \ell-1}$ as opposed to $\Delta'_{k, \ell} - \Delta'_{k, \ell-1}$. The following three lemmas provide results in this direction.

\begin{lemma}
\label{L:upper bound of difference of Delta kell}
Fix $k =k_\varepsilon + (p-1)k_{\bullet}$ and $\ell \in \{1,2,\dots, \frac 12 d_{k}^\new\}$, and define $\eta$ and $\theta$ as in the proof of Lemma~\ref{L:estimate Delta kell}.  Let $\beta$ denote the maximal $p$-adic valuations of integers in $\big[\eta-\frac{p+1}2(\ell-1),\  \eta+\theta+ \frac{p+1}2(\ell-1)\big]$. Then we have
$$
\Delta'_{k, \ell} - \Delta'_{k, \ell-1}\leq 
\tfrac {p-1}2\cdot \ell+ \tfrac 32+ \beta + \lfloor \ln \ell / \ln p\rfloor
$$
\end{lemma}
\begin{proof}
We first point out that, when $\ell > 1$, $\beta$ is automatically greater than or equal to $\lfloor \ln \ell / \ln p\rfloor + 1$.
By \eqref{E:Delta'-Delta'3} and the estimate $\theta \leq p-2$, it remains to show that
\begin{small}
$$
\frac{ \Dig\big(\eta -\frac{p+1}2(\ell-1)\big) + \Dig(\ell-1)+\Dig(p(\ell-1) + \theta) - \Dig\big(\eta +\theta+\frac{p+1}2(\ell-1)\big)}{p-1} 
\leq 
\beta +2+\Big \lfloor\frac{ \ln \ell}{ \ln p}\Big\rfloor .
$$\end{small}The left hand side exactly counts the number of times we carry over digits to the next one when computing the sum of numbers $\eta- \frac {p+1}2(\ell-1)$, $(\ell-1)$, and $p(\ell-1)+\theta$. This is clear.
\end{proof}

\begin{lemma}
\label{L:failure of convexity of Delta'}
Fix $k =k_\varepsilon + (p-1)k_{\bullet}$ and $\ell \in \{1,\dots, \frac 12 d_{k}^\new-1\}$. 
Let $\theta_\ell \in \{a+2, p-1-a\}$ be the $\theta$ as in the proof of Lemma~\ref{L:estimate Delta kell} for $\ell$.
Then we have
\begin{equation}
\label{E:failure of convexity of Delta'}
\Delta'_{k, \ell+1} -2 \Delta'_{k, \ell}+\Delta'_{k, \ell-1}\geq p-1-\theta_\ell-2v_p(\ell) \geq 1-2v_p(\ell).
\end{equation}
\end{lemma}
\begin{proof}
Write $\theta_\ell$ and $\eta_\ell$ (resp. $\theta_{\ell+1}$ and $\eta_{\ell+1}$) as in the proof of Lemma~\ref{L:estimate Delta kell} for $\ell$ (resp. $\ell+1$). 
By \eqref{E:Delta'-Delta'3}, we deduce that
\begin{align*}
&\Delta'_{k, \ell+1}-2\Delta'_{k,\ell}+\Delta'_{k,\ell-1} = \frac{p-1}2 +\frac{\theta_{\ell+1} - \theta_{\ell}}2 +\frac{\Dig\big(\eta_{\ell+1}-\frac{p+1}2\ell\big) - \Dig \big(\eta_{\ell}- \frac{p+1}2(\ell-1)\big)}{p-1} 
\\ &+ \frac{2\big( \Dig(\ell)-\Dig(\ell-1)\big)}{p-1} + \frac{\theta_{\ell+1} - \theta_{\ell}+\Dig\big(\eta_{\ell}+\theta_{\ell}+\frac{p+1}2(\ell-1)\big) - \Dig \big(\eta_{\ell+1}+\theta_{\ell+1}+ \frac{p+1}2\ell\big)}{p-1}.
\end{align*}
Write $n = \frac 12d_k^\Iw-\ell$ as in the proof of Lemma~\ref{L:estimate Delta kell}. We note that $\theta_\ell + \theta_{\ell+1} = p+1$ by their definitions and, for the fourth term, we have
$$
\frac2{p-1}\big(\Dig(\ell) - \Dig(\ell-1)\big) =\frac2{p-1}-2v_p(\ell).
$$
For the third term, we note that
$$
\big(\eta_{\ell}- \tfrac{p+1}2(\ell-1)\big)- \big(\eta_{\ell+1}-\tfrac{p+1}2\ell\big)  = \tfrac{p+1}2 +\eta_{\ell} -\eta_{\ell+1} = \tfrac{p+1}2 + \beta_{[n]}-\beta_{[n+1]} = \theta_{\ell+1}.
$$
Similarly, for the fifth term, we have
$$
\big(\eta_{\ell+1}+\theta_{\ell+1}+ \tfrac{p+1}2\ell\big) - \big(\eta_{\ell}+\theta_{\ell}+\tfrac{p+1}2(\ell-1)\big) = \tfrac{p+1}2 + \beta_{[n+1]}- \beta_{[n]}+ \theta_{\ell+1}-\theta_{\ell} =\theta_{\ell+1}.
$$
So the sum of the third term and the fifth term above is greater than or equal to
$$
\frac 1{p-1}\Big( -\theta_{\ell+1} + \big(\theta_{\ell+1} - \theta_{\ell} \big) - \theta_{\ell+1}\Big) =- \frac1{p-1} \big(\theta_{\ell+1} + \theta_{\ell}\big) =- \frac 1{p-1} \cdot (p+1). 
$$
Combining all inequalities above gives
\begin{align*}
\Delta'_{k, \ell+1}-2\Delta'_{k,\ell} + \Delta'_{k, \ell-1} \geq\ & \frac{p-1}2 + \frac{p+1-2\theta_\ell}2 +\Big(\frac2{p-1}-2v_p(\ell)\Big) - \frac{p+1}{p-1} 
\\
=\ & p-1-\theta_\ell  -2 v_p(\ell)
\end{align*}
The inequality \eqref{E:failure of convexity of Delta'} follows from combining the above inequalities.
\end{proof}

\begin{notation}
\label{N:linear shift down}
For two points $P = (m, a), Q = (n,b) \in \RR^2$, we write $\overline{PQ}$ for the line segment connecting them.

In the rest of this paper, we will frequently use the following concept. Suppose that we have a list of points $$P_i = (i,A_i) \quad\textrm{with} \quad A_i \in \RR, \textrm{ for } i = m,m+1, \dots, n$$ (or some subset of this list). We may \emph{shift them down relative to a linear function} $y = ax+b$ (for some $a,b \in \RR$), by transforming them into $$Q_i  = (i, A_i - ai-b) \textrm{\quad for \quad}i = m,m+1, \dots, n$$ (or some subset of this list).
This way, for every $i\neq j$, the slope of $\overline {Q_iQ_j}$ is precisely the slope of $\overline{P_iP_j}$ minus $a$. In particular, 
\begin{enumerate}
\item 
if the convex hull of all $P_i$'s is the straight line $\overline{P_mP_n}$, then the convex hull of all $Q_i$'s is the straight line $\overline{Q_mQ_n}$;
\item for a fixed $i_0$, the point $P_{i_0}$ is a vertex of the convex hull of the $P_i$'s if and only if the point $Q_{i_0}$ is a vertex of the convex hull of the $Q_i$'s; and
\item if $i \in [m,n]$, the distance from $P_i$ to the point on $\overline{P_mP_n}$ with $x$-coordinate $i$, is the same as the distance from $Q_i$ to the point on $\overline{Q_mQ_n}$ with $x$-coordinate $i$.
\end{enumerate}
We shall freely use these elementary facts without make reference for the rest of this paper.
\end{notation}

\begin{lemma}
\label{L:Delta - Delta' bound}
Assume $p\geq 7$. For $k =k_\varepsilon  + (p-1)k_{\bullet}$ and $\ell = 1,\dots, \frac 12 d_{k}^\new$, we have
\begin{equation}
\label{E:Delta - Delta' bound}\Delta'_{k, \ell} - 
\Delta_{k, \ell} \leq  3 \big( \ln \ell / \ln p \big)^2.
\end{equation}
Moreover, if we only assume $p\geq 5$, the following statment holds: when $\ell < 2p$ but $\ell \neq p$, $\Delta'_{k, \ell} = \Delta_{k, \ell}$; and when $\ell =p$, $\Delta'_{k, \ell}-\Delta_{k, \ell} \leq 1$.
\end{lemma}
\begin{proof}
If $(\ell, \Delta'_{k,\ell})$ lies on $\underline \Delta_k$, the lemma holds trivially. Now we assume that $(\ell,\Delta'_{k,\ell})$ does not lie on $\underline \Delta_k$. Assume that $\ell_-$ is the largest integer smaller than $\ell$ such that $P_-: =(\ell_-, \Delta'_{k,\ell_-})$ lies on $\underline \Delta_k$, and  $\ell_+$ is the smallest integer larger than $\ell$ such that $P_+: =(\ell_+, \Delta'_{k,\ell_+})$ lies on $\underline \Delta_k$.  Let $\lambda$ denote the slope of the segment $\overline{P_-P_+}$. Then 
\begin{equation}
\label{E:Delta segment inequality}
\Delta'_{k, \ell_+}-\Delta'_{k, \ell_+-1}<  \lambda< \Delta'_{k, \ell_-+1}-\Delta'_{k, \ell_-} .
\end{equation}

By \eqref{E:failure of convexity of Delta'}, for every $\ell' \in [\ell_-+1, \ell_+-1]$, $ \Delta'_{k, \ell'+1} -  \Delta'_{\ell'} \geq  \Delta'_{k, \ell'}- \Delta'_{k, \ell'-1} + 1$ whenever $v_p(\ell')=0$. It is clear from \eqref{E:Delta segment inequality} that there must be some $\ell' \in [\ell_-+1, \ell_+-1]$ whose valuation is strictly positive. Let $\ell_0$ denote the $\ell' \in [\ell_-+1, \ell_+-1]$ with maximal valuation. Write $v_p(\ell_0)=m_0$.

We first note that the maximality of $v_p(\ell_0)$ implies that $v_p(\ell_0+i)=v_p(i)$ for $1\leq i\leq \ell_+-\ell_0-1$ and $v_p(\ell_0-i)=v_p(i)$ for $1\leq i\leq \ell_0-\ell_--1$. Then taking the sum of the first inequality of \eqref{E:failure of convexity of Delta'} over $\ell_-+1,\ell_-+2, \dots, \ell_+-1$ (and also noting that two consecutive $\theta_{\ell}$'s have sum equal to $p+1$), we deduce that
\begin{align*}
&(\Delta'_{k, \ell_+}-\Delta'_{k, \ell_+-1})  -  (\Delta'_{k, \ell_-+1}-\Delta'_{k, \ell_-})
\\
\geq \ & \frac{p-3}2( \ell_+-\ell_--1) - \frac{p-5}2 -2\Big( \sum_{i=1}^{\ell_0-\ell_--1}v_p(i) + m_0 + \sum_{i=1}^{\ell_+-\ell_0-1}v_p(i)\Big)
\\
\geq \ & \frac{p-3}2\big( \ell_+-\ell_--1\big) - \frac{p-5}2 -\frac{2}{p-1}\big((\ell_0-\ell_--1) + (\ell_+-\ell_0-1)\big) -2m_0
\\
\geq\ & \frac{p-4}{2}(\ell_+-\ell_--2)+1-2m_0.
\end{align*}
Here the last inequality uses that $\frac{2}{p-1}\leq \frac 12$ because $p \geq 5$.
Combining with \eqref{E:Delta segment inequality}, we obtain
\begin{equation}
\label{E:plus-minus}
m_0 \geq
\frac{p-4}{4}(\ell_+-\ell_--2) + 1.
\end{equation}
This in particular shows that either $m_0 \geq 2$ (in which case $\ell \geq \ell_- \geq p^{m_0} - 4m_0\geq 2p$) or $m_0=1$ which must be the case when $\ell_+-1=\ell_-+1 = \ell_0 = \ell$.

We claim that $\Delta'_{k, \ell_0} - \Delta_{k, \ell_0} \leq m_0^2$. The case $\ell=\ell_0$ and the last statement of the lemma follow directly from this inequality. First note that, as explained in Notation~\ref{N:linear shift down}\,(3), we may shift all points under consideration by a linear function; this will not affect our discussion (but gives a better visualization). Through such shifting, we may suppose $\Delta'_{k, \ell_0+1}=\Delta'_{k, \ell_0-1}$. Then we can iterate (the second inequality of) \eqref{E:failure of convexity of Delta'} to deduce that, for $1\leq n \leq \ell_+-\ell_0$
\[
\Delta'_{k, \ell_0+ n} - \Delta'_{k, \ell_0+n-1}  \geq -m_0 +n - \tfrac{1}2- 2v_p((n-1)!)\geq -m_0+\tfrac n2.
\]
Symmetrically, for $1\leq n \leq \ell_0-\ell_-$, we have
$$ 
\Delta'_{k, \ell_0+1-n} - \Delta'_{k, \ell_0-n} \leq m_0-\tfrac n2.
$$
Then we can use these to deduce that
\[
\max\{\Delta'_{k, \ell_0} - \Delta'_{k, \ell_-}, \Delta'_{k, \ell_0} - \Delta'_{k, \ell_+}  \}
\leq \max_{n_0\geq 1}\Big( \sum_{n=1}^{n_0} \big(m_0-\tfrac n2\big)\Big) \leq  m_0^2.
\]
This yields $\Delta'_{k, \ell_0} - \Delta_{k, \ell_0} \leq m_0^2$. 

Now we assume $\ell \neq \ell_0$; in particular, $m_0 \geq 2$. Without loss of generality we may assume that $\ell_-<\ell \leq \ell_0$. Once again, we shift all points by a linear function with slope $\lambda$ so that we may assume that $\lambda = 0$ and $\overline{P_-P_+}$ is horizontal. Iterating 
\eqref{E:failure of convexity of Delta'}, for $1\leq n\leq \ell_0-\ell_--1$, we obtain
\begin{align*}
\Delta'_{k, \ell_-+ n+1} - \Delta'_{k, \ell_-+n}  &\geq n- 2\sum_{i=\ell_-+1}^{\ell_-+n}v_p(i)\geq n- 2\sum_{i=\ell_-+1}^{\ell_0-1}v_p(i)=n- 2v_p((\ell_0-\ell_--1)!)
\end{align*}
Using \eqref{E:plus-minus}, we get $$n- 2v_p((\ell_0-\ell_--1)!)\geq n - 2\cdot  \frac{4}{(p-1)(p-4)}(m_0-1) \geq  n-2(m_0-1).$$ We thus deduce that
\[
\Delta'_{k, \ell_0} - \Delta'_{k, \ell}\geq \sum_{n=\ell-\ell_-}^{\ell_0-\ell_--1}(n-2m_0+2)> - 2m_0^2 + m_0.
\]
Since $\lambda=0$, this implies that
\begin{equation}
\label{E:main}
\Delta'_{k, \ell} - \Delta_{k, \ell}<3m_0^2-m_0.
\end{equation}
By \eqref{E:plus-minus} and the assumption $p\geq 7$, we get 
\[
\ell_0-\ell\leq \ell_+-\ell_--2<\frac 43 m_0<2m_0.
\]
As $\ell \geq \ell_0-2m_0\geq p^{m_0}-2m_0$,
\begin{align*}
3\left(m_0^2 - (\frac{\ln (p^{m_0}-2m_0)}{\ln p})^2\right)=&3(m_0+\frac{\ln(p^{m_0}-2m_0)}{\ln p})(m_0-\frac{\ln(p^{m_0}-2m_0)}{\ln p})\\
\leq& 6m_0(m_0-\frac{\ln(p^{m_0}-2m_0)}{\ln p})=6m_0^2\cdot(-\frac{\ln(1-\frac{2m_0}{p^{m_0}})}{m_0\ln p})\\
\leq&6m_0^2\cdot\frac{2\cdot \frac{2m_0}{p^{m_0}}}{m_0\ln p}=\frac{24m_0^2}{p^{m_0}\ln p}\leq m_0 
\end{align*}
holds when $m_0 \geq 2$ and $p \geq 7$ (here we use the inequality $-\ln(1-x)\leq 2x$ when $x\in (0,\frac 12)$), the inequality \eqref{E:Delta - Delta' bound} follows from this and \eqref{E:main}.

\end{proof}

As a consequence of the above estimate, Lemma~\ref{L:estimate Delta kell} and Corollary~\ref{C:finer inequality for Delta} can be restated in the form of $\Delta_{k, \ell}- \Delta'_{k, \ell-1}$.
\begin{corollary}
\label{C:Delta - Delta'}
Let $k=k_\varepsilon + (p-1)k_\bullet $ and $\ell = 1, \dots, \frac 12d_k^\new$ be as above. Let $k' = k_\varepsilon + (p-1)k'_\bullet$ be any weight such that 
$$
d_{k'}^\ur, \textrm{ or } d_{k'}^\Iw - d_{k'}^\ur \textrm{ belongs to } \big(\tfrac 12d_k^\Iw-\ell, \tfrac 12 d_k^\Iw+\ell\big),
$$
or $\tfrac 12 d_{k'}^\Iw$ belongs to $\big[\tfrac 12d_k^\Iw-\ell, \tfrac 12 d_k^\Iw+\ell\big]$
then we have
\begin{equation}
\label{E:Delta - Delta' strong}
\Delta_{k, \ell} - \Delta'_{k,\ell-1} -v_p(w_{k} - w_{k'}) \geq \frac 12 (2\ell-1).
\end{equation}
Moreover, when $p \geq 7$ and $\ell \geq 2$, We have $\Delta_{k, \ell}-\Delta'_{k, \ell-1} -v_p(w_{k} - w_{k'}) \geq \frac 12 (4\ell-1).$
\end{corollary}
\begin{proof}
The general case follows from applying the first statement repeatedly. For the first statement, if $\ell < 2p$ but $\ell \neq p$, $\Delta_{k, \ell} = \Delta'_{k, \ell}$ by Lemma~\ref{L:Delta - Delta' bound}; so \eqref{E:Delta - Delta' strong} and the last statement are proved in Corollary~\ref{C:finer inequality for Delta}.
If $\ell \geq 2p$ or $\ell =p$, \eqref{E:Delta - Delta' strong} also follows from Corollary~\ref{C:finer inequality for Delta} and Lemma~\ref{L:Delta - Delta' bound} by noting that
\begin{equation}
\label{E:Delta - Delta' special}
\frac 12+
\frac{p-1}2(\ell-1) -\Big \lfloor \frac{\ln((p+1) \ell)}{\ln p}\Big\rfloor  \geq \frac 12(2\ell-1) +3 \Big( \frac{\ln \ell }{\ln p}\Big)^2
\end{equation}
holds when $\ell \geq 2p$ or $\ell = p$ if we replace the last term by $1$ (citing the special case in Lemma~\ref{L:Delta - Delta' bound} instead). When $p \geq 7$, \eqref{E:Delta - Delta' special} holds with $2\ell-1$ replaced by $4 \ell -1$.
\end{proof}

The following slight generalization of Corollary~\ref{C:Delta - Delta'} is tailored for our need in the sequel of this series.
\begin{corollary}
\label{C:Delta - Delta' multi}
Assume $p \geq 7$. Take integers $\ell, \ell', \ell'' \in \{0,1, \dots, \frac 12d_k^\new\}$ with $\ell \leq \ell'\leq  \ell''$ and $\ell''\geq 1$.
Let $k' = k_\varepsilon + (p-1)k'_{\bullet}$ be a weight such that
$$
d_{k'}^\ur \textrm{ or } d_{k'}^\Iw - d_{k'}^\ur \textrm{ belongs to } \big[\tfrac 12d_k^\Iw-\ell', \tfrac 12 d_k^\Iw+\ell'\big],
$$
then we have
\begin{equation}
\label{E:Delta l'' -Delta l}
\Delta_{k, \ell''}- \Delta'_{k, \ell} - (\ell''-\ell') \cdot  v_p(w_k - w_{k'}) \geq (\ell'-\ell) \cdot \Big\lfloor \frac{\ln((p+1)\ell'')}{\ln p}+1\Big\rfloor + \frac 12\big( \ell''^2 - \ell^2\big).
\end{equation}
Moreover, if there exists $k'$ such that $v_p(w_{k'}-w_k) \geq \big\lfloor \frac{\ln((p+1)\ell'')}{\ln p}+2\big\rfloor $, then there at most two such $k'$ satisfying $v_p(w_{k'}-w_k) \geq \big\lfloor \frac{\ln((p+1)\ell'')}{\ln p}+2\big\rfloor $ and
$$
d_{k'}^\ur \textrm{ or } d_{k'}^\Iw - d_{k'}^\ur \textrm{ belongs to } \big(\tfrac 12d_k^\Iw-\ell'', \tfrac 12 d_k^\Iw+\ell''\big).
$$ 
Suppose that there are two such $k'$'s, up to swapping $k'_1$ and $k'_2$, we have $d_{k'_1}^\ur, d_{k'_2}^\Iw-d_{k'_2}^\ur \in \big[\frac 12d_k^\Iw-\ell', \frac 12d_k^\Iw+\ell'\big]$; and in $\{d_{k'_1}^\ur,\, d_{k'_2}^\Iw-d_{k'_2}^\ur\}$, one element is $ \geq \frac 12 d_k^\Iw$ and one element is $ \leq \frac 12 d_k^\Iw$.
\end{corollary}
\begin{proof}
We first prove the moreover part. By the proof of Corollary~\ref{C:finer inequality for Delta}, the condition $d_{k'}^\ur \in (\frac 12d_k^\Iw - \ell'', \frac 12d_k^\Iw+\ell'')$ implies that
$$
k'_\bullet - k_\bullet \in \Big( \eta-\frac{p+1}2(\ell''-1), \ \eta+\theta + \frac{p+1}2 (\ell''-1)\Big].
$$
So if one such $k'$ satisfies $v_p(w_{k'}-w_k) \geq \big\lfloor \frac{\ln((p+1)\ell'')}{\ln p}+2\big\rfloor $, there is only such $k'$.  On the other hand, if $d_{k'}^\Iw - d_{k'}^\ur \in (\frac 12d_k^\Iw - \ell'', \frac 12d_k^\Iw+\ell'')$, then there exists $k''$ such that
$$
k''_\bullet - k_\bullet = p(k_\bullet - k'_\bullet) \quad \textrm{and} \quad d_{k''}^\ur \in \big(\tfrac 12d_k^\Iw - \ell'', \tfrac 12d_k^\Iw+\ell''\big).
$$
Moreover, for such $k''$, $d_{k'}^\Iw-d_{k'}^\ur \in (\frac 12d_k^\Iw - \ell'', \frac 12d_k^\Iw)$ if and only if $d_{k''}^\ur \in (\frac 12d_k^\Iw, \frac 12d_k^\Iw + \ell'')$, by \S~\ref{S:proof of ghost compatible with AL}(2). From this, one easily deduce the moreover part of the corollary.

Now, we move to the general case.
If $\ell'' = 1$, \eqref{E:Delta l'' -Delta l} follows from Corollary~\ref{C:Delta - Delta'}.
If $\ell'' \geq 2$, we apply the stronger statement in  Corollary~\ref{C:Delta - Delta'} iteratively to deduce that
\begin{align*}
\Delta_{k, \ell''}-\Delta'_{k,\ell} - (\ell''-\ell')\cdot v_p(w_k-w_{k'}) \geq\ & (\ell''^2-\ell^2) + (\ell'-\ell)
\\
\geq\ & \frac 12\big( \ell''^2 - \ell^2\big)+ (\ell'-\ell) \cdot \Big\lfloor \frac{\ln((p+1)\ell'')}{\ln p}+1\Big\rfloor.
\end{align*}
The last inequality follows from $\frac{\ell''+\ell}2 \geq \big\lfloor \frac{\ln((p+1)\ell'')}{\ln p}\big\rfloor$.
\end{proof}

\begin{definition}
\label{D:near steinberg}
Fix $\varepsilon$ as above and fix $w_\star \in \gothm_{\CC_p}$. For each $k=k_\varepsilon + (p-1)k_\bullet$, we let $L_{w_\star,k}$ denote the largest number (if it exists) in $\{1, \dots, \frac 12 d_{k}^\new\}$ such that 
\begin{equation}
\label{E:maximal L}
v_p(w_\star - w_k) \geq  \Delta_{k,L_{w_\star, k}} - \Delta_{k, L_{w_\star, k}-1}.
\end{equation}
When such $L_{w_\star, k}$ exists (in which case $(L_{w_\star, k}, \Delta_{k, L_{w_\star, k}})$ must be a vertex of $\underline \Delta_k$ of Notation~\ref{N:Delta kell}), we call the open interval $\nS_{w_\star, k}: =\big (\frac 12 d_{k}^\Iw - L_{w_\star,k}, \frac 12 d_{k}^\Iw + L_{w_\star,k}\big)$ the \emph{near-Steinberg range} for  the pair $(w_\star, k)$; we write $\overline \nS_{w_\star, k}: = \big[\frac 12 d_{k}^\Iw - L_{w_\star,k}, \frac 12 d_{k}^\Iw + L_{w_\star,k}\big]$ for the corresponding closed interval.
When no such $L_{w_\star, k}$ exists, $\nS_{w_\star, k} = \overline \nS_{w_\star, k} = \emptyset$.

For a positive integer $n$, we say $(\varepsilon, w_\star, n)$ or simply $(w_\star, n)$ is \emph{near-Steinberg} if $n$ belongs to the near-Steinberg range $\nS_{w_\star, k}$ for some $k$. This is equivalent to the existence of some $k \equiv  k_\varepsilon \bmod{(p-1)}$ such that $n \in (d_{k}^\unr, d_{k}^\Iw - d_{k}^\unr)$ and
$$
v_p(w_\star - w_k) \geq  \Delta_{k, |n-\frac 12 d_k^{\Iw}|+1} - \Delta_{k, |n-\frac 12 d_k^{\Iw}|}.
$$

\end{definition}

\begin{remark}
\label{R:near-Steinberg explained}
Let us illustrate the meaning of the notion near-Steinberg in a more intuitive way: by Proposition~\ref{P:ghost compatible with theta AL and p-stabilization}\,(4) the Newton polygon $\NP(G(w_k,-))$ at the weight point $w_k$ has a long straight-line over $(d_{k}^\unr, d_{k}^\Iw - d_{k}^\unr)$ with Steinberg form slope, namely $\frac{k-2}2$.  If a point $w_\star \in \gothm_{\CC_p}$ is ``close" to $w_k$, the ``continuity" of the Newton polygon predicts that $\NP(G(w_\star, -))$ also has a long straight-line over the near-Steinberg range $\nS_{w_\star, k}$ (whose slope is ``very often" equal to $\frac{k-2}2$ \emph{but need not be so} in general). Condition \eqref{E:maximal L} exactly quantifies the distance between $w_\star$ and $w_k$ versus the length of the straight line. The ghost duality $\Delta_{k,-\ell} = \Delta_{k,\ell}$ manifests itself as the fact that the $x$-coordinates of these straight lines are always symmetric about $\frac 12d_k^\Iw$. See Example~\ref{Ex:near Steinberg} as an illustration of this. 
Lemma~\ref{L:alternative characterization of Lmax} below may be viewed as an alternative characterization of the condition~\eqref{E:maximal L}.

We also point out that, although not very often, it is possible that for a near-Steinberg pair $(w_\star, n)$,  $n$ belongs to the near-Steinberg range for more than one $k$. See the proof of Theorem~\ref{T:near Steinberg = non-vertex}.
\end{remark}

\begin{example}
\label{Ex:near Steinberg}
We illustrate the previous remark with a concrete example.
Continue with the setup of Example~\ref{Ex:p=7a=3}, with $p=7$, $a=2$, and $\epsilon = \omega^2\times 1$ (so $s_\varepsilon =4$, corresponding to the ``anti-ordinary disc"). The first five terms of the ghost series are
\begin{eqnarray*}
g_1(w) &=& w-w_6,\\
g_2(w) &=& (w-w_{12})(w-w_{18})(w-w_{24})(w-w_{30}),\\
g_3(w) & =& (w-w_{18})^2(w-w_{24})^2(w-w_{30})^2(w-w_{36})(w-w_{42})(w-w_{48})(w-w_{54}),\\
g_4(w)&=& (w-w_{18})(w-w_{24})^3(w-w_{30})^3(w-w_{36})^2\cdots (w-w_{54})^2(w-w_{60})\cdots (w-w_{78}),\\
g_5(w)&=& (w-w_{24})^2(w-w_{30})^4(w-w_{36})^3\cdots (w-w_{54})^3(w-w_{60})^2\cdots (w-w_{78})^2\\
&&\cdot\,(w-w_{84})\cdots (w-w_{102}).
\end{eqnarray*}
We are interested in the points $w_\star \in \gothm_{\CC_p}$ that are close to $w_{18}$, for which we have $d_{18}^{\unr} = 1$, $\frac 12d_{18}^\Iw = 3$, and $\frac 12d_{18}^\mathrm{new} = 2$.
The vertices of the polygon $\underline \Delta_{18}$ is given by
\begin{center}\renewcommand{\arraystretch}{1.2}
\begin{tabular}{|c|c|c|c|c|c|}
\hline
& $i = -2$ & $i=-1$ & $i=0$ & $i=1$ & $i=2$
\\
\hline
$\Delta'_{18,i} = \Delta_{18,i}$ & $17$ & $11$ & $8$ & $11$ & $17$\\
\hline
\end{tabular}
\end{center}
In particular, we have $\Delta_{18, 2}- \Delta_{18,1} = 6$ and $\Delta_{18,1}-\Delta_{18,0} = 3$.

In the following table, we record the minimal valuation $v_p(g_i(w_\star))$ among all $w_\star$ such that $v_p(w_\star -w_{18})$ is the given value $r$.\begin{center}\renewcommand{\arraystretch}{1.2}\begin{tabular}{|c|c|c|c|c|c|}
\hline
when $v_p(w_\star -w_{18}) =r$ & $i=1$ & $i=2$ & $i=3$ & $i=4$ & $i=5$\\
\hline
$v_p(g_i(w_\star))$ with $r \in (2,\infty)$ & $1$ & $3+r$ & $8+2r$ & $19+r$ & $33$
\\ \hline
coordinate at $x=i$ of
$\NP(G(w_\star, t))$ when $r\in(6,\infty)$ & $1$ & $9$ & $17$ & $25$ & $33$
\\ \hline  
coordinate at $x=i$ of
$\NP(G(w_\star, t))$ when $r\in(3,6)$ & $1$ & $3+r$ & $11+r$ & $19+r$ & $33$
\\ \hline  
coordinate at $x=i$ of
$\NP(G(w_\star, t))$ when $r\in(2,3)$ & $1$ & $3+r$ & $8+2r$ & $19+r$ & $33$
\\
\hline
\end{tabular}
\end{center}
When $v_p(w_\star -w_{18}) = r > 6$, the Newton polygon  $\NP(G(w_\star, t))$ over $x \in[1,5]$ is a straight line of slope $8 = \frac{18-2}2$.  As $r$ decrease to the range $r \in (3,6)$, $\NP(G(w_\star, t))$ over $x \in[1,5]$ ``breaks" into three segments with slopes $2+r$, $8$, and $14-r$, and width $1$, $2$, and $1$, respectively. As $r$ further decreases to $r \in (2,3)$, $\NP(G(w_\star, t))$ over $x \in[1,5]$ has four segments, each with a distinct slope and width $1$.
The ``phase-changing points" $r =3$ and $r=6$ correspond exactly to $\Delta_{18, 2}-\Delta_{18,1}$ and $\Delta_{18,1} -\Delta_{18,0}$.

Readers might be mislead to think that the middle segment will always have slope $\frac{k-2}2$ before it completely breaks into segments of width $1$. As pointed out already in Remark~\ref{R:near-Steinberg explained}, this may fail when $k$ is large and when the middle segment is ``short".  See the proof of Proposition~\ref{P:shifting points wstar to wk} and Theorem~\ref{T:near Steinberg = non-vertex} for the reasons.
\end{example}

\begin{remark}
\label{R:intuition of near-Steinberg}
We provide additional visualization on near-Steinberg ranges by trying to answer the question in vague terms: for each fixed $n$ (and a fixed $\varepsilon$), what is the set of $w_\star \in \gothm_{\CC_p}$ such that $(w_\star, n)$ is near-Steinberg?

First of all, there is a unique $k$, namely $k_{\midd}(n)$ such that $n = \frac 12 d_k^\Iw$. So if $v_p(w_\star- w_{k_\midd(n)}) \geq \Delta_{k_\midd(n), 1}-\Delta_{k_\midd(n), 0}$, then $(w_\star, n)$ is near-Steinberg. By the proof of Lemma~\ref{L:estimate Delta kell}, the difference $\Delta_{k_\midd(n), 1}-\Delta_{k_\midd(n), 0} \geq \frac{a+2}2$ or $\frac{p-1-a}2$ (depending on the parity of $n$) and this inequality is often an equality (related to the inequality \eqref{E:inequality in difference of delta}). So this type of $w_\star$ is a ball around $w_{k_\midd(n)}$ of radius roughly $p^{-(a+2)/2}$ or $p^{-(p-1-a)/2}$.

Next, we consider $k = k_\midd(n) \pm (p-1)$. Then $(w_\star, n)$ is near-Steinberg for such $k$ if and only if $v_p(w_\star- w_{k}) \geq \Delta_{k, 2}-\Delta_{k, 1}$, which by Lemma~\ref{L:estimate Delta kell}, is always greater than or equal to $\frac{p+1}2$ (with most of the case being equal). Then such $w_\star$ form a ball around each of $w_{k_\midd(n)\pm (p-1)}$ of radius roughly $p^{-(p+1)/2}$.

We may continue this process to see that the near-Steinberg range is roughly a union of these balls with centers at $w_{k_\midd(n)\pm (p-1)i}$ of radius roughly $p^{-(p+1)i/4}$. Of course, the precise radii of these balls are combinatorially complicated.  

We end this remark with the following observations: for $(w_\star, n)$ to be near-Steinberg, we need $v_p(w_\star) \geq 1$. For a fixed $n$, if we only consider classical points  $w_\star = w_{k'}$, then in the majority of the cases, $(w_{k'},n)$ is not near-Steinberg (except for about $p^{-(a+2)/2}$ or $p^{-(p-1-a)/2}$ of cases, and this percentage is roughly constant even if $n$ is large), and for the rest of the cases, most near-Steinberg pairs $(w_\star, n)$ correspond to a Steinberg range of length at least $2$. Long near-Steinberg ranges are rare when considered with a fixed $n$.
\end{remark}

\begin{lemma}
\label{L:alternative characterization of Lmax}
The condition~\eqref{E:maximal L} implies that the lower convex hull of points
\begin{equation}
\label{E:convexity condition for near-Steinberg}
\big(n, \ v_p(g_{n, \hat k}(w_k)) + m_n(k) \cdot v_p(w_\star - w_k)\big) \quad \textrm{for} \quad n \in \overline \nS_{w_\star, k}
\end{equation}
is a straight line with slope $\frac{k-2}2$.
\end{lemma}
\begin{proof}
Writing $n = \frac 12 d_k^\Iw + \ell$, the points \eqref{E:convexity condition for near-Steinberg} maybe rewritten as
$$
\big( \tfrac 12 d_k^\Iw + \ell, \ \Delta'_{k, \ell} + \ell \cdot \tfrac{k-2}2 + (\tfrac 12d_k^\new - |\ell|) \cdot v_p(w_\star-w_k) \big) \quad \textrm{with} \quad \ell \in \big[  - L_{w_\star, k}, L_{w_\star, k}].
$$
Shifting these points down relative to the linear function (in the sense of Notation~\ref{N:linear shift down}) $$y =\tfrac {k-2}2 (x-\tfrac 12d_k^\Iw)+ \Delta_{k, L_{w_\star, k}} + (\tfrac 12d_k^\new - L_{w_\star, k})\cdot v_p(w_\star-w_k),$$ these points become
\begin{equation}
\label{E:list of points Delta' - Delta shifted}
\big( \tfrac 12 d_k^\Iw + \ell, \ \Delta'_{k, \ell} -\Delta_{k, L_{w_\star, k}} + (L_{w_\star,k} - |\ell|) \cdot v_p(w_\star-w_k) \big) \quad \textrm{with} \quad \ell \in \big[  - L_{w_\star, k}, L_{w_\star, k}].
\end{equation}
Now, the condition \eqref{E:maximal L} $v_p(w_\star -w_k) \geq \Delta_{k, L_{w_\star,k}} - \Delta_{k, L_{w_\star,k}-1}$ implies (and in fact equivalent to!) that
$$
(L_{w_\star,k} - |\ell|) \cdot v_p(w_\star -w_k) \geq \Delta_{k, L_{w_\star, k}}-\Delta'_{k,\ell} \quad \textrm{for} \quad \ell \in [-L_{w_\star, k}, L_{w_\star, k}]
$$
and the equality holds when $\ell = \pm L_{w_\star, k}$. So the convex hull of points \eqref{E:list of points Delta' - Delta shifted} is a straight line of slope zero; so the convex hull of points \eqref{E:convexity condition for near-Steinberg} is $\frac{k-2}2$. The lemma is proved.
\end{proof}

Before proceeding, we give a general tool to relate the Newton polygon of ghost series at one point $w_\star$ with that at another ``special" point $w_k$. This will be frequently used in this section.
\begin{proposition}
\label{P:shifting points wstar to wk}
Let $ \nS_{w_\star, k} = \big(\tfrac 12 d_{k}^\Iw - L_{w_\star, k},\, \tfrac 12 d_{k}^\Iw+ L_{ w_\star, k} \big)$ be a  near-Steinberg range.  
\begin{enumerate}
\item For any integer $k' = k_\varepsilon + (p-1)k'_\bullet \neq k $ with 
$v_p(w_{k'}-w_k) \geq  \Delta_{k, L_{ w_\star, k}} - \Delta_{k, L_{ w_\star, k}-1}$,  we have the following exclusions
$$ \tfrac 12 d_{k'}^\Iw\, \notin \overline\nS_{w_\star, k} \quad \textrm{and} \quad 
d_{k'}^\unr,\, d_{k'}^\Iw  - d_{k'}^\unr  \notin \nS_{w_\star, k}.
$$
\item 
For any integer $k' = k_\varepsilon + (p-1)k'_\bullet \neq k $ such that 
$v_p(w_{k'}-w_k) \geq  \Delta_{k, L_{ w_\star, k}} - \Delta_{k, L_{ w_\star, k}-1}$, the ghost multiplicity $m_n(k')$ is linear in $n$ when $n \in \overline \nS_{w_\star, k}$.
\item
The following two lists of points
$$
P_n = \big(n, v_p(g_{n,\hat k}(w_\star))\big), \quad Q_n = \big(n, v_p(g_{n,\hat k}(w_k))\big)  \quad \textrm{ with }n \in \overline\nS_{w_\star,k}
$$
are shifts of each other relative to a linear function with some slope in
\begin{equation}
\label{E:possible shift slopes}
\ZZ + \ZZ\cdot\alpha, \text{~for~} \alpha:= \max \{ v_p(w_\star - w_{k'})| w_{k'} \text{~is a zero of~} g_n(w) \text{~for some~} n\in \nS_{w_\star,k} \} .
\end{equation}

\item
More generally, let $\bfk :=\{k, k_1, \dots, k_r\}$ with $k_i \equiv k_\varepsilon \bmod p-1$ be a set of integers including $k$.
Suppose that there is an interval $[n_-,n_+]$ such that, for any $k' = k_\varepsilon + (p-1)k'_\bullet \notin \bfk$ with $v_p(w_{k'}-w_k) \geq v_p(w_\star - w_k)$, the ghost multiplicity $m_n(k')$ is linear in $n$ when $n \in [n_-,n_+]$. 
Then for any set of constants $A_n$ with $n \in [n_-,n_+]$,
the two lists of points
$$
P_n = \big(n, A_n+v_p(g_{n,\hat {\bfk}}(w_\star))\big), \quad Q_n = \big(n, A_n+v_p(g_{n,\hat {\bfk}}(w_k))\big)  \quad \textrm{ with }n \in [ n_-,n_+]
$$
are shifts of each other by a linear function with some slope in 
\begin{equation}
\label{E:possible shift slopes general case}
\ZZ + \ZZ\cdot\beta, \text{~for~} \beta:= \max \{ v_p(w_\star - w_{k'})| w_{k'} \text{~is a zero of~} g_{n,\hat{\bfk}}(w) \text{~for some~} n\in \nS_{w_\star,k} \}.
\end{equation}
(See Notation~\ref{N:gnhatk} for the definition of $g_{n,\hat {\bfk}}(w)$.)
\end{enumerate}
\end{proposition}
\begin{proof}
(1) Write $L = L_{w_\star, k}$ for simplicity. Suppose that an integer $k' = k_\varepsilon + (p-1)k'_\bullet \neq k$ satisfies either $\frac 12d_{k'}^\Iw \in [\frac 12 d_k^\Iw - L, \, \frac 12 d_k^\Iw+L]$, or $d_{k'}^\ur$ or $d_{k'}^\Iw-d_{k'}^\ur$ belongs to $(\frac 12 d_k^\Iw - L, \, \frac 12 d_k^\Iw+L)$. Then Corollary~\ref{C:Delta - Delta'} implies that
$$
\Delta_{k, L}-\Delta'_{k, L-1} - v_p(w_{k'}-w_k) \geq \tfrac 12(2L-1)  >0.
$$
In particular, $v_p(w_{k'}-w_k) < \Delta_{k,L} - \Delta_{k, L-1}$. So (1) holds.

(2) follows from (1) and  the ghost zero multiplicity pattern in Definition~\ref{D:ghost series}. (In fact, we just need $\frac 12 d_{k'}^\Iw \notin \nS_{w_\star, k}$ instead of $\frac 12 d_{k'}^\Iw \notin \overline\nS_{w_\star, k}$; the stronger version of (1) is needed for later application.) 

We next prove (4). We compare each $v_p(g_{n,\hat {\bfk}}(w_\star))$ with $v_p(g_{n,\hat {\bfk}}(w_k))$ by
\begin{equation}
\label{E:gn(wstar) - gn(wk)}
v_p(g_{n,\hat {\bfk}}(w_\star))  - v_p(g_{n,\hat {\bfk}}(w_k))\ =\hspace{-5pt} \sum_{\substack{k'= k_\varepsilon+(p-1)k'_\bullet \\ k'  \notin \bfk}}\hspace{-5pt} m_n(k')\cdot \big( v_p(w_\star - w_{k'}) - v_p(w_k-w_{k'}  ) \big).
\end{equation}
We consider each term associated to $k' = k_\varepsilon + (p-1)k'_\bullet$, separating into two cases.
\begin{itemize}
\item 
If $v_p(w_{k'} - w_k)  <v_p(w_\star - w_k)$ then  $v_p(w_\star - w_{k'} ) = v_p(w_k - w_{k'})$; thus the corresponding term in \eqref{E:gn(wstar) - gn(wk)} has zero contribution to the sum.
\item
If $v_p(w_{k'} - w_k)  \geq v_p(w_\star - w_k)$, the assumption in (4) implies that $m_n(k')$ is a linear function in $n$, so the contribution to the sum of this term is linear in $n$.  Moreover, we note that $v_p(w_\star - w_{k'}) - v_p(w_k-w_{k'} \in \ZZ+\ZZ\cdot \beta$ by the definition of the number $\beta$, and hence  the slope of this linear term belongs to \eqref{E:possible shift slopes}.
\end{itemize}
Summing up, the right hand side of \eqref{E:gn(wstar) - gn(wk)} is a linear function in $n$; so (4) holds.

Finally,
(3) follows from a special case of (4) when $[n_-, n_+] = \overline\nS_{w_\star,k}$, $\bfk = \{k\}$, and (all $A_n =0$); the condition in (4) is guaranteed by (2) by noting $v_p(w_\star-w_k) \geq  \Delta_{k, L_{ w_\star, k}} - \Delta_{k, L_{ w_\star, k}-1}$.
 The proposition is proved.
\end{proof}

\begin{remark}
Aside from the standard manipulations of Newton polygons, an essential ingredient of the proof of Proposition~\ref{P:shifting points wstar to wk}\,(1) is to use Corollary~\ref{C:finer inequality for Delta}: even though $v_p(w_{k'}-w_k)$ for a ghost zero $w_{k'}$ can be sometimes large, but when that happens, $\Delta'_{k, L} - \Delta'_{k, L-1}$ is also large because it ``contains $v_p(w_{k'}-w_k)$ as a summand".
\end{remark}

The following immediate corollary of Proposition~\ref{P:shifting points wstar to wk} will be the essential ingredient for a proof of that local ghost Conjecture~\ref{Conj:local ghost conjecture} implies Breuil--Buzzard--Emerton Conjecture~\ref{Conj:Breuil-Buzzard-Emerton}.

\begin{corollary}
\label{C:near-Steinberg => straightline}
Let $ \nS_{w_\star, k} = \big(\tfrac 12 d_{k}^\Iw - L_{w_\star, k},\, \tfrac 12 d_{k}^\Iw+ L_{ w_\star, k} \big)$ be a  near-Steinberg range.  
\begin{enumerate}
\item 
Assume that $w_\star$ is not a zero of $g_n(w)$ for any (or equivalently by Proposition~\ref{P:shifting points wstar to wk}\,(2),  some) $n \in \nS_{w_\star, k}$. Then the lower convex hull of points
$$
(n, \ v_p(g_n(w_\star))) \quad \textrm{for} \quad n \in \overline \nS_{w_\star, k}
$$
is a straight line. Moreover, the slope of this straight line belongs to 
\begin{equation}
\label{E:possible slopes}
\tfrac a2+\ZZ + \ZZ\cdot \gamma, \text{~for~} \gamma:= \max\{ v_p(w_\star - w_{k'})| w_{k'} \text{~is a zero of~} g_n(w) \text{~for some~} n\in \nS_{w_\star,k}  \} .
\end{equation}

\item 
Assume that $w_\star = w_{k_1}$ is a zero of $g_n(w)$ for any (or equivalently by Proposition~\ref{P:shifting points wstar to wk}\,(2), some) $n \in \nS_{w_\star, k}$. Then  the lower convex hull of points
$$
(n, \ v_p(g_{n,\hat k_1}(w_{k_1}))) \quad \textrm{for} \quad n \in \overline \nS_{w_\star, k}
$$
is a straight line of slope in $\frac a2 +\ZZ$.
\end{enumerate}
\end{corollary}
\begin{proof}
(1) We can write $v_p(g_n(w_\star)) = v_p(g_{n,\hat k}(w_\star)) + m_n(k) v_p(w_\star -w_k)$ for each $n\in  \overline \nS_{w_\star, k} $.
Applying Proposition~\ref{P:shifting points wstar to wk}\,(3), we are reduced to prove that the lower convex hull of
$$
\big(n, \ v_p(g_{n, \hat k}(w_k)) + m_n(k) \cdot v_p(w_\star - w_k)\big) \quad \textrm{for} \quad n \in \overline \nS_{w_\star, k}
$$
is a straight line with slopes in \eqref{E:possible slopes}. But this is precisely Lemma~\ref{L:alternative characterization of Lmax} by noting $\frac{k-2}2 \equiv \frac{k_\varepsilon-2}2 \equiv \frac{a+2s_\varepsilon }2 \equiv \frac a2 \bmod \ZZ$.

(2) If $k_1=k$, then $\overline \nS_{w_k, k} = [d_k^\ur, d_k^\Iw-d_k^\ur]$ and this is the ghost duality \eqref{E:ghost duality alternative}. When $k_1 \neq k$, this follows from an argument similar to (1): for $\bfk = \{k, k_1\}$, write $$v_p(g_{n,\hat k_1}(w_{k_1}))= v_p(g_{n,\hat \bfk}(w_{k_1})) + m_n(k)v_p(w_k-w_{k_1}).$$ Applying Proposition~\ref{P:shifting points wstar to wk}\,(4) to the case $w_\star = w_{k_1}$, $[n_-,n_+] = \overline \nS_{w_{k_1},k}$ and $\bfk$, we reduce the proof of (2) to proving that the lower convex hull of points
$$
\big(n, \ v_p(g_{n,\hat \bfk}(w_{k})) + m_n(k) v_p(w_{k_1}-w_k) \big)\quad \textrm{for} \quad n \in \overline \nS_{w_\star, k}
$$
is a straight line with slope in  $\frac a2 + \ZZ$. But this list of points is the same as
$$
\big(n, \ v_p(g_{n,\hat k}(w_{k})) + m_n(k)v_p(w_{k_1}-w_k)-m_n(k_1)v_p(w_{k_1}-w_k)  \big)\quad \textrm{for} \quad n \in \overline \nS_{w_\star, k}.
$$
Proposition~\ref{P:shifting points wstar to wk}\,(2) says that $m_n(k_1)$ is linear in $n \in \overline \nS_{w_\star, k}$. So the list above is a linear shift of the list in Lemma~\ref{L:alternative characterization of Lmax} (with $w_\star = w_{k_1}$), and thus its convex hull is a straight line with slope in  $\frac{k-2}2 + \ZZ = \frac a2 + \ZZ$.
\end{proof}

The following theorem, which describes the shapes of $\NP(G(w_\star, -))$, is the main result of this paper.

\begin{theorem}
\label{T:near Steinberg = non-vertex}
Fix a relevant character $\varepsilon$.
Fix $w_\star \in \gothm_{\CC_p}$. 

\begin{enumerate}
\item The set of near-Steinberg ranges $\nS_{w_\star, k}$ for all $k$ is nested, i.e. for any two such open intervals, either they are  disjoint or one is contained in another.

A near-Steinberg range $\nS_{w_\star, k}$ is called \emph{maximal} if it is not contained in other near-Steinberg ranges.

\item 

The $x$-coordinates of the vertices of the Newton polygon $\NP(G(w_\star, -))$ are exactly those integers which do not lie in any $\nS_{w_\star, k}$. Equivalently, for an integer $n \geq 1$, the pair $(n, w_\star)$ is near-Steinberg if and only the point $(n, v_p(g_n(w_\star)))$ is not a vertex of $\NP(G(w_\star, -))$.

\item 
Over the maximal near-Steinberg ranges, the slope of $\NP(G(w_\star, -))$ belongs to
\eqref{E:possible slopes}.
\end{enumerate}

\end{theorem}
\begin{remark}
\phantomsection
\begin{enumerate}
\item In Theorem~\ref{T:near Steinberg = non-vertex}\,(1), it is allowed to have two disjoint open near-Steinberg interval with a common point at their closures.
\item 
Combined with the discussion on near-Steinberg range in Remark~\ref{R:intuition of near-Steinberg}, this Theorem implies that, for a fixed $n$, a ``majority" of classical points $w_{k'}$ satisfies that $(n, g_n(w_{k'})$ is a vertex of $\NP(G(w_{k'},-))$. 
\end{enumerate}
\end{remark}

Before proving the theorem, we first prove the following technical lemma on the comparison of sizes of two near-Steinberg ranges whose closures have a nonempty intersection.

\begin{lemma}\label{L:comparison of two near-Steinberg ranges with a nonempty intersection}
	Suppose that two positive integers $k_i = k_\varepsilon + (p-1)k_{i\bullet}$ for $i=1,2$ satisfy the condition that the closures of the near-Steinberg ranges $\overline \nS_{w_\star, k_i}=[\frac 12 d_{k_i}^\Iw-L_{w_\star,k_i},\frac 12 d_{k_i}^\Iw+L_{w_\star,k_i}]$ for $i=1,2$ have non-trivial intersection. Without loss of generality, we assume that $L_{w_\star, k_1} \geq L_{w_\star, k_2}$. Then we have $L_{w_\star, k_1}  \geq p^{\lceil\frac 12 + \frac{p-1}2(L_{w_\star, k_2}-1) \rceil} - L_{w_\star, k_2} > L_{w_\star, k_2}$ and 
	\begin{equation}
	\label{E:Delta k1>Delta k2}
	\Delta_{k_1, L_{w_\star, k_1}}
	- \Delta_{k_1, L_{w_\star, k_1}-1}>v_p(w_\star - w_{k_2}).
	\end{equation}
\end{lemma}

\begin{proof}
	We can find $n\in \NN$ which is contained in both $\overline \nS_{w_\star, k_i}$'s for $i=1,2$. By Remark~\ref{R:meaning of kmidminmax}, we have
	$$
	k_{\midd\bullet}(n) \in \big[k_{i\bullet} -L_{w_\star, k_i}, k_{i\bullet} +L_{w_\star, k_i}\big] \quad \textrm{for} \quad i = 1,2.
	$$
	It follows that
	\begin{equation}
	\label{E:k1-k2}
	|k_{1\bullet} - k_{2\bullet}| \leq |k_{1\bullet}-k_{\midd\bullet}(n)| + |k_{2\bullet}-k_{\midd\bullet}(n)|\leq L_{w_\star, k_1} + L_{w_\star, k_2}.
	\end{equation}
	On the other hand, condition~\eqref{E:maximal L} and Lemma~\ref{L:estimate Delta kell} imply that
	\begin{align}
	\label{E:vp(wbullet - wki)}
	v_p(w_\star - w_{k_i}) &\geq \Delta_{k_i, L_{w_\star, k_i}} - \Delta_{k_i, L_{w_\star, k_i}-1} =\Delta'_{k_i, L_{w_\star, k_i}} - \Delta_{k_i, L_{w_\star, k_i}-1} \\ &\geq \Delta'_{k_i, L_{w_\star, k_i}} - \Delta'_{k_i, L_{w_\star, k_i}-1} \geq \tfrac 32+ \tfrac {p-1}2 (L_{w_\star, k_i} - 1) .\nonumber
	\end{align}
	
	It then follows that
	\begin{equation}
	\label{E:vp(k1-k2)}
	v_p(k_{1\bullet} - k_{2\bullet}) \geq \tfrac 12+\tfrac{p-1}2\min\big\{L_{w_\star, k_1} - 1, L_{w_\star, k_2} - 1\big\} = \tfrac 12+\tfrac{p-1}2( L_{w_\star, k_2} - 1).
	\end{equation}
	Combining \eqref{E:k1-k2} and \eqref{E:vp(k1-k2)}, we see that $L_{w_\star, k_1} $ is much bigger than $L_{w_\star, k_2}$:
	$$
	L_{w_\star, k_1}  \geq p^{\lceil\frac 12 + \frac{p-1}2(L_{w_\star, k_2}-1) \rceil} - L_{w_\star, k_2} > L_{w_\star, k_2} .
	$$
	This is the first inequality we need to prove. Now we prove (\ref{E:Delta k1>Delta k2}). Indeed, otherwise, we may deduce from (\ref{E:vp(wbullet - wki)}) that 
	\[
	v_p(w_\star - w_{k_i}) \geq \Delta_{k_1, L_{w_\star, k_1}}
	- \Delta_{k_1, L_{w_\star, k_1}-1} \geq \tfrac 32 + \tfrac{p-1}2(L_{w_\star, k_1}-1).
	\] 
	It would then follow that
	$$
	v_p(k_{1\bullet}-k_{2\bullet}) \geq \tfrac 12 + \tfrac{p-1}2(L_{w_\star, k_1}-1) \quad \textrm{but}\quad |k_{1\bullet} - k_{2\bullet}| < 2L_{w_\star, k_1}.
	$$
	This is impossible. So \eqref{E:Delta k1>Delta k2} holds.
\end{proof}

Now we are ready to prove Theorem~\ref{T:near Steinberg = non-vertex}.
\begin{proof}
(1) We need to show that, under the setup and notations of Lemma~\ref{L:comparison of two near-Steinberg ranges with a nonempty intersection}, $\nS_{w_\star, k_2}$ cannot contain a boundary point $\frac 12 d_{k_1}^\Iw - L_{w_\star, k_1}$ or $\frac 12 d_{k_1}^\Iw + L_{w_\star, k_1}$ of the bigger near-Steinberg range $\nS_{w_\star, k_1}$. Suppose otherwise. By \eqref{E:Delta k1>Delta k2}, we see that $v_p(w_\star - w_{k_1}) > v_p(w_{\star} - w_{k_2})$. It follows that $v_p(w_{k_1} - w_{k_2}) = v_p(w_\star - w_{k_2})$. So $\nS_{w_\star, k_2}$ is a near-Steinberg range for $w_{k_1}$ too, i.e. $\nS_{w_{k_1}, k_2} = \nS_{w_\star, k_2}$.  Now apply Corollary~\ref{C:near-Steinberg => straightline}\,(2) to $\nS_{w_{k_1}, k_2}$, we deduce that
the convex hull of points $$(n, \ v_p(g_{n,\hat k_1}(w_{k_1}))) \quad \textrm{for} \quad n \in \overline \nS_{w_{k_1}, k_2}
$$
is a straight line and $w_{k_1}$ is a zero of $g_n(w)$ when $n \in  \nS_{w_{k_1}, k_2}$. After a linear shift, this means that the convex hull of points $(n,\Delta'_{k_1, n})$ for $n \in \overline \nS_{w_{k_1}, k_2}$ is a straight line. In particular, $\underline \Delta_{k_1}$ (with $x$-coordinate shifted to the right by $\frac 12d_{k_1}^\Iw$) does not have a vertex in $\nS_{w_{k_1}, k_2} = \nS_{w_\star, k_2}$.
But by the definition of near-Steinberg range (especially the definition of the integer $L_{w_\ast,k}$ in Definition~\ref{D:near steinberg}), $(\pm L_{w_\star, k_1}, \Delta'_{k_1, \pm L_{w_\star, k_1}})$ are vertices of $\underline \Delta_{k_1}$.
So $\frac 12d_{k_1}^\Iw \pm L_{w_\star, k_1}$ cannot be contained in  $\nS_{w_\star, k_2} = \nS_{w_{k_1}, k_2}$, contradicting our assumption. (1) is proved.

\medskip
(2) For this, we will prove the following list of statements for each $n \geq 1$:
\begin{itemize}
\item [(a)] If $n$ belongs to some near-Steinberg range $\nS_{w_\star, k}$, then $(n, v_p(g_n(w_\star)))$ is not a vertex of $\NP(G(w_\star, -))$.

\item [(b)]
If $n$ is not contained in any closure of the near-Steinberg range $\overline\nS_{w_\star, k}$, then $v_p(g_{n-1}(w_\star)) + v_p(g_{n+1}(w_\star)) > 2v_p(g_n(w_\star))$.

\item [(c)] If $n$
is the boundary point of  precisely one maximal $\overline \nS_{w_\star, k}$, then
\begin{align*}
\textrm{if } n = \tfrac 12 d_{k}^\Iw + L_{w_\star, k}, \quad v_p(g_{n+1}(w_\star)) - v_p(g_n(w_\star) >  \tfrac1{2L_{w_\star, k}}\big(v_p(g_n(w_\star)) - v_p(g_{n-2L_{w_\star, k}}(w_\star))\big);
\\
\textrm{if } n = \tfrac 12 d_{k}^\Iw - L_{w_\star, k}, \quad\tfrac1{2L_{w_\star, k}}\big(v_p(g_{n+2L_{w_\star, k}}(w_\star)) - v_p(g_{n}(w_\star))\big)> v_p(g_{n}(w_\star)) - v_p(g_{n-1}(w_\star)  .
\end{align*}

\item [(d)] If $n$ is contained in the closure of two maximal near-Steinberg range, i.e. $n = \frac 12 d_{k_-}^\Iw + L_{w_\star, k_-} = \frac 12 d_{k_+}^\Iw - L_{w_\star, k_+}$, then
$$
\frac{v_p(g_n(w_\star)) - v_p(g_{n-2L_{w_\star, k_-}}(w_\star))}{2L_{w_\star, k_-}} < \frac{v_p(g_{n+2L_{w_\star, k_+}}(w_\star))-v_p(g_n(w_\star)) }{2L_{w_\star, k_+}}.
$$
\end{itemize}
Statements (b),  (c), and (d) imply that for any three consecutive points $P_1$, $P_2$ and $P_3$ in the set $\Omega:= \{(n,v_p(g_n(w_\star)))| ~n \text{~does not belong to any near-Steinberg range~} \nS_{w_\star,k} \}$, then $P_2$ lies (strictly) below the line segment connecting $P_1$ and $P_3$. It follows that the union of the line segments connecting consecutive points in $\Omega$ is (lower) convex. Combing with statement (a), we prove (2) immediately.

Statement (a) is proved in Corollary~\ref{C:near-Steinberg => straightline}\,(1).
We now prove (b). By assumption, $n$ does not belong to the near-Steinberg range $\nS_{w_\star,k}$ for any $k\equiv k_\varepsilon \bmod p-1$. In particular we can take $k=k_{\midd}(n)$ defined in Lemma-Notation~\ref{L:extremal ks}. The non-near-Steinberg condition says that
$$v_p(w_\star - w_k) < \Delta_{k, 1} - \Delta_{k, 0} \stackrel{\textrm{Lemma~\ref{L:Delta - Delta' bound}}}= \Delta'_{k,1}-\Delta'_{k,0}.
$$
By ghost duality and \eqref{E:increment of ghost series valuation 2}, we deduce that
\begin{equation}
\label{E:vp(wbullet - wk)< Delta -Delta explained}
2v_p(w_\star - w_k) < \Delta'_{k, 1} + \Delta'_{k, -1} - 2\Delta'_{k, 0} =\hspace{-30pt}\sum_{
k_{{\max}\bullet}(n-1) < k'_\bullet \leq k_{{\max}\bullet}(n)} 
\hspace{-30pt} v_p(w_k - w_{k'})
+\hspace{-30pt}
\sum_{
k_{{\min}\bullet}(n-1) \leq k'_\bullet < k_{{\min}\bullet}(n)}
\hspace{-30pt}v_p(w_k - w_{k'}).
\end{equation}
We need to show that
$v_p(g_{n-1}(w_\star)) + v_p(g_{n+1}(w_\star)) > 2v_p(g_n(w_\star))$. By a formula similar to \eqref{E:increment of ghost series valuation 2}, it is equivalent to show that
\begin{equation}
\label{E:non-nS => vertex 1}
\sum_{
k_{{\max}\bullet}(n-1) < k'_\bullet \leq k_{{\max}\bullet}(n)} 
\hspace{-30pt} v_p(w_\star - w_{k'})
+\hspace{-30pt}
\sum_{
k_{{\min}\bullet}(n-1) \leq k'_\bullet < k_{{\min}\bullet}(n)}
\hspace{-30pt}v_p(w_\star - w_{k'})
> 
2 v_p(w_\star - w_{k}).
\end{equation}
Note that $k_{{\max}\bullet}(n) - k_{{\max}\bullet}(n-1)\in \{a+2,p-1-a\}$. So there are at least $3$ terms on the left hand side of \eqref{E:non-nS => vertex 1}. Thus, if $v_p(w_\star - w_k)\leq 1$, the inequality \eqref{E:non-nS => vertex 1} holds trivially.

Now we assume that $v_p(w_\star - w_k) >1$.
If $v_p(w_\star - w_{k'}) \geq  v_p(w_k-w_{k'})$ for every $k'$ appearing in \eqref{E:non-nS => vertex 1}, then \eqref{E:non-nS => vertex 1} would follow from \eqref{E:vp(wbullet - wk)< Delta -Delta explained}. Now we assume that this is not the case, and let $k'$ be in the sum \eqref{E:non-nS => vertex 1} with maximal $v_p(w_k-w_{k'})$ satisfying $v_p(w_k-w_{k'})>v_p(w_\star - w_{k'}) \geq  1$. Note that by the equality (2) in Lemma~\ref{L:three elementary identities}:
\begin{equation}
\label{E:e2in4.20}
k_{{\max}\bullet}(n) - k_\bullet = pk_\bullet - \tilde k_{{\min}\bullet}(n -1)  \quad \textrm{and} \quad k_{{\max}\bullet}(n-1) - k_\bullet = pk_\bullet - \tilde k_{{\min}\bullet}(n),
\end{equation}
the existence of $k''_\bullet \in [k_{{\min}\bullet}(n-1) , k_{{\min}\bullet}(n))$ implies that there is a $k'_\bullet\in (k_{{\max}\bullet}(n-1), k_{{\max}\bullet}(n)]$ such that 
\begin{equation}
\label{E:valuation}
v_p(k'_\bullet - k_\bullet) = v_p(k_\bullet - k''_\bullet) +1.
\end{equation}
By the maximality of $v_p(w_k-w_{k'})$, we deduce that $k'_\bullet$ must live in $(k_{{\max}\bullet}(n-1), k_{{\max}\bullet}(n)]$. Moreover, since $v_p(w_k-w_{k'})>1$, using \eqref{E:e2in4.20} again, we further deduce that there exists $k''_\bullet \in [k_{{\min}\bullet}(n-1) , k_{{\min}\bullet}(n))$ satisfying \eqref{E:valuation}. In particular, we have
$$
v_p(w_\star - w_{k''}) \geq \min\{v_p(w_\star - w_{k}), v_p(w_k - w_{k''}) \} > v_p(w_\star - w_k)-1.
$$
In conclusion, using that the first sum in \eqref{E:non-nS => vertex 1} has at least $3$ terms, we deduce that
\begin{align*}
\textrm{l.h.s. of \eqref{E:non-nS => vertex 1}}\ &\geq \ v_p(w_\star - w_{k'}) + 1 + v_p(w_\star - w_{k''}) \\
 &> \
v_p(w_\star - w_{k})+1 + (v_p(w_\star - w_k)-1) = 2v_p(w_\star - w_{k}).
\end{align*}
This is exactly what we want.

\medskip
We prove (c) and (d) using a uniform argument, with some abuse of notations: 

\begin{notation}\label{Notation:L+ and L-}
	We write $n = \frac 12 d_{k_-}^\Iw + L_- = \frac 12 d_{k_+}^\Iw - L_+$ for $L_\pm= L_{w_\star, k_\pm}$. If $n$ is not the left (resp. right) boundary of a maximal near-Steinberg range, this is understood as $L_+ = \frac 12$ (resp. $L_- = \frac 12$) and $\frac 12d_{k_+}^\Iw = \frac 12 d_k^\Iw +\frac 12$ (resp. $\frac 12d_{k_-}^\Iw = \frac 12 d_k^\Iw -\frac 12$); and we shall not refer to the meaning of $k_+$ (resp. $k_-$) in later discussion. Moreover, in this case, we shall understand $\Delta_{k_+, L_+} - \Delta_{k_+, L_+-1} = 1$ (resp. $\Delta_{k_-, L_-} - \Delta_{k_-, L_--1} = 1$).
\end{notation}

For the rest of the discussion, we assume that $L_+ \geq L_-$, and the other case can be proved similarly. In this case, $\nS_{w_\star, k_+}$ is a genuine near-Steinberg range, and by  \eqref{E:Delta k1>Delta k2} and Lemma~\ref{L:estimate Delta kell}, 
\[
\Delta_{k_+, L_+} - \Delta_{k_+, L_+-1} > \Delta_{k_-, L_-} - \Delta_{k_-, L_--1}.
\]
We need to show that the following three points 
\begin{equation}
\label{E:three points convex}
\big( \tfrac 12d_{k_-}^\Iw - L_-, v_p(g_{n-2L_-}(w_\star))\big), \quad \big( n, v_p(g_{n}(w_\star))\big), \quad \big( \tfrac 12d_{k_+}^\Iw + L_+, v_p(g_{n+2L_+}(w_\star))\big)
\end{equation}
are convex. 

Before proceeding, we need a technical result:

\begin{lemma}\label{L:a technical lemma in the proof of statements (c)(d)}
	Under the above notations, we have:
\begin{enumerate}
	\item when $L_+\geq 2$, there is at most one integer $k_1$ such that
	\begin{itemize}
		\item $k_1 \equiv k_\varepsilon \bmod p-1$,
		\item
		either $n = d_{k_1}^\unr $ or $n = d_{k_1}^\Iw - d_{k_1}^\unr$, or equivalently $k_{1} \in \big( k_{\max}(n-1), k_{\max}(n)\big] \cup \big[ k_{\min}(n-1), k_{\min}(n)\big)$,
		and
		\item
		$v_p(w_{k_1} - w_{k_+}) \geq v_p(w_\star - w_{k_+})$ (and hence $v_p(w_{k_1} - w_{k_+}) \geq \Delta_{k_\pm, L_\pm}-\Delta_{k_\pm,L_\pm -1}$). 
	\end{itemize} 
	\item when $L_+=1$, either the statement in $(1)$ holds, or there exist two integers $k_1$ and $k_1'$ such that
	\begin{itemize}
		\item $k_1 \equiv k_1'\equiv k_\varepsilon \bmod p-1$,
		\item
		 $n = d_{k_1}^\unr $ and $n = d_{k'_1}^\Iw - d_{k'_1}^\unr$, or equivalently $k_{1} \in \big( k_{\max}(n-1), k_{\max}(n)\big] $ and $k_1'\in \big[ k_{\min}(n-1), k_{\min}(n)\big)$,
		and
		\item
		$v_p(w_{k_1} - w_{k_+}) =2\geq v_p(w_\star - w_{k_+})$ and $v_p(w_{k'_1} - w_{k_+}) \geq v_p(w_\star - w_{k_+})$
		\end{itemize} 
		When the latter case happens, we say that we are in `special' case.
\end{enumerate}
\end{lemma}

\begin{proof}[Proof of Lemma~\ref{L:a technical lemma in the proof of statements (c)(d)}]

To see the uniqueness of $k_1$, we first note that there is at most one such $k_1$ in each of the range $\big( k_{\max}(n-1), k_{\max}(n)\big]$ and $ \big[ k_{\min}(n-1), k_{\min}(n)\big)$ because those are intervals of length $< p^2$. (We warn the reader again of our unusual definition of $k_{\min}$ in Lemma-Notation~\ref{L:extremal ks}.) Suppose that there are two $k_1$ and $k'_1$ in the two intervals above, respectively. Then 
\begin{equation}
\label{E:vp of k1-k+}
v_p(k_{1\bullet}-k_{+\bullet}) \geq v_p(w_\star - w_{k_+})-1 \geq \Delta_{k_+, L_+}-\Delta_{k_+, L_+-1}-1 \geq \tfrac 12+ \tfrac {p-1}2(L_+-1),
\end{equation}
where the last inequality is from Lemma~\ref{L:estimate Delta kell}, and the same inequality holds for $v_p(k'_{1\bullet}-k_{+\bullet})$.
On the other hand, using the formulae for $k_{\midd\bullet}$ and $k_{\max\bullet}$ from Lemma-Notation~\ref{L:extremal ks} and that $n = \frac12 d_{k_+}^\Iw-L_+$, we have that
\[
k_{+\bullet}=n+L_++\delta_\varepsilon-1, \qquad k_{1\bullet}= (p+1)\tfrac {n}2 +  \beta_{[n]}-1-s
\]
for some $s \in \{0, \dots, p-\theta\}$,  where $\theta = \beta_{[n-1]}-\beta_{[n]}+\frac{p+1}2$ as given in the proof of Lemma~\ref{L:estimate Delta kell}. 
It follows that
\begin{equation}
\label{E:k1-k+-s}
k_{1\bullet} - k_{+\bullet} - \big(\tfrac{p-1}2k_{+\bullet} - \tfrac{p+1}2 (L_++\delta_\varepsilon)-1\big)  = \tfrac{p+1}2 + \beta_{[n]}-s.
\end{equation}
Similarly, using the formulae for $k_{\midd\bullet}$ and $k_{\min\bullet}$  from Lemma-Notation~\ref{L:extremal ks}, we can write $pk_\bullet'=(p+1)(\frac n2+\delta_{\varepsilon })-\beta_{[n]}-s'$, for some $s'\in \{\theta,\dots, p \}$ and 
get
$$
pk_{+\bullet} -pk'_{1\bullet} -  \big(\tfrac{p-1}2k_{+\bullet} + \tfrac{p+1}2( L_+-\delta_\varepsilon)-1\big) =\tfrac{p+1}2+\beta_{[n]}+s'-p.
$$
So we deduce that
$$
v_p\big( (pk_{+\bullet}-pk'_{1\bullet}) - (k_{1\bullet} - k_{+\bullet})\big)  \geq \tfrac 12 + \tfrac {p-1}2(L_+-1), \quad \textrm{and}
$$
$$
\big|(pk_{+\bullet}-pk'_{1\bullet}) - (k_{1\bullet} - k_{+\bullet})\big| = (p+1)L_+ + s+s'-p.
$$
This is not possible unless $L_+  =1$.

Now we assume $L_+=1$. Notice that $|s+s'-p|\leq p-\theta$. Thus $v_p\big( (pk_{+\bullet}-pk'_{1\bullet}) - (k_{1\bullet} - k_{+\bullet})\big) \leq 1$. This forces $v_p(k_{1\bullet}-k_{+\bullet})= 1$,  and hence $v_p(w_{k_1}-w_{k_+})=2$.
	
\end{proof}
\medskip
We continue our proof of statements (c) and (d). As suggested by Lemma~\ref{L:a technical lemma in the proof of statements (c)(d)}, we divide our discussion into two cases: $L_+\geq 2$ and $L_+=1$. 

Assume $L_+\geq 2$ first.
Set $\bfk = \{k_+,k_1\}$ (and keep in mind that it is possible that $k_+=k_1$).
We would like to apply Proposition~\ref{P:shifting points wstar to wk}\,(4) to  $k = k_+$, the above $\bfk$, $w_\star$, and $[n_-, n_+] = [n-2L_-, n+2L_+]$. For this, we need to check that for each $k' = k_\varepsilon + (p-1)k'_\bullet \notin \bfk$ with $v_p(w_{k'}-w_{k_+}) \geq v_p(w_{k_+}-w_\star)$, the numbers $d_{k'}^\unr, d_{k'}^{\Iw}-d_{k'}^{\unr}, \frac 12d_{k'}^\Iw$ do not belong to $(n-2L_-, n+2L_+)$. 
\begin{itemize}
\item 
By Proposition~\ref{P:shifting points wstar to wk}\,(1), they do not belong to $(n, n+2L_+)$ and $\frac 12d_{k'}^\Iw \neq n$.
\item
$d_{k'}^\unr, d_{k'}^{\Iw}-d_{k'}^{\unr}$ are not equal to $n$ by our choice of $\bfk$ above.
\item
When $\nS_{w_\star, L_-}$ is a genuine near-Steinberg range, $\Delta_{k_+, L_+} - \Delta_{k_+, L_+-1} > \Delta_{k_-, L_-} - \Delta_{k_-, L_--1}$ allows us to apply Proposition~\ref{P:shifting points wstar to wk}\,(1) to conclude that $d_{k'}^\unr, d_{k'}^{\Iw}-d_{k'}^{\unr}, \frac 12d_{k'}^\Iw$ do not belong to $(n-2L_-, n)$. 
(When $L_- = \frac 12$, no more discussion is needed.)
\end{itemize}

To sum up, we can apply Proposition~\ref{P:shifting points wstar to wk}\,(4)  to deduce that the three points in \eqref{E:three points convex} is a linear 
shift of the following three points:
$$
\Big( j,\ v_p(g_{j, \hat{\bfk}}(w_{k_+})) + \sum_{k \in \bfk} m_j(k) v_p(w_\star-w_k)\Big)\quad \textrm{with} \quad j= n-2L_-,\,n,\, n+2L_+.
$$
These points are the same as
\begin{equation}
\label{E:vertex three points after shift}\Big( j,\ v_p(g_{j, \hat{ k_+}}(w_{k_+})) + m_j(k_+)v_p(w_\star-w_{k_+}) - m_j(k_1)\big(v_p(w_{k_+}-w_{k_1})- v_p(w_\star-w_{k_1})  \big)\Big).
\end{equation}
with $j = n-2L_-, n, n+2L_+$, where the last term appears only when $k_1$ exists and is not equal to $k_+$ (note that we flip the sign before the sum over $\bfk$ so that the summands are positive). we write $P_-$, $P$, $P_+$ for these three points, respectively.
We separate the discussion into the following three possibilities:

(i)
If $k_1=k_+$, that is, $d_{k_+}^\unr = n$, then $m_j(k_+) = 0$ for $j = n-2L_-, n, n+2L_+$.
Proposition~\ref{P:ghost compatible with theta AL and p-stabilization} (applied to the slope of $\overline{PP_+}$) and Proposition~\ref{P:gouvea k-1/p+1 conjecture} (applied to the slope of $\overline{P_-P}$) together imply that
\[
\textrm{slope of }\overline{PP_+}  - \textrm{slope of }\overline{P_-P} \geq \frac {k_+-2}{2} -  \frac{k_+-1}{p+1}>0.
\]
This last inequality needs $k_+ \geq 3$; but this is okay because when $k_+ = 2 $ or $ 2\frac 12$ (by our convention made in Notation~\ref{Notation:L+ and L-}), $d_{k_+}^\Iw$ is at most $1$ and we cannot be in the case of (c) and (d).
The points $P_-$, $P$, and $P_+$ are clearly convex. Thus (c) and (d) are proved in this case.

(ii) If there is no $k_1$, then obviously $d_{k_+}^\unr\neq n$. Using $\Delta_{k_+, L_+} - \Delta_{k_+, L_+-1} \geq \Delta_{k_-, L_-} - \Delta_{k_-, L_--1}$, we deduce that $d_{k_+}^\unr$ does not belong to $(n-2L_-, n)$. Thus $[n-2L_-, n]\subseteq [d_{k_+}^\unr, \tfrac 12 d_{k_+}^\Iw)$; so $m_j(k_+)$ is linear in $j \in [n-2L_-,n]$. By the ghost duality, the slope of $\overline{P_-P}$ is given by
\begin{align*}
&\frac{k_+-2}2  -\frac{ \Delta'_{k_+, L_++2L_-} - \Delta'_{k_+, L_+}}{2L_-} + v_p(w_\star- w_{k_+}) 
\\
\leq\ & \frac{k_+-2}2 - \big(\Delta_{k_+, L_++1}- \Delta_{k_+, L_+}\big) + v_p(w_\star - w_{k_+}) < \frac{k_+-2}2,
\end{align*}
where the first inequality uses the fact that $(L_+, \Delta'_{k_+, L_+})$ is a vertex in $\underline \Delta_{k_+}$. On the other hand, the slope of $\overline{PP_+}$ is simply $\frac{k_+-2}2$. Then (c) and (d) hold in this case too.

(iii) If $k_1$ exists and $k_1 \neq k_+$, then by the uniqueness of $k_1$, we get $d_{k_+}^\unr\neq n$. As in (ii), we first have $[n-2L_-, n]\subseteq [d_{k_+}^\unr, \tfrac 12 d_{k_+}^\Iw)$; so $m_j(k_+)$ is linear in $j \in [n-2L_-,n]$. 
Regardless of whether  $n=d_{k_1}^{\unr}$ or $d_{k_1}^\Iw-d_{k_1}^\ur$,  by \eqref{E:vertex three points after shift},
\begin{align*}
&\textrm{slope of }\overline{PP_+}  - \textrm{slope of }\overline{P_-P}
\\
=\ & \frac{k_+-2}2 - \Big( \frac{k_+-2}2  -\frac{\Delta'_{k_+, L_++2L_-} - \Delta'_{k_+, L_+}}{2L_-} +v_p(w_\star-w_{k_+})\Big) - \big(v_p(w_{k_+}-w_{k_1})- v_p(w_\star-w_{k_1})  \big)
\\
\geq \ &\frac{\Delta'_{k_+, L_++2L_-} - \Delta'_{k_+, L_+}}{2L_-} - v_p(w_{k_+}-w_{k_1}),
\end{align*}
where the last inequality uses the fact that $$v_p(w_{k_1} - w_{k_+}) \geq  v_p(w_\star-w_{k_+})\quad \Rightarrow \quad v_p(w_\star - w_{k_1}) \geq  v_p(w_\star-w_{k_+}).$$

Applying Corollary~\ref{C:finer inequality for Delta} to $k = k_+$, $k'=k_1$ $\ell = L_++1, \dots, L_++2L_-$, and noting that $d_{k_1}^\ur$ or $d_{k_1}^\Iw-d_{k_1}^\ur $ belong to $(\frac 12d_{k_+}^\Iw - \ell, \frac 12d_{k_+}^\Iw+\ell)$ for these $\ell$, we then have
$$
\Delta'_{k_+, \ell} - \Delta'_{k_+, \ell-1} \geq \tfrac 12(2\ell-1) + v_p(w_{k_{+}} - w_{k_1})>v_p(w_{k_{+}} - w_{k_1}).
$$
Summing this up over $\ell = L_++1, \dots, L_++2L_-$, we obtain
$$
\Delta'_{k_+, L_++2L_-} - \Delta'_{k_+, L_+} > 2L_- \cdot v_p(w_{k_{+}} - w_{k_1}).
$$
This proves (c) and (d) in the case $L_+\geq 2$.

Now we treat the case for $L_+=1$, and hence $L_-=\frac 12$. We can assume that we are in the `special case'. Replace the set $\bfk=\{k_+,k_1 \}$ by $\bfk=\{k_+,k_1,k_1' \}$, and a similar argument as above reduces the proof to show that the following three points:
\begin{equation}
\Big( j,\ v_p(g_{j, \hat{ k_+}}(w_{k_+})) + m_j(k_+)v_p(w_\star-w_{k_+}) - \sum _{k=k_1,k_1'}m_j(k)\big(v_p(w_{k_+}-w_{k})- v_p(w_\star-w_{k})  \big)\Big).
\end{equation}
with $j=n-1,n,n+2$, are convex. We denote these three points by $P_-,P,P_+$ as before. By a similar argument as in (iii), we have 
\begin{align*}
&\textrm{slope of }\overline{PP_+}  - \textrm{slope of }\overline{P_-P}
\\
=\ & \frac{k_+-2}2 - \Big( \frac{k_+-2}2  -(\Delta'_{k_+, 2} - \Delta'_{k_+, 1}) +v_p(w_\star-w_{k_+})\Big) -
\sum_{k=k_1,k_1'} \big(v_p(w_{k_+}-w_{k})- v_p(w_\star-w_{k})  \big)
\\
\geq \ &(\Delta'_{k_+, 2} - \Delta'_{k_+, 1})- v_p(w_{k_+}-w_{k'_1})-\Big(\big(v_p(w_{k_+}-w_{k_1})- v_p(w_\star-w_{k_1})  \big)\Big).
\end{align*}
Recall $v_p(w_{k_1}-w_{k_+})=2$ and $v_p(w_\star-w_{k_1})\geq 1$, so we have 
\[
\textrm{slope of }\overline{PP_+}  - \textrm{slope of }\overline{P_-P}\geq (\Delta'_{k_+, 2} - \Delta'_{k_+, 1}) - v_p(w_{k_+}-w_{k'_1})-1
\]

Since $d_{k_1'}^\Iw-d_{k_1'}^\ur=n\in (\frac 12 d_{k_+}^\Iw-2,\frac 12 d_{k_+}^\Iw+2)$, we can apply Corollary~\ref{C:finer inequality for Delta} to $k=k_+$, $k'=k_1'$ and $l=2$, and get 
\[
\Delta'_{k_+, 2} - \Delta'_{k_+, 1}- v_p(w_{k_+}-w_{k'_1})\geq \frac 32>1.
\]
This completes the proof of (c) and (d), and part (2) of the theorem as well.

\medskip
(3) By (2), over a maximal near-Steinberg range $\overline\nS_{w_\star, k}$, the slopes are given by the slopes of the segment connecting
$$
\big( \tfrac 12d_{k}^\Iw- L_{w_\star, k},\ v_p(g_{\frac 12d_{k}^\Iw- L_{w_\star, k}}(w_\star)) \big) \quad \textrm{and} \quad \big( \tfrac 12d_{k}^\Iw+ L_{w_\star, k},\ v_p( g_{\frac 12d_{k}^\Iw+ L_{w_\star, k}}(w_\star)) \big).
$$
Then (3) follows from Corollary~\ref{C:near-Steinberg => straightline}\,(1).
\end{proof}

This proposition above implies the following, which is in close relation with Breuil--Buzzard--Emerton's Conjecture~\ref{Conj:Breuil-Buzzard-Emerton}.
\begin{corollary}
\label{C:integrality of slopes of G}
Fix $\varepsilon$ a relevant character. Then for any $k \equiv k_\varepsilon  \bmod p-1$, 
\begin{enumerate}
\item the slopes of $\NP(G(w_k, - ))$ with multiplicity one are all integers, and
\item 
other slopes of $\NP(G(w_k, - ))$ always have even multiplicity and they belong to $\frac a2 +\ZZ$.
\end{enumerate}
\end{corollary}
\begin{proof}
Note that $v_p(g_n(w_k))$ is always an integer by definition. So a slope with multiplicity $1$ must be an integer. (2) follows from Theorem~\ref{T:near Steinberg = non-vertex}\,(3) above.
\end{proof}

\begin{example}
\label{Ex:pathological Steinberg range}
We give a example of a situation where near-Steinberg ranges are nested. 
A pathological situation could be when there exists a sequence of distinct positive integers $k_{1\bullet},\dots, k_{r\bullet}$ such that $k_{{i+1}\bullet} \equiv k_{i\bullet} \bmod{p^{pk_{i\bullet}}}$. Set $k_i: =k_\varepsilon + (p-1)k_{i\bullet}$. If $n = \frac 12 d_{k_r}^\Iw-1$ and $v_p(w_\star - w_{k_r}) \geq k_{r-1\bullet}$, then $(n,w_\star, k_i)$ is near-Steinberg for all $i$.
\end{example}

\begin{proposition}
\label{P:vertices of Deltak}
Fix $\varepsilon$ a relevant character and $k_0 \equiv k_\varepsilon \bmod p-1$.
The following statements are equivalent for $\ell \in\{0, \dots, \frac12d_{k_0}^\new-1\}$.
\begin{enumerate}
\item 
The point $(\ell, \Delta'_{{k_0},\ell})$ is not a vertex of $\underline \Delta_{k_0}$,
\item $(\frac 12d_{k_0}^\Iw + \ell, w_{k_0}, k_1)$ is near-Steinberg for some  $k_1 > {k_0}$, and
\item 
$(\frac 12d_{k_0}^\Iw - \ell, w_{k_0}, k_2)$ is near-Steinberg for some  $k_2 < {k_0}$.
\end{enumerate}
Moreover, the slopes of $\underline \Delta_{k_0}$ with multiplicity one belong to $\ZZ$. Other slopes all have even multiplicity and the slopes belong to $\frac a2 + \ZZ$.
\end{proposition}
\begin{proof}
The implication (2)$\,\Rightarrow\,$(1) and (3)$\,\Rightarrow\,$(1) is proved in Corollary~\ref{C:near-Steinberg => straightline}\,(2). We will prove (1)$\,\Rightarrow\,$(2) and the proof of (1)$\,\Rightarrow\,$(3) is similar.
We need to show that, for each integer $\ell$, if $\big(\frac 12d_{k_0}^\Iw + \ell, w_{k_0}, k_1\big)$ is not  near-Steinberg for any $k_1 > {k_0}$, then $(\ell, \Delta'_{{k_0}, \ell})$ is a vertex. Let $2L_+$ (resp. $2L_-$) denote the largest size of near-Steinberg range $\nS_{w_k, k_+}$ (resp. $\nS_{w_k, k_-}$) such that $\frac 12d_{k_0}^\Iw+\ell = \frac 12d_{k_+}^\Iw-L_+$ (resp. $\frac 12d_k^\Iw+\ell = \frac 12d_{k_-}^\Iw+L_-$). If such near-Steinberg range does not exist, we put $L_+ = \frac 12$ (resp. $L_- = \frac 12$).
Then we need to show that
\begin{equation}
\label{E:Delta' convex points}
\frac{v_p(\Delta'_{k_0, \ell - 2L_-}) - v_p(\Delta'_{k_0, \ell})}{2L_-} \;<\; \frac{v_p(\Delta'_{k_0, \ell + 2L_+})-v_p(\Delta'_{k_0, \ell}) }{2L_+}.
\end{equation}
We shall indicate how to modify the proof of (b)(c)(d) of Theorem~\ref{T:near Steinberg = non-vertex}\,(2), to prove the corresponding situation in our case.
Indeed, when $L_+ = L_-=\frac 12$, we may prove \eqref{E:Delta' convex points} using the exactly same argument as in the proof of (a) of Theorem~\ref{T:near Steinberg = non-vertex}\,(2) for $w_\star = w_{k_0}$ and $n = \frac 12d_{k_0}^\Iw+\ell$, except that $g_j(w_\star)$ is replaced by $\Delta'_{k_0, j - \frac 12 d_{k_0}^\Iw}$ for $\big|j -\big(\frac 12 d_{k_0}^\Iw+\ell\big)\big| \leq 1$ and the role of $w_\star$ is played by $w_{k_0}$.

Now, we assume that at least one of $L_+$ and $L_-$ is greater or equal to $1$. Without loss of generality, we assume that $L_+ \geq L_-$ and the other case can be treated similarly. Then we can argue in a same way as in the proof of (c) and (d) of Theorem~\ref{T:near Steinberg = non-vertex}\,(2), with $w_\star = w_{k_0}$, $n = \frac 12d_{k_0}^\Iw+ \ell$, and $g_j(w_\star)$ replaced by $g_{j, \hat k_0}(w_{k_0})$. After we apply Proposition~\ref{P:shifting points wstar to wk}\,(4), the points $P_-$, $P$, $P_+$ in \eqref{E:vertex three points after shift} should become
\begin{equation}
\label{E:vertex three points after shift Delta} \Big( j,\ v_p(g_{j, \hat{ k_+}}(w_{k_+})) + (m_j(k_+)- m_j(k_0))v_p(w_{k_0}-w_{k_+}) - m_j(k_1)\big(v_p(w_{k_+}-w_{k_1})- v_p(w_{k_0}-w_{k_1})  \big)\Big).
\end{equation}
But $m_j(k_0)$ is linear in $k_0$; so we may remove this term, and completely reduce to the proof of (c) and (d) therein.
\end{proof}

\appendix

\section{Recollection of representation theory of $\FF[\GL_2(\FF_p)]$}

In this appendix, we recall some basic facts on (right) $\FF$-representations of $\GL_2(\FF_p)$. Our main references are \cite{breuil} and \cite{paskunas} which work with left representations. Our convention will be the corresponding \emph{right representation given by transpose}.

\subsection{Setup}
Throughout this appendix, we fix an \emph{odd} prime number $p$. We set
$$\bar \rmG = \GL_2(\FF_p), \quad  \bar \rmU = \MATRIX{1}{\FF_p}{0}{1}, \quad \textrm{and }\ \bar  \rmB = \MATRIX{\FF_p^\times}{\FF_p}0{\FF_p^\times}.
$$
$$\bar \rmT = \MATRIX {\FF_p^\times} 0 0 {\FF_p^\times}, \quad
\bar \rmU^- = \MATRIX{1}{0}{\FF_p}{1}, \quad \textrm{and }\ \bar  \rmB^- = \MATRIX{\FF_p^\times}0{\FF_p}{\FF_p^\times}.
$$ The coefficient field will be a finite extension $\FF$ of $\FF_p$.
Write $\omega: \FF_p^\times \to \FF^\times$  for the character induced by inclusion.
For a character $\chi_{a,b} :=\omega^a \times \omega^b$ and an $\FF$-representation $V$ of $\bar \rmG$, we write $V(\chi)$ for the subspace of $V$ on which $\bar \rmT$ acts by $\chi$.

Let $\bar s=\Matrix 0110$ represent the nontrivial element in the Weyl group of $\bar \rmG$.  If $\chi$ is a character of $\bar \rmT$, we use
$\chi^s$ to denote the character defined by
$$
\chi^s(\bar t)=\chi(\bar s^{-1}\bar t\bar s),  \quad \textrm{for all }\bar t\in \bar \rmT.
$$

For a pair of non-negative integers $(a,b)$, we use $\sigma_{a,b}$ to denote the \emph{right} $\FF$-representation $\Sym^a\FF^{\oplus 2} \otimes \det^b$ of $\bar \rmG$ (in particular, $b$ is taken modulo $p-1$). The right action is the transpose of the left action in \cite{breuil, paskunas}; explicitly, it is, realized on the space of homogeneous polynomials of degree $a$:
$$ h(X, Y) \big|_{\Matrix {\bar \alpha}{\bar \beta}{\bar \gamma}{\bar \delta}} : =  h(\bar \alpha X + \bar \beta Y, \bar \gamma X + \bar \delta Y) \quad \textrm{for } \Matrix {\bar \alpha}{\bar \beta}{\bar \gamma}{\bar \delta} \in \bar \rmG.$$ 
When $a \in \{0, \dots, p-1\}$ and $b \in \{0, \dots, p-2\}$, $\sigma_{a,b}$ is irreducible. These representations exhaust all irreducible (right) $\FF$-representations of $\GL_2(\FF_p)$, and are called \emph{Serre weights}.
We use $\Proj_{a,b}$ to denote the projective envelope of $\sigma_{a,b}$ as a (right) $\FF[\GL_2(\FF_p)]$-module.

\begin{lemma}
\label{L:projective envelope}
The projective envelope $\Proj_{a,b}$ can be described explicitly as follows.
\begin{enumerate}
\item
For any $b$, $\Proj_{0,b}$ is of dimension $p$ over $\FF$, and admits a socle filtration
\[
\begin{tikzcd}
\Proj_{0,b}: = \sigma_{0,b} \ar[r, -]& \sigma_{p-3, b+1} \ar[r, -] &\sigma_{0,b},
\end{tikzcd}
\]
where the $\sigma_{0,b}$ on the left represents the socle, the $\sigma_{0,b}$ on the right represents the cosocle, each line represents a (unique up to isomorphism) nontrivial extension between the two Jordan--H\"older factors.

\item
When $1 \leq a \leq p-2$, $\Proj_{a,b}$ is of dimensional $2p$ over $\FF$, and admits a socle filtration 
\[
\begin{tikzcd}[row sep = 0pt]
& \sigma_{p-1-a,a+b} \ar[rd,-] \ar[dd, phantom, "\oplus" description]
\\
\Proj_{a,b}: = \sigma_{a,b} \ar[ru,-] \ar[rd,-] && \sigma_{a,b}, \\
& \sigma_{p-3-a,a+b+1} \ar[ru,-]
\end{tikzcd}
\]
where the $\sigma_{a,b}$ on the left represents the socle, the $\sigma_{a,b}$ on the right represents the cosocle, each line represents a (unique up to isomorphism) nontrivial extension between the two factors. (Here when $a=p-2$, the term $\sigma_{p-3-a, a+b+1}$ is interpreted as $0$.)

\item
For any $b$, the Serre weight $\sigma_{p-1,b}$ is projective, i.e. $\Proj_{p-1,b} = \sigma_{p-1,b}$. In particular, $\dim_\FF\Proj_{p-1,b} = p$.
\end{enumerate}

Moreover, in either case $\Proj_{a,b}$ is also the injective envelope of $\sigma_{a,b}$ as an $\FF$-representation of $\GL_2(\FF_p)$.
\end{lemma}
\begin{proof}
The dimensions of the projective envelopes follows from \cite[Lemma~4.1.7]{paskunas}. The socle filtrations of the projective envelopes in $(1)$ and $(2)$ follows from \cite[Lemma~3.4, Lemma~3.5]{breuil-paskunas}. Part $(3)$ follows from \cite[Corollary~4.2.22]{paskunas}.
\end{proof}

\subsection{Induced representations}

Consider the following character $\eta$ of $\bar \rmB$
$$
\eta: \bar \rmB\rightarrow \FF^\times,\qquad  \Matrix {\bar \alpha}{\bar \beta}0{\bar \delta} \mapsto \bar \alpha.
$$
For $k \in \ZZ$, we have an induced representation
$$ \Ind_{\bar \rmB}^{\bar \rmG}(\eta^k): = \{f: \bar \rmG \to \FF\; |\; f(\bar b \bar g) = \eta^k(\bar b)f(\bar g), \textrm{ for all } \bar b \in \bar \rmB\},$$
which we equip with the \emph{right action} given by
$$ f|_{\bar h}(\bar g): = f(\bar g \bar h^\rmT), \quad \textrm{for } \bar g, \bar h \in \bar \rmG.$$
This is the transpose of the usual left action.

\begin{lemma}\fakephantomsection
\label{L:exact sequence of induced representation}
\begin{enumerate}
\item 

For any integer $1\leq k\leq p-1$, we have an exact sequence of
$\bar \rmG$-representations:
\begin{equation}
\label{E:exact sequence of induced representation}
0\rightarrow \sigma_{k,0}\rightarrow
\Ind_{\bar \rmB}^{\bar \rmG}\eta^{k}\rightarrow \sigma_{p-1-k, k}\rightarrow 0.\end{equation}
The above exact sequence splits if and only if $k=p-1$.

\item 
Taking $\bar U$-invariants of \eqref{E:exact sequence of induced representation} gives an exact sequence of $\bar \rmU$-invariants.
\[
0\rightarrow \sigma_{k,0}^{\bar \rmU}\rightarrow
\big(\Ind_{\bar \rmB}^{\bar \rmG}\eta^{k}\big)^{\bar \rmU}\rightarrow \sigma_{p-1-k, k}^{\bar \rmU}\rightarrow 0.
\]
\item 

In the description of $\Proj_{a,b}$ of Lemma~\ref{L:projective envelope}\,(2), the extensions 
$$
\begin{tikzcd}
\sigma_{a,b}\ar[r,-] & \sigma_{p-1-a,a+b} & \textrm{and} &  \sigma_{p-1-a,a+b}\ar[r,-] &\sigma_{a,b}
\end{tikzcd}
$$
are twists of induced representations $\Ind_{\bar \rmB}^{\bar \rmG}\eta^{a} \otimes \det ^b$ and $\Ind_{\bar \rmB}^{\bar \rmG}\eta^{p-1-a} \otimes \det^{a+b}$, respectively.
\end{enumerate}

\end{lemma}
\begin{proof}
$(1)$ follows from \cite[Proposition~3.2.2,Lemma~4.1.3]{paskunas}. $(2)$ follows from \cite[Lemma~3.1.11]{paskunas}. $(3)$ follows from \cite[Lemma~4.1.1]{paskunas}.
\end{proof}

\begin{lemma}
\label{L:U invariants of Proj}
When $1 \leq a\leq p-2$, 
in terms of the socle filtration in Lemma~\ref{L:projective envelope}\,(2), the  $\bar \rmU$-invariant subspace of $\Proj_{a,b}$ is a $2$-dimensional subspace given by
\[
(\Proj_{a,b})^{\bar \rmU} \cong \sigma_{a,b}^{\bar \rmU} \oplus  \sigma_{p-1-a, a+b}^{\bar \rmU},
\]
where the first term is the $\sigma_{a,b}$ in the socle.
In particular, the actions of $\bar \rmT$ on these two $1$-dimensional subspaces are given by the characters
\[
\omega^{b} \times \omega^{a+b} \qquad \textrm{and} \qquad \omega^{a+b} \times \omega^{b}, \quad \textrm{ respectively}.
\]
\end{lemma}
\begin{proof}
Since $\Proj_{a,b}$ is a free module over $\FF[\bar \rmU]$, we know that $\dim (\Proj_{a,b})^{\bar \rmU} = 2$.  But $\Ind_{\bar \rmB}^{\bar \rmG}\eta^a \otimes \det^b$ is a subrepresentation of $\Proj_{a,b}$ by Lemma~\ref{L:exact sequence of induced representation}\,(3) and has  already $2$-dimensional  $\bar\rmU$-invariants by Lemma~\ref{L:exact sequence of induced representation}\,(2).
So we have
$$
(\Proj_{a,b})^{\bar \rmU} \cong \big(\Ind_{\bar \rmB}^{\bar \rmG}\eta^{a}\big)^{\bar \rmU} \otimes \det ^b \cong  \sigma_{a,b}^{\bar \rmU} \oplus  \sigma_{p-1-a, a+b}^{\bar \rmU}.\qedhere
$$
\end{proof}

\begin{lemma}
\label{L:exact sequence of (p+1)-steps jump}
For any integer $k\geq p+1$, we have an exact sequence of
$\bar \rmG$-representations:
$$
0\rightarrow \sigma_{k-(p+1),1}\xrightarrow{i}
\sigma_{k,0}\xrightarrow{\pi} \Ind_{\bar \rmB}^{\bar \rmG}\eta^{k}\rightarrow
0.
$$
\end{lemma}
\begin{proof}
This is  \cite[Proposition~2.4]{reduzzi}. For the convenience of the readers, we recall the definition of the two maps.
The map $i$ is given by multiplication by the homogeneous polynomial $
f_p(X,Y)=X^pY-XY^p$ of degree $p+1$. It is straightforward to check that $i$
is $\bar \rmG$-equivariant and injective.
The map $\pi$ comes from adjunction
\[
\Hom_{\bar \rmG}(\sigma_{k,0}, \Ind_{\bar \rmB}^{\bar \rmG}\eta^k) \cong \Hom_{\bar \rmB^-}(\sigma_{k,0}, \eta^k).
\] (Note that both $\sigma_{k,0}$ and the induced representations are right representations, coming from the transpose of the natural left action.)
The vector $X^k \in \sigma_{k,0}$ is $\bar \rmU^-$-invariants and is an eigenvector of $\bar \rmT$ with eigencharacter $\eta^k$.
\end{proof}

We conclude the appendix with the following result that is claimed in Remark~\ref{R:after the definition of completed homology piece}$(2)$.

\begin{proposition}
\label{P:reducible => ordinary}
Let $\bar\rho = \Matrix{\omega_1^{a+b+1}}{*\neq 0}{0}{\omega_1^b}$ with $1\leq a \leq p-4$ and $0 \leq b \leq p-2$ be as in Definition~\ref{D:primitive type} and let $\widetilde \rmH$ is a primitive $\calO\llbracket K_p\rrbracket$-projective augmented module of type $\bar\rho$. In particular, $\overline \rmH = \widetilde \rmH / (\varpi, \rmI_{1+p\rmM_2(\ZZ_p)})$ is isomorphic to $\Proj_{a,b}$ as a right $\FF[\GL_2(\FF_p)]$-module.
Then $\rmS_2^\Iw(\omega^b\times \omega^{a+b})$ is of rank $1$ over $\calO$, and the $U_p$-action on this space is invertible (over $\calO$); in particular, the $U_p$-slope on this space is zero.
\end{proposition}
\begin{proof} 
By Proposition~\ref{P:dimension of SIw}, $\rmS_2^\Iw(\omega^b\times \omega^{a+b})$ is free of rank $1$ over $\calO$.

Consider three characters $\varepsilon=1\times \omega^a$, $\varepsilon'=\omega^a\times 1$ and $\psi=\omega^a\times\omega^a$. By Corollary~\ref{C:dIw is even} and Proposition~\ref{P:dimension of Sunr}, we have $d_{p+1-a}^\Iw(\psi)=2$ and $d_{p+1-a}^\unr(\psi)=0$. We choose a set $\{\tilde{u}_j=\Matrix 10j1,j=0,\dots, p-1,\tilde{u}_\infty=\Matrix 0110 \}$ of coset representatives of $\Iw_p\setminus \rmK_p$. The formula
\[
\pr(\varphi)(m)=\sum_{j=0,\dots,p-1,\infty}\varphi(m\tilde{u}_j^{-1})|_{\tilde{u}_j},\varphi\in \rmS_{p+1-a}^\Iw(\psi),m\in \tilde{\rmH}
\]
defines a map $\pr:\rmS_{p+1-a}^\Iw(\psi)\rightarrow \rmS_{p+1-a}^\ur(\psi)$. Here $(\cdot)|_{\tilde{u}_j}$ in the above formula is the right action of $\rmK_p$ defined in (\ref{E:right action on Sym}). On the other hand, a straightforward computation gives an equality of operators $\AL_{p+1-a,\psi}\circ \pr-\AL_{p+1-a,\psi}=U_p$ on  $\rmS_{p+1-a}^\Iw(\psi)$. Since $\rmS_{p+1-a}^\ur(\psi)=0$, we actually have $\rmU_p=-\AL_{p+1-a,\psi}$ on $\rmS_{p+1-a}^\Iw(\psi)$. So by Proposition~\ref{P:Atkin-Lehner duality}, the $U_p$-slopes on $\rmS_{p+1-a}^\Iw(\psi)$ are $\frac{p-1-a}{2}$ with multiplicity $2$. By Proposition~\ref{P:theta map}, we have $v_p(c_1^{(\varepsilon')}(w_{p+1-a}))\geq \frac{p-1-a}{2}>1$. Since $c_1^{(\varepsilon')}(w)\in \calO\llbracket w\rrbracket$ and $v_p(w_{p+1-a})\geq 1$, $v_p(w_2)\geq 1$, we have $v_p(c_1^{(\varepsilon')}(w_2))\geq 1$. Since $d_2^\Iw(\varepsilon')=1$, by Proposition~\ref{P:theta map} and Proposition~\ref{P:Atkin-Lehner duality}, we see that the $U_p$-slope on $\rmS_2^\Iw(\varepsilon')$ is exactly $1$. Now by Proposition~\ref{P:Atkin-Lehner duality}, the $U_p$-slope on $\rmS_2^\Iw(\varepsilon)$ is $0$.
\end{proof}

\section{Errata for \cite{liu-wan-xiao}}
We include some errata for \cite{liu-wan-xiao} here.

\subsection{\cite[Proposition~3.4]{liu-wan-xiao}}
This is just a bad convention, the action $||_{\delta_p}^{[\cdot]}$ is a \emph{right} action; however, \cite[Proposition~3.4]{liu-wan-xiao} uses the column vector convention.  Ideally, one should swap the indices of $P_{m,n}(\delta_p)$. This does not create an mathematical error.

\subsection{\cite[\S\,5.4]{liu-wan-xiao}} 
\label{ASS:errata integral model}
In \cite[(5.4.1)]{liu-wan-xiao}, we defined a closed subspace
$$
\Ind_{B^\op(\ZZ_p)}^{\Iw_q}([\cdot]')^{\mathrm{mod}} = \widehat{\bigoplus}_{n \geq 0} T^n \Lambda^{>1/p}\cdot \binom zn
$$
of the space of continuous functions $\calC^0(\ZZ_p; \Lambda^{>1/p})$.
Unfortunately, as stated, this subspace is \emph{not} stable under the $\bfM_1$-action. If one insists on having an $\bfM_1$-action, one has to replace the completed direct sum above with direct product, then the rest of the estimate stays valid, except that the space $S_{\mathrm{int}}^{D, \dagger, 1}$ is not a Banach module over $\Lambda^{>1/p}$ in the usual sense; it is the $\Lambda^{>1/p}$-linear dual of a Banach module over $\Lambda^{>1/p}$. The compactness of $U_p$ needs be interpreted otherwise.

\end{document}